\documentclass[leqno]{amsart}
\usepackage[left=1.2in, right=1.2in, top=1in, bottom=1in]{geometry}
\usepackage{boilerplate}
\usepackage{mathrsfs}
\usepackage{tikz-cd}
\usepackage{mathdots}
\usetikzlibrary{arrows,arrows.meta}
\usepackage[mathscr]{euscript}
\usepackage{algorithm}
\usepackage{algpseudocode}
\usepackage{hyperref}
\hypersetup{
    colorlinks,
    citecolor=black,
    filecolor=black,
    linkcolor=black,
    urlcolor=black
}
\usepackage{enumitem}
\usepackage{subcaption}

\makeatletter
\renewcommand\subsection{\leftskip 0pt\@startsection{subsection}{2}{\z@}%
                                     {-3.25ex\@plus -1ex \@minus -.2ex}%
                                     {1.5ex \@plus .2ex}%
                                     {\normalfont\normalsize\bfseries}}
\renewcommand\subsubsection{\@startsection{subsubsection}{3}{\z@}%
                                     {-3.25ex\@plus -1ex \@minus -.2ex}%
                                     {1.5ex \@plus .2ex}%
                                     {\normalfont\normalsize\bfseries\leftskip 3ex}}
\makeatother

\def\mfp{\underline{\mathbb{F}_p}}

\setlist[enumerate]{label*=\arabic*.}

\title{On the parametrized Tate construction and two theories of real $p$-cyclotomic spectra}

\author{J.D. Quigley}
\address{
Dept. of Mathematics \\
Cornell University \\
Ithaca, NY, U.S.A.
}
\author{Jay Shah}
\address{
Dept. of Mathematics \\
University of Notre Dame \\
Notre Dame, IN, U.S.A.
}

\begin{document}

\tikzcdset{arrow style=tikz, diagrams={>=stealth}}

\begin{abstract} We give a new formula for $p$-typical real topological cyclic homology that refines the fiber sequence formula discovered by Nikolaus and Scholze for $p$-typical topological cyclic homology to one involving genuine $C_2$-spectra. To accomplish this, we give a new definition of the $\infty$-category of real $p$-cyclotomic spectra that replaces the usage of genuinely equivariant dihedral spectra with the parametrized Tate construction $(-)^{t_{C_2} \mu_p}$ associated to the dihedral group $D_{2p} = \mu_p \rtimes C_2$. We then define a $p$-typical and $\infty$-categorical version of H{\o}genhaven's $O(2)$-orthogonal cyclotomic spectra, construct a forgetful functor relating the two theories, and show that this functor restricts to an equivalence between full subcategories of appropriately bounded below objects.
\end{abstract}

\date{\today}
\maketitle

\tableofcontents

\section{Introduction}

The main contribution of this work is to establish a new formula for computing $p$-typical real topological cyclic homology, assuming a certain bounded below hypothesis is satisfied. In fact, inspired by Nikolaus and Scholze's re-imagining of the theory of cyclotomic spectra \cite{NS18}, we give a new definition of the $\infty$-category of real $p$-cyclotomic spectra, and relate this to a $p$-typical and $\infty$-categorical version of H{\o}genhaven's $O(2)$-orthogonal cyclotomic spectra \cite{Hog16}. To contextualize our work, we begin by recalling the theory of $p$-cyclotomic spectra \cite{BHM93} \cite{HM97} \cite{BM16} \cite{BG16} \cite{NS18}. Let $\mu_{p^n} \subset S^1$ be the subgroup of $p^n$th roots of unity, and let $\mu_{p^{\infty}} = \bigcup_n \mu_{p^n}$ be the Pr\"ufer $p$-group.

\begin{dfn}[{\cite[Def.~II.3.1]{NS18}}] \label{dfn:GenuineCycSp} A \emph{genuine $p$-cyclotomic spectrum} is a genuine $\mu_{p^\infty}$-spectrum $X$, together with an equivalence $\Phi^{\mu_p} X \overset{\simeq}{\to} X$ in $\Sp^{\mu_{p^\infty}}$.\footnote{Here and throughout this paper, we implicitly use the $p$th power isomorphism $\mu_{p^{\infty}}/\mu_p \xto{\cong} \mu_{p^{\infty}}$.} The $\infty$-category of \emph{genuine $p$-cyclotomic spectra} is then the equalizer
\[
\begin{tikzcd}
\CycSp^{\mr{gen}}_p = \Eq(\Sp^{\mu_{p^\infty}} \arrow[r,shift left,"\Phi^{\mu_p}"] \arrow[r,shift right,swap,"id"] & \Sp^{\mu_{p^\infty}} ).
\end{tikzcd}
\]
\end{dfn}

We refer the reader to \cite[Thm.~II.3.7]{NS18} for a comparison of Def.~\ref{dfn:GenuineCycSp} with Hesselholt and Madsen's definition of cyclotomic spectra \cite[Def.~1.2]{HM97} and Blumberg and Mandell's definition \cite{BM16}. 

If $A$ is an $E_1$-ring spectrum, then its \emph{topological Hochschild homology} $\THH(A)$ obtains the structure of a genuine $p$-cyclotomic spectrum. This structure induces maps
$$R,F : \THH(A)^{\mu_{p^{n}}} \to \THH(A)^{\mu_{p^{n-1}}}$$
which may be used to define \emph{$p$-typical topological cyclic homology}
$$\TC^{\mr{gen}}(A,p) = \lim_{R,F} \THH(A)^{\mu_{p^n}}.$$
$\TC^{\mr{gen}}(A,p)$ is the receptacle of a \emph{trace map} from $p$-typical algebraic K-theory $\mr{K}(A,p)$. Moreover, by work of Dundas, Goodwillie, and McCarthy \cite{DGM12}, if $f : A \to B$ is a map of $E_1$-ring spectra such that $\ker \pi_0(f)$ is nilpotent, then the commutative square
\[
\begin{tikzcd}
\mr{K}(A,p) \arrow{d} \arrow{r} & \mr{K}(B,p) \arrow{d} \\
\TC^{\mr{gen}}(A,p) \arrow{r} & \TC^{\mr{gen}}(B,p)
\end{tikzcd}
\]
is homotopy cartesian. Cyclotomic structures and the attendant mechanism of the trace thereby furnish a powerful methodology for accessing the $p$-typical algebraic K-theory of ring spectra, which has spawned numerous computations -- for example, see \cite{BHM93} \cite{HM97} \cite{HM03} \cite{Rog00} \cite{AR02} \cite{Rog03}. 

Even though topological cyclic homology is a more computable theory than algebraic K-theory, the complexity of its definition is nonetheless formidable. In \cite{NS18}, Nikolaus and Scholze give a new and simpler definition of cyclotomic spectra that replaces the use of genuinely equivariant spectra with Borel equivariant spectra and the $\mu_p$-Tate construction. Let $\Sp^{h\mu_{p^\infty}} = \Fun(B \mu_{p^\infty} , \Sp)$ be the $\infty$-category of spectra with $\mu_{p^{\infty}}$-action.

\begin{dfn}[{\cite[Def.~II.1.1]{NS18}}] \label{dfn:NaiveCycSp} A \emph{$p$-cyclotomic spectrum} is a spectrum $X$ with $\mu_{p^\infty}$-action, together with a $\mu_{p^\infty}$-equivariant map $\varphi : X \to X^{t\mu_p}$. The $\infty$-category of \emph{$p$-cyclotomic spectra} is then the lax equalizer
\[
\begin{tikzcd}
\CycSp_p = \LEq (\Sp^{h\mu_{p^\infty}} \arrow[r,shift left,"id"] \arrow[r,shift right,swap,"(-)^{t\mu_p}"] & \Sp^{h\mu_{p^\infty}} ).
\end{tikzcd}
\]
\end{dfn}

Def.~\ref{dfn:NaiveCycSp} gives rise to a new definition $\TC(-,p)$ of $p$-typical topological cyclic homology \cite[Def.~II.1.8]{NS18}, which is computed by the fiber sequence \cite[Prop.~II.1.9]{NS18}
\[ \begin{tikzcd}
\TC(X,p) \simeq \fib(X^{h\mu_{p^\infty}} \arrow{rr}{\varphi^{h\mu_{p^\infty}}-\can} &  & (X^{t\mu_p})^{h\mu_{p^\infty}}),
\end{tikzcd} \]
and is used in \cite[\S IV]{NS18} to recover and expand on several fundamental results in trace methods.

The structure map $\Phi^{\mu_p} X \to X^{t\mu_p}$ induces a `forgetful' functor $$\sU : \CycSp_p^{\mr{gen}} \to \CycSp_p.$$ Nikolaus and Scholze proved the following remarkable result:

\begin{thm}[{\cite[Thm.~II.4.10]{NS18} and \cite[Thm.~II.6.3]{NS18}}] \label{thm:MainTheoremNikolausScholze} If $X$ is a genuine $p$-cyclotomic spectrum whose underlying non-equivariant spectrum is bounded below, then there is a canonical equivalence $$\TC^{\mr{gen}}(X,p) \simeq \TC(\sU(X),p).$$ More generally, $\sU$ restricts to an equivalence between the full subcategories of bounded below objects.
\end{thm}

In particular, the \emph{a priori} more intricate data of a genuine $p$-cyclotomic spectrum can be extracted from the data of a $p$-cyclotomic spectrum in the bounded below case.

\subsection{Real \texorpdfstring{$p$}{p}-cyclotomic spectra}

We now turn to the theory of real $p$-cyclotomic spectra \cite{HM13} \cite{Hog16} \cite{DMPR17}. To motivate this discussion, we begin by recalling some facts about real algebraic K-theory.

Let $\mathcal{C}$ be an exact category with weak equivalences equipped with a duality structure $(D,\eta)$. Hesselholt and Madsen defined \emph{real algebraic K-theory} $\mr{KR}(\mathcal{C},D,\eta)$ in \cite{HM13}. This is a genuine $C_2$-spectrum whose underlying spectrum $\mr{KR}(\mathcal{C},D,\eta)^e$ is equivalent to the algebraic K-theory of $\mathcal{C}$. Its categorical $C_2$-fixed points $\mr{KR}(\mathcal{C},D,\eta)^{C_2}$ recover Schlichting's higher Grothendieck-Witt groups (or hermitian algebraic K-groups) $\mr{GW}(\mathcal{C},D,\eta)$, and its geometric $C_2$-fixed points $\Phi^{C_2}\mr{KR}(\mathcal{C},D,\eta)$ are `genuine L-theory' \cite{Dotto2019}, which agrees rationally with quadratic L-theory. 

\begin{rem} Let us mention a few ways in which these fixed points arise in mathematics. The Grothendieck-Witt groups $\mr{GW}_i(k)$ of a field form the target of the degree map in $\mathbb{A}^1$-homotopy theory \cite{Mor12} that has be used to extend results in enumerative geometry to more general base fields -- see \cite{WW19} for a survey of results. The hermitian K-groups of a nice scheme $X$ participate in an exact sequence relating Milnor K-groups and motivic stable homotopy groups as predicted by Morel's $\pi_1$-conjecture \cite{RSO19}. The motivic slice spectral sequence for hermitian K-theory was used by R{\"o}ndigs and {\O}stv{\ae}r to reprove Milnor's conjecture on quadratic forms \cite{RO16}. Finally, genuine L-theory plays a central role in Dotto and Ogle's approach to the Novikov conjecture \cite{Dotto2019}.
\end{rem}

\begin{rem} Forthcoming work of Calm\`{e}s et al. \cite{CDH+19} constructs real algebraic K-theory for the more general input of a stable $\infty$-category equipped with a Poincar\'e structure.
\end{rem}

Real algebraic K-theory is expected to be computable via trace methods. Let $A$ be an $E_\sigma$-algebra for $\sigma$ the $C_2$-sign representation. In \cite{Dotto2019}, Dotto and Ogle constructed a trace map
$$\mr{KR}(A) \to \THR(A)$$
where $\THR(A)$ is \emph{real topological Hochschild homology} \cite{Dot12}. $\THR(A)$ is the motivating example of an \emph{$O(2)$-cyclotomic spectrum} \cite[Def.~2.6]{Hog16}. In this paper, we will restrict ourselves to a $p$-typical version of this notion by instead considering the dihedral Pr\"ufer group $D_{2p^{\infty}} = \mu_{p^{\infty}} \rtimes C_2 \subset O(2)$.

\begin{dfn}[See Def.~\ref{dfn:GenRealCycSp}] A \emph{genuine real $p$-cyclotomic spectrum} is a genuine $D_{2p^\infty}$-spectrum $X$, together with an equivalence $\Phi^{\mu_p} X \overset{\simeq}{\to} X$ in $\Sp^{D_{2p^\infty}}$. The $\infty$-category of \emph{genuine real $p$-cyclotomic spectra} is then the equalizer
\[ \begin{tikzcd}
\RCycSp^{\mr{gen}}_p = \Eq ( \Sp^{D_{2p^\infty}} \arrow[r,shift left,"\Phi^{\mu_p}"] \arrow[r,shift right,swap,"\id"] & \Sp^{D_{2p^\infty}} ).
\end{tikzcd} \]
\end{dfn}

A genuine real $p$-cyclotomic structure induces maps of genuine $C_2$-spectra\footnote{See \S\ref{section:EquivariantConventions} for our conventions and notation regarding fixed point functors.}
$$R, F : \Psi^{\mu_{p^{n}}}X \to \Psi^{\mu_{p^{n-1}}}X$$
which may be used to define \emph{$p$-typical real topological cyclic homology}
$$\TCR^{\mr{gen}}(X,p) = \lim_{R, F} \Psi^{\mu_{p^n}}X.$$

It is believed that $\TCR^{\mr{gen}}$ is a good approximation to $\mr{KR}$ in the same way that $\TC$ is a good approximation to $\mr{K}$ via the Dundas-Goodwillie-McCarthy theorem. Computations of $\TCR^{\mr{gen}}$ have been made by H{\o}genhaven \cite{Hog16} and in forthcoming work of Dotto-Moi-Patchkoria.

As with the passage from Def.~\ref{dfn:GenuineCycSp} to Def.~\ref{dfn:NaiveCycSp}, we wish to reformulate Def.~\ref{dfn:GenRealCycSp} by discarding most of the genuinely equivariant structure. However, we cannot discard all genuineness, as $\TCR^{\mr{gen}}(-,p)$ is valued in genuine $C_2$-spectra. Therefore, in order to accomplish this, we need a \emph{parametrized} version of Borel equivariant homotopy and the Tate construction. To explain, we first recall two distinct perspectives on the usual Tate construction $X^{tG}$ for $G$ a finite group:

\begin{enumerate}
\item $X^{tG}$ is the cofiber of an additive \emph{norm} map $\Nm: X_{h G} \to X^{h G}$, which is a homotopical version of the homomorphism $\overline{x} \mapsto \Sigma_{g \in G} g x$ from $G$-coinvariants to $G$-invariants.
\item Let $\Sp^{h G} = \Fun(BG, \Sp)$. The functor $(-)^{t G}: \Sp^{h G} \to \Sp$ is the composite
\[ \begin{tikzcd}[row sep=4ex, column sep=6ex, text height=1.5ex, text depth=0.25ex]
\Sp^{h G} \ar[hookrightarrow]{r}{j_{\ast}} & \Sp^G \ar{r}{- \wedge \widetilde{E G}} & \Sp^G \ar{r}{\Psi^G} & \Sp
\end{tikzcd} \]
where $j_{\ast}$ is the embedding of spectra with $G$-action as Borel complete objects, and
\[ EG_+ \to S^0 \to \widetilde{EG} \]
is the cofiber sequence of pointed $G$-spaces as in Constr.~\ref{cnstr:GspaceFromGfamily} for the trivial $G$-family. If $X$ is the underlying spectrum of a genuine $G$-spectrum $Y$, then we obtain the formula
$$X^{t G} \simeq (F(E G_+,Y) \wedge \widetilde{E G})^G.$$
\end{enumerate}

These two perspectives are connected by the \emph{Adams isomorphism}: there is a canonical equivalence
$$ X_{h G} \simeq (Y \wedge EG_+)^G. $$

More generally, for an extension $\psi=[N \to G \to G/N]$, we can consider the \emph{$N$-free} $G$-family $\Gamma_N$ and seek to identify $\Gamma_N$-complete objects in terms of $G/N$-spectra with `$\psi$-twisted' $N$-action. To make rigorous sense of this notion, we use the formalism of $G$-$\infty$-categories as developed by Barwick, Dotto, Glasman, Nardin, and the second author \cite{Exp0}.

\begin{dfn}[See Def.~\ref{dfn:BorelGSpectraRelativeToNormalSubgroup}] A \emph{$G/N$-spectrum $X$ with $\psi$-twisted $N$-action} is a $G/N$-functor
\[ X: B^{\psi}_{G/N} N \to \ul{\Sp}^{G/N} \]
where $B^{\psi}_{G/N} N$ is the $G/N$-space of $N$-free $G$-orbits (Def.~\ref{dfn:TwistedClassifyingSpace}) and $\ul{\Sp}^{G/N}$ is the $G/N$-$\infty$-category of $G/N$-spectra (Def.~\ref{dfn:GCategoryGSpectra}).
\end{dfn}

By Prop.~\ref{prp:BorelSpectraAsCompleteObjects}, the $\infty$-category $\Fun_{G/N}(B^{\psi}_{G/N} N, \ul{\Sp}^{G/N})$ canonically embeds as $\Gamma_N$-complete objects in $\Sp^G$, so we can make the following definition.

\begin{dfn}[See Rmk.~\ref{rem:PointSetModels}] \label{dfn:PointSetModelDef} The \emph{parametrized Tate construction} $(-)^{t[\psi]}$ is the composite
\[ \begin{tikzcd}[row sep=4ex, column sep=4ex, text height=1.5ex, text depth=0.25ex]
\Fun_{G/N}(B^{\psi}_{G/N} N, \ul{\Sp}^{G/N}) \ar[hookrightarrow]{r}{j_{\ast}} & \Sp^G \ar{rr}{- \wedge \widetilde{E \Gamma_N}} &  & \Sp^G \ar{r}{\Psi^N} & \Sp^{G/N}.
\end{tikzcd} \]
\end{dfn}

In fact, in \S\ref{section:NormMaps} we generalize the Hopkins-Lurie ambidexterity theory for local systems \cite[\S4.3]{hopkins2013ambidexterity} in order to construct \emph{parametrized norm maps}
\[ \Nm: X_{h[\psi]} \to X^{h[\psi]}, \]
where $X_{h[\psi]}$ is the $G/N$-colimit of $X$, i.e., the \emph{parametrized orbits}, and $X^{h[\psi]}$ is the $G/N$-limit of $X$, i.e., the \emph{parametrized fixed points}. We then define the parametrized Tate construction $X^{t[\psi]}$ to be the cofiber of $\Nm$ (Def.~\ref{dfn:ParamTateCnstr}), and prove the equivalence of Def.~\ref{dfn:ParamTateCnstr} with Def.~\ref{dfn:PointSetModelDef} by way of the Adams isomorphism for the normal subgroup $N$ of $G$ (Prop.~\ref{prp:EquivalentTateConstructions}).

\begin{rem} The parametrized Tate construction as written in Def.~\ref{dfn:PointSetModelDef} is a special case of the generalized Tate construction associated to an arbitrary $G$-family that was studied by Greenlees and May \cite[\S 17]{GreenleesMay}. However, the identification of $\Gamma_N$-complete spectra as a parametrized functor $\infty$-category, and the resulting connection with parametrized norm maps, appears to be new.
\end{rem}

For the semidirect product extension that defines $D_{2p^n}$, let us instead write $B^t_{C_2} \mu_{p^n}$ and $$X_{h_{C_2} \mu_{p^n}} \to X^{h_{C_2} \mu_{p^n}} \to X^{t_{C_2} \mu_{p^n}}$$ for $C_2$-spectra $X$ with twisted $\mu_{p^n}$-action. Extending the above picture to $D_{2p^{\infty}}$ (c.f. Setup~\ref{setup:Dihedral}), we may now formulate the main definition of this work.

\begin{dfn}[See Def.~\ref{dfn:RealCycSp}] A \emph{real $p$-cyclotomic spectrum} is a $C_2$-spectrum $X$ with a twisted $\mu_{p^{\infty}}$-action, together with a twisted $\mu_{p^{\infty}}$-equivariant map $\varphi: X \to X^{t_{C_2} \mu_p}$. The $\infty$-category of \emph{real $p$-cyclotomic spectra} is then the lax equalizer
\[ \begin{tikzcd}
\RCycSp_p := \LEq (\Fun_{C_2}(B^t_{C_2} \mu_{p^{\infty}}, \ul{\Sp}^{C_2}) \arrow[rr,shift left,"id"] \arrow[rr,shift right,swap,"(-)^{t_{C_2}\mu_p}"] & & \Fun_{C_2}(B^t_{C_2} \mu_{p^{\infty}}, \ul{\Sp}^{C_2}) ).
\end{tikzcd} \]
\end{dfn}

This gives rise to a new definition of $p$-typical real topological cyclic homology (Def.~\ref{dfn:newTCR}), which is computed by the fiber sequence (Prop.~\ref{prp:TCRfiberSequence})
\[ \begin{tikzcd}
\TCR(X,p) \simeq \fib(X^{h_{C_2}\mu_{p^\infty}} \arrow{rrr}{\varphi^{h_{C_2}\mu_{p^{\infty}}}-\can} & & & (X^{t_{C_2}\mu_p})^{h_{C_2}\mu_{p^\infty}}).
\end{tikzcd} \]

As with $p$-cyclotomic spectra, there is a forgetful functor $$\sU_{\RR} : \RCycSp_p^{\mr{gen}} \to \RCycSp_p.$$

Here is the main theorem of this work.

\begin{thm}[See Thm.~\ref{thm:MainTheoremEquivalenceBddBelow} and Cor.~\ref{cor:TCRFormulasEquivalent}] If $X$ is a genuine real $p$-cyclotomic spectrum whose underlying spectrum is bounded below,\footnote{We emphasize that the bounded below condition concerns the underlying spectrum and not the underlying $C_2$-spectrum.} then there is a canonical equivalence $$\TCR^{\mr{gen}}(X,p) \simeq \TCR(\sU_{\RR}(X),p).$$ More generally, $\sU_{\RR}$ restricts to an equivalence between the full subcategories of underlying bounded below objects.
\end{thm}

We also give a simple application of Cor.~\ref{cor:TCRFormulasEquivalent} to computing the $C_2$-equivariant homotopy groups of $\TCR^{\mr{gen}}(H \mfp,p)$ for $p$ an odd prime (Thm.~\ref{thm:TCROddPrimeComputation}), deferring more sophisticated computations (in particular, $\TCR^{\mr{gen}}(H \ul{\FF_2},2)$) to a future work.

\begin{wrn} In contrast to \cite{NS18}, we do not construct the real topological Hochschild homology of an $E_{\sigma}$-algebra as a real $p$-cyclotomic spectrum in our sense, deferring such a construction to a future work. Nonetheless, the `decategorified' version of our main theorem in the form of the fiber sequence formula for $\TCR^{\mr{gen}}(-,p)$ (Cor.~\ref{cor:decategorifiedEasy}) already suffices for making calculations involving homotopy groups.
\end{wrn}

\subsection{Methods of proof}

The key computational input for Nikolaus and Scholze's proof of Thm.~\ref{thm:MainTheoremNikolausScholze} is the \emph{Tate orbit lemma}:

\begin{lem}[{\cite[Lem.~I.2.1]{NS18}}] Let $X$ be a spectrum with $\mu_{p^2}$-action that is bounded below. Then
$$(X_{h\mu_p})^{t(\mu_{p^2}/\mu_p)} \simeq 0.$$
\end{lem}

Correspondingly, in order to prove Thm.~\ref{thm:MainTheoremEquivalenceBddBelow}, we will need the \emph{dihedral Tate orbit lemma}:

\begin{lem}[See Lem.~\ref{lem:dihedralTOLEven} and Lem.~\ref{lem:dihedralTOLOdd}] Let $X$ be a $C_2$-spectrum with twisted $\mu_{p^2}$-action whose underlying spectrum is bounded below. Then
$$(X_{h_{C_2} \mu_p})^{t_{C_2}(\mu_{p^2}/\mu_p)} \simeq 0.$$
\end{lem}

On the other hand, instead of \cite[\S II.5]{NS18}, the key categorical input for us will be a theorem of Ayala, Mazel-Gee, and Rozenblyum \cite[Thm.~A]{AMGR-NaiveApproach}, extending work of Glasman \cite{Glasman17}, that reconstructs $\Sp^G$ from the $\infty$-categories $\Fun(B W_G H, \Sp)$ ranging over subgroups $H \leq G$, along with the data of `generalized Tate' functors interpolating between them.\footnote{The left-lax functoriality of the generalized Tate construction makes a precise statement slightly complicated to state: see Def.~\ref{dfn:GeometricLocus}.} To make effective use of this theorem in our context, we will reprove it in a slightly different form (Thm.~\ref{thm:GeometricFixedPointsDescriptionOfGSpectra}), which also makes no explicit use of $(\infty,2)$-category theory (in contrast to the proof in \cite{AMGR-NaiveApproach}).\footnote{However, we note that our proof only pertains to the situation where $G$ is a finite group, whereas \cite{AMGR-NaiveApproach} consider the more general case of a compact Lie group.} For this, and also more generally, we will need an elaborate understanding of the formalism of \emph{recollements} \cite[\S A.8]{HA} \cite{BarwickGlasmanNoteRecoll}, which plays a fundamental role in equivariant stable homotopy theory via the recollement on $\Sp^G$ defined by a $G$-family \cite[Part~IV]{GreenleesMay}. Given the dihedral Tate orbit lemma and a variant of Thm.~\ref{thm:GeometricFixedPointsDescriptionOfGSpectra} (Var.~\ref{vrn:ReconstructionEquivalence}), we may then deduce Thm.~\ref{thm:MainTheoremEquivalenceBddBelow} from generic theorems regarding recollements (for instance, compare Thm.~\ref{thm:OneGenerationAndExtension} and Prop.~\ref{prp:EquivalenceOnBoundedBelowAtFiniteLevel}).

\begin{rem} For greater logical clarity, we have separated out our study of recollements so as to comprise the first two sections of this paper. We encourage the reader primarily interested in real $p$-cyclotomic spectra to begin with \S\ref{section:FirstEquivariantSection} and refer to these sections as needed.
\end{rem}

\subsection{Conventions and notation}

We now state a few miscellaneous conventions that are used throughout the paper. Our conventions on equivariant stable homotopy theory are indicated in \S\ref{section:EquivariantConventions}, and our terminology concerning parametrized $\infty$-categories is recalled in \S\ref{section:ParamTerminology}.

\begin{itemize}[leftmargin=4ex] \item Throughout, we work with the formalism of $\infty$-categories. For us, an \emph{$\infty$-category} is a quasi-category, i.e., a simplicial set that satisfies the inner horn filling condition. We implicitly identify categories as $\infty$-categories via the nerve construction.
\item For an $\infty$-category $C$, we let $\sO(C) = \Fun(\Delta^1,C)$ denote the $\infty$-category of arrows in $C$.
\item For a simplicial set $S$, let $S^{\sharp}$ be the marked simplicial set with all edges marked, and let $S^{\flat}$ be the marked simplicial set with only the degenerate edges marked. For a (locally) cocartesian fibration $p: C \to S$, let $\leftnat{C}$ be the marked simplicial set with the (locally) $p$-cocartesian edges marked, and let $\leftnat{\Lambda^n_0}$, $\leftnat{\Delta^n}$ indicate that the edge $\{0,1\}$ is marked. Dually, we may consider $\rightnat{C}$ if $p$ is a cartesian fibration, and $\rightnat{\Lambda^n_n}$, $\rightnat{\Delta^n}$ with the edge $\{n-1,n\}$ marked.
\item Constructions made internal to an $\infty$-category, such as limits and colimits, are necessarily homotopy invariant, so we will typically suppress the adjective `homotopy' in our discussion. We also suppress routine arguments that concern the homotopy invariance of constructions involving $\infty$-categories that are made in simplicial sets or marked simplicial sets.
\item For an $\infty$-category $C$, we denote its mapping spaces by $\Map_C(-,-)$ or $\Map(-,-)$ if $C$ is understood. Likewise, if $C$ is a stable $\infty$-category, then we denote its mapping spectra by $\map_C(-,-)$ or $\map(-,-)$.
\item Let $C = (C, \otimes, 1)$ be a symmetric monoidal $\infty$-category. If $C$ is stable, then we require that the tensor product on $C$ is exact separately in each variable, and if $C$ is presentable, then we require that the tensor product on $C$ commutes with colimits separately in each variable.
\item If $C$ is a closed symmetric monoidal $\infty$-category (e.g., $C$ is presentably symmetric monoidal), then we typically denote its internal hom by $F_C(-,-)$ or just $F(-,-)$.
\item Let $F: C \to D$ be a functor between two symmetric monoidal $\infty$-categories. Then we say that $F$ is \emph{lax monoidal} if $F$ lifts to the structure of a functor $F^{\otimes}: C^{\otimes} \to D^{\otimes}$ of $\infty$-operads (so $F^{\otimes}$ is a functor over $\Fin_{\ast}$ that preserves inert edges). If $F^{\otimes}$ moreover preserves all cocartesian edges, then we say that $F$ is \emph{symmetric monoidal} or simply \emph{monoidal}.
\item Let $C$ and $D$ be symmetric monoidal $\infty$-categories and let $\adjunct{L}{C}{D}{R}$ be an adjunction. Then $L \dashv R$ is \emph{monoidal} if $L$ is symmetric monoidal (so then $R$ is necessarily lax monoidal).
\item In contrast to the introduction, we will typically denote the smash product of $G$-spectra (and related $\infty$-categories) by the symbol $\otimes$ instead of $\wedge$.
\item For a finite group $G$, we let $\FF_G$ be the category of finite $G$-sets and $\sO_G \subset \FF_G$ be the full subcategory on the nonempty transitive $G$-sets.
\end{itemize}

\begin{rem} The full subcategory of $\sO_G$ spanned by the orbits $\{G/H : H \leq G \}$ constitutes a skeleton of $\sO_G$, where given a finite nonempty transitive $G$-set $U$, a choice of basepoint $b \in U$ specifies an isomorphism $U \cong G/H$ with $H = \{ h \in G : h \cdot b = b \}$ and $b \mapsto 1 H$. To avoid some basepoint technicalities, we opt for the basepoint-free definition of $\sO_G$. Note that we may always pass to a skeleton of $\sO_G$ when checking \emph{conditions} that involve $\sO_G$ in some way -- e.g., to check if a natural transformation of presheaves on $\sO_G$ is an equivalence, it suffices to check on orbits.
\end{rem}

\subsection{Acknowledgments}

The authors thank Mark Behrens, Andrew Blumberg, Emanuele Dotto, Jeremy Hahn, Kristian Moi, Irakli Patchkoria, Dylan Wilson, and Mingcong Zeng for helpful discussions. The authors were partially supported by NSF grant DMS-1547292.

\section{Recollements}

In this section, we establish the basic theory of recollements, expanding upon \cite[\S A.8]{HA} and \cite{BarwickGlasmanNoteRecoll}. After setting up the definitions and summarizing Lurie's results on recollements, we give a treatment of the monoidal structure on a recollement,\footnote{Although our results on monoidal recollements are presumably well-known, we do not know of an alternative reference.} connect the theory of stable monoidal recollements to that of smashing localizations, and record some useful projection formulas. We conclude by proving a few necessary lemmas concerning families of recollements.

\begin{dfn} \label{dfn:recollement} Let $\cX$ be an $\infty$-category that admits finite limits and let $\cU, \cZ \subset \cX$ be full subcategories that are stable under equivalences. Then ($\cU$, $\cZ$) is a \emph{recollement} of $\cX$ if the inclusion functors $j_{\ast}: \cU \subset \cX$ and $i_{\ast}: \cZ \subset \cX$ admit left exact left adjoints $j^{\ast}$ and $i^{\ast}$ such that
\begin{enumerate} \item $j^{\ast} i_{\ast}$ is equivalent to the constant functor at the terminal object $0$ of $\cU$.
\item $j^{\ast}$ and $i^{\ast}$ are jointly conservative, i.e., if $f: x \to y$ is a morphism in $\cX$ such that $j^{\ast} f$ and $i^{\ast} f$ are equivalences, then $f$ is an equivalence.
\end{enumerate}
We will call $\cU$ the \emph{open} part of the recollement, $\cZ$ the \emph{closed} part of the recollement, and $i^{\ast} j_{\ast}$ the \emph{gluing functor}.\footnote{Our convention on which subcategory is open and which is closed matches that for constructible sheaves, whereas other authors (e.g., \cite{BarwickGlasmanNoteRecoll}) use the opposite convention, which matches that for quasi-coherent sheaves. Our convention is also consistent with viewing sieves as closed subsets and cosieves as open subsets of a poset, and thus seems more appropriate for applications in equivariant homotopy theory -- for instance, see Def.~\ref{dfn:GeometricLocus}.} \footnote{In \cite[Def.~A.8.1]{HA}, Lurie calls the open part $C_1$ and the closed part $C_0$.}

Now suppose that $(\cU_1, \cZ_1)$ and $(\cU_2, \cZ_2)$ are recollements on $\cX_1$ and $\cX_2$. Then a functor $F: \cX_1 \to \cX_2$ is a \emph{morphism of recollements} if $F$ sends $j^{\ast}_1$-equivalences to $j^{\ast}_2$-equivalences and $i^{\ast}_1$-equivalences to $i^{\ast}_2$-equivalences. Let $\Recoll$ denote the resulting $\infty$-category of recollements, and let $\Recoll^{\lex}$ be the full subcategory on those morphisms of recollements that are also left-exact.
\end{dfn}

\begin{nul}[\textbf{Fracture square}] \label{recollementFractureSquare} Let $(\cU, \cZ)$ be a recollement of $\cX$ and let $\eta_j: \id \to j_{\ast} j^{\ast}$, $\eta_i: \id \to i_{\ast} i^{\ast}$ denote the unit transformations. Then we have a pullback square of functors
\[ \begin{tikzcd}[row sep=4ex, column sep=6ex, text height=1.5ex, text depth=0.25ex]
\id \ar{r}{\eta_i} \ar{d}[swap]{\eta_j} & i_{\ast} i^{\ast} \ar{d}{i_{\ast} i^{\ast} \eta_j} \\
j_{\ast} j^{\ast} \ar{r}{\eta_i j_{\ast} j^{\ast}} & i_{\ast} i^{\ast} j_{\ast} j^{\ast}.
\end{tikzcd} \]
\end{nul}

\begin{nul} \label{LaxVsStrictMorphismsOfRecollements} Suppose that $F:\cX_1 \to \cX_2$ is a morphism of recollements $(\cU_1, \cZ_1) \to (\cU_2, \cZ_2)$. Then we may define $F_U = j_2^{\ast} F {j_1}_{\ast}: \cU_1 \to \cU_2$ and $F_Z = i_2^{\ast} F {i_1}_{\ast}$ so that we have a commutative diagram
\[ \begin{tikzcd}[row sep=4ex, column sep=4ex, text height=1.5ex, text depth=0.25ex]
\cU_1 \ar{d}{F_U} & \cX_1 \ar{d}{F} \ar{l}[swap]{j_1^{\ast}} \ar{r}{i_2^{\ast}} & \cZ_1 \ar{d}{F_Z} \\
\cU_2 & \cX_2 \ar{l}[swap]{j_2^{\ast}} \ar{r}{i_2^{\ast}} & \cZ_2,
\end{tikzcd} \]
such that $F$ is left-exact if and only if $F_U$ and $F_Z$ are left-exact. By adjunction, we get natural transformations $\nu: F {j_1}_{\ast} \to {j_2}_{\ast} F_U$ and $\nu': F {i_1}_{\ast} \to {i_2}_{\ast} F_Z$. If $F$ preserves the terminal object, then $\nu'$ is an equivalence -- indeed, for all $z \in \cZ_1$ we then have
\[ {j_2}^{\ast} F {i_1}_{\ast}(z) \simeq F_U {j_1}^{\ast} {i_1}_{\ast}(z) \simeq F_U (0) \simeq 0, \]
so the unit map $F {i_1}_{\ast} (z) \to {i_2}_{\ast} i_2^{\ast} F {i_1}_{\ast} (z) = {i_2}_{\ast} F_Z (z)$ is an equivalence. In particular, if $F$ is left exact, then $\nu'$ is an equivalence \cite[Rmk.~A.8.10]{HA}. On the other hand, $\nu$ is an equivalence if and only if $$\nu'': F_Z {i_1}^{\ast} {j_1}_{\ast} \to i_2^{\ast} {j_2}_{\ast} F_U$$ is an equivalence -- indeed, the `only if' direction is obvious, and for the `if' direction we may readily check that ${j_2}^{\ast} \nu$ and ${i_2}^{\ast} \nu$ are equivalences and then invoke the joint conservativity of ${j_2}^{\ast}$ and ${i_2}^{\ast}$.
\end{nul}

\begin{dfn} If $\nu''$ in \ref{LaxVsStrictMorphismsOfRecollements} is an equivalence, then we call $F$ a \emph{strict} morphism of recollements. Let $\Recoll_0 \subset \Recoll$ and $\Recoll^{\lex}_0 \subset \Recoll^{\lex}$ be the wide subcategories on the strict morphisms.
\end{dfn}

\begin{rem} \label{rem:TwoOutOfThreePropertyEquivalencesStrictMorphismRecoll} If $F: \cX_1 \to \cX_2$ is a strict left-exact morphism of recollements, then $F$ is an equivalence if and only if $F_U$ and $F_Z$ are equivalences \cite[Prop.~A.8.14]{HA}.
\end{rem}

\begin{dfn} Let $\pi: \cM \to \Delta^1$ be a functor of $\infty$-categories with fibers $\cM_0 = \cZ$ and $\cM_1 = \cU$. Then $\pi$ is a \emph{left-exact correspondence} \cite[Def.~A.8.6]{HA} if
\begin{enumerate} \item $\pi$ is a cartesian fibration, so determines a functor $\phi: \cU \to \cZ$.
\item $\phi$ is left-exact, i.e., the $\infty$-categories $\cU$ and $\cZ$ admit finite limits and $\phi$ preserves finite limits.
\end{enumerate}
A morphism of left-exact correspondences is a functor $F: \cM_1 \to \cM_2$ over $\Delta^1$. In terms of the left-exact functors $\phi_1$ and $\phi_2$, this corresponds to a right-lax commutative diagram
\[ \begin{tikzcd}[row sep=4ex, column sep=4ex, text height=1.5ex, text depth=0.25ex]
\cU_1 \ar{r}{\phi_1} \ar{d}[swap]{F_U} \ar[phantom]{rd}{\SWarrow} & \cZ_1 \ar{d}{F_Z} \\
\cU_2 \ar{r}[swap]{\phi_2} & \cZ_2.
\end{tikzcd} \]
Let $\sO^{\rlax}_{\lex}(\Cat_{\infty})$ denote the resulting $\infty$-category of left-exact correspondences as a full subcategory of $(\Cat_{\infty})_{/\Delta^1}$, and let $\sO_{\lex}(\Cat_{\infty})$ be the wide subcategory on those morphisms that preserve cartesian edges, so that the right-lax commutativity is actually strict. Note that under the straightening correspondence, $\sO_{\lex}(\Cat_{\infty})$ is the full subcategory of $\sO(\Cat_{\infty})$ on left-exact functors $\phi: \cU \to \cZ$.

If $F_U$ and $F_Z$ are also left-exact, we say that the morphism $F$ of left-exact correspondences is \emph{left-exact}. We may then view (lax) commutative squares as residing inside $\Cat_{\infty}^{\lex}$ itself. Let $\sO^{\rlax}(\Cat^{\lex}_{\infty}) \subset \sO^{\rlax}_{\lex}(\Cat_{\infty})$ and $\sO(\Cat^{\lex}_{\infty}) \subset \sO_{\lex}(\Cat_{\infty})$ denote the resulting wide subcategories.
\end{dfn}

\begin{nul} \label{recollEquivalenceToOplaxLim} Let $\cM \to \Delta^1$ be a left-exact correspondence and let $\cX = \Fun_{/\Delta^1}(\Delta^1, \cM)$ be its $\infty$-category of sections. Let $\cU \subset \cX$ be the full subcategory on the cartesian sections and let $\cZ \subset \cX$ be the full subcategory on those sections $\sigma$ such that $\sigma(1)$ is a terminal object of $\cU$. Then $(\cU, \cZ)$ is a recollement of $\cX$ \cite[Prop.~A.8.7]{HA}. Moreover, the formation of sections
\[ \goesto{\cM}{\Fun_{/\Delta^1}(\Delta^1, \cM)} \]
carries morphisms of left-exact correspondences to morphisms of recollements, and thereby defines a functor\footnote{We denote this by $\Rlax\lim$ in view of the interpretation of the sections of a cartesian fibration as defining the right-lax limit of the corresponding functor.}
\[ \Rlax\lim : \sO^{\rlax}_{\lex}(\Cat_{\infty}) \xto{\simeq} \Recoll, \]
which is an equivalence of $\infty$-categories by \cite[Prop.~A.8.8]{HA} (for full faithfulness) and \cite[Prop.~A.8.11]{HA} (which shows that if $(\cU,\cZ)$ is a recollement of $\cX$, then $\cX$ is equivalent to the lax limit of $i^{\ast} j_{\ast}: \cU \to \cZ$). Furthermore, in view of the discussion in \ref{LaxVsStrictMorphismsOfRecollements}, $\Rlax \lim$ restricts to equivalences of subcategories
\begin{align*} \sO_{\lex}(\Cat_{\infty}) \xto{\simeq} \Recoll_0, \; \sO^{\rlax}(\Cat^{\lex}_{\infty}) \xto{\simeq} \Recoll^{\lex}, \; \sO(\Cat^{\lex}_{\infty}) \xto{\simeq} \Recoll^{\lex}_0.
\end{align*}
\end{nul}


\begin{nul} Let $\pi: \cM \to \Delta^1$ be a cartesian fibration. By the dual of \cite[Lem.~2.22]{Exp2}, we have a trivial fibration $\sO^{\cart}(\cM) \to \sO(\Delta^1) \times_{\ev_1, \Delta^1, \pi} \cM$, which restricts to a trivial fibration $\ev_1: \Fun^{\cart}_{/\Delta^1}(\Delta^1,\cM) \to \cM_1$. Let $\chi$ be a section of $\ev_1$. 

Because $i: \rightnat{\Lambda^2_2} \to \rightnat{\Delta^2}$ is right marked anodyne, with the structure map $\sigma^0: \Delta^2 \to \Delta^1$, $(\sigma^0)^{-1}(0)= \{0, 1\}$ and $(\sigma^0)^{-1}(1)= \{2\}$, we have a trivial fibration
\[ i^\ast: \Fun_{/\Delta^1}(\rightnat{\Delta^2},\rightnat{\cM}) \to \Fun_{/\Delta^1}(\rightnat{\Lambda_2^2},\rightnat{\cM}) \cong \Fun_{/\Delta^1}(\Delta^1,\cM) \times_{\ev_1, X_1, \ev_1} \Fun^{\cart}_{/\Delta^1}(\Delta^1,\cM). \]
Let $\kappa$ be a section of $i^\ast$. The section $\chi$ yields a functor
\[ f = (\id,\chi \circ \ev_1): \Fun_{/\Delta^1}(\Delta^1,\cM) \to \Fun_{/\Delta^1}(\Delta^1,\cM) \times_{X_1} \Fun^{\cart}_{/\Delta^1}(\Delta^1,\cM).\]
Let $g = \kappa \circ f$. Then the various maps fit into the commutative diagram
\[ \begin{tikzcd}[row sep=4ex, column sep=4ex, text height=1.5ex, text depth=0.25ex]
\Fun_{/\Delta^1}(\Delta^1,\cM) \ar{r}{g} \ar{d}{\ev_1} & \Fun_{/\Delta^1}(\rightnat{\Delta^2},\rightnat{\cM}) \ar{r}{\ev_{01}} \ar{d}{\ev_{12}} & \Fun(\Delta^1,\cM_0) \ar{d}{\ev_1} \\
\cM_1 \ar{r}{\chi} &  \Fun^{\cart}_{/\Delta^1}(\Delta^1,\cM) \ar{r}{\ev_0} & \cM_0.
\end{tikzcd} \]
\end{nul}

\begin{lem} \label{lm:sectionsPullbackSquare} The natural map $\Fun_{/\Delta^1}(\Delta^1,\cM) \to \sO(\cM_0) \times_{\cM_0} \cM_1$ is an equivalence, so the outer square is a homotopy pullback square of $\infty$-categories.
\end{lem}
\begin{proof} Because the sections $\chi$ and $\kappa$ are equivalences, the map $g$ is an equivalence. Moreover, because the map $\Lambda^2_1 \to \Delta^2$ is inner anodyne, the rightmost square is a homotopy pullback square. The claim follows.
\end{proof}

\begin{cor} \label{cor:RecollementAsPullbackSquare} Suppose that $(\cU,\cZ)$ is a recollement of $\cX$ and consider the commutative\footnote{We can obtain a commutative diagram of simplicial sets using standard techniques in quasi-category theory.} diagram
\[ \begin{tikzcd}[row sep=4ex, column sep=6ex, text height=1.5ex, text depth=0.25ex]
\cX \ar{r}{i^{\ast} \eta_j} \ar{d}[swap]{j^{\ast}} & \sO(\cZ) \ar{d}{\ev_1} \\
\cU \ar{r}{\phi = i^{\ast} j_{\ast}} & \cZ
\end{tikzcd} \]
where $\eta_j: \cX \to \sO(\cX)$ is the functor that sends $x$ to the unit map $x \to j_{\ast} j^{\ast} x$. Then the induced map
\[ \cX \xto{\simeq} \sO(\cZ) \times_{\ev_1, \cZ, \phi} \cU \]
is an equivalence of $\infty$-categories.
\end{cor}
\begin{proof} Combine Lem.~\ref{lm:sectionsPullbackSquare} with the equivalence $\rlax \lim: \sO^{\rlax}_{\lex}(\Cat_{\infty}) \xto{\simeq} \Recoll$ of \ref{recollEquivalenceToOplaxLim}.
\end{proof}

\begin{rem} In view of Cor.~\ref{cor:RecollementAsPullbackSquare}, given a recollement $(\cU,\cZ)$ of $\cX$ we will sometimes write objects $x \in \cX$ as $[u,z,\alpha:z \to \phi(u)]$ for $\phi = i^{\ast} j_{\ast}$.
\end{rem}

Given a left-exact functor $\phi: \cU \to \cZ$, we may also extract the resulting recollement from the \emph{cocartesian} fibration classified by $\phi$, even though it is difficult to encode the right-lax functoriality when working with cocartesian fibrations.

\begin{nul} Let $S$ be an $\infty$-category and $C \to S$ a cocartesian fibration. Recall from \cite{BGN} or \cite[Rec.~5.15]{Exp2} that the \emph{dual cartesian fibration} $C^\vee \to S^{\op}$ is defined to have $n$-simplices\footnote{Here, $\widetilde{\sO}(-)$ is the \emph{twisted arrow $\infty$-category}. We use the directionality convention of \cite{M1} instead of \cite[\S 5.2.1]{HA}, so twisted arrows are contravariant in the source and covariant in the target.}
\[ \begin{tikzcd}[row sep=4ex, column sep=4ex, text height=1.5ex, text depth=0.25ex]
\leftnat{\widetilde{\sO}((\Delta^n)^{\op})} \ar{r} \ar{d}[swap]{\ev_1}  & \leftnat{C} \ar{d} \\
((\Delta^n)^{\op})^\sharp \ar{r} & S^\sharp.
\end{tikzcd} \]
In fact, because the functor $\widetilde{\sO}'(-): s\Set^+_{/S} \to s\Set^+_{/S}$ of \cite[Prop.~5.16]{Exp2} preserves colimits, it follows that for all simplicial sets $A$ over $S^{\op}$
\[ \Hom_{/S^{\op}}(A,C^{\vee}) \cong \Hom_{/S}(\widetilde{\sO}'(A^{\op}),\leftnat{C}). \]
Consequently, we obtain an equivalence
\[ \Fun_{/S^{\op}}(S^{\op},C^{\vee}) \simeq \Fun^{\cocart}_{/S}(\widetilde{\sO}(S),C). \]
\end{nul}

\begin{nul} \label{dualizingOneSimplex} The barycentric subdivision $\sd(\Delta^1) = [0 \rightarrow 01 \leftarrow 1]$ is isomorphic to the twisted arrow category $\widetilde{\sO}(\Delta^1)$. Therefore, for a cocartesian fibration $C \to \Delta^1$, we deduce that
\[ \Fun^{\cocart}_{/\Delta^1}(\sd(\Delta^1),C) \simeq \Fun_{/\Delta^1}(\Delta^1,C^{\vee}) \]
and hence by Lem.~\ref{lm:sectionsPullbackSquare} we can decompose $\Fun^{\cocart}_{/\Delta^1}(\sd(\Delta^1),C)$ as a pullback square $\sO(\cZ) \times_{\ev_1, \cZ, \phi} \cU$ for a choice of pushforward functor $\phi: \cU \to \cZ$ (where $\cU \simeq C_0$ and $\cZ \simeq C_1$). This observation will be important for us when we discuss recollements on right-lax limits in the sequel.
\end{nul}

\subsection{Stable recollements}

\begin{dfn} Let $\cX$ be a stable $\infty$-category and let ($\cU$, $\cZ$) be a recollement of $\cX$. Then this recollement is \emph{stable} if $\cU$ and $\cZ$ are stable subcategories. Let $\Recoll^{\st}$, resp. $\Recoll^{\st}_0$ be the full subcategory of $\Recoll^{\lex}$, resp. $\Recoll_0^{\lex}$ whose objects are the stable recollements.
\end{dfn}

\begin{dfn} If $\cM \to \Delta^1$ is a left-exact correspondence, then $\cM$ is \emph{exact} if the functor $\phi: \cM_1 \to \cM_0$ is an exact functor of stable $\infty$-categories. Let $\sO^{\rlax}(\Cat^{\st}_{\infty})$, resp. $\sO(\Cat^{\st}_{\infty})$ be the full subcategory of $\sO^{\rlax}(\Cat^{\lex}_{\infty})$, resp. $\sO(\Cat^{\lex}_{\infty})$ on the exact correspondences.
\end{dfn}

\begin{rem} The functor $\Rlax\lim$ of \ref{recollEquivalenceToOplaxLim} restricts to equivalences
\begin{align*} \sO^{\rlax}(\Cat^{\st}_{\infty}) \xto{\simeq} \Recoll^{\st}, \quad \sO(\Cat^{\st}_{\infty}) \xto{\simeq} \Recoll^{\st}_0.
\end{align*}
\end{rem}

\begin{nul} Let $(\cU,\cZ)$ be a stable recollement of $\cX$. Then $j^{\ast}: \cX \to \cU$ admits a fully faithful left adjoint\footnote{For the existence of $j_!$, we only need that $\cZ$ admits an initial object $\emptyset$ \cite[Cor.~A.8.13]{HA}. Then $j_!$ is defined by the formula $j_!(u) = [u,\emptyset, \emptyset \to \phi(u)]$.} $j_!$, $i_{\ast}$ admits a right adjoint $i^!$, and we have norm maps $\Nm: j_! \to j_{\ast}$ and $\Nm': i^! \to i^{\ast}$ that fit into fiber sequences
\begin{align*} j_! \to j_{\ast} \to i_{\ast} i^{\ast} j_{\ast}  \quad \text{and} \quad
i^! \to i^{\ast} \to i^{\ast} j_{\ast} j^{\ast} \:,
\end{align*}
where the other maps are induced by the unit transformations for $j^{\ast} \dashv j_{\ast}$ and $i^{\ast} \dashv i_{\ast}$. On objects $x=[u,z,\alpha] \in \cX$, these amount to the fiber sequences
\begin{align*} [u,0,0] \to [u,\phi u, \id] \to [0,\phi u, 0] \quad \text{and} \quad \fib(\alpha) \to z \to \phi u \: .
\end{align*}
Considering the various unit and counit transformations and the norm maps, we may extend the pullback square of \ref{recollementFractureSquare} to a commutative diagram 
\[ \begin{tikzcd}[row sep=4ex, column sep=4ex, text height=1.5ex, text depth=0.25ex]
 & i_{\ast} i^! \ar{r}{\simeq} \ar{d} & i_{\ast} i^! \ar{d}{i_{\ast} \Nm'} \\
j_! j^{\ast} \ar{r} \ar{d}{\simeq} & \id \ar{r} \ar{d} & i_{\ast} i^{\ast} \ar{d} \\
j_! j^{\ast} \ar{r}[swap]{\Nm j^{\ast}} & j_{\ast} j^{\ast} \ar{r} & i_{\ast} i^{\ast} j_{\ast} j^{\ast}
\end{tikzcd} \]
in which every row and column is a fiber sequence.
\end{nul}

\begin{nul} \label{stableRecollementComment} In the stable case, the datum of the closed part of a recollement determines the entire recollement. More precisely, if $\cZ \subset \cX$ is a stable reflective and coreflective subcategory of $\cX$ and we define $\cU$ to be the full subcategory on those objects $u \in \cX$ such that $\Map_{\cX}(z,u) \simeq \ast$ for all $z \in \cZ$, then ($\cU$, $\cZ$) is a stable recollement of $\cX$ \cite[Prop.~A.8.20]{HA}, and conversely, if $(\cU,\cZ)$ is a stable recollement of $\cX$ then $j_{\ast}: \cU \subset \cX$ is defined as above from $\cZ$. We may also identify $j_!(\cU)$ as given by those objects $u \in \cX$ such that $\Map_{\cX}(u,z) \simeq \ast$ for all $z \in \cZ$. 

Moreover, $F: \cX_1 \to \cX_2$ is a morphism of stable recollements $(\cU_1, \cZ_1) \to (\cU_2, \cZ_2)$ if and only if $F|_{\cZ_1} \subset \cZ_2$ and $F|_{j_!(\cU_1)} \subset j_!(\cU_2)$ (in particular, we then have ${j_2}_! F_U \simeq F {j_1}_!$). This is because $\cZ$ coincides with the $j^{\ast}$-null objects and $j_!(\cU)$ with the $i^{\ast}$-null objects. Given this, $F$ is then a strict morphism of stable recollements if and only if we also have that $F|_{j_{\ast}(\cU_1)} \subset j_{\ast}(\cU_2)$.
\end{nul}


\subsection{Monoidal recollements}

We now extend the theory of recollements to the situation where $\cX$ admits a symmetric monoidal structure $(\cX, \otimes, 1)$.

\begin{dfn} \label{dfn:monoidalRecollement} Let $\cX$ be a symmetric monoidal $\infty$-category that admits finite limits. Then a recollement $(\cU,\cZ)$ of $\cX$ is \emph{monoidal} if the localization functors $j_{\ast} j^{\ast}$ and $i_{\ast} i^{\ast}$ are compatible with the symmetric monoidal structure in the sense of \cite[Def.~2.2.1.6]{HA}, i.e., for every $j^{\ast}$, resp. $i^{\ast}$-equivalence $f: x \to x'$ and any $y \in \cX$, $f \otimes \id: x \otimes y \to x' \otimes y$ is a $j^{\ast}$, resp. $i^{\ast}$-equivalence.

A morphism $F:(\cU, \cZ) \to (\cU',\cZ')$ of recollements on $\cX$ and $\cX'$ is \emph{monoidal} if the functor $F: \cX \to \cX'$ is symmetric monoidal. Let $\Recoll^{\otimes}$ denote the $\infty$-category of monoidal recollements and morphisms thereof.
\end{dfn}

\begin{nul} In the situation of Def.~\ref{dfn:monoidalRecollement}, by \cite[Prop.~2.2.1.9]{HA} $\cU$ and $\cZ$ obtain symmetric monoidal structures such that the adjunctions $j^{\ast} \dashv j_{\ast}$ and $i^{\ast} \dashv i_{\ast}$ are monoidal. In particular, the gluing functor $i^{\ast} j_{\ast}$ is lax monoidal. Furthermore, if $F$ is a morphism of monoidal recollements, then the induced functors $F_U$ and $F_Z$ of \ref{LaxVsStrictMorphismsOfRecollements} are also symmetric monoidal.
\end{nul}

We first show that given a lax monoidal functor $\phi: \cU \to \cZ$, the recollement $\Rlax \lim \phi$ is monoidal. Recall that the arrow $\infty$-category $\sO(C) = C^{\Delta^1}$ admits a pointwise monoidal structure $(C^{\otimes})^{\Delta^1}$ (\ref{pointwiseMonoidalStructure}).

\begin{dfn} \label{dfn:canonicalSMConOplaxLimit} Suppose $\phi^\otimes: \cU^\otimes \to \cZ^\otimes$ is a lax monoidal functor of symmetric monoidal $\infty$-categories (i.e., a map of $\infty$-operads). Consider the pullback square of $\infty$-operads
\[ \begin{tikzcd}[row sep=4ex, column sep=4ex, text height=1.5ex, text depth=0.25ex]
(\cZ^\otimes)^{\Delta^1} \times_{\cZ^\otimes} \cU^\otimes \ar{r} \ar{d} & (\cZ^\otimes)^{\Delta^1} \ar{d}{\ev_1} \\ 
\cU^\otimes \ar{r}{\phi^\otimes} & \cZ^\otimes. 
\end{tikzcd} \]
By Lem.~\ref{lm:evaluationCocartesianMonoidal}, $\ev_1$ is a cocartesian fibration, so $(\cZ^\otimes)^{\Delta^1} \times_{\cZ^\otimes} \cU^\otimes \to \cU^\otimes \to \Fin_{\ast}$ is a cocartesian fibration and therefore a symmetric monoidal $\infty$-category. This defines the \emph{canonical} symmetric monoidal structure on the right-lax limit of $\phi$.
\end{dfn}

\begin{rem} In Def.~\ref{dfn:canonicalSMConOplaxLimit}, at the level of objects the tensor product on $\sO(\cZ) \times_{\cZ} \cU$ is defined in the following way: suppose given two objects $x=[u,z,\alpha:z \to \phi(u)]$ and $x' = [u',z',\alpha':z' \to \phi(u')]$. Then $x \otimes x' = [u \otimes u', z \otimes z', \gamma]$, where $\gamma$ is given by the composite map
\[ z \otimes z' \xto{\alpha \otimes \alpha'} \phi(u) \otimes \phi(u') \to \phi(u \otimes u') \]
using the lax monoidality of $\phi$ for the second map. 
\end{rem}

\begin{lem} If $\phi: \cU \to \cZ$ is a lax monoidal left-exact functor, then $\Rlax \lim \phi$ is a monoidal recollement with respect to the canonical symmetric monoidal structure on $\sO(\cZ) \times_{\cZ} \cU$.
\end{lem}
\begin{proof} We only need to observe that in Def.~\ref{dfn:canonicalSMConOplaxLimit}, the two evaluation maps $j^{\ast}: \sO(\cZ) \times_{\cZ} \cU \to \cU$ and $i^{\ast}: \sO(\cZ) \times_Z U \to \sO(\cZ) \xto{\ev_0} \cZ$ are symmetric monoidal.
\end{proof}

We next wish to show that given a monoidal recollement $(\cU, \cZ)$ of $\cX$, the symmetric monoidal structure on $\cX$ is the canonical one of Def.~\ref{dfn:canonicalSMConOplaxLimit}. We first observe that the unit transformation of a monoidal adjunction is itself a lax monoidal functor.

\begin{lem} \label{lem:unitLaxMonoidal} Let $C^\otimes$ and $D^\otimes$ be symmetric monoidal $\infty$-categories and let $\adjunct{F}{C}{D}{G}$ be a monoidal adjunction. Then the unit transformation $\eta: C \to \sO(C)$ lifts to a lax monoidal functor $\eta^{\otimes}: C^{\otimes} \to (C^{\otimes})^{\Delta^1}$ such that $\ev_1 \eta^{\otimes} \simeq G^{\otimes} F^{\otimes}$ and $\ev_0 \eta^{\otimes} \simeq \id$.
 \end{lem}
\begin{proof} Let $\cM \to \Delta^1$ be the bicartesian fibration classified by the adjunction. We may factor (or define) $\eta$ as the composition
\[ C \simeq \Fun^{\cocart}_{/ \Delta^1}(\Delta^1,\cM) \subset \Fun_{/\Delta^1}(\Delta^1,\cM) \simeq \sO(C) \times_{C} D  \to \sO(C) \]
where we use Lem.~\ref{lm:sectionsPullbackSquare} for the identification of the sections of $\cM$. Let $\Fun_{/\Delta^1}(\Delta^1,\cM)$ be equipped with its canonical symmetric monoidal structure. Because $F$ is symmetric monoidal, the inclusion $\Fun^{\cocart}_{/ \Delta^1}(\Delta^1,\cM) \subset \Fun_{/\Delta^1}(\Delta^1,\cM)$ defines a symmetric monoidal structure on $\Fun^{\cocart}_{/ \Delta^1}(\Delta^1,\cM)$ by restriction such that the equivalence $\ev_0: \Fun^{\cocart}_{/ \Delta^1}(\Delta^1,\cM) \xto{\simeq} C$ is an equivalence of symmetric monoidal $\infty$-categories. Also, the projection $\Fun_{/\Delta^1}(\Delta^1,\cM) \to \sO(C)$ is lax monoidal by definition. We deduce that $\eta$ lifts to a lax monoidal functor $\eta^{\otimes}$ with the indicated properties.
\end{proof}

\begin{cor} Let ($\cU$, $\cZ$) be a monoidal recollement of $\cX$. Then the functor $\cX \to \Fun(\Delta^1 \times \Delta^1,\cX)$ realizing the pullback square of functors
\[ \begin{tikzcd}[row sep=4ex, column sep=4ex, text height=1.5ex, text depth=0.25ex]
\id \ar{r} \ar{d} & i_\ast i^\ast \ar{d} \\
j_\ast j^\ast \ar{r} & i_\ast i^\ast j_\ast j^\ast
\end{tikzcd} \]
lifts to a lax monoidal functor $\cX^{\otimes} \to (\cX^{\otimes})^{\Delta^1 \times \Delta^1}$. Consequently, if $A \in \cX$ is an algebra object, then we have an equivalence of algebras
\[ A \simeq (j_\ast j^\ast)(A) \times_{(i_\ast i^\ast j_\ast j^\ast)(A)} (i_\ast i^\ast)(A). \]
\end{cor}
\begin{proof} By Lem.~\ref{lem:unitLaxMonoidal}, the monoidal adjunction $j^{\ast} \dashv j_{\ast}$ yields a lax monoidal functor $$(\eta_j)^{\otimes}: \cX^{\otimes} \to (\cX^{\otimes})^{\Delta^1}.$$ We also have the induced monoidal adjunction $\adjunct{\widehat{i}^\ast}{\sO(X)}{\sO(Z)}{\widehat{i}_{\ast}}$ which yields a lax monoidal functor $$(\eta_{\widehat{i}})^{\otimes}: (\cX^{\otimes})^{\Delta^1} \to (\cX^{\otimes})^{\Delta^1 \times \Delta^1}.$$ The composite $(\eta_{\widehat{i}})^{\otimes} \circ (\eta_j)^{\otimes}$ then defines the desired functor.
\end{proof}

\begin{prp}\label{prp:CanonicalMonoidalStructureOnMonoidalRecollement} Suppose $(\cU, \cZ)$ is a monoidal recollement of $\cX$. Then the equivalence $$\cX \xto{\simeq} \sO(\cZ) \times_{\cZ} \cU$$ of Cor.~\ref{cor:RecollementAsPullbackSquare} refines to an equivalence of symmetric monoidal $\infty$-categories, where we equip $\sO(\cZ) \times_{\cZ} \cU$ with the canonical symmetric monoidal structure of Def.~\ref{dfn:canonicalSMConOplaxLimit}.
\end{prp}
\begin{proof} By Lem.~\ref{lem:unitLaxMonoidal} and Lem.~\ref{lem:strictificationOperadSquare}, we have a commutative diagram of $\infty$-operads
\[ \begin{tikzcd}[row sep=4ex, column sep=8ex, text height=1.5ex, text depth=0.5ex]
\cX^{\otimes} \ar{r}{(i^{\ast})^{\otimes} (\eta_j)^{\otimes}} \ar{d}[swap]{(j^{\ast})^{\otimes}} & (\cZ^{\otimes})^{\Delta^1} \ar{d}{\ev_1} \\
\cU^{\otimes} \ar{r}{(i^{\ast})^{\otimes} (j_{\ast})^{\otimes} } & Z^{\otimes}
\end{tikzcd} \]
such that the induced functor $\theta^{\otimes}: \cX^{\otimes} \to (\cZ^{\otimes})^{\Delta^1} \times_{\cZ^{\otimes}} \cU^{\otimes}$ covers the map $\theta$ of Cor.~\ref{cor:RecollementAsPullbackSquare}. Since $\theta$ is an equivalence, $\theta^{\otimes}$ is an equivalence.
\end{proof}

We include the following simple strictification result for completeness.

\begin{lem} \label{lem:strictificationOperadSquare} Suppose we have a homotopy commutative square of $\infty$-operads
\[ \begin{tikzcd}[row sep=4ex, column sep=4ex, text height=1.5ex, text depth=0.25ex]
A^\otimes \ar{r}{F'} \ar{d}{G'} & B^\otimes \ar{d}{G} \\
C^\otimes \ar{r}{F} & D^\otimes
\end{tikzcd} \]
in the sense that there is the data of a homotopy $\theta: G \circ F' \overset{\simeq}{\Rightarrow} F \circ G'$ over $\Fin_{\ast}$
\[ \begin{tikzcd}[row sep=4ex, column sep=4ex, text height=1.5ex, text depth=0.25ex]
A^\otimes \times \{0\} \ar{r}{F'} \ar{d} & B^\otimes \ar{d}{G} \\
A^\otimes \times \Delta^1 \ar{r}{\theta} & D^\otimes \\
A^\otimes \times \{1\} \ar{r}{G'} \ar{u} & C^\otimes \ar{u}[swap]{F}
\end{tikzcd} \]
such that $\theta$ sends every edge $(a,0) \to (a,1)$ to an equivalence. Suppose also that $G$ is a fibration of $\infty$-operads, i.e., a categorical fibration \cite[2.1.2.10]{HA}. Then there exists a functor $F'': A^\otimes \to B^\otimes$ homotopic to $F'$ as a map of $\infty$-operads such that the square
\[ \begin{tikzcd}[row sep=4ex, column sep=4ex, text height=1.5ex, text depth=0.25ex]
A^\otimes \ar{r}{F''} \ar{d}{G'} & B^\otimes \ar{d}{G} \\
C^\otimes \ar{r}{F} & D^\otimes
\end{tikzcd} \]
strictly commutes.
\end{lem}
\begin{proof} Given an $\infty$-operad $O^\otimes$, let $O^{\otimes,\natural}$ denote the marked simplicial set $(O^\otimes,\cE)$ where $\cE$ is the collection of inert morphisms in $O^\otimes$ \cite[2.1.4.5]{HA}. Consider the lifting problem in marked simplicial sets
\[ \begin{tikzcd}[row sep=4ex, column sep=4ex, text height=1.5ex, text depth=0.25ex]
A^{\otimes,\natural} \times \{0\} \ar{r}{F'} \ar{d} & B^{\otimes,\natural} \ar{d}{G} \\
A^{\otimes,\natural} \times (\Delta^1)^{\sharp} \ar{r}{\theta} \ar[dotted]{ru}{\overline{\theta}} & D^{\otimes,\natural}.
\end{tikzcd} \]
Because $G$ is assumed to be a fibration of $\infty$-operads, $G$ is a fibration in the model structure on $\infty$-preoperads \cite[2.1.4.6]{HA}. Therefore, the dotted lift $\overline{\theta}$ exists. If we then let $F'' = \overline{\theta}|_{A^\otimes \times \{1\}}$, the claim follows.
\end{proof}

\begin{nul} Suppose we have a commutative diagram of symmetric monoidal $\infty$-categories and lax monoidal functors
\[ \begin{tikzcd}[row sep=4ex, column sep=4ex, text height=1.5ex, text depth=0.25ex]
{\cU}^{\otimes} \ar{r}{{\phi}^{\otimes}} \ar{d}[swap]{{F_U}^{\otimes}} & {\cZ}^{\otimes} \ar{d}{{F_Z}^{\otimes}} \\
{\cU'}^{\otimes} \ar{r}{{\phi'}^{\otimes}} & {\cZ'}^{\otimes}.
\end{tikzcd} \]
Then by way of the commutative diagram
\[ \begin{tikzcd}[row sep=4ex, column sep=4ex, text height=1.5ex, text depth=0.25ex]
(\cZ^\otimes)^{\Delta^1} \times_{\cZ^{\otimes}} \cU^{\otimes} \ar{r} \ar{d} & (\cZ^\otimes)^{\Delta^1} \ar{r}{F_Z^{\otimes}} \ar{d}{\ev_1} & (\cZ'^{\otimes})^{\Delta^1} \ar{d}{\ev_1} \\
\cU^{\otimes} \ar{r}{\phi^{\otimes}} \ar{rd}[swap]{F_U^{\otimes}} & \cZ^{\otimes} \ar{r}{F_Z^{\otimes}} & \cZ'^{\otimes} \\
& \cU'^{\otimes} \ar{ru}[swap]{\phi'^{\otimes}}
\end{tikzcd} \]
we obtain a lax monoidal functor $F^{\otimes}: (\cZ^\otimes)^{\Delta^1} \times_{\cZ^{\otimes}} \cU^{\otimes} \to (\cZ'^\otimes)^{\Delta^1} \times_{\cZ'^{\otimes}} \cU'^{\otimes}$, which is symmetric monoidal if $F_U^{\otimes}$ and $F_Z^{\otimes}$ are symmetric monoidal.

Let $\sO_{\lex}(\Cat^{\otimes, \lax}_{\infty}) \subset \sO(\Cat^{\otimes,\lax}_{\infty})$ be the subcategory whose objects are left-exact lax monoidal functors and whose morphisms are through symmetric monoidal functors. Then by the above construction\footnote{Technically, to make a rigorous construction we may work at the level of preoperads and then pass to the underlying $\infty$-categories.} we may lift the functor $\Rlax \lim: \sO_{\lex}(\Cat_{\infty}) \to \Recoll_0$ to
\[ {\Rlax \lim}^{\otimes}: \sO_{\lex}(\Cat^{\otimes, \lax}_{\infty}) \to \Recoll_0^{\otimes}. \]
An elaboration of Prop.~\ref{prp:CanonicalMonoidalStructureOnMonoidalRecollement} shows that ${\Rlax \lim}^{\otimes}$ is an equivalence -- we leave the details to the reader.

One also has a lift of $\Rlax \lim: \sO^{\rlax}_{\lex}(\Cat_{\infty}) \to \Recoll$ if one considers right-lax commutative squares of $\infty$-operads. Since the details in this case are more involved, we leave a precise formulation to the reader.
\end{nul}

\begin{nul}[\textbf{Closed monoidal structure}] Suppose now that $\cX$ is also closed monoidal and let $F(-,-)$ denote the internal hom. If $(\cU,\cZ)$ is a monoidal recollement of $\cX$, then we define
$$F_{\cU}(u,u') = j^{\ast} F(j_{\ast} u, j_{\ast} u') \: \text{ and } \: F_{\cZ}(z,z') = i^{\ast} F(i_{\ast} z, i_{\ast} z')$$
to be internal homs for $\cU$ and $\cZ$, so that $\cU$ and $\cZ$ are closed monoidal. Indeed, since $j^{\ast} \dashv j_{\ast}$ is monoidal, we have
\begin{align*} \Map_{\cU}(w,j^{\ast} F(j_{\ast} u, j_{\ast} v)) &\simeq \Map_{\cX}(j_{\ast} w,  F(j_{\ast} u, j_{\ast} v)) \simeq \Map_{\cX}(j_{\ast} w \otimes j_{\ast}v, j_{\ast}v) \\
& \Map_{\cU}(j^{\ast}(j_{\ast} w \otimes j_{\ast} u), v) \simeq \Map_{\cU}(w \otimes u,v),
\end{align*}
and similarly for $F_{\cZ}(-,-)$. Moreover we have natural equivalences
\begin{align*} F(x,j_{\ast} u) \simeq j_{\ast} F_{\cU}(j^{\ast} x,u), \quad F(x,i_{\ast} z) \simeq i_{\ast} F_{\cZ}(i^{\ast} x,z).
\end{align*}
For example, we may check
\begin{align*}
\Map_{\cX}(x,F(y,j_{\ast}u)) & \simeq \Map_{\cX}(x \otimes y, j_{\ast} u) \simeq \Map_{\cU}(j^{\ast} x \otimes j^{\ast} y,u) \\
& \simeq \Map_{\cU}(j^{\ast} x, F_{\cU}(j^{\ast} y ,u)) \simeq \Map_{\cX}(x,j_{\ast} F_{\cU}(j^{\ast} y ,u)).
\end{align*}
This implies that the unit maps
\begin{align*} F(j_{\ast} u , j_{\ast} u') & \to j_{\ast} j^{\ast} F(j_{\ast} u, j_{\ast} u') = j_{\ast} F_{\cU}(u,u') \\
F(i_{\ast} z, i_{\ast} z') & \to i_{\ast} i^{\ast} F(i_{\ast} z, i_{\ast} z') = i_{\ast} F_{\cZ}(z,z')
\end{align*}
are equivalences.
\end{nul}


\begin{prp}[Projection formulas] \label{prp:ProjectionFormulasMonoidalRecollement} Let $(\cU, \cZ)$ be a stable\footnote{We do not require stability for the $i^{\ast} \dashv i_{\ast}$ projection formula. For the assertions that only involve $j_!$, we only need that $\cX$ be pointed.} monoidal recollement of $\cX$.
\begin{enumerate} 
\item The natural maps $\alpha: i_{\ast}(z) \otimes x \to i_{\ast}(z \otimes i^\ast x)$ and $\beta: j_!(u \otimes j^\ast x) \to j_!(u) \otimes x$ are equivalences.
\item The fiber sequence $j_! j^{\ast} x \to x \to i_{\ast} i^{\ast} x$ is equivalent to
\[ j_!(1_U) \otimes x \to x \to i_{\ast}(1_Z) \otimes x. \]
\end{enumerate}
Now suppose also that $\cX$ is closed monoidal.
\begin{enumerate}
    \setcounter{enumi}{2}
    \item We have natural equivalences $F(j_! u, x) \simeq j_{\ast} F_{\cU}(u,j^{\ast} x)$ and $F(i_{\ast} z,x) \simeq i_{\ast} F_{\cZ}(z,i^! x)$.
    \item The fiber sequence $i_{\ast} i^! x \to x \to j_{\ast} j^{\ast} x$ is equivalent to
    \[ F(i_{\ast} 1_Z,x) \to x \to F(j_! 1_U,x). \]
    \item  We have natural equivalences $j^{\ast} F(x,y) \simeq F_{\cU}(j^{\ast} x, j^{\ast} y)$ and $F_{\cZ}(i^{\ast} x, i^! y) \simeq i^! F(x,y)$.    
\end{enumerate}
\end{prp}
\begin{proof} For (1), it's easily checked that $i^{\ast} \alpha$, $j^{\ast} \alpha$ and $i^{\ast} \beta$, $j^{\ast} \beta$ are equivalences, hence $\alpha$ and $\beta$ are equivalences. (2) then follows as a corollary. For (3), we have sequences of equivalences
\begin{align*} \Map_{\cX}(y,F(j_!u,x)) & \simeq \Map_{\cX}(y \otimes j_! u, x) \simeq \Map_{\cX}(j_!(j^{\ast} y \otimes u),x) \simeq \Map_{\cU}(j^{\ast} y \otimes u,j^{\ast} x) \\
& \simeq \Map_{\cU}(j^{\ast} y , F_{\cU}(u,j^{\ast} x)) \simeq \Map_{\cX}(y,j_{\ast} F_{\cU}(u,j^{\ast} x)), \: \text{and} \\
\Map_{\cX}(y, F(i_{\ast} z, x)) & \simeq \Map_{\cX}(y \otimes i_{\ast} z, x) \simeq \Map_{\cX}(i_{\ast}(i^{\ast} y \otimes z), x) \simeq \Map_{\cZ}(i^{\ast} y \otimes z, i^! x) \\
& \simeq \Map_{\cZ}(i^{\ast} y, F_{\cZ}(z,i^! x)) \simeq \Map_{\cZ}(y, i_{\ast} F_{\cZ}(z,i^! x)).
\end{align*}
If we let $u = 1_U$, then $F_{\cU}(1_U, v) \simeq v$, hence $F(j_! 1_U, x) \simeq j_{\ast} F_{\cU}(1_U,j^{\ast} x) \simeq j_{\ast} j^{\ast} x$. (4) then follows as a corollary. For (5), we have sequences of equivalences
\begin{align*} \Map_{\cU}(u,j^{\ast} F(x,y)) & \simeq \Map_{\cX}(j_! u, F(x,y)) \simeq \Map_{\cX}(j_! u \otimes x,y) \simeq \Map_{\cX}(j_!(u \otimes j^{\ast} x),y) \\
& \simeq \Map_{\cU}(u \otimes j^{\ast} x, j^{\ast} y) \simeq \Map_{\cU}(u,F_{\cU}(j^{\ast} x, j^{\ast} y)), \: \text{and} \\
\Map_{\cZ}(z,F_{\cZ}(i^{\ast} x, i^! y)) & \simeq \Map_{\cZ}(z \otimes i^{\ast} x, i^! y) \simeq \Map_{\cX}(i_{\ast}(z \otimes i^{\ast} x),y) \simeq \Map_{\cX}(i_{\ast} z \otimes x, y) \\
 & \simeq \Map_{\cX}(i_{\ast} z, F(x,y)) \simeq \Map_{\cZ}(z,i^! F(x,y)).
\end{align*}
\end{proof}

\begin{cor} \label{cor:FractureSquareMonoidal} Suppose that $(\cU, \cZ)$ is a stable monoidal recollement of a closed symmetric monoidal stable $\infty$-category $\cX$. Then for all $x \in \cX$, we have a commutative diagram
\[ \begin{tikzcd}[row sep=4ex, column sep=4ex, text height=1.5ex, text depth=0.25ex]
x \otimes j_! (1_U) \ar{r} \ar{d}{\simeq} & x \ar{r} \ar{d} & x \otimes i_{\ast}(1_Z) \ar{d} \\
F(j_!(1_U), x) \otimes j_!(1_U) \ar{r} & F(j_!(1_U), x) \ar{r} & F(j_!(1_U), x) \otimes i_{\ast}(1_Z)
\end{tikzcd} \]
in which the righthand square is a pullback square.
\end{cor}

\begin{nul}[\textbf{Relation to smashing localizations}] \label{SmashingLocalizationsAreStableMonoidalRecollements} Suppose $\cX$ is a symmetric monoidal stable $\infty$-category and $\cZ \subset \cX$ is a reflective and coreflective subcategory that determines a stable recollement $(\cU, \cZ)$ on $\cX$. Then this recollement is monoidal if and only if $i_{\ast} i^{\ast}$ is compatible with the symmetric monoidal structure on $\cX$ and the resulting projection formula for $i^{\ast} \dashv i_{\ast}$ holds, i.e., the natural map $i_{\ast} z \otimes x \to i_{\ast} ( z \otimes i^{\ast} x)$ is an equivalence for all $x \in \cX$ and $z \in \cZ$. Indeed, the `only if' direction hold by Prop.~\ref{prp:ProjectionFormulasMonoidalRecollement}, and for the `if' direction, we only need to show that for every $x \in \cX$ such that $j^{\ast} x \simeq 0$, $j^{\ast} (x \otimes y) \simeq 0$ for every $y \in \cX$. But $j^{\ast} x \simeq 0$ if and only if $x \simeq i_{\ast} i^{\ast} x$, and then
\[ j^{\ast}( x \otimes y) \simeq j^{\ast}( i_{\ast} i^{\ast} x \otimes y) \simeq j^{\ast} (i_{\ast} (i^{\ast} x \otimes i^{\ast} y)) \simeq 0. \]

In view of \cite[Prop.~5.29]{MATHEW2017994}, $\cZ$ is a \emph{smashing localization} of $\cX$ in the sense that $\cZ \simeq \Mod_{\cX}(A)$ for $A = i_{\ast}i^{\ast}1$ an idempotent $E_{\infty}$-algebra in $\cX$. We deduce that smashing localizations of $\cX$ are in bijective correspondence with stable monoidal recollements of $\cX$. Moreover, if $F: \cX \to \cX'$ is a morphism of monoidal recollements $(\cU,\cZ) \to (\cU', \cZ')$, then
\[ F i_{\ast} i^{\ast} 1 \simeq i'_{\ast} i'^{\ast} F(1) \simeq i'_{\ast} i'^{\ast} 1, \]
so $F$ preserves the defining idempotent $E_{\infty}$-algebras.
\end{nul}


\subsection{Families of recollements}

We conclude this section with a few extensions of recollement theory to the parametrized setting. Let $S$ be an $\infty$-category, let $\cX_{\bullet}: S \to \Recoll_0^{\lex}$ be a functor, and let $\cX, \cU, \cZ \to S$ be the cocartesian fibrations obtained via the Grothendieck construction. Then in view of \ref{LaxVsStrictMorphismsOfRecollements} and the strictness assumption, we have $S$-adjunctions \cite[Def.~8.1]{Exp2}
\[ \begin{tikzcd}[row sep=4ex, column sep=4ex, text height=1.5ex, text depth=0.25ex]
\cU \ar[shift right=1,right hook->]{r}[swap]{j_{\ast}} & \cX \ar[shift right=2]{l}[swap]{j^{\ast}} \ar[shift left=2]{r}{i^{\ast}} & \cZ \ar[shift left=1,left hook->]{l}{i_{\ast}}.
\end{tikzcd} \]

We first show that the procedure of taking $S$-functor categories yields a recollement.

\begin{lem} \label{lem:FunctorCategoryRecollement} For any $S$-$\infty$-category $K$, $(\Fun_S(K,\cU), \Fun_S(K, \cZ))$ is a recollement of $\Fun_S(K,\cX)$.
\end{lem}
\begin{proof} By \cite[Prop.~8.2]{Exp2}, we have induced adjunctions given by postcomposition
\[ \begin{tikzcd}[row sep=4ex, column sep=4ex, text height=1.5ex, text depth=0.25ex]
\Fun_S(K,\cU) \ar[shift right=1,right hook->]{r}[swap]{\overline{j}_{\ast}} & \Fun_S(K,\cX) \ar[shift right=2]{l}[swap]{\overline{j}^{\ast}} \ar[shift left=2]{r}{\overline{i}^{\ast}} & \Fun_S(K,\cZ) \ar[shift left=1,left hook->]{l}{\overline{i}_{\ast}},
\end{tikzcd} \]
where it is clear that $\overline{j}^{\ast} \overline{j}_{\ast} \simeq \id$ and $\overline{i}^{\ast} \overline{i}_{\ast} \simeq \id$, hence $\overline{j}_{\ast}$ and $\overline{i}_{\ast}$ are fully faithful. By \cite[Prop.~5.4.7.11]{HTT}, the hypothesis that for all $f: s \to t$ the restriction functors $f^{\ast}: \cX_t \to \cX_s$ preserve finite limits ensures that $\Fun_S(K, \cX)$ admits finite limits (which are computed fiberwise), and similarly the induced restriction functors $f_U^{\ast}$ and $f_Z^{\ast}$ preserve finite limits, so $\Fun_S(K, \cU)$, $\Fun_S(K, \cZ)$ admit finite limits and $\overline{j}^{\ast}, \overline{i}^{\ast}$ preserve finite limits. Since $j^{\ast} i_{\ast} \simeq 0$ and the terminal object $0 \in \Fun_S(K, \cU)$ is given by $K \to S \xto{\underline{0}} \cU$ for the cocartesian section $\underline{0}: S \to \cU$ that selects the terminal object in each fiber, we get that $\overline{j}^{\ast} \overline{i}_{\ast} \simeq 0$. Finally, since a morphism $f$ in $\Fun_S(K,\cX)$ is an equivalence if and only if $f(k)$ is an equivalence for all $k \in K$, we deduce that $\overline{j}^{\ast}$ and $\overline{i}^{\ast}$ are jointly conservative using the joint conservativity of $j^{\ast}$ and $i^{\ast}$.
\end{proof}

\begin{cor} \label{cor:LimitsOfRecollements} The forgetful functors $\Recoll_0^{\lex} \to \Cat_{\infty}$ and $\Recoll_0^{\st} \to \Cat_{\infty}^{\st}$ create limits.
\end{cor}
\begin{proof} The first statement follows from Lem.~\ref{lem:FunctorCategoryRecollement} by taking $K=S$ and using that the $\infty$-category of cocartesian sections computes the limit of a diagram of $\infty$-categories \cite[\S 3.3.3]{HTT}. We note that the proof of Lem.~\ref{lem:FunctorCategoryRecollement} shows that the evaluation functors at any $s \in S$ are left-exact and strict morphisms of recollements, so the limit resides in $\Recoll_0^{\lex}$. Finally, because limits in $\Cat_{\infty}^{\st}$ are created in $\Cat_{\infty}$, the second statement follows.
\end{proof}

We can also use Lem.~\ref{lem:FunctorCategoryRecollement} to compute $S$-colimits in $\cX$. For clarity, let us temporarily revert to the non-parametrized case $S = \ast$ for the next two results; the $S$-analogues will also hold by the same reasoning.

\begin{lem} \label{lem:ColimitExistenceInRecollement} Let $(\cU, \cZ)$ be a recollement of $\cX$ and suppose that $\cU$ and $\cZ$ admit $K$-indexed colimits. Then $\cX$ admits $K$-indexed colimits.
\end{lem}
\begin{proof} With respect to the recollement of $\Fun(K, \cX)$ of Lem.~\ref{lem:FunctorCategoryRecollement}, the constant diagram functor $\delta: \cX \to \Fun(K, \cX)$ is obviously a morphism of recollements. Passing to left adjoints, we obtain a right-lax commutative diagram
\[ \begin{tikzcd}[row sep=4ex, column sep=4ex, text height=1.5ex, text depth=0.25ex]
\Fun(K, \cU) \ar{r}{\overline{i}^{\ast} \overline{j}_{\ast}} \ar{d}[swap]{\colim} \ar[phantom]{rd}{\SWarrow} & \Fun(K, \cZ) \ar{d}{\colim} \\
\cU \ar{r}[swap]{i^{\ast} j_{\ast}} & \cZ,
\end{tikzcd} \]
which induces a morphism of recollements $\colim: \Fun(K, \cX) \to \cX$. We claim that $\colim$ is left adjoint to $\delta$. In fact, if $\cM, \cM^K \to \Delta^1$ are the cartesian fibrations classified by $i^{\ast} j_{\ast}$ and $\overline{i}^{\ast} \overline{j}_{\ast}$ respectively, then we have a map $\delta: \cM^K \to \cM$ of cartesian fibrations and by \cite[Prop.~7.3.2.6]{HA} a relative left adjoint $\colim: \cM^K \to \cM$. The formation of sections sends relative adjunctions to adjunctions, which proves the claim. We deduce that $\cX$ admits $K$-indexed colimits.
\end{proof}

\begin{cor} Suppose $\cU$ and $\cZ$ are presentable $\infty$-categories and $\phi: \cU \to \cZ$ is a left-exact accessible functor. Then $\cX = \rlax\lim \phi$ is a presentable $\infty$-category.
\end{cor}
\begin{proof} By Lem.~\ref{lem:ColimitExistenceInRecollement}, $\cX$ admits all small colimits. By \cite[Cor.~5.4.7.17]{HTT}, $\cX$ is accessible. We conclude that $\cX$ is presentable.
\end{proof}


Finally, we describe how recollements interact with an ambidextrous adjunction (e.g., the adjunction between restriction and induction for equivariant spectra). 

\begin{lem} \label{lem:AmbidexterityRecollement} Let $(\cU,\cZ)$ and $(\cU',\cZ')$ be stable recollements on $\cX$ and $\cX'$ and let $f^{\ast}: \cX \to \cX'$ be an exact functor such that $f^{\ast}|_{i_{\ast}(\cZ)} \subset i_{\ast}(\cZ')$ (so $f^{\ast}$ is not necessarily a morphism of recollements, but we still may define ${f_U}^{\ast} \coloneq j'^{\ast} f^{\ast} j_{\ast}$, ${f_Z}^{\ast} \coloneq i'^{\ast} f^{\ast} i_{\ast}$, and have ${f_U}^{\ast} j^{\ast} \simeq j'^{\ast} {f_U}^{\ast}$).
\begin{enumerate}
    \item Suppose that $f^{\ast}|_{j_!( \cU)} \subset j'_!(\cU')$ and $f^{\ast}$ admits a right adjoint $f_{\ast}$. Then
    \begin{enumerate}
        \item The essential image of $f_{\ast} j'_{\ast}$ lies in $j_{\ast}(\cU)$, so $f^{\ast} \dashv f_{\ast}$ restricts to an adjunction
        \[ \adjunct{{f_U}^{\ast}}{\cU}{\cU'}{{f_U}_{\ast}} \]
        with $j_{\ast} {f_U}_{\ast} \simeq f_{\ast} j'_{\ast}$.
        \item The natural map $j^{\ast} f_{\ast} \to {f_U}_{\ast} j'^{\ast}$ is an equivalence. 
        \item The essential image of $f_{\ast} i'_{\ast}$ lies in $i_{\ast}(\cZ)$, so $f^{\ast} \dashv f_{\ast}$ restricts to an adjunction
        \[ \adjunct{{f_Z}^{\ast}}{\cZ}{\cZ'}{{f_Z}_{\ast}} \]
        with $i_{\ast} {f_Z}_{\ast} \simeq f_{\ast} i'_{\ast}$.
    \end{enumerate}
    \item Suppose that $f^{\ast}|_{j_{\ast}( \cU)} \subset j'_{\ast}(\cU')$ and $f^{\ast}$ admits a left adjoint $f_!$. Then
    \begin{enumerate}
        \item The essential image of $f_{\ast} j'_!$ lies in $j_!(\cU)$, so $f_! \dashv f^{\ast}$ restricts to an adjunction
        \[ \adjunct{{f_U}_!}{\cU'}{\cU}{{f_U}^{\ast}} \]
        with $j_! {f_U}_! \simeq f_! j'_!$.
        \item The natural map ${f_U}_! j^{\ast} \to j'^{\ast} f_!$ is an equivalence. 
        \item The essential image of $f_! i'_{\ast}$ lies in $i_{\ast}(\cZ)$, so $f_! \dashv f^{\ast}$ restricts to an adjunction
        \[ \adjunct{{f_Z}_!}{\cZ'}{\cZ}{{f_Z}^{\ast}} \]
        with $i_{\ast} {f_Z}_! \simeq f_! i'_{\ast}$.
        \item The natural map $i^{\ast} {f_Z}_! \to {f_Z}_! i'^{\ast}$ is an equivalence.
    \end{enumerate}
    \item Suppose that $f^{\ast} \in \Recoll^{\st}_0$, $f^{\ast}$ admits left and right adjoints $f_!$ and $f_{\ast}$, and we have the ambidexterity equivalence $f_! \simeq f_{\ast}$. Then $f_{\ast} \in \Recoll^{\st}_0$ and we additionally have ambidexterity equivalences ${f_U}_! \simeq {f_U}_{\ast}$ and ${f_Z}_! \simeq {f_Z}_{\ast}$.
\end{enumerate}
\end{lem}
\begin{proof} We first prove the assertions of (1). For (1.1), for any $u' \in \cU'$ because we have for all $z \in \cZ$ that
\begin{align*} \Map_{\cX}(i_{\ast} z,f_{\ast} j'_{\ast} u') \simeq \Map_{\cU'}(j'^{\ast} f^{\ast} i_{\ast} z, u') \simeq \Map_{\cU'}(f_U^{\ast} j'^{\ast} i_{\ast} z, u') \simeq \ast,
\end{align*}
we get $f_{\ast} j'_{\ast} u' \in j_{\ast}(\cU)$. For (1.2), the assertion holds because the map is adjoint to the equivalence $f^{\ast} j_! \to j'_! {f_U}^{\ast}$. For (1.3), for any $z' \in \cZ'$ we have
\[ j^{\ast} f_{\ast} i'_{\ast} z' \simeq {f_U}_{\ast} j^{\ast} i'_{\ast} z' \simeq {f_U}_{\ast} 0 \simeq 0, \]
hence $f_{\ast} i'_{\ast} z' \in i_{\ast}(\cZ)$. Next, the assertions of (2) hold by a dual argument; we note that the extra assertion (2.4) holds because $f_!$ now commutes with $j_!$ instead of $j_{\ast}$. Finally, for (3) the functor $f_! \simeq f_{\ast}$ is in $\Recoll^{\st}_0$ by combining (1.1), (1.3), and (2.1). For the ambidexterity assertions, the equivalence ${f_Z}_! \simeq {f_Z}_{\ast}$ is clear because the embedding $i_{\ast}: \cZ \subset \cX$ is unambiguous, whereas for ${f_U}_! \simeq {f_U}_{\ast}$ we note that the sequence of equivalences
\begin{align*} \Map_{\cU}(u, {f_U}_! u') & \simeq \Map_{\cX}(j_! u, f_! j'_! u') \simeq \Map_{\cX} (j_! u, f_{\ast} j'_! u') \simeq \Map_{\cX'}(f^{\ast} j_! u, j'_! u') \\
& \simeq \Map_{\cX'} (j'_! {f_U}^{\ast} u, j'_! u') \simeq \Map_{\cU'} ({f_U}^{\ast} u, u')
\end{align*}
demonstrates that ${f_U}_!$ is right adjoint to ${f_U}^{\ast}$ and hence ${f_U}_! \simeq {f_U}_{\ast}$.
\end{proof}


\begin{cor} \label{cor:RecollementGivesParamStableSubcategories} Let $G$ be a finite group. Suppose that $\cX_{\bullet}: \sO^{\op}_G \to \Recoll^{\st}_0$ is a functor such that the underlying $G$-$\infty$-category $\cX$ is $G$-stable \cite[Def.~7.1]{Exp4}. Then $\cU$ and $\cZ$ are $G$-stable and all of the functors appearing in the diagram of $G$-adjunctions
\[ \begin{tikzcd}[row sep=4ex, column sep=4ex, text height=1.5ex, text depth=0.25ex]
\cU \ar[shift right=1,right hook->]{r}[swap]{j_{\ast}} & \cX \ar[shift right=2]{l}[swap]{j^{\ast}} \ar[shift left=2]{r}{i^{\ast}} & \cZ \ar[shift left=1,left hook->]{l}{i_{\ast}}
\end{tikzcd} \]
are $G$-exact.
\end{cor}
\begin{proof} By Lem.~\ref{lem:AmbidexterityRecollement}, it only remains to check the Beck-Chevalley condition for $\cU$ and $\cZ$ to show the existence of finite $G$-products. But this follows from the same condition on $\cX$, since the restriction and induction functors $(f_{-})^{\ast}, (f_{-})_{\ast}$ commute with the inclusion functors $(j_{\bullet})_{\ast}$, $(j_{\bullet})_!$, and $(i_{\bullet})_{\ast}$.
\end{proof}

\begin{dfn} \label{dfn:ParamStableRecollement} In the situation of Cor.~\ref{cor:RecollementGivesParamStableSubcategories}, we say that $(\cU,\cZ)$ is a \emph{$G$-stable $G$-recollement} of $\cX$.
\end{dfn}

\section{Recollements on lax limits of \texorpdfstring{$\infty$}{infinity}-categories}

Suppose $p: C \to S$ is a locally cocartesian fibration classified by a $2$-functor $f: \fC[S] \to \Cat_{\infty}$ (\cite[Def.~1.1.5.1]{HTT} and \cite[\S 3]{G}), so for every $2$-simplex $\Delta^2 \to S$, we have a lax commutative diagram of $\infty$-categories
\[ \begin{tikzcd}[row sep=4ex, column sep=4ex, text height=1.5ex, text depth=0.25ex]
C_{0} \ar{rd}[swap]{F_{01}} \ar{rr}{F_{02}} & \ar[phantom]{d}{\Downarrow} & C_{2}, \\
& C_{1} \ar{ru}[swap]{F_{12}} &
\end{tikzcd} \]
and the higher-dimensional simplices of $S$ supply coherence data. Then the $2$-functoriality of $f$ yields two notions of lax limit corresponding to choosing two possible orientations for morphisms -- informally, the \emph{left-lax} limit of $f$ has objects given by tuples $(x_i \in C_i, \alpha_{ij}: F_{ij}(x_i) \to x_j)$, whereas the \emph{right-lax} limit of $f$ has objects given by tuples $(x_i \in C_i, \alpha_{ij}: x_j \to F_{ij}(x_i))$. To give rigorous meaning to these notions, we may circumvent giving a precise formulation of the lax universal property (for instance, as carried out in \cite{GHN}) and instead \emph{define} the left-lax limit to be the $\infty$-category of sections
\[ \llax\lim(f) = \Fun_{/S}(S,C) \]
and the right-lax limit to be the $\infty$-category
\[ \rlax\lim(f) = \Fun^{\cocart}_{/S}(\sd(S),C), \]
where $\sd(S)$ is the \emph{barycentric subdivision} of $S$ (Def.~\ref{dfn:barycentricSubdivision}) that is locally cocartesian over $S$ via the $\mathit{max}$ functor (Constr.~\ref{cnstr:MaxFunctorSubdivision}), and we let $\Fun^{\cocart}_{/S}(-,-)$ be the full subcategory on those functors over $S$ that preserve \emph{locally cocartesian} edges. Viewing $f$ itself as a \emph{left-lax diagram} in $\Cat_{\infty}$, we may thereby speak of left-lax and right-lax limits of left-lax diagrams of $\infty$-categories; dually, we may also speak of left-lax and right-lax limits of right-lax diagrams of $\infty$-categories encoded as locally cartesian fibrations. We refer to \cite[\S 1]{AMGR-NaiveApproach} for a more detailed discussion.\footnote{We follow \cite[\S 1]{AMGR-NaiveApproach} in referring to these two types of lax limits as `left' and `right', even though lax and oplax are more standard nomenclature. The terminology is consistent with the usage of left for cocartesian-type constructions and right for cartesian-type constructions (e.g., left and right fibrations).}

\begin{dfn} Let $S' \subset S$ be a full subcategory. Then $S'$ is a \emph{sieve} if for every morphism $x \to y$ in $S$, if $y \in S'$, then $x \in S'$. Dually, $S'$ is a \emph{cosieve} if $(S')^{\op}$ is a sieve in $S^{\op}$.

Given a sieve $S_0 \subset S$ and cosieve $S_1 \subset S$, we say that $S_0$ and $S_1$ form a \emph{sieve-cosieve decomposition} of $S$ if $S_0$ and $S_1$ are disjoint and any object $x \in S$ lies either in $S_0$ or $S_1$.
\end{dfn}

\begin{rem} Note that sieves and cosieves are necessarily stable under equivalences. Given a sieve-cosieve decomposition $(S_0,S_1)$ of $S$, we may define a functor $\pi: S \to \Delta^1$ that sends each object $x \in S$ to the integer $i \in \{0,1\}$ such that $x \in S_i$. Conversely, any functor $\pi: S \to \Delta^1$ determines a sieve-cosieve decomposition of $S$ by taking its fibers over $0$ and $1$.
\end{rem}

Our main goal in this section is to describe how a sieve-cosieve decomposition of $S$ produces recollements on right-lax limits of left-lax diagrams of $\infty$-categories.

\begin{rem} As we saw in \ref{recollEquivalenceToOplaxLim}, a recollement itself is an example of a right-lax limit over $\Delta^1$. Given a working theory of (pointwise) right-lax Kan extensions, our results should follow from the usual transitivity property of Kan extensions applied to the factorization $S \xto{\pi} \Delta^1 \to \ast$. However, we are not aware of such a theory that also affords the explicit description of the gluing functor given in Prop.~\ref{prp:ExistenceLaxRightKanExtension}.
\end{rem}

\subsection{Recollements on right-lax limits of strict diagrams}

Before entering into our study of left-lax diagrams, let us consider the simpler case of strict diagrams $f: S \to \Cat_{\infty}$. For this case, right-lax limits are modeled by sections of the \emph{cartesian} fibration that classifies $f$. Thus suppose that $p: C \to S$ is a cartesian fibration, $\pi: S \to \Delta^1$ is a functor, and let $p_0: C_0 \to S_0$, $p_1: C_1 \to S_1$ denote the pullbacks of $p$ to the fibers $S_0$, $S_1$ of $\pi$. Given a section $F: S \to C$ of $p$, let $j^{\ast} F: S_1 \to C_1$ be its restriction over $S_1$ and let $i^{\ast} F: S_0 \to C_0$ be its restriction over $S_0$. We obtain functors
\[ j^{\ast}: \Fun_{/S}(S,C) \to \Fun_{/S_1}(S_1,C_1), \quad i^{\ast}: \Fun_{/S}(S,C) \to \Fun_{/S_0}(S_0, C_0). \]
We first explain when $j^{\ast}$ and $i^{\ast}$ admit right adjoints. Suppose $G: S_1 \to C_1$ is a section of $p_1$. For every $x \in S$, let
\[ G_x: (S_1)_{x/} \coloneq S_1 \times_S S_{x/} \to S_1 \xto{G} C_1 \subset C \] 
be the composite functor and consider the commutative diagram
\[ \begin{tikzcd}[row sep=4ex, column sep=6ex, text height=1.5ex, text depth=0.25ex]
(S_1)_{x/} \ar{r}{G_x} \ar{d} & C \ar{d}{p} \\
(S_1)_{x/}^{\lhd} \ar{r} \ar[dotted]{ru}[swap]{\overline{G_x}} & S
\end{tikzcd} \]
where the cone point is sent to $x$. By \cite[Cor.~4.3.1.11]{HTT}, if for every $s \in S$, $C_s$ admits $(S_1)_{x/}$-indexed limits, and for every $f: s \to t$, the pullback functor $f^{\ast}: C_t \to C_s$ preserves $(S_1)_{x/}$-indexed limits, then there exists a dotted lift $\overline{G_x}$ which is a $p$-limit of $G_x$. If this holds for all $x \in S$,  then by the dual of \cite[Lem.~4.3.2.13]{HTT}, the $p$-right Kan extension $j_{\ast} G$ exists and is computed pointwise by these $p$-limits. Moreover, by \cite[Prop.~4.3.2.17]{HTT}, the right adjoint $j_{\ast}$ then exists and is computed objectwise by $j_{\ast} G$.

Now let $H: S_0 \to C_0$ be a section of $p_0$. The same results hold for computing $i_{\ast} H$. However, the slice $\infty$-categories $(S_0)_{x/}$ are empty when $x \in S_1$. Therefore, the hypotheses above amount to supposing that for all $s \in S$, $C_s$ admits a terminal object, and for all $f: s \to t$, the pullback functor $f^{\ast}$ preserves this terminal object.

Finally, let $\cK = \{K_{\alpha}\}_{\alpha \in A}$ be a class of simplicial sets and suppose that for all $K \in \cK$ and $s \in S$, the fiber $C_s$ admits $K$-indexed limits, and for all $f: s \to t$, the pullback functor $f^{\ast}$ preserves $K$-indexed limits. Then by the dual of \cite[Prop.~5.4.7.11]{HTT} and \cite[Rmk.~5.4.7.13]{HTT}, $\Fun_{/S}(S,C)$ admits $K$-indexed limits such that the evaluation functors $\ev_s: \Fun_{/S}(S,C) \to C_s$ preserve $K$-indexed limits -- in other words, the $K$-indexed limits in $\Fun_{/S}(S,C)$ are computed fiberwise. 

Let us now suppose that $p$ satisfies this condition for $\cK$ the class of finite simplicial sets and also satisfies the existence hypotheses for $j_{\ast}$.

\begin{prp} \label{prp:RecollementSectionCategoryCartesianFibration} The adjunctions 
\[ \begin{tikzcd}[row sep=4ex, column sep=4ex, text height=1.5ex, text depth=0.25ex]
\Fun_{/S_1}(S_1,C_1) \ar[shift right=1,right hook->]{r}[swap]{j_{\ast}} & \Fun_{/S}(S,C) \ar[shift right=2]{l}[swap]{j^{\ast}} \ar[shift left=2]{r}{i^{\ast}} & \Fun_{/S_0}(S_0,C_0) \ar[shift left=1,left hook->]{l}{i_{\ast}}
\end{tikzcd} \]
together exhibit $\Fun_{/S}(S,C)$ as a recollement of $\Fun_{/S_1}(S_1,C_1)$ and $\Fun_{/S_0}(S_0,C_0)$.
\end{prp}
\begin{proof} Note the functors $j^{\ast}$ and $i^{\ast}$ are left exact by the fiberwise computation of limits in section $\infty$-categories. Because $(S_0)_{x/} = \emptyset$ for all $x \in S_1$, we get that $j^{\ast} i_{\ast}$ is the constant functor at the terminal object of $\Fun_{/S_1}(S_1,C_1)$. Finally, $i^{\ast}$ and $j^{\ast}$ are jointly conservative because equivalences are detected objectwise in $\Fun_{/S}(S,C)$.
\end{proof}

\begin{rem} If the fibers of $p$ are moreover stable $\infty$-categories, then the left-exact pullback functors $f^{\ast}$ are necessarily exact and the recollement of Prop.~\ref{prp:RecollementSectionCategoryCartesianFibration} is stable.
\end{rem}

\begin{exm} \label{exm:SieveCosieveRecollementOnFunctorCategory} Let $C \simeq D \times S$ and $p$ be the projection to $S$. Then the recollement of Prop.~\ref{prp:RecollementSectionCategoryCartesianFibration} simplifies to 
\[ \begin{tikzcd}[row sep=4ex, column sep=4ex, text height=1.5ex, text depth=0.25ex]
\Fun(S_1,D) \ar[shift right=1,right hook->]{r}[swap]{j_{\ast}} & \Fun(S,D) \ar[shift right=2]{l}[swap]{j^{\ast}} \ar[shift left=2]{r}{i^{\ast}} & \Fun(S_0,D) \ar[shift left=1,left hook->]{l}{i_{\ast}}
\end{tikzcd} \]
where $j: S_1 \to S$ and $i: S_0 \to S$ now denote the inclusions. Recollement theory then gives a calculational technique for computing the right Kan extension $\phi_{\ast} F$ of a functor $F: S \to D$ along $\phi: S \to T$. Namely, if we let $\phi_0 = \phi \circ i$, $\phi_1 = \phi \circ j$, $F_0 = F|_{S_0}$, and $F_1 = F|_{S_1}$, the pullback square \ref{recollementFractureSquare} yields a pullback square
\[ \begin{tikzcd}[row sep=4ex, column sep=4ex, text height=1.5ex, text depth=0.25ex]
\phi_{\ast} F \ar{r} \ar{d} & (\phi_0)_{\ast} F_0 \ar{d} \\
(\phi_1)_{\ast} F_1 \ar{r} & (\phi_0)_{\ast} \left( \left(j_{\ast} F_1 \right)|_{S_0} \right).
\end{tikzcd} \]
\end{exm}

\subsection{Recollements on right-lax limits of left-lax diagrams} \label{subsection:recollRLaxLLax}

We now seek to establish the analogue of Prop.~\ref{prp:RecollementSectionCategoryCartesianFibration} for right-lax limits of locally cocartesian fibrations. Although the ideas are straightforward, the categorical details turn out to be considerably more involved. We begin by proving some needed extensions to the theory of relative right Kan extensions initiated in \cite[\S 4.1-3]{HTT}, which play a technical role in our construction of the recollement adjunctions. We then give an explicit construction of the barycentric subdivision $\sd(S)$ as a quasi-category (Def.~\ref{dfn:barycentricSubdivision}, but also see Rmk.~\ref{rem:barycentricSubdivisionOrdinaryCategory}), and extend the cocartesian pushforward of \cite[Lem.~2.22]{Exp2} to the locally cocartesian situation (Prop.~\ref{prp:LocallyCocartesianPushforward} and Prop.~\ref{prp:subdivisionExtension}). Finally, given a sieve-cosieve decomposition of $S$ and suitable hypotheses on $p: C \to S$, we establish localizations in Prop.~\ref{prp:ExistenceLaxRightKanExtension}, Cor.~\ref{cor:openPartOfRecollement}, and Prop.~\ref{prp:closedPartOfRecollement}, and show that these together constitute a recollement of the right-lax limit of $p$ in Thm.~\ref{thm:RecollementRlaxLimitOfLlaxFunctor}.

\subsubsection{Relative right Kan extension}

In \cite[Prop.~4.3.1.10]{HTT}, Lurie gives a criterion for when a colimit diagram in a fiber of a locally cocartesian fibration is a relative colimit. In contrast, we will also need a separate understanding of when a \emph{limit} diagram in a fiber is a relative limit. As indicated in Lem.~\ref{lem:LimitIsRelativeLimit}, in this situation we can give an unconditional statement.

\begin{lem} \label{lem:LimitIsRelativeLimit} Let $S$ be an $\infty$-category and let $f: C \to S$ be a locally cocartesian fibration. Let $s \in S$ be an object and $\overline{p}: K^{\lhd} \to C_s$ a limit diagram that extends $p$. Then, viewed as a diagram in $C$, $\overline{p}$ is a $f$-limit diagram \cite[4.3.1.1]{HTT}, i.e., the commutative square
\[ \begin{tikzcd}[row sep=4ex, column sep=4ex, text height=1.5ex, text depth=0.25ex]
C_{/\overline{p}} \ar{r} \ar{d} & C_{/p} \ar{d} \\
S_{/f \overline{p}} \ar{r} & S_{/f p}
\end{tikzcd} \]
is a homotopy pullback square.
\end{lem}
\begin{proof} It suffices to show that $C_{/\overline{p}} \to C_{/p} \times_{S_{/f p}} S_{/f \overline{p}}$ is a trivial Kan fibration. To this end, let $A \to B$ be a monomorphism of simplicial sets and consider the lifting problem
\[ \begin{tikzcd}[row sep=4ex, column sep=4ex, text height=1.5ex, text depth=0.25ex]
A \ar{r} \ar{d} & C_{/\overline{p}} \ar{d} \\
B \ar{r} \ar[dotted]{ru} &  C_{/p} \times_{S_{/f p}} S_{/f \overline{p}}.
\end{tikzcd} \]
This transposes to the lifting problem
\[ \begin{tikzcd}[row sep=4ex, column sep=4ex, text height=1.5ex, text depth=0.25ex]
A \star K^{\lhd} \bigcup_{A \star K} B \star K \ar{r}{\beta} \ar{d} & C \ar{d}{f} \\
B \star K^{\lhd} \ar{r}[swap]{\alpha} \ar[dotted]{ru}[swap]{\gamma} & S.
\end{tikzcd} \]
Our approach will be to first pushforward to the fiber $C_s$ using that $f$ is a locally cocartesian fibration and then solve the lifting problem in $C_s$ using that $\overline{p}$ is a limit diagram.

To begin, because $\overline{p}$ is a diagram in the fiber $C_s$, the map $\alpha$ factors as $B \star K^\rhd \to B \star \Delta^0 \xto{\alpha'} S$ with $\alpha'|_{\Delta^0} = \{s\}$. We may define a map $r: (B \star \Delta^0) \times \Delta^1 \to B \star \Delta^0$ such that $r_0 = \id$ and $r_1$ is constant at $\Delta^0$ in the following way: let $\pi: B \star \Delta^0 \to \Delta^1$ be the structure map of the join which sends $B$ to $\{0\}$ and $\Delta^0$ to $\{1\}$, and let $\rho$ be the composite $(B \star \Delta^0) \times \Delta^1 \xto{\pi \times \id} \Delta^1 \times \Delta^1 \xto{\text{max}} \Delta^1$, so the fiber of $\rho$ over $\{0\}$ is $B \times \{0\}$. Then, recalling that maps $L \to X \star Y$ of simplicial sets over $\Delta^1$ are equivalently specified by pairs of maps $(f_0:L_0 \to X, f_1: L_1 \to Y)$, $r$ is the map over $\Delta^1$ with respect to $\rho$ and $\pi$ given by $B \subset B \star \Delta^0$ and the constant map to $\Delta^0$. Now let
\[ h^\alpha: (B \star K^{\lhd}) \times \Delta^1 \to (B \star \Delta^0) \times \Delta^1 \xto{r} B \star \Delta^0 \xto{\alpha'} S,  \]
so $h^\alpha_0 = \alpha$ and $h^\alpha_1$ is constant at $\{s\}$. Also denote by $h^\alpha$ the restrictions of $h^\alpha$ to $(B \star K) \times \Delta^1$, $(A \star K^{\lhd}) \times \Delta^1$, and $(A \star K) \times \Delta^1$.

Let $\mathfrak{P} = (M_S,T,\emptyset)$ be the categorical pattern on $s\Set^+_{/S}$ that yields the locally cocartesian model structure, so $M_S$ consists of all the edges in $S$, $T$ consists of all the degenerate $2$-simplices in $S$, and the fibrant objects are the locally cocartesian fibrations. By the criterion of \cite[Lem.~B.1.10]{HA} applied to $K \to B \star K$ (with the degenerate edges marked) and $\{0\} \to (\Delta^1)^\sharp$, the inclusion map of marked simplicial sets
\[ (B \star K) \times \{0\} \cup_{(K \times \{0\})} K \times (\Delta^1)^\sharp \to (B \star K) \times (\Delta^1)^\sharp \]
is $\mathfrak{P}$-anodyne, and likewise replacing $K \to B \star K$ with $K^\lhd \to A \star K^\lhd$ and $K \to A \star K$. Using left properness of the locally cocartesian model structure, we deduce that the morphism
\[ \begin{tikzcd}[row sep=4ex, column sep=4ex, text height=1.5ex, text depth=0.25ex]
(A \star K^\lhd \cup_{A \star K} B \star K) \times \{0\} \cup_{K^\lhd \times \{0\}} K^\lhd \times (\Delta^1)^\sharp \ar{d} \\
(A \star K^\lhd \cup_{A \star K} B \star K) \times (\Delta^1)^\sharp
\end{tikzcd} \] 
is $\mathfrak{P}$-anodyne. Consider the commutative square
\[ \begin{tikzcd}[row sep=4ex, column sep=4ex, text height=1.5ex, text depth=0.25ex]
(A \star K^\lhd \cup_{A \star K} B \star K) \times \{0\} \cup_{K^\lhd \times \{0\}} K^\lhd \times (\Delta^1)^\sharp \ar{d} \ar{r} & \leftnat{C} \ar{d}{f} \\
(A \star K^\lhd \cup_{A \star K} B \star K) \times (\Delta^1)^\sharp \ar{r}[swap]{h^{\alpha}} \ar[dotted]{ru}[swap]{h^\beta} & S^\sharp
\end{tikzcd} \]
where the top horizontal map restricted to the first factor is $\beta$ and to the second factor $K^\lhd \times (\Delta^1)^\sharp$ is the constant homotopy $K^{\lhd} \times \Delta^1 \xto{\pr} K^{\lhd} \xto{\overline{p}} C$. Then the dotted lift $h^\beta$ exists, and the image of $h^\beta_1$ is contained in the fiber $C_s$.

Now consider the commutative triangle
\[ \begin{tikzcd}[row sep=4ex, column sep=4ex, text height=1.5ex, text depth=0.25ex]
A \star K^\lhd \cup_{A \star K} B \star K \ar{r}{h^\beta_1} \ar{d} & C_s \\
B \star K^{\lhd} \ar[dotted]{ru}[swap]{\gamma_1}
\end{tikzcd} \]
Because $\overline{p}: K^\lhd \to C_s$ is a limit diagram, the map $(C_s)_{/\overline{p}} \to (C_s)_{/p}$ is a trivial Kan fibration. Therefore, the dotted lift $\gamma_1$ exists. 

Next, define a map
\[ \theta = (\theta', \theta''): (B \times \Delta^1) \star K^{\lhd} \to (B \star K^{\lhd}) \times \Delta^1 \]
by its factors
\begin{align*} \theta' &: (B \times \Delta^1) \star K^{\lhd} \xto{\pr \star \id} B \star K^{\lhd} \\
\theta'' &: (B \times \Delta^1) \star K^{\lhd} \xto{\pr \star \id} \Delta^1 \star K^\lhd \to \Delta^1 \star \Delta^0 \cong \Delta^2 \xto{\sigma^1} \Delta^1.
\end{align*}
Here $\sigma^1: \Delta^2 \to \Delta^1$ is the standard degeneracy map, so $\sigma^1(0) = 0$, $\sigma^1(1) = 1$, and $\sigma^1(2) = 1$. Also denote by $\theta$ the restriction to $(A \times \Delta^1) \star K^{\lhd}$, etc. Let
 \[ X = (A \times \Delta^1) \star K^{\lhd} \cup_{(A \times \Delta^1) \star K} (B \times \Delta^1) \star K  \bigcup\limits_{ (A \times \{1\}) \star K^{\lhd} \cup_{(A \times \{1\}) \star K} (B \times \{1 \}) \star K} B \star K^{\lhd} \]
 and consider the commutative diagram
\[ \begin{tikzcd}[row sep=4ex, column sep=4ex, text height=1.5ex, text depth=0.25ex]
 X \ar{r}{(h^\beta \circ \theta) \cup \gamma_1} \ar{d}[swap]{\lambda} & C \ar{d}{f} \\
 (B \times \Delta^1) \star K^{\lhd} \ar{r}[swap]{h^{\alpha} \circ \theta} \ar[dotted]{ru}[swap]{h^\gamma} & S
\end{tikzcd} \]
(where for commutativity, we use that $\theta_1: (B \times \{1\}) \star K^{\lhd} \to (B \star K^{\lhd}) \times \{1\}$ is an isomorphism). By the dual of \cite[Lem.~2.1.2.4]{HTT} applied to $A \to B$ and the right anodyne map $\{1\} \to \Delta^1$, the map
\[ \lambda': A \times \Delta^1 \cup_{A \times \{1\}} B \times \{1\} \to B \times \Delta^1 \]
is right anodyne. Then by \cite[Lem.~2.1.2.3]{HTT} applied to $\lambda'$ and the map $K \to K^{\lhd}$, $\lambda$ is inner anodyne. Thus the dotted lift $h^\gamma$ exists. Finally, let $\gamma = h^\gamma_0$ and observe that $\gamma$ is a solution to the original lifting problem of interest.
\end{proof}

We briefly digress to complete the theory of Kan extensions by constructing relative Kan extensions along general functors (c.f. Lurie's remark at the beginning of \cite[\S 4.3.3]{HTT}). Recall the relative join construction $- \star_{-} -$ of \cite[Def.~4.1]{Exp2} along with its bifibration property \cite[Lem.~4.8]{Exp2}.

\begin{dfn} \label{Dfn:RelativeKanExtension} Consider the commutative diagram of $\infty$-categories
\[ \begin{tikzcd}[row sep=4ex, column sep=4ex, text height=1.5ex, text depth=0.25ex]
X \ar{r}{F} \ar{d}{\phi} & C \ar{d}{p} \\
Y \ar{r}{\alpha} & S
\end{tikzcd} \]
where $p: C \to S$ is a categorical fibration. Suppose given the data of a functor $G: Y \to C$ over $S$ and a homotopy $h: X \times \Delta^1 \to C$ over $S$ with $h_0 = G \circ \phi$ and $h_1 = F$. Let $\pi: Y \star_Y X \to Y$ be the structure map and let $\overline{G}: Y \star_Y X \xto{\pi} Y \xto{G} C$. Since $\Fun(Y \star_Y X, C) \to \Fun(Y, C) \times \Fun(X, C)$ is a bifibration, we may select an edge $\overline{G} \to \overline{F}$ that is cocartesian over $h: G \circ \phi \to F$ in $\Fun(X,C)$ with degenerate image $\id_{G}$ in $\Fun(Y,C)$. Then we say that $G$ is a \emph{$p$-right Kan extension of $F$} along $\phi$ (exhibited via $h$) if the commutative diagram
\[ \begin{tikzcd}[row sep=4ex, column sep=4ex, text height=1.5ex, text depth=0.25ex]
X \ar{r}{F} \ar[hookrightarrow]{d}{\iota_X} & C \ar{d}{p} \\
Y \star_Y X \ar{r}{\alpha \circ \pi} \ar{ru}{\overline{F}} & S
\end{tikzcd} \]
exhibits $\overline{F}$ as a $p$-right Kan extension of $F$ in the sense of \cite[Def.~4.3.2.2]{HTT}.
\end{dfn}

\begin{rem} \label{rem:ExistenceOfRelativeRKE} In the initial setup of Def.~\ref{Dfn:RelativeKanExtension}, given $\overline{F}: Y \star_Y X \to C$ a map over $S$ extending $F: X \to C$, let $G = \overline{F}|_Y: Y \to C$ and let $h: X \times \Delta^1 \xto{h'} Y \star_Y X \xto{\overline{F}} C$ with $h'$ specified by the pair $(\phi, \id_Y)$ (c.f. the definition \cite[Def.~4.1]{Exp2} of $- \star_Y -$ as $j_\ast: s\Set_{/Y \times \partial \Delta^1} \to s\Set_{/Y \times \Delta^1}$ for the inclusion $j: Y \times \partial \Delta^1 \to Y \times \Delta^1$). Then $\overline{F}$ is a $p$-right Kan extension in the sense of \cite[Def.~4.3.2.2]{HTT} if and only if $G$ is a $p$-right Kan extension along $\phi$ in the sense of Def.~\ref{Dfn:RelativeKanExtension}. Moreover, we have an equivalence of $\infty$-categories $X \times_{Y \star_Y X} (Y \star_Y X)_{y/} \simeq X \times_Y Y_{y/}$ implemented by pulling back the functors $\iota_Y: Y \subset Y \star_Y X$ and $\pi: Y \star_Y X \to Y$ and the respective induced functors on the slice categories via $X \subset Y \star_Y X$. Because of this, Lurie's existence and uniqueness theorem \cite[Prop.~4.3.2.15]{HTT} for $p$-right Kan extensions applies to show that the $p$-right Kan extension $G$ of $F$ along $\phi$ exists if and only if for every $y \in Y$, the diagram $X \times_Y Y_{y/} \to X \xto{F} C$ extends to a $p$-limit diagram (which then computes the value of $G$ on $y$). Moreover, there is then a contractible space of choices for $G$.
\end{rem}

\begin{rem} \label{rem:AdjunctionForRKE} The situation of Def.~\ref{Dfn:RelativeKanExtension} globalizes in the following manner. Suppose every functor $F: X \to C$ admits a $p$-right Kan extension to $\overline{F}: Y \star_Y X \to C$. By \cite[Prop.~4.3.2.17]{HTT}, the restriction functor $(\iota_X)^\ast: \Fun_{/S}(Y \star_Y X, C) \to \Fun_{/S}(X,C)$ then admits a right adjoint $(\iota_X)_\ast$ which is computed on objects as $F \mapsto \overline{F}$. We also have a relative adjunction (\cite[Def.~7.3.2.2]{HA}) $$\adjunct{\iota_Y}{Y}{Y \star_Y X}{\pi}$$ over $Y$ (hence over $S$) where $\iota_Y$ is left adjoint to $\pi$. From this, we obtain an adjunction $$\adjunct{\pi^\ast}{\Fun_{/S}(Y,C)}{\Fun_{/S}(Y \star_Y X,C)}{(\iota_Y)^\ast}$$ where $\pi^\ast$ is left adjoint to $(\iota_Y)^\ast$. Composing these two adjunctions, we obtain the adjunction
\[ \adjunct{\phi^\ast}{\Fun_{/S}(Y,C)}{\Fun_{/S}(X,C)}{\phi_\ast} \]
where $\phi_{\ast}$ is given on objects by sending $F$ to its $p$-right Kan extension along $\phi$.
\end{rem}

\begin{cor} \label{cor:RightKanExtensionComputedInFiber} Suppose we have a commutative diagram of $\infty$-categories
\[ \begin{tikzcd}[row sep=4ex, column sep=4ex, text height=1.5ex, text depth=0.25ex]
X \ar{r}{F} \ar{d}{\phi} & C \ar{d}{p} \\
Y \ar{r}{\alpha} & S
\end{tikzcd} \]
where $p$ is a locally cocartesian fibration and $\phi$ is a cartesian fibration. Suppose that for every $y \in Y$, the limit of $F|_{X_y}: X_y \to C_{\alpha(y)}$ exists. Then the $p$-right Kan extension $G: Y \to C$ of $F$ along $\phi$ exists and $G(y) \simeq \lim\limits_{\ot} F|_{X_y}$. If $G$ exists for all $F$, then we have an adjunction
\[ \adjunct{\phi^\ast}{\Fun_{/S}(Y,C)}{\Fun_{/S}(X,C)}{\phi_\ast} \]
where $\phi_{\ast}(F) \simeq G$.
\end{cor}
\begin{proof} We need to show that for every $y \in Y$, the $p$-limit of $F^y: X \times_Y Y_{y/} \to X \xto{F} C$ exists. By Lem.~\ref{lem:LimitIsRelativeLimit}, the $p$-limit of $F|_{X_y}$ exists and is computed as the limit of $F|_{X_y}$ viewed as a diagram in $C_{\alpha(y)}$. Because $\phi$ is a cartesian fibration, we have a retraction $r: X \times_Y Y_{y/} \to X_y$ to the inclusion $i: X_y \to X \times_Y Y_{y/}$ such that $r$ is right adjoint to $i$ (on objects, $r$ is given by the formula $r(x,y \xto{e} \phi(x)) = e^\ast(x)$, where $e^\ast: X_{\phi(x)} \to X_y$ is the pullback functor encoded by the lifting property of the cartesian fibration $\phi$). As a left adjoint, $i$ is right cofinal.\footnote{We adopt Lurie's terminology in \cite{HA}: recall that a map $q: K \to L$ is right cofinal if and only if $q^{\op}$ is cofinal.} However, since $r \circ i = \id$, we moreover have that $r$ is right cofinal by the right cancellative property of right cofinal maps \cite[Prop.~4.1.1.3(2)]{HTT}. Hence, by \cite[Prop.~4.3.1.7]{HTT} applied to $r$ and a $p$-limit diagram $(X_y)^{\lhd} \to C$, the $p$-limit of $F^y$ exists and is computed as the limit of $F|_{X_y}$ in $C_{\alpha(y)}$. The claim now follows from Rmk.~\ref{rem:ExistenceOfRelativeRKE}.
\end{proof}

\subsubsection{Barycentric subdivision and locally cocartesian pushforward}

Let $\Delta$ be the category with objects the finite ordinals $\{[n] = \{0<1< ... < n\} : n \in \NN \}$ and morphisms the order-preserving maps. Let $\xi: \cE \Delta \to \Delta$ denote the relative nerve \cite[Def.~3.2.5.2]{HTT} of the canonical inclusion $i: \Delta \to s\Set$. Then $\xi$ is a cocartesian fibration classified by $i$, which is an explicit model for the tautological cocartesian fibration over $\Delta$. Explicitly, an $n$-simplex $\Delta^n \to \cE \Delta$ is given by a sequence $[a_0] \xto{\alpha_0} [a_1] \xto{\alpha_1} ... \xto{g_{n-1}} [\alpha_n]$ of order-preserving maps in $\Delta$ together with morphisms $\kappa_i: \Delta^{\{0,...,i\}} \cong \Delta^i \to \Delta^{a_i}$ which fit into a commutative diagram
\[ \begin{tikzcd}[row sep=4ex, column sep=4ex, text height=1.5ex, text depth=0.25ex]
\Delta^{\{0\}} \ar[hookrightarrow]{r} \ar{d}{\kappa_0} & \Delta^{\{0,1\}} \ar[hookrightarrow]{r} \ar{d}{\kappa_1} & \cdots \ar[hookrightarrow]{r} & \Delta^{\{0,...,n-1\}} \ar[hookrightarrow]{r} \ar{d}{\kappa_{n-1}} & \Delta^n \ar{d}{\kappa_n} \\
\Delta^{a_0} \ar{r}{\alpha_0} & \Delta^{a_1} \ar{r}{\alpha_1} & \cdots \ar{r} & \Delta^{a_{n-1}} \ar{r}{\alpha_{n-1}} & \Delta^{a_n}.
\end{tikzcd} \]
Let $\cE \Delta^{\inj} \subset \cE \Delta$ denote the pullback over the subcategory $\Delta^{\inj} \subset \Delta$ of injective order-preserving maps and also denote the structure map of $\cE \Delta^{\inj}$ by $\xi$. Consider the span of marked simplicial sets
\[ \begin{tikzcd}[row sep=4ex, column sep=4ex, text height=1.5ex, text depth=0.25ex]
(\Delta^{\inj})^{\sharp} &  \leftnat{(\cE \Delta^{\inj})} \ar{r}{\xi} \ar{l}[swap]{\xi} & (\Delta^{\inj})^{\sharp}
\end{tikzcd} \]
where we mark the $\xi$-cocartesian edges in $\cE \Delta^{\inj}$. Similar to the definition in \cite[Exm.~2.24]{Exp2} (which considers the source input to be instead a cartesian fibration), let $$\widetilde{\Fun}_{\Delta^{\inj}}(\cE \Delta^{\inj},-) \coloneq \xi_{\ast} \xi^{\ast}(-): s\Set^+_{/\Delta^{\inj}} \to s\Set^+_{/\Delta^{\inj}}.$$

Note that with $\xi$ a cocartesian fibration, $\xi_{\ast} \xi^{\ast}$ is right Quillen with respect to the \emph{cartesian} model structure on $s\Set^+_{/ \Delta^{\inj}}$ by the dual of \cite[Thm.~2.23]{Exp2}.

\begin{dfn} The $\infty$-category of \emph{paths}\footnote{For us, a path in $C$ is any $n$-simplex $\Delta^n \to C$. In contrast, we reserve the term `string' for objects of the barycentric subdivision $\sd(C)$ (c.f. Def.~\ref{dfn:barycentricSubdivision}).} in an $\infty$-category $C$ is
$$\widehat{\sO}(C) = \widetilde{\Fun}_{\Delta^{\inj}}(\cE \Delta^{\inj}, C \times \Delta^{\inj}).$$
Let $\xi_C: \widehat{\sO}(C) \to \Delta^{\inj}$ denote the structure map of the cartesian fibration and note that its fiber over $[n] \in \Delta^{\inj}$ is $\Fun(\Delta^n,C)$.

In addition, let $\widehat{\sO}^{\simeq}(S) \subset \widehat{\sO}(S)$ be the wide subcategory on the $\xi_S$-cartesian edges over $\Delta^{\inj}$ (so the fiber of $\widehat{\sO}^{\simeq}(S)$ over $[n]$ is $\Map(\Delta^n,S)$), and for a functor $p: C \to S$, let $$\widehat{\sO}^{\simeq}_S(C) = \widehat{\sO}^{\simeq}(S) \times_{\widehat{\sO}(S)} \widehat{\sO}(C).$$
\end{dfn}

Intuitively, under the straightening correspondence $\xi_C$ is classified by the functor $(\Delta^{\inj})^{\op} \to \Cat_{\infty}$ that sends $[n]$ to $\Fun(\Delta^n,C)$ and is functorial with respect to precomposition in the first variable; we will not need a precise articulation of this fact.

\begin{rem} If $C \to S$ is a categorical fibration, then $\widehat{\sO}(C) \to \widehat{\sO}(S)$ is also a categorical fibration by \cite[Thm.~2.23]{Exp2} and \cite[B.2.7]{HA}.
\end{rem}

\begin{cnstr}[Variants associated to a sieve] \label{cnstr:sieveVariantsPathCategories} Let $\pi: S \to \Delta^1$ be a functor and $S_0$ the fiber over $0$. Let $\widehat{\sO}(S)_0 \subset \widehat{\sO}(S)$ be the full subcategory on those objects $\sigma: \Delta^n \to S$ such that $\pi \sigma(0) = 0$, and let $\widehat{\sO}^\simeq(S)_0 = \widehat{\sO}(S)_0 \cap \widehat{\sO}^{\simeq}(S)$. Define the `initial segment' functor $$\lambda_S: \widehat{\sO}(S)_0 \to \widehat{\sO}(S_0)$$ by the following rule: 

\begin{itemize} \item[($\ast$)] Suppose $\sigma: \Delta^n \to \widehat{\sO}(S)_0$ is a $n$-simplex, which corresponds to a sequence of inclusions
 \[ \begin{tikzcd}[row sep=4ex, column sep=4ex, text height=1.5ex, text depth=0.25ex]
\Delta^{a_0} \ar[hookrightarrow]{r}{\alpha_1} & \Delta^{a_1} \ar[hookrightarrow]{r}{\alpha_2} & \cdots \ar[hookrightarrow]{r}{\alpha_n} & \Delta^{a_n}
\end{tikzcd} \]
determining a map $a: \Delta^n \to \Delta^{\inj}$ and a functor $f: \Delta^n \times_{a, \Delta^{\inj}} \cE \Delta^{\inj} \to S$ such that for every $0 \leq i \leq n$, the restriction $f_i: \Delta^{a_i} \to S$ has $f_i(0) \in S_0$. Let $b_i \in \Delta^{a_i}$ be the maximum element such that $f_i(b_i) \in S_0$, and note that $a$ restricts to yield a sequence of inclusions  
\[ \begin{tikzcd}[row sep=4ex, column sep=4ex, text height=1.5ex, text depth=0.25ex]
\Delta^{b_0} \ar[hookrightarrow]{r}{\beta_1} \ar[hookrightarrow]{d} & \Delta^{b_1} \ar[hookrightarrow]{r}{\beta_2} \ar[hookrightarrow]{d} & \cdots \ar[hookrightarrow]{r}{\beta_n} & \Delta^{b_n} \ar[hookrightarrow]{d} \\
\Delta^{a_0} \ar[hookrightarrow]{r}{\alpha_1} & \Delta^{a_1} \ar[hookrightarrow]{r}{\alpha_2} & \cdots \ar[hookrightarrow]{r}{\alpha_n} & \Delta^{a_n}
\end{tikzcd} \]
because we always have that $\alpha_i(b_{i-1}) \leq b_i$ as $S_0$ is a sieve in $S$ stable under equivalences. Let $b: \Delta^n \to \Delta^{\inj}$ be the map determined by the sequence of upper horizontal inclusions. $f$ then restricts to yield a map $f_0$:
\[ \begin{tikzcd}[row sep=4ex, column sep=4ex, text height=1.5ex, text depth=0.25ex]
\Delta^n \times_{b, \Delta^{\inj}} \cE \Delta^{\inj} \ar{r}{f_0} \ar[hookrightarrow]{d} & C_0 \ar[hookrightarrow]{d} \\
\Delta^n \times_{a, \Delta^{\inj}} \cE \Delta^{\inj} \ar{r}{f} & C.
\end{tikzcd} \]
Define $\lambda_S(\sigma): \Delta^n \to \widehat{\sO}(S_0)$ to be the $n$-simplex determined by $f_0$. Now observe that this assignment is natural in $\Delta^n$, hence defines a map of simplicial sets.
\end{itemize}

Observe that $\lambda_S$ is a retraction of the inclusion $\widehat{\sO}(S_0) \to \widehat{\sO}(S)_0$ induced by $S_0 \to S$.

An edge $e: \Delta^1 \to \widehat{\sO}(S)_0$ is $\xi_S$-cartesian if and only if the corresponding functor $f: \Delta^1 \times_{a, \Delta^{\inj}} \cE \Delta^{\inj} \to S$ sends every edge $(i \in [a_0]) \to (\alpha_1(i) \in [a_1])$ to an equivalence, and similarly for $\xi_{S_0}$-cartesian edges in $\widehat{\sO}(S_0)$. Therefore, $\lambda_S$ preserves cartesian edges and restricts to a map $$\lambda_S: \widehat{\sO}^{\simeq}(S)_0 \to \widehat{\sO}^{\simeq}(S_0).$$

Now let $p: C \to S$ be a locally cocartesian fibration and let $p_0: C_0 \to S_0$ be its fiber over $0$. Let $$\widehat{\sO}^{\simeq}_S(C)_0 = \widehat{\sO}^{\simeq}(S)_0 \times_{\widehat{\sO}(S)_0} \widehat{\sO}(C)_0,$$ so $\widehat{\sO}^{\simeq}_S(C)_0 \subset \widehat{\sO}^{\simeq}_S(C)$ is the full subcategory on objects $c: \Delta^n \to C$ with $c(0) \in C_0$.  The initial segment functor $\lambda_{(-)}$ fits into a commutative diagram
\[ \begin{tikzcd}[row sep=4ex, column sep=4ex, text height=1.5ex, text depth=0.25ex]
\widehat{\sO}^{\simeq}(S)_0 \ar[hookrightarrow]{r} \ar{d}{\lambda_S} & \widehat{\sO}(S)_0 \ar{d}{\lambda_S} & \widehat{\sO}(C)_0 \ar{l}[swap]{p} \ar{d}{\lambda_C} \\
\widehat{\sO}^{\simeq}(S_0) \ar[hookrightarrow]{r} & \widehat{\sO}(S_0) & \widehat{\sO}(C_0) \ar{l}[swap]{p_0}
\end{tikzcd} \]

and therefore defines a functor $\lambda_{p}: \widehat{\sO}_S^{\simeq}(C)_0 \to \widehat{\sO}_{S_0}^{\simeq}(C_0)$.

Finally, let $\widehat{\sO}^{\simeq}_S(C)_0^{\cocart} \subset \widehat{\sO}^{\simeq}_S(C)_0$ be the full subcategory on those objects $c: \Delta^n \to C$ such that if $i \in \Delta^n$ is the maximum element with $c(i) \in C_0$, then $c$ sends every edge $\{j, j+1 \}$, $j \geq i$ to a locally-$p$ cocartesian edge (i.e., a cocartesian edge over $\Delta^1$ in the pullback $\Delta^1 \times_S C$).
\end{cnstr}

The subsequent proposition indicates that we can construct a `locally cocartesian pushforward' extending from $C_0$ to $C$ along paths in the base $S$ that originate in $S_0$.

\begin{prp} \label{prp:LocallyCocartesianPushforward} The map $(\lambda_{p}, p): \widehat{\sO}^{\simeq}_S(C)_0^{\cocart} \to \widehat{\sO}^{\simeq}_{S_0}(C_0) \times_{p_0, \widehat{\sO}^{\simeq}(S_0), \lambda_S} \widehat{\sO}^{\simeq}(S)_0$ is a trivial fibration of simplicial sets.
\end{prp}
\begin{proof} We need to solve the lifting problem
\[ \begin{tikzcd}[row sep=4ex, column sep=4ex, text height=1.5ex, text depth=0.25ex]
\partial \Delta^n \ar[hookrightarrow]{d} \ar{r} & \widehat{\sO}^{\simeq}_S(C)_0^{\cocart} \ar{d}{(\lambda_p,p)} \\
\Delta^n \ar{r} \ar[dotted]{ru} &  \widehat{\sO}^{\simeq}_{S_0}(C_0) \times_{\widehat{\sO}^{\simeq}(S_0)} \widehat{\sO}^{\simeq}(S)_0.
\end{tikzcd} \]
Let $a: \Delta^n \to \widehat{\sO}^{\simeq}(S)_0 \to \Delta^{\inj}$ and $b: \Delta^n \to \widehat{\sO}^{\simeq}_{S_0}(C_0) \to \Delta^{\inj}$ be as discussed in the definition of $\lambda$. This lifting problem transposes to
\[ \begin{tikzcd}[row sep=4ex, column sep=4ex, text height=1.5ex, text depth=1ex]
\Delta^n \times_{b, \Delta^{\inj}} \cE \Delta^{\inj} \bigcup_{\partial \Delta^n \times_{b, \Delta^{\inj}} \cE \Delta^{\inj}} \partial \Delta^n \times_{a, \Delta^{\inj}} \cE \Delta^{\inj} \ar{r} \ar[hookrightarrow]{d}{f} & C \ar{d}{p} \\
\Delta^n \times_{a, \Delta^{\inj}} \cE \Delta^{\inj} \ar{r} \ar[dotted]{ru} & S.
\end{tikzcd} \]
Consider $\Delta^n \times_{a, \Delta^{\inj}} \cE \Delta^{\inj}$ as a marked simplicial set where an edge $(i \in \Delta^{a_k}) \to (j \in \Delta^{a_l})$, $\alpha: \Delta^{a_k} \to \Delta^{a_l}$, $\alpha(i) \leq j$ is marked if and only if $k = l$ (so $\alpha = \id$), $b_k \leq i$ and $j = i+1$, and let the domain of $f$ also inherit this marking. Then it suffices to show that $f$ is a trivial cofibration in the locally cocartesian model structure on $s\Set^+_{/S}$, defined by the categorical pattern $\mathfrak{P} = (M_S,T,\emptyset)$ with $M_S$ all of the edges in $S$ and $T$ consisting of the $2$-simplices $\tau$ in $S$ with the edge $\tau(\{1,2\})$ an equivalence. Proceeding by induction on $n$, by a two-out-of-three argument it suffices to show that the inclusion $f': \Delta^n \times_{b, \Delta^{\inj}} \cE \Delta^{\inj} \to \Delta^n \times_{a, \Delta^{\inj}} \cE \Delta^{\inj}$ is a trivial cofibration. We define a filtration of the poset inclusion $f'$ as follows:
\begin{itemize}
    \item[($\ast$)] Let $a_n - b_n = t$. For $0 \leq k \leq n$, let $\alpha_k: \Delta^{a_k} \to \Delta^{a_n}$ denote the inclusion. Let $P_r \subset \Delta^n \times_{a, \Delta^{\inj}} \cE \Delta^{\inj}$ be the subposet on those objects $(i \in \Delta^{a_k})$ such that $\alpha_k(i)-b_n \leq r$. Note that $P_0 = \Delta^n \times_{b, \Delta^{\inj}} \cE \Delta^{\inj}$, because if $(i \in \Delta^{a_k})$ is such that $i > b_k$, then necessarily $\alpha_k(i) > b_n$, and likewise if $i \leq b_k$, then $\alpha_k(i) \leq b_n$ (this follows from the definitions of the $b_i$ and that $S_0$ is a sieve stable under equivalences). Then we have that $f'$ factors as a sequence of poset sieve inclusions $\Delta^n \times_{b, \Delta^{\inj}} \cE \Delta^{\inj} = P_0 \subset P_1 \subset \cdots \subset P_t = \Delta^n \times_{a, \Delta^{\inj}} \cE \Delta^{\inj}$.
\end{itemize}
It now suffices to show that $P_i \subset P_{i+1}$ is a trivial cofibration for all $0 \leq i < t$. For simplicity, let us suppose $i=0$ (and $t>0$ for non-triviality), the other cases being proved similarly. Let $k \in [n]$ be the smallest element such that $b_n + 1 \in \Delta^{a_n}$ is in the image of $\alpha_k: \Delta^{a_k} \to \Delta^{a_n}$. Note then that for all $k \leq l \leq n$, $\alpha_l(b_l+1) = b_n+1$. View the poset $\Delta^{\{k,...,n\}} \times \Delta^1$ as a cosieve $U$ in $P_1$ via the inclusion which sends $(l,0)$ to $(b_l \in \Delta^{a_l})$ and $(l,1)$ to $(b_l + 1 \in \Delta^{a_l})$. Then as a marked simplicial set, we have $U = (\Delta^{\{k,...,n\}})^{\flat} \times (\Delta^1)^{\sharp}$. By \cite[B.1.10]{HA}, the inclusion
\[ U \cap P_0 = (\Delta^{\{k,...,n\}})^{\flat} \times \{0\} \to U = (\Delta^{\{k,...,n\}})^{\flat} \times (\Delta^1)^{\sharp} \]
is $\mathfrak{P}$-anodyne. Noting that $P_0$ and $U$ together cover $P_1$, it thus suffices to show that we have a homotopy pushout square of $\infty$-categories
\[ \begin{tikzcd}[row sep=4ex, column sep=4ex, text height=1.5ex, text depth=0.25ex]
U \cap P_0 \ar{r} \ar{d} & U \ar{d} \\
P_0 \ar{r} & P_1
\end{tikzcd} \]
as we would then deduce the lower horizontal map to be $\mathfrak{P}$-anodyne. For this, the criterion of Lem.~\ref{lem:posetPushoutViaFlatness} is easily verified.
\end{proof}

\begin{lem} \label{lem:posetPushoutViaFlatness} Suppose $P$ is a poset, $Z \subset P$ is a sieve and $U \subset P$ is a cosieve such that $P = Z \cup U$. Then the commutative square
\[ \begin{tikzcd}[row sep=4ex, column sep=4ex, text height=1.5ex, text depth=0.25ex]
U \cap Z \ar{r} \ar{d} & U \ar{d} \\
Z \ar{r} & P
\end{tikzcd} \]
is a homotopy pushout square of $\infty$-categories if and only if for every $a \notin U$ and $c \notin Z$ such that $a \leq c$, the subposet $P_{a//c} = \{ b \in U \cap Z : a \leq b \leq c\}$ is weakly contractible.
\end{lem}
\begin{proof}  Define a map $\pi: P \to \Delta^2$ by
\begin{equation*} \pi(x) = \begin{cases} 0 & x \notin U \\
2 & x \notin Z \\
1 & x \in U \cap Z 
\end{cases}
\end{equation*}
Observe that $P \times_{\Delta^2} \Delta^{\{0,1\}} = Z$,  $P \times_{\Delta^2} \Delta^{\{1,2\}} = U$, and $P \times_{\Delta^2} \{1\} = U \cap Z$. We may therefore apply the flatness criterion of \cite[B.3.2]{HA} to $\pi$ in order to deduce the criterion in question.
\end{proof}

We now introduce our quasi-categorical model of the barycentric subdivision $\sd(S)$.

\begin{dfn} \label{dfn:barycentricSubdivision} An $n$-simplex $\sigma: \Delta^n \to S$ is a \emph{string} if for every $0 \leq i < n$, $\sigma(\{i,i+1\})$ is not an equivalence in $S$. The \emph{barycentric subdivision} (or \emph{subdivision})
\[ \sd(S) \subset \widehat{\sO}^{\simeq}(S) \]
is the full subcategory of $\widehat{\sO}^{\simeq}(S)$ on the strings in $S$.\footnote{Note that given a string $\sigma: \Delta^n \to S$, we may still have that `longer' edges in $\Delta^n$ are sent to equivalences in $S$ by $\sigma$, so $\sd(S)$ may fail to be a cartesian fibration over $\Delta^{\inj}$. However, if every retract in $S$ is an equivalence, then this possibility is excluded, and $\sd(S) \subset \widehat{\sO}^{\simeq}(S) \to \Delta^{\inj}$ remains a cartesian fibration. Our definition of the barycentric subdivision only seems reasonable under this hypothesis, although we do not need to demand it for our theorems.} Given a functor $C \to S$, the \emph{$S$-relative subdivision} $\sd_S(C)$ is the pullback $$\sd(S) \times_{\widehat{\sO}^{\simeq}(S)} \widehat{\sO}^{\simeq}_S(C) \cong \sd(S) \times_{\widehat{\sO}(S)} \widehat{\sO}(C).$$ Similarly, parallel to Constr.~\ref{cnstr:sieveVariantsPathCategories} we may define $\sd(S)_0$, $\sd_S(C)_0$, and $\sd_S(C)_0^{\cocart}$ for a locally cocartesian fibration $C \to S$ and a functor $S \to \Delta^1$.
\end{dfn}

\begin{rem} \label{rem:barycentricSubdivisionOrdinaryCategory} Suppose that $S$ is the nerve of a category, which we also denote as $S$. Then $\sd(S)$ is the nerve of the category whose objects are functors $\sigma: \Delta^n \to S$ such that $\sigma(\{i, i+1 \})$ is not an equivalence in $S$, and where a morphism $[\sigma: \Delta^n \to S] \to [\tau: \Delta^m \to S]$ is given by the data of a map $\alpha: [n] \to [m]$ in $\Delta^{\inj}$ and a natural transformation $\sigma \to \alpha^{\ast} \tau$ through equivalences. In particular, if $S$ is the nerve of a poset $P$, then $\sd(P)$ is the nerve of the usual barycentric subdivision of $P$.

On the other hand, the usual definition of the subdivision of an $\infty$-category \cite[Def.~1.15]{AMGR-NaiveApproach} is as the left Kan extension of the functor $\sd: \Delta \to \Cat_{\infty}$ along the restricted Yoneda embedding $\Delta \subset \Cat_{\infty}$. Although we expect our quasi-categorical definition of the subdivision to recover this more abstract definition, we will not prove this here.
\end{rem}

\begin{cnstr}[Maximum functor] \label{cnstr:MaxFunctorSubdivision} Define a `last vertex' map $\max_S: \widehat{\sO}(S) \to S$ by the following rule:
\begin{itemize} \item[($\ast$)] Suppose $\sigma: \Delta^n \to \widehat{\sO}(S)$ is a $n$-simplex, which corresponds to a sequence of inclusions \[ \begin{tikzcd}[row sep=4ex, column sep=4ex, text height=1.5ex, text depth=0.25ex]
\Delta^{a_0} \ar[hookrightarrow]{r}{\alpha_1} & \Delta^{a_1} \ar[hookrightarrow]{r}{\alpha_2} & \cdots \ar[hookrightarrow]{r}{\alpha_n} & \Delta^{a_n}
\end{tikzcd} \]
determining a map $a: \Delta^n \to \Delta^{\inj}$ and a functor $f: \Delta^n \times_{a, \Delta^{\inj}} \cE \Delta^{\inj} \to S$. Define a functor $\chi: \Delta^n \to \Delta^n \times_{a, \Delta^{\inj}} \cE \Delta^{\inj}$ to be the identity on the first component and the unique $n$-simplex of $\cE \Delta^{\inj}$
\[ \begin{tikzcd}[row sep=4ex, column sep=4ex, text height=1.5ex, text depth=0.25ex]
\Delta^{\{0\}} \ar[hookrightarrow]{r} \ar{d}{\kappa_0} & \Delta^{\{0,1\}} \ar[hookrightarrow]{r} \ar{d}{\kappa_1} & \cdots \ar[hookrightarrow]{r} & \Delta^n \ar{d}{\kappa_n} \\
\Delta^{a_0} \ar[hookrightarrow]{r}{\alpha_1} & \Delta^{a_1} \ar[hookrightarrow]{r}{\alpha_2} & \cdots \ar[hookrightarrow]{r}{\alpha_n} & \Delta^{a_n}.
\end{tikzcd} \]
specified by $\kappa_i(i) = a_i$ on the second component. Then $\max_S(\sigma) = f \circ \chi: \Delta^n \to S$.
\end{itemize}

In other words, $\max_S$ is the functor induced by precomposing by the section $\Delta^{\inj} \to \cE \Delta^{\inj}$ which selects the maximal vertex in every fiber.
\end{cnstr}

\begin{lem} \label{lm:subdivisionLocallyCocartesianByMaxFunctor} \begin{enumerate}[leftmargin=*] \item The functor $\max_S: \widehat{\sO}(S) \to S$ is a categorical fibration.
\end{enumerate}
\begin{enumerate}
\setcounter{enumi}{1}
\item The restricted functor $\max_S: \widehat{\sO}^{\simeq}(S) \to S$ is a locally cocartesian fibration.
\item The restricted functor $\max_S: \sd(S) \to S$ is a locally cocartesian fibration.
\end{enumerate}
\end{lem}
\begin{proof} (1) We first verify that $\max_S$ is an inner fibration. For this, let $n \geq 2$, $0 < k < n$, and consider the lifting problem
\[ \begin{tikzcd}[row sep=4ex, column sep=4ex, text height=1.5ex, text depth=0.25ex]
\Lambda^n_k \ar{r} \ar{d} & \widehat{\sO}(S) \ar{d}{\max_S} \\
\Delta^n \ar{r} \ar[dotted]{ru} & S. 
\end{tikzcd} \]
Let $a: \Delta^n \to \Delta^{\inj}$ be the unique extension of the given $\Lambda^n_k \to \Delta^{\inj}$. The lifting problem then transposes to
\[ \begin{tikzcd}[row sep=4ex, column sep=4ex, text height=1.5ex, text depth=0.25ex]
\Delta^n \bigcup_{\Lambda^n_k} \Lambda^n_k \times_{\Delta^{\inj}} \cE \Delta^{\inj} \ar{r} \ar{d} & S \\
\Delta^n \times_{\Delta^{\inj}} \cE \Delta^{\inj} \ar[dotted]{ru}
\end{tikzcd} \]
and it suffices to show the vertical arrow is inner anodyne. Since $\cE \Delta^{\inj} \to \Delta^{\inj}$ is a cocartesian fibration, it is in particular a flat inner fibration, and the desired result follows. 

We next show that $\max_S$ is a categorical fibration by lifting equivalences from the base. So suppose $e: \Delta^1 \to S$ is an equivalence and $\sigma: \Delta^n \to S$ is an object of $\widehat{\sO}(S)$ such that $\max_S(\sigma) = \sigma(n) = e(0)$. The restriction of $\max_S$ to $\Fun(\Delta^n,S) \subset \widehat{\sO}(S)$ is evaluation at $\{n\}$, which is a categorical fibration, so $e$ lifts to an equivalence in $\Fun(\Delta^n,S)$ and hence in $\widehat{\sO}(S)$.

(2) First observe that since $\widehat{\sO}^{\simeq}(S) \subset \widehat{\sO}(S)$ is a subcategory stable under equivalences, the restricted $\max_S$ functor is a categorical fibration by (1). To prove that $\max_S$ is a locally cocartesian fibration, it then suffices to prove that for any edge $e: s \to t$ in $S$ that is \emph{not} an equivalence, the pullback $\max_S(e): \widehat{\sO}^{\simeq}(S) \times_{S} \Delta^1 \to \Delta^1$ is a cocartesian fibration. To this end, we claim that an edge $\widetilde{e}: x \to y$ lifting $e$ is $\max_S(e)$-cocartesian if and only if the corresponding data of an inclusion $\alpha: \Delta^{a_0} \to \Delta^{a_1}$ and a functor $f: \Delta^1 \times_{\Delta^{\inj}} \cE \Delta^{\inj} \to S$ is such that in addition $a_1 = a_0 +1$ and $\alpha$ is the inclusion of the initial segment. Note that given an object $x: \Delta^{a_0} \to S$ with $s= x(a_0)$, such a lift $\widetilde{e}$ of $e$ may be defined by `appending' $e$ to $x$: indeed, let $y: \Delta^{a_0+1} \to S$ be an extension of $x \cup e: \Delta^{a_0} \cup_{a_0,\Delta^0,0} \Delta^1 \to S$, let
\[ r: \Delta^1 \times_{\alpha, \Delta^{\inj}} \cE \Delta^{\inj} \to \Delta^{a_{0}+1} \]
be the retraction functor which fixes $\Delta^{a_{0} + 1}$ and is given by $\alpha$ on $\Delta^{a_0}$, and define $\widetilde{e}$ as $y \circ r$. Hence, establishing the claim will complete the proof.

The 'only if' direction will follow from the 'if' direction together with the stability of cocartesian edges under equivalence. For the `if' direction, fix such an edge $\widetilde{e}$. Recall from the definition that $\widetilde{e}: x \to y$ is $\max_S(e)$-cocartesian if and only if for all objects $z \in \widehat{\sO}^{\simeq}(S)$ with $\max_S(z) = t$, the commutative square
\[ \begin{tikzcd}[row sep=4ex, column sep=4ex, text height=1.5ex, text depth=0.5ex]
\Map_{\widehat{\sO}^{\simeq}(S)_{\max_S = t}}(y,z) \ar{r}{(\widetilde{e})^\ast} \ar{d} & \Map_{\widehat{\sO}^{\simeq}(S)}(x,z) \ar{d}{\max_S} \\
\{ e \} \ar{r} & \Map_S(s,t)
\end{tikzcd} \]
is a homotopy pullback square.  Viewing $x$ as $x: \Delta^{a_0} \to S$, $y$ as $y: \Delta^{a_0+1} \to S$, and $z$ as $z: \Delta^{a_2} \to S$, and computing the mapping spaces in $\widehat{\sO}^{\simeq}(S)$ as a cartesian fibration over $\Delta^{\inj}$, we see that
\[ \Map_{\widehat{\sO}^{\simeq}(S)}(x,z) \simeq \bigsqcup_{\gamma: [a_0] \subset [a_2]} \Map_{\Map(\Delta^{a_0},S)}(x,\gamma^{\ast}z). \]
Therefore, it suffices to show that for any \emph{fixed} inclusion $\gamma: \Delta^{a_0} \to \Delta^{a_2}$ with $\gamma(a_0) < a_2$, letting $\beta: \Delta^{a_0+1} \to \Delta^{a_2}$ be the unique extension of $\gamma$ with $\beta(a_0+1) = a_2$, we have that the square of mapping spaces
\[ \begin{tikzcd}[row sep=4ex, column sep=4ex, text height=1.5ex, text depth=0.5ex]
\Map_{\Map(\Delta^{a_{0}+1},S)}(y,\beta^{\ast}z) \ar{r}{\alpha^{\ast}} \ar{d} & \Map_{\Map(\Delta^{a_{0}},S)}(x,\gamma^{\ast}z) \ar{d} \\
\{e\} \ar{r} & \Map_{\iota S}(x(a_0),z(a_2))
\end{tikzcd} \]
is a homotopy pullback square (where the right vertical map sends $x \to \gamma^{\ast}z$ to the composite $x(a_0) \to z(\gamma(a_0)) \to z(a_2)$). (Here we implicitly use that maps in $\widehat{\sO}^{\simeq}(S)$ are natural transformations through equivalences to account for the $\max_S = t$ condition for the upper-left mapping space.) But this follows since $\ev_{a_0+1}: \Fun(\Delta^{a_0+1},S) \to S$ is a cocartesian fibration with $\overline{x} \to y$ a cocartesian edge lifting $e$, where $\overline{x}$ is the degeneracy $s_{a_0}$ applied to $x$ (we note that $\Map_{\Map(\Delta^{a_{0}+1},S)}(\overline{x},\beta^{\ast}z) \simeq \Map_{\Map(\Delta^{a_{0}},S)}(x,\gamma^{\ast}z)$).

(3) This is clear from the description of the locally $\max_S$-cocartesian edges given in (2).
\end{proof}

Lem.~\ref{lm:subdivisionLocallyCocartesianByMaxFunctor} ensures that the following proposition is well-formulated; also note that $\sd(S)_0 \subset \sd(S)$ is a sub-locally cocartesian fibration via $\max_S$ as it is the inclusion of a cosieve stable under equivalences.

\begin{prp} \label{prp:subdivisionExtension} Let $p: C \to S$ be a locally cocartesian fibration and $\pi: S \to \Delta^1$ a functor. Let $p_0: C_0 \to S_0$ be the fiber of $p$ over $0$. 
\begin{enumerate}
    \item Restricting the domain and codomain of the map of Prop.~\ref{prp:LocallyCocartesianPushforward} yields the map
    \[ \sd_S(C)_0^{\cocart} \to \sd_{S_0}(C_0) \times_{\sd(S_0)} \sd(S)_0 \]
     which is also a trivial fibration of simplicial sets.
    \item Precomposition by the inclusion $S_0 \to S$ defines a trivial fibration of simplicial sets
\[ \Fun^{\cocart}_{/S}(\sd(S)_0,C) \to \Fun^{\cocart}_{/S_0}(\sd(S_0),C_0). \]
\end{enumerate}
\end{prp}

For the proof, it will be convenient to introduce an auxiliary construction. Define a functor $\delta: \widehat{\sO}(S) \to \widehat{\sO}(\widehat{\sO}(S))$ by the following rule:
\begin{itemize}
\item[($\ast$)] Suppose $\sigma: \Delta^n \to \widehat{\sO}(S)$ is a $n$-simplex, which corresponds to a sequence of inclusions
 \[ \begin{tikzcd}[row sep=4ex, column sep=4ex, text height=1.5ex, text depth=0.25ex]
\Delta^{a_0} \ar[hookrightarrow]{r}{\alpha_1} & \Delta^{a_1} \ar[hookrightarrow]{r}{\alpha_2} & \cdots \ar[hookrightarrow]{r}{\alpha_n} & \Delta^{a_n}
\end{tikzcd} \]
determining a map $a: \Delta^n \to \Delta^{\inj}$ and a functor $f: \Delta^n \times_{a, \Delta^{\inj}} \cE \Delta^{\inj} \to S$. Define a map $\overline{a}: \Delta^n \times_{a, \Delta^{\inj}} \cE \Delta^{\inj} \to \Delta^{\inj}$ on objects by $\overline{a}(i \in \Delta^{a_k}) = \Delta^{\{0,...,i\}}$ and on morphisms $(i \in \Delta^{a_k}) \to (j \in \Delta^{a_l})$, $\alpha_{k l}: \Delta^{a_k} \to \Delta^{a_l}$, $\alpha_{k l}(i) \leq j$ by restriction of $\alpha_{k l}$ to $\Delta^{\{0,...,i\}} \subset \Delta^{a_k}$ (which then is valued in $\Delta^{\{0,...,j\}} \subset \Delta^{a_l}$). Then define a functor of categories
\[ \phi: (\Delta^n \times_{a, \Delta^{\inj}} \cE \Delta^{\inj}) \times_{\overline{a}, \Delta^{\inj}} \cE \Delta^{\inj} \to \Delta^n \times_{a, \Delta^{\inj}} \cE \Delta^{\inj} \]
by sending objects $(i \in \Delta^{a_k}, i' \leq i)$ to $(i' \in \Delta^{a_k})$ and morphisms $(i \in \Delta^{a_k}, i' \leq i) \to (j \in \Delta^{a_l}, j' \leq j)$ (specified by the data of a map $\alpha_{k l}: \Delta^{a_k} \to \Delta^{a_l}$ such that $\alpha_{k l}(i) \leq j$ and $\alpha_{k l}(i') \leq j'$) to the morphism $(i' \in \Delta^{a_k}) \to (j' \in \Delta^{a_l})$ specified by the same data.

We may then specify a map $g: \Delta^n \times_{a, \Delta^{\inj}} \cE \Delta^{\inj} \to \widehat{\sO}(S)$ defined over $\Delta^{\inj}$ via $\overline{a}$ and the structure map $\xi_S$ as adjoint to the map $f \circ \phi: (\Delta^n \times_{a, \Delta^{\inj}} \cE \Delta^{\inj}) \times_{\overline{a}, \Delta^{\inj}} \cE \Delta^{\inj} \to S$. $g$ in turn defines the desired $n$-simplex $\delta(\sigma): \Delta^n \to \widehat{\sO}(\widehat{\sO}(S))$.
\end{itemize}
Informally, $\delta$ sends objects $s_0 \to s_1 \to ... \to s_n$ to their `initial segment parametrization'
\[ [s_0] \to [s_0 \to s_1] \to ... \to [s_0 \to s_1 \to ... \to s_n]. \]
Next, using the functor $\max_S$ to make sense of the next statement, we may use $\delta$ to define functors
\begin{align*}
 \delta &: \widehat{\sO}^{\simeq}(S) \to \widehat{\sO}^{\simeq}_S(\widehat{\sO}^{\simeq}(S)) = \widehat{\sO}^{\simeq}(S) \times_{\widehat{\sO}(S)} \widehat{\sO}(\widehat{\sO}^{\simeq}(S)) \\
 \delta &: \sd(S) \to \sd_S(\sd(S)) = \sd(S) \times_{\widehat{\sO}(S)} \widehat{\sO}(\sd(S))
\end{align*}
as the identity on the first factor and a restriction of $\delta$ on the second factor.

\begin{proof}[Proof of Prop.~\ref{prp:subdivisionExtension}] (1) follows from Prop.~\ref{prp:LocallyCocartesianPushforward} in view of the pullback square
\[ \begin{tikzcd}[row sep=4ex, column sep=4ex, text height=1.5ex, text depth=0.25ex]
\sd_S(C)^{\cocart}_0 \ar{r} \ar{d} & \widehat{\sO}^{\simeq}_S(C)_0^{\cocart} \ar{d} \\
\sd_{S_0}(C_0) \times_{\sd(S_0)} \sd(S)_0 \ar{r} & \widehat{\sO}^{\simeq}_{S_0} (C_0) \times_{\widehat{\sO}^{\simeq}(S_0)} \widehat{\sO}^{\simeq}(S)_0.
\end{tikzcd} \]
For (2), we need to solve the lifting problem
\[ \begin{tikzcd}[row sep=4ex, column sep=4ex, text height=1.5ex, text depth=0.25ex]
A \ar[hookrightarrow]{d} \ar{r} & \Fun^{\cocart}_{/S}(\sd(S)_0,C) \ar{d} \\
B \ar{r} \ar[dotted]{ru} & \Fun^{\cocart}_{/S_0}(\sd(S_0),C_0).
\end{tikzcd} \]
This transposes to
\[ \begin{tikzcd}[row sep=4ex, column sep=4ex, text height=1.5ex, text depth=0.25ex]
A \times \sd(S)_0 \bigcup_{A \times \sd(S_0)} B \times \sd(S_0) \ar{r}{G \cup F} \ar[hookrightarrow]{d} & C \ar{d}{p} \\
B \times \sd(S)_0 \ar{r}{\max_S} \ar[dotted]{ru} & S.
\end{tikzcd} \]
The functoriality of $\sd_{S_0}(-)$ in its argument results in a functor
\[ \sd_{S_0}: \Fun_{/S_0}(\sd(S_0),C_0) \to \Fun_{/S_0}(\sd_{S_0}(\sd(S_0)),\sd_{S_0}(C_0)).\]
Given $F: B \times \sd(S_0) \to C_0$, let $\sd_{S_0}(F): B \times \sd_{S_0}(\sd(S_0)) \to \sd_{S_0}(C_0)$ denote the image. We then define $\overline{F}$ as the composite
\[ \begin{tikzcd}[row sep=4ex, column sep=6ex, text height=1.5ex, text depth=0.5ex]
B \times \sd(S_0) \ar{r}{\id \times \delta} & B \times \sd_{S_0}(\sd(S_0)) \ar{r}{\sd_{S_0}(F)} & \sd_{S_0}(C_0).
\end{tikzcd} \]
Also let $\overline{F}'$ denote $\overline{F}$ with codomain $\sd_S(C)^{\cocart}_0$ via the inclusion $\sd_{S_0}(C_0) \subset \sd_S(C)^{\cocart}_0$.

Similarly, given $G: A \times \sd(S)_0 \to C$, we may define $\overline{G}$ as the composite
\[ \begin{tikzcd}[row sep=4ex, column sep=6ex, text height=1.5ex, text depth=0.5ex]
A \times \sd(S)_0 \ar{r}{\id \times \delta} & A \times \sd_S(\sd(S)_0) \ar{r}{\sd_S(G)} & \sd_S(C)^{\cocart}_0
\end{tikzcd} \]
where we note that the codomain of $\sd_S(G)$ necessarily lies in $\sd_S(C)^{\cocart}_0$ by definition of the locally $\max_S$-cocartesian edges in $\sd(S)_0$ (here it is essential that we use $\sd(S)$ rather than $\widehat{\sO}^{\simeq}(S)$). Clearly, $\overline{G}$ and $\overline{F}'$ are compatible on their common domain $A \times \sd(S_0)$ since $G$ and $F$ are. We thereby may factor the square above as
\[ \begin{tikzcd}[row sep=4ex, column sep=4ex, text height=1.5ex, text depth=0.25ex]
A \times \sd(S)_0 \bigcup_{A \times \sd(S_0)} B \times \sd(S_0) \ar{r}{\overline{G} \cup \overline{F}'} \ar[hookrightarrow]{d} & \sd_S(C)^{\cocart}_0 \ar{r}{\max_C} \ar[->>]{d}{\simeq} & C \ar{d}{p} \\
B \times \sd(S)_0 \ar{r}{(\overline{F} \lambda,\pr)} \ar[dotted]{ru} & \sd_{S_0}(C_0) \times_{\sd(S_0)} \sd(S)_0 \ar{r}{\max_S} & S
\end{tikzcd} \]
The dotted lift exists by (1), and postcomposition of such a lift by $\max_C$ defines the desired lift.
\end{proof}

\subsubsection{Main results}

We begin by constructing a factorization system \cite[Def.~5.2.8.8]{HTT} on $\sd(S)$ associated to a sieve-cosieve decomposition of $S$. To do this, we need a few preparatory lemmas.

\begin{lem} \label{lem:cartesianSlice} Let $p: X \to S$ be a cartesian fibration. Given a functor $\phi: K \to X$, let
\[ \overline{p}: X^{\phi/} = \Fun(K^{\rhd},X) \times_{\Fun(K, X)} \{\phi\} \to S^{p \phi/} = \Fun(K^{\rhd},X) \times_{\Fun(K,X)} \{ p \phi\} \]
be the functor induced by $p$. Then $\overline{p}$ is a cartesian fibration, and an edge $\overline{e}: \overline{x} \to \overline{y} \in X^{\phi/}$ is $\overline{p}$-cartesian if and only if the underlying edge $e: x \to y \in X$ is $p$-cartesian.
\end{lem}
\begin{proof} We may duplicate the proof of \cite[3.1.2.1]{HTT} to prove the lemma, the essential tool being \cite[3.1.2.3]{HTT}. In more detail, let $E$ be the described collection of edges in $X^{\phi/}$ and suppose given a lifting problem in marked simplicial sets of the form
\[ \begin{tikzcd}[row sep=4ex, column sep=4ex, text height=1.5ex, text depth=0.25ex]
\rightnat{\Lambda^n_n} \ar{r} \ar{d} & (X^{\phi/}, E) \ar{d}{\overline{p}} \\
\rightnat{\Delta^n} \ar{r} \ar[dotted]{ru} & (S^{p \phi/})^{\sharp} 
\end{tikzcd} \]
where we mark the edge $\{n-1,n\}$ of $\Lambda^n_n$ (if $n>1$) and of $\Delta^n$. This transposes to a lifting problem of the form
\[ \begin{tikzcd}[row sep=4ex, column sep=4ex, text height=1.5ex, text depth=0.25ex]
\rightnat{\Lambda^n_n} \times K^{\rhd} \bigcup_{\rightnat{\Lambda^n_n} \times K} \rightnat{\Delta^n} \times K \ar{r}{f} \ar{d}{i} & \rightnat{X} \ar{d}{p} \\
\rightnat{\Delta^n} \times K^{\rhd} \ar[dotted]{ru} \ar{r} & S^{\sharp}
\end{tikzcd} \]
where we mark the $p$-cartesian edges in $X$. Note that $f$ is indeed a map of marked simplicial sets: this is by definition of $E$ for $f$ on the edge $\{n-1,n\} \times \{v\}$ ($v \in K^{\rhd}$ the cone point), and by definition of $f$ on $\Delta^n \times K$ as given by $\phi \circ \pr_K$ for the other marked edges. Applying \cite[3.1.2.3]{HTT}, we deduce that $i$ is marked right anodyne, so the dotted lift exists.
\end{proof}

\begin{lem} \label{lem:mappingSpacesCartesianFibration} Let $p: X \to S$ be a cartesian fibration. Suppose we have a commutative square in $X$
\[ \begin{tikzcd}[row sep=4ex, column sep=4ex, text height=1.5ex, text depth=0.25ex]
x \ar{r}{h} \ar{d}{f} & z \ar{d}{g} \\
y \ar{r}{k}  & w.
\end{tikzcd} \]
If the edge $g$ is $p$-cartesian, then we have an equivalence
\[ \Map_{x//w}(y,z) \xto{\simeq} \Map_{px//pw}(py,pz). \]
\end{lem}
\begin{proof} By Lem.~\ref{lem:cartesianSlice}, $\overline{p}: X^{x/} \to S^{px/}$ is a cartesian fibration and $g$, viewed as an edge $h \to kf$, is a $\overline{p}$-cartesian edge. Therefore, we have a homotopy pullback square of spaces
\[ \begin{tikzcd}[row sep=4ex, column sep=4ex, text height=1.5ex, text depth=0.25ex]
\Map_{x/}(y,z) \ar{r}{g_{\ast}} \ar{d}{p} & \Map_{x/}(y,w) \ar{d}{p} \\
\Map_{px/}(py,pz) \ar{r}{pg_{\ast}} & \Map_{px/}(py,pw)
\end{tikzcd} \]
Taking fibers over $k \in \Map_{x/}(y,w)$ and $pk \in \Map_{px/}(py,pw)$ yields the claimed equivalence.
\end{proof}

Fix a functor $\pi: S \to \Delta^1$ and let $S_i$ denote the fiber over $i \in \{0,1\}$. We now define a factorization system on $\widehat{\sO}^{\simeq}(S)$ that will restrict to a factorization system on the full subcategory $\sd(S)$. Recall that the data of a morphism $e: x \to y$ in $\widehat{\sO}^{\simeq}(S)$ is given by an inclusion $\alpha: \Delta^{a_0} \to \Delta^{a_1}$ and a map $f: \Delta^1 \times_{\Delta^{\inj}} \cE \Delta^{\inj} \to S$ that restricts to $x: \Delta^{a_0} \to S$ and $y: \Delta^{a_1} \to S$, such that $f$ sends morphisms $(i \in \Delta^{a_0}) \to (\alpha(i) \in \Delta^{a_1})$ to equivalences in $S$.

\begin{dfn} Let $\cL$ be the subclass of morphisms $(\alpha,f): x \to y$ such that for every $i \notin \im \alpha$, we have that $y(i) \in S_0$, and let $\cR$ be the subclass of morphisms $(\alpha,f): x \to y$ such that for every $i \notin \im \alpha$, we have that $y(i) \in S_1$.
\end{dfn}

\begin{prp} $(\cL, \cR)$ defines a factorization system on $\widehat{\sO}^{\simeq}(S)$ and on $\sd(S)$.
\end{prp}
\begin{proof} We will check the assertion concerning $\widehat{\sO}^{\simeq}(S)$; the second assertion will then be an obvious consequence. We first explain how to factor morphisms. Suppose that $\gamma: \Delta^{a_0} \to \Delta^{a_2}$, $h: \Delta^1 \times_{a, \Delta^{\inj}} \cE \Delta^{\inj} \to S$ is the data of a morphism in $\widehat{\sO}^{\simeq}(S)$ from $x$ to $z$. Let $\Delta^{a_1} \subset \Delta^{a_2}$ be the subset on those $i \in \Delta^{a_2}$ such that $i \in \im \gamma$ or $z(i) \in S_0$. We then obtain a factorization of $\gamma$ as
\[ \begin{tikzcd}[row sep=4ex, column sep=4ex, text height=1.5ex, text depth=0.25ex]
\Delta^{a_0} \ar[hookrightarrow]{r}{\alpha} & \Delta^{a_1} \ar[hookrightarrow]{r}{\beta} & \Delta^{a_2}.
\end{tikzcd} \]
defining $\overline{a}: \Delta^2 \to \Delta^{\inj}$ extending the given $a: \Delta^{\{0,2\}} \to \Delta^{\inj}$.
Let $r: \Delta^2 \times_{\overline{a}, \Delta^{\inj}} \cE \Delta^{\inj} \to \Delta^1 \times_{a, \Delta^{\inj}} \cE \Delta^{\inj}$ be the unique retraction which is the identity on $\Delta^{a_0}$ and $\Delta^{a_2}$ and is given by $\beta$ on $\Delta^{a_1}$. Let $\overline{h} = h \circ r$. Then $\overline{h}$ is the desired factorization of $h$, as it corresponds to a factorization
\[ \begin{tikzcd}[row sep=4ex, column sep=4ex, text height=1.5ex, text depth=0.25ex]
x \ar{r}{f} \ar[bend right]{rr}[swap]{h} & y \ar{r}{g} & z \\
\end{tikzcd} \]
with $y = z \circ \beta: \Delta^{a_1} \to S$ defined so that $y(i) \in S_0$ for all $i \notin \im \alpha$ and $z(j) \in S_1$ for all $j \notin \im \beta$, hence $f$ in $\cL$, and $g$ in $\cR$.

Next, observe that because $S_0$ and $S_1$ are closed under retracts, so are $\cL$ and $\cR$. It only remains to check that $\cL$ is left orthogonal to $\cR$. For this, suppose given a commutative square in $\widehat{\sO}^{\simeq}(S)$ on the left with $f \in \cL$ and $g \in \cR$ covering the square in $\Delta^{\inj}$ on the right
\[ \begin{tikzcd}[row sep=4ex, column sep=4ex, text height=1.5ex, text depth=0.25ex]
x \ar{r}{h} \ar{d}{f} & z \ar{d}{g} \\
y \ar{r}{k} \ar[dotted]{ru} & w
\end{tikzcd} \qquad, \qquad
\begin{tikzcd}[row sep=4ex, column sep=4ex, text height=1.5ex, text depth=0.25ex]
\Delta^a \ar{r}{\delta} \ar{d}{\alpha} & \Delta^c \ar{d}{\beta} \\
\Delta^b \ar{r}{\kappa} \ar[dotted]{ru}{\gamma} & \Delta^d.
\end{tikzcd} \]
Because $\xi_S: \widehat{\sO}^{\simeq}(S) \to \Delta^{\inj}$ is a right fibration, by Lem.~\ref{lem:mappingSpacesCartesianFibration} it suffices to show that $\Map_{\Delta^a//\Delta^d}(\Delta^b,\Delta^c)$ is contractible. This holds if and only if $\Delta^b \subset \Delta^c$ when viewed as subsets of $\Delta^d$, so that the mapping space is non-empty. Our hypothesis ensures that if $i \notin \im \beta$, then $w(i) \in S_1$, and if $i \in \Delta^b$, either $i \in \im \alpha$ or $y(i) \in S_0$. Therefore, we must have that for every $i \in \Delta^b$ with $i \notin \im \alpha$ that $w(\kappa(i)) \in S_0$, and hence $\kappa(i) \in \im \beta$. We conclude that the dotted lift $\gamma$ exists, which completes the proof.
\end{proof}

Let $\sO^L(\sd(S)) \subset \sO(\sd(S))$ denote the full subcategory on those morphisms $x \to y$ in the class $\cL$.

\begin{lem} \label{lem:relativeAdjSubdivision} \begin{enumerate}[leftmargin=*] \item The inclusion $i: \sO^L(\sd(S)) \subset \sO(\sd(S))$ admits a right adjoint $r$ that on objects sends $h: x \to y$ to $f: x \to z$ where $h$ factors as $g \circ f$ according to the $(\cL, \cR)$ factorization system.
\end{enumerate}
\begin{enumerate}
\setcounter{enumi}{1}
\item $i \dashv r$ defines a relative adjunction with respect to evaluation $\ev_0$ at the source, and therefore for every $x \in \sd(S)$ we obtain an adjunction
\[ \adjunctb{\{x\} \times_{\sd(S)} \sO^L(\sd(S)) }{ \sd(S)^{x/}  }. \]
\item The relative adjunction $i \dashv r$ restricts to a relative adjunction
\[ \adjunct{i}{\sO^L(\sd(S)) \times_{\ev_1, \sd(S)} \sd(S)_0 }{\sO(\sd(S)) \times_{\ev_1, \sd(S)} \sd(S)_0 }{r} \]
and therefore for every $x \in \sd(S)$ we obtain an adjunction
\[ \adjunctb{\{x\} \times_{\sd(S)} \sO^L(\sd(S)) \times_{\sd(S)} \sd(S)_0 }{ \sd(S)_0^{x/}  }. \]
\end{enumerate}
\end{lem}
\begin{proof} Claim (1) is the dual formulation of \cite[5.2.8.19]{HTT}. Claims (2) and (3) then follow by the definition of relative adjunction \cite[7.3.2.1]{HA} and its pullback property \cite[7.3.2.5]{HA}.
\end{proof}

We are now prepared to construct the recollement adjunctions. Note that the hypotheses of the following proposition are satisfied if $S$ is equivalent to a finite poset and $p: C \to S$ is a locally cocartesian fibration such that the fibers admit finite limits and the pushforward functors preserve finite limits.

\begin{prp} \label{prp:ExistenceLaxRightKanExtension} Let $p: C \to S$ be a locally cocartesian fibration, let $\pi: S \to \Delta^1$ be a functor, and suppose we have a commutative diagram
\[ \begin{tikzcd}[row sep=4ex, column sep=4ex, text height=1.5ex, text depth=0.25ex]
\sd(S)_0 \ar[hookrightarrow]{d}{\phi} \ar{r}{F} & C \ar{d}{p} \\
\sd(S) \ar{r}{\max_S} & S
\end{tikzcd} \]
where $F$ preserves locally cocartesian edges. Given $x \in \sd(S_1)$, let
\[ J_x = \{x\} \times_{\sd(S)} \sO^L(\sd(S)) \times_{\sd(S)} \sd(S)_0 .\]
Note that $(\max_S \circ \ev_1)|_{J_x}$ is constant at $\max_S(x)$.
\begin{enumerate} \item If for every $x \in \sd(S_1)$, the limit of $(F \ev_1)|_{J_x}: J_x \to C_{\max_S(x)}$ exists, then the $p$-right Kan extension $G$ of $F$ along $\phi$ exists and $G(x) \simeq \lim\limits_{\ot} F|_{J_x}$.
\item If for every $f: s \to t$ in $S$, the pushforward functor $f_!: C_s \to C_t$ preserves all limits appearing in (1), then $G$ preserves all locally cocartesian edges.
\item If the hypotheses of (1) and (2) hold for all $F$, then we have an adjunction
\[ \adjunct{\phi^{\ast}}{\Fun^{\cocart}_{/S}(\sd(S),C)}{\Fun^{\cocart}_{/S}(\sd(S)_0,C)}{\phi_{\ast}}. \]
\end{enumerate}
\end{prp}
\begin{proof} Note that $\sd(S_1) \subset \sd(S)$ is the complementary sieve inclusion to the cosieve $\sd(S)_0 \subset \sd(S)$. For (1), to show existence of the $p$-right Kan extension it suffices for every $x \in \sd(S_1)$ to show that the $p$-limit of $F \circ \pr_1: \sd(S)_0^{x/} \to \sd(S)_0 \to C$ exists. But by the argument of Cor.~\ref{cor:RightKanExtensionComputedInFiber} applied to the adjunction $\adjunctb{J_x}{\sd(S)_0^{x/}}$ of Lem.~\ref{lem:relativeAdjSubdivision}, this follows from the given hypothesis.

For (2), first note that there are no locally $\max_S$-cocartesian edges $e: x \to y$ such that $x \in \sd(S_1)$ and $y \in \sd(S)_0$, or vice-versa, so it suffices to handle the case where $e: x \to y$ is a locally $\max_S$-cocartesian edge in $\sd(S_1)$ only. Let $f: \max_S(x) = s \to \max_S(y) = t$ be the edge in $S_1 \subset S$. If $f$ is an equivalence, then $e$ is an equivalence and $G(e)$ is an equivalence, so we may suppose $f$ is not an equivalence. Then by the description of the locally $\max_S$-cocartesian edges in Lem.~\ref{lm:subdivisionLocallyCocartesianByMaxFunctor}, $y$ is obtained from $e$ by appending the edge $f$. Correspondingly, the functor $J_y \xto{\simeq} J_x$ defined via sending $y \to z$ to $x \to z$ by precomposing is an equivalence, using that such edges are constrained to only add objects in $S_0$. Examining how the functoriality of $G$ is obtained from the pointwise existence criterion for Kan extensions, we see that the comparison morphism in $C_t$
\[ \psi: f_! G(x) \simeq f_! (\lim\limits_{\ot} F \ev_1|_{J_x}) \to G(y) \simeq \lim_{\ot} F \ev_1|_{J_y} \]
is induced via the functoriality of limits (contravariant in the diagram, covariant in the target) from the commutative diagram
\[ \begin{tikzcd}[row sep=4ex, column sep=4ex, text height=1.5ex, text depth=0.25ex]
J_x \ar{r}{F \ev_1} & C_s \ar{d}{f_!} \\
J_y \ar{u}{\simeq} \ar{r}{F \ev_1} & C_t.
\end{tikzcd} \]
The hypothesis that $f_!$ preserve limits indexed by $J_x$ together with $J_y \simeq J_x$ then proves that $\psi$ is an equivalence.

Finally, for (3) it is clear that if $G: \sd(S) \to C$ preserves locally cocartesian edges, then the restriction $\phi^{\ast} G$ of $G$ to $\sd(S)_0$ does as well. (1) and (2) establish the same fact for $\phi_{\ast} F$. Hence, the characteristic adjunction
\[ \adjunct{\phi^{\ast}}{\Fun_{/S}(\sd(S),C)}{\Fun_{/S}(\sd(S)_0,C)}{\phi_{\ast}} \]
of the $p$-right Kan extension along $\phi$ restricts to the full subcategories of functors preserving locally cocartesian edges in order to yield the desired adjunction.
\end{proof}

\begin{rem} \label{rem:IdentifyCategoryForLimitWhenPoset} Suppose that $S$ is a poset and $x \in S_1 \subset \sd(S_1)$. Then the $\infty$-category $J_x$ that appears in Prop.~\ref{prp:ExistenceLaxRightKanExtension} is the poset whose objects are strings $[a_0 < \cdots < a_n < x]$, $n \geq 0$ with $a_i \in S_0$ and whose morphisms are string inclusions.
\end{rem}

\begin{cor} \label{cor:openPartOfRecollement} Suppose the hypotheses of Prop.~\ref{prp:ExistenceLaxRightKanExtension} are satisfied. Let $j: \sd(S_0) \to \sd(S)$ denote the inclusion. Then the functor $j^{\ast}$ of restriction along $j$ participates in an adjunction
\[\adjunct{j^{\ast}}{\Fun^{\cocart}_{/S}(\sd(S),C)}{\Fun^{\cocart}_{/S_0}(\sd(S_0),C_0)}{j_{\ast}} \]
with fully faithful right adjoint $j_{\ast}$.
\end{cor}
\begin{proof} Combine Prop.~\ref{prp:ExistenceLaxRightKanExtension} and Prop.~\ref{prp:subdivisionExtension}(2).
\end{proof}

We also have a far simpler result concerning the calculation of the left adjoint $j_!$ of $j^{\ast}$.

\begin{prp} \label{prp:openPartOfRecollementLeftAdjoint} Suppose that for every $s \in S$, the fiber $C_s$ admits an initial object $\emptyset$, and the pushforward functors all preserve initial objects. Then $j^{\ast}$ admits a fully faithful left adjoint $j_!$ such that for $F: \sd(S_0) \to C_0$, $j_! F(x) \simeq \emptyset$ for all $x \in \sd(S_1)$.
\end{prp}
\begin{proof} Suppose we have a commutative diagram
\[ \begin{tikzcd}[row sep=4ex, column sep=4ex, text height=1.5ex, text depth=0.25ex]
\sd(S)_0 \ar[hookrightarrow]{d}{\phi} \ar{r}{F} & C \ar{d}{p} \\
\sd(S) \ar{r}{\max_S} & S.
\end{tikzcd} \]
For all $x \in \sd(S_1)$, the fiber product $\sd(S)^{/x} \times_{\sd(S)} \sd(S)_0$ is the empty category. Therefore, under our assumption the $p$-left Kan extension $\phi_! F$ of $F$ along $\phi$ exists and is computed by $\phi_! F(x) = \emptyset$ on $\sd(S_1)$. Combining this observation with Prop.~\ref{prp:subdivisionExtension}(2), we obtain the desired adjunction
\[ \adjunct{j_!}{\Fun^{\cocart}_{/S_0}(\sd(S_0),C_0)}{\Fun^{\cocart}_{/S}(\sd(S),C)}{j^{\ast}}. \]
\end{proof}

We next turn to the cosieve inclusion $S_1 \subset S$. Note that the inclusion $i: \sd(S_1) \to \sd(S)$ is a sub-locally cocartesian fibration with respect to $\max_S: \sd(S) \to S$, and is in addition a \emph{sieve} inclusion, and hence a cartesian fibration. In fact, the cosieve inclusion $j: \sd(S)_0 \to \sd(S)$ is complementary to $i$.

\begin{prp} \label{prp:closedPartOfRecollement} Suppose the fibers of $p: C \to S$ admit terminal objects and the pushforward functors preserve terminal objects. Then we have the adjunction
\[ \adjunct{i^{\ast}}{\Fun^{\cocart}_{/S}(\sd(S),C)}{\Fun^{\cocart}_{/S_1}(\sd(S_1),C_1)}{i_{\ast}} \]
with $i_{\ast}$ fully faithful, where $i^{\ast}$ is given by restriction along $i$ and $i_{\ast}$ is $p$-right Kan extension along $i$. Moreover, for a functor $G: \sd(S_1) \to C_1$, we have $(i_{\ast} G)(x) \simeq 1 \in C_{\max_S(x)}$ for all $x \in \sd(S)_0$.
\end{prp}
\begin{proof} By Cor.~\ref{cor:RightKanExtensionComputedInFiber}, using the hypothesis that the fibers of $p$ admit terminal objects we have the adjunction
\[ \adjunct{i^{\ast}}{\Fun_{/S}(\sd(S),C)}{\Fun_{/S_1}(\sd(S_1),C_1)}{i_{\ast}}  \]
with $i^{\ast}$ and $i_{\ast}$ as described. Then using that the pushforward functors preserve terminal objects, we see that this adjunction restricts to the one of the proposition.
\end{proof}

\begin{lem} \label{lem:rlaxLimitAdmitsFiniteLimitsAndStable} Suppose that the fibers $C_s$ of $p: C \to S$ admit $K$-(co)limits and the pushforward functors preserve $K$-(co)limits. Then $\Fun^{\cocart}_{/S}(\sd(S), C)$ admits $K$-indexed (co)limits, and for all $\sigma \in \sd(S)$ over $s = \max_S(\sigma)$, the evaluation functor $\ev_{\sigma}: \Fun^{\cocart}_{/S}(\sd(S), C) \to C_s$ preserves $K$-indexed (co)limits. Moreover, if the fibers $C_s$ are stable $\infty$-categories and the pushforward functors are exact, then $\Fun^{\cocart}_{/S}(\sd(S), C)$ is a stable $\infty$-category.
\end{lem}
\begin{proof} Apply \cite[Prop.~5.4.7.11]{HTT} to the locally cocartesian fibration $\sd(S) \times_S C \to \sd(S)$, with the subcategory of $\widehat{\Cat}_{\infty}$ either taken to be those $\infty$-categories that admit $K$-indexed (co)limits and functor that preserve $K$-indexed (co)limits, or the subcategory $\Cat_{\infty}^{\st}$ of stable $\infty$-categories and exact functors thereof.
\end{proof}

\begin{thm} \label{thm:RecollementRlaxLimitOfLlaxFunctor} Suppose that the hypothesis of Prop.~\ref{prp:ExistenceLaxRightKanExtension} hold, and also that the fibers of $p: C \to S$ admit finite limits and the pushforward functors preserve finite limits. Then the two adjunctions of Cor.~\ref{cor:openPartOfRecollement} and Prop.~\ref{prp:closedPartOfRecollement} combine to exhibit $\Fun^{\cocart}_{/S}(\sd(S),C)$ as a recollement of $\Fun^{\cocart}_{/S_0}(\sd(S_0),C_0)$ and $\Fun^{\cocart}_{/S_1}(\sd(S_1),C_1)$.
\end{thm}
\begin{proof} We verify the conditions to be a recollement. By our hypothesis on $p$ and Lem.~\ref{lem:rlaxLimitAdmitsFiniteLimitsAndStable}, finite limits in $\Fun^{\cocart}_{/S}(\sd(S),C)$ exist and are computed fiberwise. Therefore, the restriction functors $j^{\ast}$ and $i^{\ast}$ are left exact. By the formula for $i_{\ast}$ given in Prop.~\ref{prp:closedPartOfRecollement}, it is clear that $j^{\ast} i_{\ast}$ is constant at the terminal object. Finally, we check that $j^{\ast}$ and $i^{\ast}$ are jointly conservative. Suppose given a morphism $\alpha: F \to F'$ in $\Fun^{\cocart}_{/S}(\sd(S),C)$ such that $j^{\ast} \alpha$ and $i^{\ast} \alpha$ are equivalences. Observe that $\alpha$ is an equivalence if and only if for all $x \in S$, $\alpha_x: F(x) \to F'(x)$ is an equivalence (viewing $x$ as an object in $\sd(S)$). Because any object of $S$ lies in either $S_0$ or $S_1$, we deduce that $\alpha$ is an equivalence.
\end{proof}

We conclude this subsection by giving an application of Thm.~\ref{thm:RecollementRlaxLimitOfLlaxFunctor} to the presentability of the right-lax limit $\Fun^{\cocart}_{/S}(\sd(S),C)$.

\begin{dfn} Given an object $s \in S$, its \emph{dimension} $\dim_S(s)$ is the supremum over all $n$ such that there exists a functor $\sigma: \Delta^n \to S$, $\sigma(n) = s$ with $\sigma(\{i,i+1 \})$ not an equivalence for all $0 \leq i < n$. The \emph{dimension} $\dim(S)$ of $S$ is the supremum of $\{ \dim_S(s): s \in S \}$.
\end{dfn}

Let us now suppose that $S$ is equivalent to a \emph{finite} poset and write $P = S$. 

\begin{prp} \label{prp:rightLaxLimitPresentable} Suppose that the fibers $C_s$ of $p: C \to P$ are presentable and the pushforward functors are left-exact and accessible. Then $\Fun^{\cocart}_{/P}(\sd(P),C)$ is presentable, and for all $s \in P$, the evaluation functor $\ev_s: \Fun^{\cocart}_{/P}(\sd(P),C) \to C_{s}$ preserves (small) colimits and is accessible.
\end{prp}
\begin{proof} The accessibility statements follow from \cite[Prop.~5.4.7.11]{HTT} as in Lem.~\ref{lem:rlaxLimitAdmitsFiniteLimitsAndStable}, so we only need to show the existence and preservation of small colimits. Our strategy is to proceed by induction on the dimension of $P$. If $\dim(P) = 0$, then the statement is clear. Suppose for the inductive hypothesis that we have established the statement for all $\infty$-categories $P'$ equivalent to finite posets of dimension $<n$ where $n = \dim(P)$. Let $\pi: P \to \Delta^1$ be a functor such that $\dim(P_0)<n$ and $\dim(P_1)<n$; for instance, we may take $P_0 \subset P$ to be the sieve on those objects in $P$ of dimension $<n$. Under our finiteness assumption on $P$, the diagrams in Prop.~\ref{prp:ExistenceLaxRightKanExtension} are finite. Thus, we may apply Thm.~\ref{thm:RecollementRlaxLimitOfLlaxFunctor} to decompose $\Fun^{\cocart}_{/P}(\sd(P),C)$ as a recollement of $\Fun^{\cocart}_{/P_0}(\sd(P_0),C_0)$ and $\Fun^{\cocart}_{/P_1}(\sd(P_1),C_1)$. By the inductive hypothesis, both these $\infty$-categories admit all small colimits such that the evaluation functors at objects in $P_0$ and $P_1$ are colimit-preserving. By Lem.~\ref{lem:ColimitExistenceInRecollement}, we conclude that $\Fun^{\cocart}_{/P}(\sd(P),C)$ admits all small colimits. Moreover, because $P_0$ and $P_1$ cover $P$, we also have that the evaluation functors for objects $s \in P$ are colimit-preserving.
\end{proof}

\subsection{\texorpdfstring{$1$}{1}-generated and extendable objects}

Suppose $S = \Delta^2$ and $p: C \to \Delta^2$ is a locally cocartesian fibration classified by a $2$-functor
\[ \begin{tikzcd}[row sep=4ex, column sep=6ex, text height=1.5ex, text depth=0.25ex]
C_{0} \ar{rd}[swap]{F} \ar{rr}{H} & \ar[phantom]{d}{\Downarrow} & C_{2}. \\
& C_{1} \ar{ru}[swap]{G} &
\end{tikzcd} \]
Then the data of a functor $\sd(\Delta^2) \to C$ over $\Delta^2$ that preserves locally cocartesian edges can be summarized as follows:
\begin{itemize}
    \item Objects $c_i \in C_i$ for $i = 0,1,2$.
    \item Morphisms $f: c_1 \to F(c_0)$, $g: c_2 \to G(c_1)$, and $h: c_2 \to H(c_0)$.
    \item A commutative square
\[ \begin{tikzcd}[row sep=4ex, column sep=6ex, text height=1.5ex, text depth=0.25ex]
c_2 \ar{r}{h} \ar{d}{g} & H(c_0) \ar{d}{\can} \\
G(c_1) \ar{r}{G(f)} & G F(c_0).
\end{tikzcd} \]
\end{itemize}

Furthermore, if the map $\mit{can}$ is an equivalence, then the data of the commutative square and the morphism $h$ is redundant, since then $h \simeq G(f) \circ g$ and compositions in an $\infty$-category are unique up to contractible choice. More precisely, if we let $\gamma_2: \sd_1(\Delta^2) \subset \sd(\Delta^2)$ be the subposet on $\{[0], [1], [2], [0<1], [1<2] \}$, then the functor
\[ \gamma_2^{\ast}: \Fun^{\cocart}_{/\Delta^2}(\sd(\Delta^2), C) \to \Fun^{\cocart}_{/\Delta^2}(\sd_1(\Delta^2), C) \]
is a trivial fibration onto its image when restricted to objects for which $\mit{can}$ is an equivalence.

Our goal in this subsection is to generalize this observation to the case where $S=\Delta^n$. We introduce subcategories of $1$-generated and extendable objects (Def.~\ref{dfn:OneGenerated} and Def.~\ref{dfn:extendability}) and show their equivalence under the restriction functor $\gamma_n^{\ast}$ (Thm.~\ref{thm:OneGenerationAndExtension}), given a stability hypothesis on $C \xto{p} \Delta^n$.

\begin{ntn} \label{ntn:convexStringsLengthOne} Let $\gamma_n: \sd_1(\Delta^n) \subset \sd(\Delta^n)$ be the subposet on strings $[k]$ and $[k<k+1]$.
\end{ntn}

We also introduce convenient notation for convex subposets of $\Delta^n$.

\begin{ntn} Let $[i:j] \subset \Delta^n$ denote the subposet on $i \leq k \leq j$.
\end{ntn}

Via its inclusion into $\sd(\Delta^n)$, we regard $\sd_1(\Delta^n)$ as a simplicial set over $\Delta^n$ (i.e., by the functor that takes the maximum) and as a marked simplicial set (so that each edge $[k] \to [k<k+1]$ is marked). We first state the analogue of Thm.~\ref{thm:RecollementRlaxLimitOfLlaxFunctor} for $\sd_1$, whose proof is far simpler.

\begin{prp} \label{prp:StaircaseRecollement} Let $p: C \to \Delta^n$ be a locally cocartesian fibration such that the fibers admit finite limits and the pushforward functors preserve finite limits. Let $0 \leq k < n$, so the subcategories $[0:k] \cong \Delta^k$ and $[k+1:n] \cong \Delta^{n-k-1}$ of $\Delta^n$ give a sieve-cosieve decomposition. Then we have adjunctions
\[ \begin{tikzcd}[row sep=4ex, column sep=4ex, text height=1.5ex, text depth=0.25ex]
\Fun^{\cocart}_{/[0:k]}(\sd_1([0:k]), C_{[0:k]}) \ar[shift right=1,right hook->]{r}[swap]{j_{\ast}} & \Fun^{\cocart}_{/\Delta^n}(\sd_1(\Delta^n), C) \ar[shift right=2]{l}[swap]{j^{\ast}} \ar[shift left=2]{r}{i^{\ast}} & \Fun^{\cocart}_{/[k+1:n]}(\sd_1([k+1:n]), C_{[k+1:n]}) \ar[shift left=1,left hook->]{l}{i_{\ast}}
\end{tikzcd} \]
that exhibit $\Fun^{\cocart}_{/\Delta^n}(\sd_1(\Delta^n), C)$ as a recollement.
\end{prp}
\begin{proof} Let $j: \sd_1([0:k]) \to \sd_1(\Delta^n)$ and $i: \sd_1([k+1:n]) \to \sd_1(\Delta^n)$ be the inclusions, so $j^{\ast}$ and $i^{\ast}$ are defined by restriction along $j$ and $i$. As in the proof of Lem.~\ref{lem:rlaxLimitAdmitsFiniteLimitsAndStable}, our hypotheses on $p$ ensure that the three $\infty$-categories admit finite limits and the functors $j^{\ast}$ and $i^{\ast}$ are left-exact. Moreover, since equivalences are detected on strings $[k]$, $j^{\ast}$ and $i^{\ast}$ are jointly conservative. The functor $i_{\ast}$ is obtained by $p$-right Kan extension as in the proof of Prop.~\ref{prp:closedPartOfRecollement}, and its essential image consists of functors $F: \sd_1(\Delta^n) \to C$ such that $F(i)$ is a terminal object in $C_i$ for all $0 \leq i \leq k$, so $j^{\ast} i_{\ast}$ is the constant functor at the terminal object.

Finally, we show existence of $j_{\ast}$. Let $\sd_1([0:k])^+$ be the subposet of $\sd_1([0:n])$ on all objects in $\sd_1([0:k])$ and $\{[k<k+1]\}$, with marking inherited from $\sd(\Delta^n)$. Then we have a pushout square of marked simplicial sets
\[ \begin{tikzcd}[row sep=4ex, column sep=4ex, text height=1.5ex, text depth=0.25ex]
\Delta^0 \ar{r} \ar{d} & (\Delta^1)^{\sharp} \ar{d} \\
\sd_1([0:k]) \ar{r} & \sd_1([0:k])^+ 
\end{tikzcd} \]
so the inclusion $\sd_1([0:k]) \subset \sd_1([0:k])^+$ is $\mathfrak{P}$-anodyne for the categorical pattern $\mathfrak{P}$ defining the locally cocartesian model structure on $s\Set^+_{/\Delta^n}$. We thus obtain a trivial fibration
\[ \Fun^{\cocart}_{/[0:k+1]}(\sd_1([0:k])^+, C_{[0:k+1]}) \to \Fun^{\cocart}_{/[0:k]}(\sd_1([0:k]), C_{[0:k]}). \]
On the other hand, given a commutative diagram
\[ \begin{tikzcd}[row sep=4ex, column sep=4ex, text height=1.5ex, text depth=0.25ex]
\sd_1([0:k])^+ \ar{r}{F} \ar{d} & C \ar{d}{p} \\
\sd_1([0:k+1]) \ar{r} \ar[dotted]{ru}[swap]{G} & \Delta^n,
\end{tikzcd} \]
since $\sd_1([0:k])^+ \times_{\sd_1([0:k+1])} \sd_1([0:k+1])_{[k+1]/} \cong \{[k < k+1]\}$, $F$ admits a $p$-right Kan extension along $\sd_1([0:k])^+ \subset \sd_1([0:k+1])$ and $G$ is a $p$-right Kan extension of $F$ if and only if $G$ sends the edge $[k+1] \to [k<k+1]$ to an equivalence. Therefore, we may alternate between anodyne extension and $p$-right Kan extension along the filtration
\[ \sd_1([0:k]) \subset \sd_1([0:k])^+ \subset \sd_1([0:k+1]) \subset \cdots \sd_1([0:n-1])^+ \subset \sd_1(\Delta^n) \]
to define the functor $j_{\ast}$. Moreover, we see that the essential image of $j_{\ast}$ consists of those functors $\sd_1(\Delta^n) \to C$ that send the edges $[l+1] \to [l < l+1]$ to equivalences for all $l \geq k$.
\end{proof}

We next wish to introduce a condition on objects of $\Fun^{\cocart}_{/\Delta^n}(\sd(\Delta^n), C)$, which we term `$1$-generated', that indicates that the data of such objects is essentially determined by their restriction to $\sd_1(\Delta^n)$.

\begin{ntn} Given a string $\sigma = [i<i+k]$ in $\sd(\Delta^n)$, let $Q_{\sigma} \subset \sd(\Delta^n)$ be the subposet on all strings $[i<\cdots<i+k]$. Note that $Q_{\sigma}$ is a $(k-1)$-dimensional cube lying in the fiber $\sd(\Delta^n)_{\max=i+k}$ with $\sigma$ as its minimal element.
\end{ntn}

\begin{dfn} \label{dfn:OneGenerated} Let $C \to \Delta^n$ be a locally cocartesian fibration and $F: \sd(\Delta^n) \to C$ be a functor that preserves locally cocartesian edges. We say that $F$ is \emph{$1$-generated} if for all strings $\sigma = [i<i+k]$ in $\sd(\Delta^n)$, $F|_{Q_{\sigma}}$ is a limit diagram in $C_{i+k}$.

Let $\Fun^{\cocart}_{/\Delta^n}(\sd(\Delta^n), C)_{\gen{1}}$ be the full subcategory on the $1$-generated objects.
\end{dfn}


\begin{lem} \label{lm:equivalentOneGenerationConditions} Let $C \to \Delta^n$ be a locally cocartesian fibration whose fibers are stable $\infty$-categories and whose pushforward functors are exact. Then $F: \sd(\Delta^n) \to C$ is $1$-generated if and only if for all string inclusions $e: [i<i+k] \to [i < i+1 < i+k]$ in $\sd(\Delta^n)$, $F(e)$ is an equivalence in $C_{i+k}$.
\end{lem}
\begin{proof} We will prove the stronger claim that for fixed $k \geq 2$ and all string inclusions $e_{ij}: \sigma_{ij} = [i<i+j] \to [i < i+1 < i+j]$ with $2 \leq j \leq k$, $F|_{Q_{\sigma_{ij}}}$ is a limit diagram for all $Q_{\sigma_{ij}}$ if and only if $F(e_{ij})$ is an equivalence for all $e_{ij}$.

We proceed by induction on $k$. For the base case $k=2$, given a string inclusion $\sigma = [i<i+2] \to [i<i+1<i+2]$, the edge is the $1$-dimensional cube $Q_{\sigma}$, so $F|_{Q_{\sigma}}$ is a limit diagram if and only if $F(e)$ is an equivalence. Now let $k>2$ and suppose we have proven the statement for all $l<k$. Note that in proving either direction of the `if and only if' statement, we may suppose that $F|_{Q_{\sigma_{ij}}}$ is a limit diagram \emph{and} $F(e_{ij})$ for all $2 \leq j < k$, so let us do so.

Consider an edge $e: \sigma = [i<i+k] \to [i < i+1 < i+k]$. For $1<j<k$, let $Q_{\sigma,j} \subset Q_{\sigma}$ be the subposet on strings excluding vertices $i+j, ..., i+k-1$. Then we have a descending filtration of sieve inclusions
\[ Q_{\sigma} \coloneq Q_{\sigma,k} \supset Q_{\sigma, k-1} \supset Q_{\sigma, k-2} \supset \cdots \supset Q_{\sigma,2} \]
where $Q_{\sigma, j}$ is a $(j-1)$-dimensional cube and $Q_{\sigma,2}$ consists only of the edge $e$. Note that if we let $Q_{\sigma,j}' = Q_{\sigma,j+1} \setminus Q_{\sigma,j}$ for $1<j<k$, then the minimal element of $Q_{\sigma,j}'$ is given by $\sigma_j = [i<i+j<i+k]$, and if we let $\sigma'_j = [i<i+j]$, then $Q_{\sigma,j}'$ is obtained from $Q_{\sigma_j'}$ by concatenating $i+k$. By the inductive hypothesis and using that the pushforward functors are exact, we get that $F|_{Q_{\sigma,j}'}$ is a limit diagram. Taking total fibers of cubes then shows that $F|_{Q_{\sigma,j}}$ is a limit diagram if and only if $F|_{Q_{\sigma,j-1}}$ is a limit diagram. Traversing the filtration, we conclude that $F|_{Q_{\sigma}}$ is a limit diagram if and only if $F(e)$ is an equivalence.
\end{proof}


\begin{lem} \label{lm:limitCubeDegenerateIfOneGenerated} Let $Q = \sd(\Delta^n)_{\max = n}$, $D$ a stable $\infty$-category, and $f: Q \to D$ a functor. Suppose the following condition holds:
\begin{itemize} \item[($\ast$)] For all string inclusions $e: \sigma \to \sigma'$ in $Q$ obtained by concatenating $[i<k] \to [i<i+1<k] $ by a (possibly empty) suffix $\tau$, $f(e)$ is an equivalence.
\end{itemize}
Then $f$ is a limit diagram if and only if $f([n] \to [n-1<n])$ is an equivalence.
\end{lem}
\begin{proof} The proof is similar to that of Lem.~\ref{lm:equivalentOneGenerationConditions}. For $0 \leq j < n$, let $Q_{\geq j}$, $Q_{=j}$ be the subposet on strings $\sigma$ with minimum $\geq j$, resp $=j$. Then $Q_{\geq j}$ is a $(n-j)$-dimensional cube, $Q_{=j} = Q_{\geq j} \setminus Q_{\geq j+1}$ is a $(n-j-1)$-dimensional cube, and we have a descending filtration
\[ Q = Q_{\geq 0} \supset Q_{\geq 1} \supset Q_{\geq 2} \supset \cdots \supset Q_{\geq n-1}. \] 
Observe that $Q_{=j} = Q_{[j<n]}$, so $f|_{Q_{=j}}$ is a limit diagram under our hypotheses by the proof of Lem.~\ref{lm:equivalentOneGenerationConditions}. Therefore, taking total fibers shows that $f|_{ Q_{\geq j}}$ is a limit diagram if and only if $f|_{ Q_{\geq j+1}}$ is a limit diagram. Traversing the filtration then proves the claim.
\end{proof}

We continue to assume $C \to \Delta^n$ is a locally cocartesian fibration whose fibers are stable $\infty$-categories and whose pushforward functors are exact. Observe that we have a commutative diagram
\[ \begin{tikzcd}[row sep=4ex, column sep=6ex, text height=1.5ex, text depth=0.5ex]
\Fun^{\cocart}_{/[0:n-1]}(\sd([0:n-1]), C_{[0:n-1]}) \ar{r}{\gamma_{n-1}^{\ast}} & \Fun^{\cocart}_{/[0:n-1]}(\sd_1([0:n-1]), C_{[0:n-1]}) \\
\Fun^{\cocart}_{/\Delta^n}(\sd(\Delta^n), C) \ar{r}{\gamma_n^{\ast}} \ar{u}{j^{\ast}} \ar{d}[swap]{i^{\ast}} & \Fun^{\cocart}_{/\Delta^n}(\sd_1(\Delta^n), C) \ar{u}{j^{\ast}} \ar{d}[swap]{i^{\ast}}  \\
C_n \ar{r}{\id} & C_n,
\end{tikzcd} \]
so in particular $\gamma_n^{\ast}$ is a morphism of stable recollements. However $\gamma_n$ generally fails to be a \emph{strict} morphism of stable recollements, i.e., the natural transformation
\[ i^{\ast} j_{\ast} \to i^{\ast} j_{\ast} \gamma_{n-1}^{\ast} \]
is typically not an equivalence.

\begin{lem} \label{lem:OneGeneratedStrictMorphismOfRecollements} Suppose $F: \sd(\Delta^n) \to C$ is $1$-generated. Then the comparison map
\[  i^{\ast} j_{\ast} j^{\ast} F = (j_{\ast} j^{\ast} F)(n) \to i^{\ast} j_{\ast} \gamma_{n-1}^{\ast} j^{\ast} F =  (j_{\ast}( F|_{\sd_1([0:n-1])}))(n) \]
is an equivalence.
\end{lem}
\begin{proof} Let $K \subset \sd(\Delta^n)$ be the subposet on strings $\sigma$ with $\max(\sigma) = n$ and $\sigma \neq n$. By the formulas computing $j_{\ast}$ given in Prop.~\ref{prp:ExistenceLaxRightKanExtension} and Prop.~\ref{prp:StaircaseRecollement}, we see that the comparison map is given by the canonical map from the limit of $F|_{K}$ to $F([n-1<n])$. Since $F$ is $1$-generated, by Lem.~\ref{lm:equivalentOneGenerationConditions} the conditions of Lem.~\ref{lm:limitCubeDegenerateIfOneGenerated} are satisfied, so this canonical map is an equivalence.
\end{proof}

\begin{dfn} For the functor $j_{\ast}$ defined as in Prop.~\ref{prp:ExistenceLaxRightKanExtension} with respect to $[0:n-1]$ and $\{ n \}$, we say that a functor $F: \sd([0:n-1]) \to C_{[0:n-1]}$ is \emph{$+$-1-generated} if both $F$ and $j_{\ast} F$ are $1$-generated. Let
\[ \Fun^{\cocart}_{/[0:n-1]}(\sd([0:n-1]), C_{[0:n-1]})_{\gen{1}}^+ \]
be the full subcategory on the $+$-$1$-generated objects.
\end{dfn}

\begin{lem} \label{lem:OneGeneratedRecollement} We have adjunctions
\[ \begin{tikzcd}[row sep=4ex, column sep=4ex, text height=1.5ex, text depth=0.25ex]
\Fun^{\cocart}_{/[0:n-1]}(\sd([0:n-1]), C_{[0:n-1]})_{\gen{1}}^+ \ar[shift right=1,right hook->]{r}[swap]{j_{\ast}} & \Fun^{\cocart}_{/\Delta^n}(\sd(\Delta^n), C)_{\gen{1}} \ar[shift right=2]{l}[swap]{j^{\ast}} \ar[shift left=2]{r}{i^{\ast}} & C_n \ar[shift left=1,left hook->]{l}{i_{\ast}}
\end{tikzcd} \]
that exhibit $\Fun^{\cocart}_{/\Delta^n}(\sd(\Delta^n), C)_{\gen{1}}$ as a stable recollement.
\end{lem}
\begin{proof} Clearly, we may define $j_{\ast}$, $i^{\ast}$, and $i_{\ast}$ to be the restrictions of the corresponding functors for the adjunctions of Thm.~\ref{thm:RecollementRlaxLimitOfLlaxFunctor}. The only subtle point is that given $F: \sd(\Delta^n) \to C$ which is $1$-generated, we require that the localization $j_{\ast} j^{\ast} F$ is also $1$-generated. But this holds, since $F \simeq j_{\ast} j^{\ast} F$ except possibly at $n \in \sd(\Delta^n)$ and the $1$-generated condition ignores $n$. Therefore, we may also define $j^{\ast}$ as the restricted functor, and the recollement conditions are then immediate. 
\end{proof}

\begin{cor} \label{cor:OneGeneratedStrictMorphismOfRecollements} The restriction $\gamma_n^{\ast}: \Fun^{\cocart}_{/\Delta^n}(\sd(\Delta^n), C)_{\gen{1}} \to \Fun^{\cocart}_{/\Delta^n}(\sd_1(\Delta^n), C)$ is a strict morphism of stable recollements with respect to Lem.~\ref{lem:OneGeneratedRecollement} and Prop.~\ref{prp:StaircaseRecollement}.
\end{cor}
\begin{proof} This follows immediately from Lem.~\ref{lem:OneGeneratedStrictMorphismOfRecollements}.
\end{proof}

We want to apply Cor.~\ref{cor:OneGeneratedStrictMorphismOfRecollements} to show that $\gamma_n^{\ast}$ is an equivalence (in fact, a trivial fibration) onto its essential image. To understand this image as a condition on objects in the codomain, we introduce the following definition. For $0 \leq i < j \leq n$, let $\tau^j_i: C_i \to C_j$ denote the pushforward functor encoded by the locally cocartesian fibration.

\begin{dfn} \label{dfn:extendability} We say that a functor $f: \sd_1(\Delta^n) \to C$ is \emph{extendable} if for every string $[i<i+1<i+k]$ in $\sd(\Delta^n)$, the canonical map in $C_{i+k}$
\[ \tau^{i+k}_i f(i) \to (\tau^k_{i+1} \circ \tau^{i+1}_i) f(i) \]
encoded by the locally cocartesian fibration is an equivalence. Let
\[ \Fun^{\cocart}_{/\Delta^n}(\sd_1(\Delta^n),C)_{\ext} \]
denote the full subcategory on the extendable objects.
\end{dfn}

\begin{dfn} For the functor $j_{\ast}$ defined as in Prop.~\ref{prp:StaircaseRecollement} with respect to $[0:n-1]$ and $\{ n \}$, we say that a functor $f: \sd_1([0:n-1]) \to C$ is \emph{$+$-extendable} if both $f$ and $j_{\ast} f$ are extendable. Let
\[ \Fun^{\cocart}_{/[0:n-1]}(\sd_1([0:n-1]), C_{[0:n-1]})_{\ext}^+ \]
be the full subcategory on the $+$-extendable objects.
\end{dfn}

Note that the extendability condition becomes stronger through considering the additional strings in $\sd(\Delta^n)$; for example, extendability is no condition on $f: \sd_1([0:1]) \to C_{[0:1]}$, but we acquire the condition that the map $\tau^2_0 f(0) \to \tau^2_1 \tau^1_0 f(0)$ is an equivalence upon enlarging to $\Delta^2$. Let us first state the evident counterpart to Lem.~\ref{lem:OneGeneratedRecollement}.

\begin{lem} \label{lem:ExtendableObjectsRecollement} We have adjunctions
\[ \begin{tikzcd}[row sep=4ex, column sep=4ex, text height=1.5ex, text depth=0.25ex]
\Fun^{\cocart}_{/[0:n-1]}(\sd_1([0:n-1]), C_{[0:n-1]})_{\ext}^+ \ar[shift right=1,right hook->]{r}[swap]{j_{\ast}} & \Fun^{\cocart}_{/\Delta^n}(\sd_1(\Delta^n), C)_{\ext} \ar[shift right=2]{l}[swap]{j^{\ast}} \ar[shift left=2]{r}{i^{\ast}} & C_n \ar[shift left=1,left hook->]{l}{i_{\ast}}
\end{tikzcd} \]
that exhibit $\Fun^{\cocart}_{/\Delta^n}(\sd_1(\Delta^n), C)_{\ext}$ as a stable recollement.
\end{lem}
\begin{proof} This is immediate from restricting the recollement of Prop.~\ref{prp:StaircaseRecollement}.
\end{proof}

We have assembled all the ingredients needed to prove Thm.~\ref{thm:OneGenerationAndExtension}. Note that by Lem.~\ref{lm:limitCubeDegenerateIfOneGenerated}, $\gamma_n^{\ast}$ of a $1$-generated object is extendable, so the functor of Thm.~\ref{thm:OneGenerationAndExtension} is well-defined.

\begin{thm} \label{thm:OneGenerationAndExtension} Suppose $C \to \Delta^n$ is a locally cocartesian fibration whose fibers are stable $\infty$-categories and whose pushforward functors are exact. Then the functor
\[ \gamma_n^{\ast}: \Fun^{\cocart}_{/\Delta^n}(\sd(\Delta^n), C)_{\gen{1}} \to \Fun^{\cocart}_{/\Delta^n}(\sd_1(\Delta^n),C)_{\ext} \]
is an equivalence of $\infty$-categories.
\end{thm}
\begin{proof} We proceed by induction on $n$. For the base cases $n = 0$ and $n=1$, the result is trivial. Let $n>1$ and suppose we have proven the theorem for all $k<n$. By the inductive hypothesis, $\gamma_{n-1}^{\ast}$ is an equivalence. Observe that $\gamma_{n-1}^{\ast}$ restricts to a functor
\[ (\gamma_{n-1}^{\ast})^+: \Fun^{\cocart}_{/[0:n-1]}(\sd([0:n-1]), C_{[0:n-1]})_{\gen{1}}^+ \to \Fun^{\cocart}_{/[0:n-1]}(\sd_1([0:n-1]), C_{[0:n-1]})_{\ext}^+. \]
If we let $(\gamma_{n-1}^{\ast})^{-1}$ be an inverse functor, then by Lem.~\ref{lm:equivalentOneGenerationConditions}, if $f: \sd_1([0:n-1]) \to C_{[0:n-1]}$ is $+$-extendable, then $(\gamma_{n-1}^{\ast})^{-1}(f)$ is $+$-1-generated. Therefore, $(\gamma_{n-1}^{\ast})^+$ is also an equivalence. By Cor.~\ref{cor:OneGeneratedStrictMorphismOfRecollements} (but replacing the codomain there with the recollement of Lem.~\ref{lem:ExtendableObjectsRecollement}) and the two-out-of-three property of equivalences for a strict morphism of stable recollements (Rmk.~\ref{rem:TwoOutOfThreePropertyEquivalencesStrictMorphismRecoll}), we deduce that $\gamma_n^{\ast}$ is an equivalence.
\end{proof}

\begin{nul} \label{dualizedDescriptionOfSpine} To make better use of Thm.~\ref{thm:OneGenerationAndExtension}, let us further unpack the $\infty$-category $\Fun_{/\Delta^n}^{\cocart}(\sd_1(\Delta^n),C)$. Note that we may write $\sd_1(\Delta^n)$ as the union of marked simplicial sets
\[ \sd([0:1]) \cup_{1} \sd([1:2]) \cup_{2} \cdots \cup_{n} \sd([n-1:n]), \]
so we obtain a fiber product decomposition
\[ \Fun_{/\Delta^n}^{\cocart}(\sd_1(\Delta^n),C) \simeq \Fun_{/[0:1]}^{\cocart}(\sd([0:1]), C_{[0:1]}) \times_{C_1} \cdots \times_{C_{n-1}} \Fun_{/[n-1:n]}^{\cocart}(\sd([n-1:n]), C_{[n-1:n]}). \]

Let $\tau^{i+1}_i: C_i \to C_{i+1}$ be the pushforward functors as before, and with respect to the trivial fibration (induced by the inner anodyne spine inclusion $[0:1] \cup_1 \cdots \cup_{n-1} [n-1:n] \to \Delta^n$)
\[ \Fun(\Delta^n, \Cat_{\infty}) \xto{\simeq} \Fun([0:1], \Cat_{\infty}) \times_{1} \cdots \times_{n-1} \Fun([n-1:n], \Cat_{\infty}), \]
let $\tau_{\bullet}: \Delta^n \to \Cat_{\infty}$ be a functor lifting the $\tau^{i+1}_i$. Let $C^{\vee} \to (\Delta^n)^{\op}$ be a cartesian fibration classified by $\tau_{\bullet}$. Then if we let $[i+1:i] = [i:i+1]^{\op}$, we have that $(C^{\vee})_{[i+1:i]} \simeq (C_{[i:i+1]})^{\vee}$ where the righthand $(-)^{\vee}$ denotes the dual cartesian fibration of the cocartesian fibration $C_{[i:i+1]} \to [i:i+1]$. Then by \ref{dualizingOneSimplex}, we have an equivalences of $\infty$-categories
\[ \Fun^{\cocart}_{/[i:i+1]}(\sd([i:i+1]), C_{[i:i+1]}) \simeq \Fun_{/[i+1:i]} ([i+1:i], C^{\vee}_{[i+1:i]}) \simeq \sO(C_{i+1}) \times_{\ev_1,C_{i+1}, \tau^{i+1}_i} C_i. \]
\end{nul}

Again using that the spine inclusion is inner anodyne, we obtain the following proposition.

\begin{prp} \label{prp:DualDescriptionOfSections} We have equivalences of $\infty$-categories
\begin{align*} \Fun_{/\Delta^n}^{\cocart}(\sd_1(\Delta^n),C) & \simeq \Fun_{/(\Delta^n)^{\op}}((\Delta^n)^{\op},C^{\vee}) \\
 & \simeq \sO(C_n) \times_{C_n} \sO(C_{n-1}) \times_{C_{n-1}} \cdots \times_{C_2} \sO(C_1) \times_{C_1} C_0,
\end{align*}
where in the fiber product, the maps $\sO(C_k) \to C_k$ are given by evaluation at the target, and the maps $\sO(C_k) \to C_{k+1}$ are given by composing evaluation at the source with $\tau^{k+1}_k: C_k \to C_{k+1}$.
\end{prp}

\begin{ntn} Under the equivalence of Prop.~\ref{prp:DualDescriptionOfSections}, let $\Fun_{/(\Delta^n)^{\op}}((\Delta^n)^{\op},C^{\vee})_{\ext}$ denote the extendable objects. Then we will also write (abusing notation)
\[ \begin{tikzcd}[row sep=4ex, column sep=6ex, text height=1.5ex, text depth=0.5ex]
\Fun^{\cocart}_{/\Delta^n}(\sd(\Delta^n), C)_{\gen{1}} \ar{r}{\gamma_n^{\ast}}[swap]{\simeq} \ar[hook]{d} & \Fun_{/(\Delta^n)^{\op}}((\Delta^n)^{\op},C^{\vee})_{\ext} \ar[hook]{d} \\
\Fun^{\cocart}_{/\Delta^n}(\sd(\Delta^n), C) \ar{r}{\gamma_n^{\ast}} & \sO(C_n) \times_{C_n} \cdots \times_{C_1} C_0.
\end{tikzcd} \]
\end{ntn}

\begin{rem} The type of iterated fiber product occuring in Prop.~\ref{prp:DualDescriptionOfSections} appears in the work of Nikolaus and Scholze when they describe the data of a $C_{p^n}$-spectrum $X$ whose geometric fixed points (except possibly $\Phi^{C_{p^n}} X$) are all bounded below -- see \cite[Cor.~II.4.7]{NS18} and Prop.~\ref{prp:EquivalenceOnBoundedBelowAtFiniteLevel}.
\end{rem}

\section{The \texorpdfstring{$\cF$}{F}-recollement on \texorpdfstring{$\Sp^G$}{SpG}}
\label{section:FirstEquivariantSection}

In this section, we introduce and study recollements on the $\infty$-category $\Sp^G$ of $G$-spectra determined by a family $\cF$ of subgroups of $G$, for $G$ a finite group. We then apply our results in \S \ref{subsection:recollRLaxLLax} to reprove a theorem of Ayala, Mazel-Gee, and Rozenblyum that reconstructs $\Sp^G$ from its geometric fixed points (Thm.~\ref{thm:GeometricFixedPointsDescriptionOfGSpectra}). As a corollary, we deduce a limit formula (Cor.~\ref{cor:FormulaForGeomFixedPointsOfCompleteSpectrum}) for the geometric fixed points of an $\cF$-complete spectrum by means of Prop.~\ref{prp:ExistenceLaxRightKanExtension}, which will play an important role in our proof of the dihedral Tate orbit lemma (Exm.~\ref{exm:DihedralEven} and Lem.~\ref{lem:dihedralTOLEven}, Exm.~\ref{exm:DihedralOdd} and Lem.~\ref{lem:dihedralTOLOdd}).

\subsection{Conventions on equivariant stable homotopy theory}
\label{section:EquivariantConventions}

At the outset, let us be clear about which foundations for equivariant stable homotopy theory are employed in this paper. In their monograph, Nikolaus and Scholze choose to work with the classical point-set model of orthogonal $G$-spectra \cite[Def.~II.2.3]{NS18}, then obtaining the $\infty$-category $\Sp^G$ of $G$-spectra\footnote{In this paper, the term \emph{$G$-spectrum} is synonymous with genuine $G$-spectrum.} via inverting equivalences \cite[Def.~II.2.5]{NS18}. In contrast, we will use the foundations laid out by Bachmann and Hoyois in \cite[\S 9]{BachmannHoyoisNorms}, which attaches to every profinite groupoid $X$ a presentable, stable, and symmetric monoidal $\infty$-category $\SH(X)$ such that for $X = B G$, $\SH(B G)$ is equivalent to $\Sp^G$ as defined in \cite{NS18} (c.f. the remark prior to \cite[Lem.~9.5]{BachmannHoyoisNorms}). In fact, we will only need the Bachmann-Hoyois construction for finite groupoids.

\begin{dfn} \label{dfn:BachmannHoyoisFunctor} Let $\Gpd_{\fin}$ be the $(2,1)$-category of finite groupoids, and let $$\sH, \sH_{\sbullet}, \SH: \Gpd_{\fin}^{\op} \to \CAlg(\Pr^{\mr{L}})$$
denote the (restriction of the) functors constructed in \cite[\S 9.2]{BachmannHoyoisNorms}. For a map $f: X \to Y$ of finite groupoids, write $f^{\ast}$ for the associated functor and $f_{\ast}$ for its right adjoint.
\end{dfn}

\begin{rem} Let $X = BG$. Then $\sH(BG) \simeq \Spc^G \coloneq \Fun(\sO^{\op}_G, \Spc)$, the $\infty$-category of $G$-spaces defined as presheaves on the orbit category $\sO_G$, and likewise $\sH_{\sbullet}(BG)$ is the $\infty$-category $\Spc^G_{\ast}$ of pointed $G$-spaces. As we already mentioned, $\SH(BG) \simeq \Sp^G$ is the $\infty$-category of $G$-spectra, defined as the filtered colimit taken in $\Pr^L$
\[ \Spc^G_{\ast} \xto{\Sigma^{\rho}} \Spc^G_{\ast} \xto{\Sigma^{\rho}} \Spc^G_{\ast}  \xto{\Sigma^{\rho}} \cdots, \]
where $\rho$ is the regular $G$-representation. In addition, by \cite[Exm.~9.11]{BachmannHoyoisNorms} $\Sp^G$ is equivalent to the $\infty$-category of \emph{spectral Mackey functors} on finite $G$-sets that was studied by Barwick \cite{M1} and Guillou-May \cite{guillou2}.
\end{rem}

Note that by definition, $f^{\ast}: \SH(Y) \to \SH(X)$ is the symmetric monoidal left Kan extension of $\sH_{\sbullet}(Y) \xto{f^{\ast}} \sH_{\sbullet}(X) \xto{\Sigma^{\infty}} \SH(X)$ along $\Sigma^{\infty}: \sH_{\sbullet}(Y) \to \SH(Y)$. Therefore:
\begin{enumerate} \item Suppose $f: BH \to BG$ is the map of groupoids induced by an injective group homomorphism $H \to G$. Then $f^{\ast}: \Sp^G \to \Sp^H$ is homotopic to the usual restriction functor, and $f_{\ast}: \Sp^H \to \Sp^G$ is homotopic to the usual induction functor. Instead of $f^{\ast} \dashv f_{\ast}$, we will typically write this adjunction as $\res^G_H \dashv \ind^G_H$. Note that this adjunction is ambidextrous and satisfies the projection formula (in fact, \cite[Lem.~9.4(3)]{BachmannHoyoisNorms} establishes the projection formula for any finite covering map).
\item Suppose $f: BG \to BG/N$ is the map of groupoids induced by a surjective group homomorphism $G \to G/N$. Then $f^{\ast}: \Sp^{G/N} \to \Sp^G$ is homotopic to the usual inflation functor, which we denote as $\inf^N$. The right adjoint to $\inf^N$ is the \emph{categorical fixed points} functor
$$ \Psi^N: \Sp^G \to \Sp^{G/N}. $$
Now suppose $H \leq G$ is any subgroup and let $W_G H = N_G H / H$ be the Weyl group of $H$. Then we will also write
$$ \Psi^H: \Sp^G \xtolong{\res^G_{N_G H}}{1.5} \Sp^{N_G H} \xto{\Psi^H} \Sp^{W_G H} $$
\end{enumerate}

Given a $G$-spectrum $X$, we introduce notation to distinguish the underlying spectrum of $\Psi^H X$.

\begin{ntn} For a $G$-spectrum $X$ and subgroup $H \leq G$, we let $X^H = \res^{W_G H} \Psi^H (X)$.\footnote{With respect to the description of $\Sp^G$ as spectral Mackey functors, $X^H$ is given by evaluation at $G/H$.}
\end{ntn}

Since the restriction functor $\Sp^{W_G H} \to \Sp$ lifts to $\Fun(B W_G H, \Sp)$, the spectrum $X^H$ also comes endowed with a $W_G H$-action.

\begin{rem} By stabilizing the adjointability relations in \cite[Lem.~9.4]{BachmannHoyoisNorms}, it follows that that for any pullback square of finite groupoids
\[ \begin{tikzcd}[row sep=4ex, column sep=4ex, text height=1.5ex, text depth=0.25ex]
W \ar{r}{f} \ar{d}{g} & Y \ar{d}{g} \\
X \ar{r}{f} & Z,
\end{tikzcd} \]
the canonical natural transformation $f^{\ast} g_{\ast} \to f_{\ast} g^{\ast}$ of functors $\SH(X) \to \SH(Y)$ is an equivalence. In particular, we have an equivalence $X^H \simeq \Psi^H \res^G_H (X)$.
\end{rem}

We now turn to the \emph{geometric fixed points} and \emph{Hill-Hopkins-Ravenel norm} functors.

\begin{dfn} \label{BachmannHoyoisFunctorNorms} Let $\sH^{\otimes}, \sH_{\sbullet}^{\otimes}, \SH^{\otimes}: \Span(\Gpd_{\fin}) \to \CAlg(\Cat_{\infty}^{\mr{sift}})$ be the (restrictions of the) functors defined as in \cite[\S 9.2]{BachmannHoyoisNorms}, which on the subcategory $\Gpd_{\fin}^{\op}$ restrict to the functors $\sH, \sH_{\sbullet}, \SH$ of Def.~\ref{dfn:BachmannHoyoisFunctor}. For a map of finite groupoids $f: X \to Y$, write $f_{\otimes}$ for the associated covariant functor.
\end{dfn}

Parallel to the discussion above, we note \cite[Rmk.~9.9]{BachmannHoyoisNorms}:
\begin{enumerate} \item Suppose $f: BH \to BG$ for a subgroup $H \leq G$. Then $f_{\otimes}: \Sp^H \to \Sp^G$ is homotopic to the multiplicative norm functor $N^G_H$ introduced by Hill, Hopkins, and Ravenel \cite{HHR}.
\item Suppose $f: BG \to B(G/N)$. Then $f_{\otimes}: \Sp^G \to \Sp^{G/N}$ is homotopic to the usual geometric fixed points functor $\Phi^N$. For $H \leq G$ any subgroup, we also write
$$ \Phi^H: \Sp^G \xtolong{\res^G_{N_G H}}{1.5} \Sp^{N_G H} \xto{\Phi^H} \Sp^{W_G H}. $$
\end{enumerate}

\begin{ntn} For a $G$-spectrum $X$ and subgroup $H \leq G$, we let $X^{\phi H} = \res^{W_G H} \Phi^H (X)$. Also let
$$ \phi^H: \Sp^G \xto{\Phi^H} \Sp^{W_G H} \xto{\res} \Fun(B W_G H, \Sp). $$
\end{ntn}

\begin{rem} Because $\SH^{\otimes}$ is defined on $\Span(\Gpd_{\fin})$, we have that for any pullback square of finite groupoids
\[ \begin{tikzcd}[row sep=4ex, column sep=4ex, text height=1.5ex, text depth=0.25ex]
W \ar{r}{f} \ar{d}{g} & Y \ar{d}{g} \\
X \ar{r}{f} & Z,
\end{tikzcd} \]
there is a canonical equivalence $f^{\ast} g_{\otimes} \simeq f_{\otimes} g^{\ast}$ of functors $\SH(X) \to \SH(Y)$. In particular, we have an equivalence $X^{\phi H} \simeq \Phi^H \res^G_H X$.
\end{rem}

\begin{rem} We will use some additional features of these fixed points functors:
\begin{enumerate} \item For any subgroup $H \leq G$, the functor $\Psi^H$ is colimit-preserving, since the inflation functors preserve dualizable and hence compact objects; indeed, by equivariant Atiyah duality \cite[\S III.5.1]{MR866482} every compact object in $\Sp^G$ is dualizable, and conversely, since the unit in $\Sp^G$ is compact, all dualizable objects in $\Sp^G$ are compact.
\item The functors $\{ (-)^H : H \leq G \}$ are jointly conservative, since the orbits $\Sigma^{\infty}_+ G/H$ corepresent $(-)^H$ and form a set of compact generators for $\Sp^G$.
\item The functors $\{ \phi^H : H \leq G \}$ are jointly conservative, since the evaluation functors $\ev_{G/H}$ are jointly conservative for $\Spc^G$, $\phi^H \Sigma^{\infty}_+ \simeq \Sigma^{\infty}_+ \ev_{G/H}$, and suspension spectra generate $\Sp^G$ under desuspensions and sifted colimits.
\end{enumerate}
\end{rem}

We will also need to use some aspects of the theory of $G$-$\infty$-categories in this work.

\begin{dfn} Let $\omega_G: \FF_G \to \Gpd_{\fin}$ be the functor that sends a finite $G$-set $U$ to its action groupoid $U//G$.
\end{dfn}

\begin{dfn} \label{dfn:GCategoryGSpectra} Define the \emph{$G$-$\infty$-category of $G$-spectra} $\underline{\Sp}^G \to \sO^{\op}_G$ to be the cocartesian fibration classified by $\SH \circ (\omega^{\op}_G|_{\sO^{\op}_G})$. In addition, let $\underline{\Sp}^{G, \otimes} \to \sO^{\op}_G \times \Fin_{\ast}$ be the cocartesian $\sO_G^{\op}$-family of symmetric monoidal $\infty$-categories classified by $\SH \circ (\omega^{\op}_G|_{\sO^{\op}_G})$ (when viewed as valued in $\CMon(\Cat_{\infty})$).
\end{dfn}

\begin{rem} \label{rem:sliceCategoryPassage} For a subgroup $H$ of $G$, let
\[ \adjunct{\ind^G_H}{\FF_H}{\FF_G}{\res^G_H} \]
 denote the induction-restriction adjunction, where $\ind^G_H(U) = G \times_H U$. Then $\ind^G_H: \sO_H \to \sO_G$ factors as $\sO_H \simeq (\sO_G)_{/(G/H)} \to \sO_G$. Moreover, $\omega_G \circ \ind^G_H$ and $\omega_H$ are canonically equivalent, so we have an equivalence of $H$-$\infty$-categories
\[ \underline{\Sp}^H \simeq \sO_H^{\op} \times_{\sO_G^{\op}} \underline{\Sp}^G. \]
\end{rem}

\begin{rem} Given a $G$-$\infty$-category $K$, we may endow $\Fun_G(K, \ul{\Sp}^G)$ with the pointwise monoidal structure of Def.~\ref{dfn:S-PointwiseMonoidal} with respect to the construction $\underline{\Sp}^{G, \otimes}$ of Def.~\ref{dfn:GCategoryGSpectra}.
\end{rem}

\subsection{Basic theory of families}

\begin{dfn} \label{dfn:subconjugacyPoset} Given a finite group $G$, its \emph{subconjugacy category} $\fS[G]$ is the category whose objects are subgroups $H$ of $G$, and whose morphism sets are defined by
\begin{align*} \Hom_{\fS[G]}(H,K) = \begin{cases} \ast \quad \text{ if } H \text{ is subconjugate to } K, \\
\emptyset \quad \text{ otherwise}.
\end{cases}
\end{align*}
We will also write $\fS = \fS[G]$ if the ambient group $G$ is clear from context.
\end{dfn}

\begin{dfn} \label{dfn:family} A \emph{$G$-family} $\cF$ is a sieve in $\fS$, i.e., a full subcategory of $\fS$ whose set of objects is a set of subgroups of $G$ closed under subconjugacy.
\end{dfn}

\begin{rem} Abusing notation, we will also denote the set of objects of $\fS$ or a family $\cF$ by the same symbol. If we view morphisms in $\fS$ as defining a binary relation $\leq$ on the set of subgroups of $G$, then $\fS$ is a preordered set, which is a poset if $G$ is abelian. Although we generally reserve the expression $H \leq K$ for $H$ a subgroup of $K$, when discussing strings in the preordered set $\fS$ we will also write $\leq$ for its binary relation -- we trust the meaning to be clear from context.
\end{rem}

\begin{cnstr} \label{cnstr:GspaceFromGfamily} Given a $G$-family $\cF$, define $G$-spaces $E \cF$ and $\widetilde{E \cF}$ by the formulas
\begin{align*} E \cF^K = \begin{cases} \emptyset \; \text{ if } K \notin \cF \\
\ast \; \text{ if } K \in \cF
\end{cases}, \quad
\widetilde{E \cF}^K = \begin{cases} S^0 \; \text{ if } K \notin \cF \\
\ast \; \text{ if } K \in \cF
\end{cases}.
\end{align*}
We have a cofiber sequence of pointed $G$-spaces
\[ E \cF_+ \to S^0 \to \widetilde{E \cF}. \]
The unit map $S^0 \to \widetilde{E \cF}$ exhibits $\widetilde{E \cF}$ as an idempotent object \cite[Def.~4.8.2.1]{HA} of $\Spc^G_{\ast}$ with respect to the smash product, hence $\widetilde{E \cF}$ is a idempotent $E_{\infty}$-algebra by \cite[Prop.~4.8.2.9]{HA}.\footnote{This is also obvious since we are considering presheaves of sets.} Let $E \cF_+$ and $\widetilde{E \cF}$ also denote $\Sigma^{\infty}$ of the same pointed $G$-spaces. Then $\widetilde{E \cF}$ is an idempotent $E_{\infty}$-algebra in $\Sp^G$, and hence by the discussion in \ref{SmashingLocalizationsAreStableMonoidalRecollements} defines a stable monoidal recollement
\[ \begin{tikzcd}[row sep=4ex, column sep=4ex, text height=1.5ex, text depth=0.25ex]
\Sp^{h \cF} \ar[shift right=1,right hook->]{r}[swap]{j_{\ast}} & \Sp^G \ar[shift right=2]{l}[swap]{j^{\ast}} \ar[shift left=2]{r}{i^{\ast}} & \Sp^{\Phi \cF} \ar[shift left=1,left hook->]{l}{i_{\ast}}
\end{tikzcd} \]
such that $\Sp^{\Phi \cF} \simeq \Mod_{\Sp^G}(\widetilde{E \cF})$. By Cor.~\ref{cor:FractureSquareMonoidal}, for any $X \in \Sp^G$ we have the \emph{$\cF$-fracture square}
\[ \begin{tikzcd}[row sep=4ex, column sep=4ex, text height=1.5ex, text depth=0.25ex]
X \ar{r} \ar{d} & X \otimes \widetilde{E \cF} \ar{d} \\
F(E \cF_+, X) \ar{r} & F(E \cF_+, X) \otimes \widetilde{E \cF}.
\end{tikzcd} \]
Following standard terminology, we say that a $G$-spectrum $X$ is \emph{$\cF$-torsion}, \emph{$\cF$-complete}, or \emph{$\cF^{-1}$-local} if it is in the essential image of $j_!$, $j_{\ast}$, or $i_{\ast}$, respectively. Note that for a $G$-spectrum $X$,
\begin{itemize}
    \item $X$ is $\cF$-torsion if and only if $X \otimes E \cF_+ \xto{\simeq} X$ or $X \otimes\widetilde{E \cF} \simeq 0$.
    \item $X$ is $\cF$-complete if and only if $X \xto{\simeq} F(E \cF_+, X)$ or $F(\widetilde{E \cF},X) \simeq 0$.
    \item $X$ is $\cF^{-1}$-local if and only if $X \xto{\simeq} X \otimes \widetilde{E \cF}$ or $X \otimes E \cF_+ \simeq 0$.
\end{itemize}
\end{cnstr}

\begin{ntn} For a $G$-family $\cF$, we have already set $\Sp^{h \cF} \subset \Sp^G$ to be the full subcategory of $\cF$-complete $G$-spectra and $\Sp^{\Phi \cF} \subset \Sp^G$ to be the full subcategory of $\cF^{-1}$-local $G$-spectra. We also let $\Sp^{\tau \cF} \subset \Sp^G$ denote the full subcategory of $\cF$-torsion $G$-spectra.

In addition, if $\cF$ is the trivial family $\{1\}$, we will also write $E \cF = E G$, $\Sp^{h \cF} = \Sp^{h G}$, and refer to $\cF$-torsion or complete objects as \emph{Borel} torsion or complete.\footnote{Other authors refer to Borel torsion spectra as \emph{free} and Borel complete spectra as \emph{cofree}.} It is well-known that $\Sp^{h G} \simeq \Fun(B G, \Sp)$ (\cite[Prop.~6.17]{MATHEW2017994}, \cite[Thm.~II.2.7]{NS18}) -- we will later give two different generalizations of this fact (Lem.~\ref{lem:LocallyClosedFibersAreBorel} and Prop.~\ref{prp:BorelSpectraAsCompleteObjects}).
\end{ntn}

\begin{rem} \label{rem:torsionCompleteEquivalence} The functor $j_! j^{\ast}: \Sp^{h \cF} \xto{\simeq} \Sp^{\tau \cF}$ implements an equivalence between $\cF$-complete and $\cF$-torsion objects \cite[Prop.~7]{BarwickGlasmanNoteRecoll}.
\end{rem}

\begin{rem} \label{rem:MathewComparison} The endofunctors $j_! j^{\ast}$, $j_{\ast} j^{\ast}$, and $i_{\ast} i^{\ast}$ of $\Sp^G$ attached to a family $\cF$ agree with the $A_\cF$-acyclization, $A_\cF$-completion, and $A_\cF^{-1}$-localization functors in \cite{MATHEW2017994} defined with respect to the $E_{\infty}$-algebra $A_{\cF} \coloneq \prod_{H \in \cF} F(G/H_+, 1)$ by \cite[Prp 6.5-6.6]{MATHEW2017994}. Moreover, the theory of $A$-torsion, $A$-complete, and $A^{-1}$-local objects for a dualizable $E_{\infty}$-algebra $A$ (\cite[Part 1]{MATHEW2017994} under the hypotheses \cite[2.26]{MATHEW2017994}) extends the more general monoidal recollement theory for the idempotent object $1 \to U_A$ of \cite[Constr.~3.12]{MATHEW2017994}. For example, the $\cF$-fracture square for $\Sp^G$ given by Cor.~\ref{cor:RecollementAsPullbackSquare} agrees with the $A_{\cF}$-fracture square given by \cite[Thm.~3.20]{MATHEW2017994} (although we additionally consider the monoidal refinement Prop.~\ref{prp:CanonicalMonoidalStructureOnMonoidalRecollement}).

As a separate consequence, we also have that $\Sp^{\tau \cF} \subset \Sp^G$ is the localizing subcategory generated by the orbits $\{G/H_+ : H \in \cF \}$. Also, $G/H_+$ is both $\cF$-complete and $\cF$-torsion.
\end{rem}

In the remainder of this subsection, we collect some basic results concerning $\cF$-recollements that we will need in the sequel. Classical references for this material are \cite[\S II]{MR866482} and \cite[\S 17]{GreenleesMay}, and other references include \cite[\S 6]{MATHEW2017994} and \cite[\S 2]{AMGR-NaiveApproach}.

\begin{lem} \label{lem:GeometricFixedPointsDetectionCriterion} Let $\cF$ be a $G$-family and let $X \in \Sp^G$.
\begin{enumerate} \item $X$ is $\cF^{-1}$-local if and only if $X^{\phi K} \simeq 0$ for all $K \in \cF$.
\item $X$ is $\cF$-torsion if and only if $X^{\phi K} \simeq 0$ for all $K \notin \cF$.
\end{enumerate}
Therefore, for a map $f: X \to Y$ in $\Sp^G$, $f$ is a $j^{\ast}$-equivalence if and only if $f^{\phi K}$ is an equivalence for all $K \in \cF$, and $f$ is an $i^{\ast}$-equivalence if and only if $f^{\phi K}$ is an equivalence for all $K \notin \cF$.
\end{lem}
\begin{proof} First note that for any $X \in \Sp^G$ and subgroup $K$ of $G$,
\begin{align*} (X \otimes E \cF_+)^{\phi K} \simeq X^{\phi K} \otimes (E \cF_+)^{\phi K} & \simeq \begin{cases} 0 \; \text{ if } K \notin \cF \\
X^{\phi K} \; \text{ if } K \in \cF
\end{cases}, \\
(X \otimes \widetilde{E \cF})^{\phi K} \simeq X^{\phi K} \otimes \widetilde{E \cF} {}^{\phi K} & \simeq \begin{cases} X^{\phi K} \; \text{ if } K \notin \cF \\
0 \; \text{ if } K \in \cF
\end{cases}.
\end{align*}
Thus, if $X$ is $\cF^{-1}$-local so that $X \simeq X \otimes \widetilde{E \cF}$, then $X^{\phi K} \simeq 0$ for all $K \in \cF$. Conversely, if $X^{\phi K} \simeq 0$ for all $K \in \cF$, then $(X \otimes E \cF_+)^{\phi K} \simeq 0$ for all subgroups $K$, so by the joint conservativity of the functors $\phi^K$, $X \otimes E \cF_+ \simeq 0$ and $X$ is $\cF^{-1}$-local. This proves (1), and the proof of (2) is similar.
\end{proof}

\begin{rem}[Extension to $G$-recollement] \label{ParamRecollementFamily}  Suppose $\cF$ is a $G$-family, and let $\cF^H \subset \fS[H]$ denote the $H$-family obtained by intersecting $\cF$ with $\fS[H] \subset \fS[G]$. For any map of $G$-orbits $f: G/H \to G/K$ with associated adjunction $\adjunct{f^{\ast}}{\Sp^K}{\Sp^H}{f_{\ast}}$, note that 
\[ f^{\ast}(E \cF^K_+ \to S^0 \to \widetilde{E \cF^K}) \simeq E \cF^H_+ \to S^0 \to \widetilde{E \cF^H}. \]
By monoidality of $f^{\ast}$, it follows that $f^{\ast}$ preserves $\cF$-torsion and $\cF^{-1}$-local objects. Furthermore, the projection formula implies that 
\[ f^{\ast} F(E \cF^K_+,X) \simeq F(E \cF^H_+, f^{\ast} X), \]
so $f^{\ast}$ preserves $\cF$-complete objects. Therefore, $\cF$ defines a lift of the functor $\SH: \sO_G^{\op} \to \Cat^{\st}_{\infty}$ to $\Recoll^{\st}_0$. Passing to Grothendieck constructions, let
\[ \begin{tikzcd}[row sep=4ex, column sep=4ex, text height=1.5ex, text depth=0.25ex]
\underline{\Sp}^{h \cF} \ar[shift right=1,right hook->]{r}[swap]{j_{\ast}} & \underline{\Sp}^G \ar[shift right=2]{l}[swap]{j^{\ast}} \ar[shift left=2]{r}{i^{\ast}} & \underline{\Sp}^{\Phi \cF} \ar[shift left=1,left hook->]{l}{i_{\ast}}
\end{tikzcd} \]
denote the resulting diagram of $G$-adjunctions. By Cor.~\ref{cor:RecollementGivesParamStableSubcategories}, $\underline{\Sp}^{h \cF}$ and $\underline{\Sp}^{\Phi \cF}$ are $G$-stable $G$-$\infty$-categories and all $G$-functors in the diagram are $G$-exact. We thereby obtain a $G$-stable $G$-recollement $(\ul{\Sp}^{h \cF}, \ul{\Sp}^{\Phi \cF})$ of $\ul{\Sp}^G$ (Def.~\ref{dfn:ParamStableRecollement}).
\end{rem}

We may also consider $\cF$-recollements of the $\infty$-category of $G$-spaces (indeed, of any $\infty$-category of $\cE$-valued presheaves on $\sO_G$).

\begin{ntn} Given a $G$-family $\cF$, let $\sO_{G,\cF} \subset \sO_G$ be the full subcategory on those orbits with stabilizer in $\cF$, and let $\sO_{G,\cF}^c$ be its complement. 
\end{ntn}

\begin{cnstr}[$\cF$-recollement of $G$-spaces] \label{cnstr:SpacesRecollement} Given a $G$-family $\cF$, we may define a functor $\pi: \sO_G^{\op} \to \Delta^1$ such that $(\sO_G^{\op})_1 = (\sO_{G,\cF})^{\op}$ and $(\sO_G^{\op})_0 = (\sO_{G,\cF}^c)^{\op}$. Let $\Spc^{h \cF} = \Fun((\sO_{G,\cF})^{\op}, \Spc)$ and $\Spc^{\Phi \cF} = \Fun((\sO^c_{G,\cF})^{\op}, \Spc)$. By Exm.~\ref{exm:SieveCosieveRecollementOnFunctorCategory}, we obtain a monoidal recollement with respect to the cartesian product on $G$-spaces
\[ \begin{tikzcd}[row sep=4ex, column sep=4ex, text height=1.5ex, text depth=0.25ex]
\Spc^{h \cF} \ar[shift right=1,right hook->]{r}[swap]{j_{\ast}} & \Spc^G \ar[shift right=2]{l}[swap]{j^{\ast}} \ar[shift left=2]{r}{i^{\ast}} & \Spc^{\Phi \cF} \ar[shift left=1,left hook->]{l}{i_{\ast}}.
\end{tikzcd} \]
Moreover, if we instead take presheaves in $\Spc_{\ast}$, we obtain a monoidal recollement with respect to the smash product of pointed $G$-spaces
\[ \begin{tikzcd}[row sep=4ex, column sep=4ex, text height=1.5ex, text depth=0.25ex]
\Spc^{h \cF}_{\ast} \ar[shift right=1,right hook->]{r}[swap]{j_{\ast}} & \Spc^G_{\ast} \ar[shift right=2]{l}[swap]{j^{\ast}} \ar[shift left=2]{r}{i^{\ast}} & \Spc^{\Phi \cF}_{\ast} \ar[shift left=1,left hook->]{l}{i_{\ast}}.
\end{tikzcd} \]
where $\widetilde{E \cF} \simeq i_{\ast} i^{\ast}(S^0)$ and the unit map exhibits $\widetilde{E \cF}$ as the same idempotent object as above.

Given a map $f: X \to Y$ in $\Spc^G$, by definition $f$ is a $j^{\ast}$-equivalence if and only if $X^K \to Y^K$ is an equivalence for all $K \in \cF$, and $f$ is a $i^{\ast}$-equivalence if and only if $X^K \to Y^K$ is an equivalence for all $K \notin \cF$. Therefore, by Lem.~\ref{lem:GeometricFixedPointsDetectionCriterion} and the compatibility of geometric fixed points with $\Sigma^{\infty}_+$, the functor $\Sigma^{\infty}_+$ is a morphism of recollements $(\Spc^{h \cF}, \Spc^{\Phi \cF}) \to (\Sp^{h \cF}, \Sp^{\Phi \cF})$, and likewise for $\Sigma^{\infty}$. In particular, we get induced functors
\[ \Sigma^{\infty}_+: \Spc^{h \cF} \to \Sp^{h \cF}, \quad \Sigma^{\infty}_+: \Spc^{\Phi \cF} \to \Sp^{\Phi \cF}. \]
On the other hand, $\Omega^{\infty}$ is not a morphism of recollements; indeed, if $X \in \Sp^G$ is $\cF$-torsion, then we may have that $i^{\ast} \Omega^{\infty} X$ is non-trivial, so $\Omega^{\infty}$ does not preserve $i^{\ast}$-equivalences. However, if $f: X \to Y$ is a $j^{\ast}$-equivalence in $\Sp^G$, so that $f^{\phi K}$ is an equivalence for all $K \in \cF$, then $\res^G_H(f)$ is an equivalence for all $H \in \cF$ because the functors $\phi^K$ for $K \leq H$ jointly detect equivalences in $\Sp^H$. Therefore, $\Omega^{\infty} (f)$ is a $j^{\ast}$-equivalence, and the $\Sigma^{\infty}_+ \dashv \Omega^{\infty}$ adjunction induces an adjunction
\[ \adjunct{\Sigma^{\infty}_+}{\Spc^{h \cF}}{\Sp^{h \cF}}{\Omega^{\infty}}. \]
Now suppose $X$ is $\cF^{-1}$-local, so that $X^{\phi K} \simeq 0$ for all $K \in \cF$. Then $\res^G_H X \simeq 0$ for all $H \in \cF$, so $(\Omega^{\infty} X)^H \simeq \ast$ for all $H \in \cF$ and thus $\Omega^{\infty} X$ lies in the essential image of $i_{\ast}$. We thereby obtain an adjunction
\[ \adjunct{\Sigma^{\infty}_+}{\Spc^{\Phi \cF}}{\Sp^{\Phi \cF}}{\Omega^{\infty}}. \]

To summarize the various compatibilities, we have that
\begin{enumerate} \item $j^{\ast} \Sigma^{\infty}_+ \simeq \Sigma^{\infty}_+ j^{\ast}: \Spc^{G} \to \Sp^{h \cF}$ and $j_{\ast} \Omega^{\infty} \simeq \Omega^{\infty} j_{\ast}: \Sp^{h \cF} \to \Spc^{G}$.
\item $j^{\ast} \Omega^{\infty} \simeq \Omega^{\infty} j^{\ast}: \Sp^G \to \Spc^{h \cF}$ and $j_! \Sigma^{\infty}_+ \simeq \Sigma^{\infty} j_!: \Spc^{h \cF} \to \Sp^G$.
\item $i^{\ast} \Sigma^{\infty}_+ \simeq \Sigma^{\infty}_+ i^{\ast}: \Spc^G \to \Sp^{\Phi \cF}$ and $i_{\ast} \Omega^{\infty} \simeq \Omega^{\infty} i_{\ast}: \Sp^{\Phi \cF} \to \Spc^{G}$.
\end{enumerate}
\end{cnstr}

Next, we study situations that arise in the presence of two $G$-families.

\begin{rem} \label{rem:FamilyIntersectionFormula} Let $\cF$ and $\cG$ be two $G$-families. Then their intersection $\cF \cap \cG$ is again a $G$-family. Note that $E(\cF \cap \cG) \simeq E \cF \times E \cG$ as $G$-spaces, so $E \cF_+ \otimes E \cG_+ \simeq E(\cF \cap \cG)_+$. Consequently, for any $X \in \Sp^G$, the $\cG$-fracture square for $F(E \cF_+, X)$ yields a commutative diagram
\[ \begin{tikzcd}[row sep=4ex, column sep=4ex, text height=1.5ex, text depth=0.25ex]
F(E \cF_+, X) \otimes {E \cG}_+ \ar{r} \ar{d}{\simeq} & F(E \cF_+, X) \ar{r} \ar{d} & F(E \cF_+, X) \otimes \widetilde{E \cG} \ar{d} \\
F(E (\cF \cap \cG)_+, X) \otimes {E \cG}_+ \ar{r} & F(E (\cF \cap \cG)_+, X) \ar{r} & F(E (\cF \cap \cG)_+, X) \otimes \widetilde{E \cG}
\end{tikzcd} \]
in which the righthand square is a pullback square.
\end{rem}

\begin{lem} \label{lem:KeyIntersectionPropertyFamilies} Let $\cF$ and $\cG$ be two $G$-families. Then $\Sp^{\Phi \cG} \cap \Sp^{h \cF} = \Sp^{\Phi(\cF \cap \cG)} \cap \Sp^{h \cF}$ and $\Sp^{\Phi \cG} \cap \Sp^{h \cF} = \Sp^{\Phi \cG} \cap \Sp^{h(\cF \cup \cG)}$.
\end{lem}
\begin{proof} We prove the first equality, the proof of the second being similar. If $X$ is $\cG^{-1}$-local, then $X$ is $(\cG \cap \cF)^{-1}$-local by Lem.~\ref{lm:subfamilyProperties}(2), so we have the forward inclusion. On the other hand, by Rmk.~\ref{rem:FamilyIntersectionFormula}, for any $X \in \Sp^G$ we have that
\[ F(E \cF_+,X) \otimes E \cG_+ \simeq F(E(\cF \cap \cG)_+, X) \otimes E \cG_+. \]
But $F(E(\cF \cap \cG)_+, X) \simeq 0$ if $X$ is $(\cF \cap \cG)^{-1}$-local, and $X \simeq F(E \cF_+,X)$ if $X$ is $\cF$-complete. Thus, if $X$ is both $\cF$-complete and $(\cF \cap \cG)^{-1}$-local, then $X$ is $\cG^{-1}$-local. We thereby deduce the reverse inclusion.
\end{proof}

\begin{lem} \label{lm:subfamilyProperties} Suppose $\cG$ is a subfamily of $\cF$. Then
    \begin{enumerate}
    \item If $X$ is $\cG$-torsion, then $X$ is $\cF$-torsion.
    \item If $X$ is $\cF^{-1}$-local, then $X$ is $\cG^{-1}$-local.
    \item If $X$ is $\cG$-complete, then $X$ is $\cF$-complete.
    \item If $X$ is $\cG^{-1}$-local, then its $\cF$-completion $F(E \cF_+, X)$ is again $\cG^{-1}$-local.
    \item If $X$ is $\cG^{-1}$-local, then its $\cF$-acyclization $X \otimes E \cF_+$ is again $\cG^{-1}$-local.
    \item If $X$ is $\cG$-complete and $\cF^{-1}$-local, then $X \simeq 0$.
    \end{enumerate}
\end{lem}
\begin{proof} (1) and (2) follow immediately from Lem.~\ref{lem:GeometricFixedPointsDetectionCriterion}. For (3), to show $X$ is $\cF$-complete, we need to show that for all $\cF^{-1}$-local $Y$, $\Map(Y,X) \simeq \ast$. But by (2), $Y$ is $\cG^{-1}$-local, so this mapping space is contractible since $X$ is $\cG$-complete by assumption. For (4), we need to show that for all $\cG$-torsion $Y$, $\Map(Y,F(E \cF_+,X)) \simeq \ast$. But
\[ \Map(Y,F(E \cF_+,X)) \simeq \Map(Y \otimes E \cF_+,X) \simeq \Map(Y,X) \simeq \ast \]
since $Y \otimes E \cF_+ \simeq Y$ by (1) and the assumption that $X$ is $\cG^{-1}$-local. The proof of (5) is similar: given $\cG^{-1}$-local $X$ and any $\cG$-complete $Y$, we have that $\Map(X \otimes E \cF_+, Y) \simeq \ast$ because $Y$ is also $\cF$-complete by (3), hence $X \otimes E \cF_+$ is $\cG^{-1}$-local. Finally, for (6) note that $X$ is then $\cG^{-1}$-local by (2), hence $X \simeq 0$.
\end{proof}

Supposing still that $\cG$ is a subfamily of $\cF$, by Lem.~\ref{lm:subfamilyProperties}(1-3), the defining adjunctions of the $\cF$ and $\cG$-recollements on $\Sp^G$ restrict to adjunctions
\[ \adjunct{(i_{\cF})^{\ast}}{\Sp^{\Phi \cG}}{\Sp^{\Phi \cF}}{(i_{\cF})_{\ast}}, \adjunct{(j_{\cG})^{\ast}}{\Sp^{h \cF}}{\Sp^{h \cG}}{(j_{\cG})_{\ast}}, \adjunct{(j'_{\cG})_!}{\Sp^{\tau \cG}}{\Sp^{\tau \cF}}{(j'_{\cG})^{\ast}}. \]
By Lem.~\ref{lm:subfamilyProperties}(4-5), the $\cF$-completion adjunction restricts to
\[  \adjunct{(j_{\cF})^{\ast}}{\Sp^{\Phi \cG}}{\Sp^{h \cF} \cap \Sp^{\Phi \cG}}{(j_{\cF})_{\ast}} \]
such that $(j_{\cF})^{\ast}$ admits a left adjoint $(j_{\cF})_!$ given by the inclusion of $\cF$-torsion and $\cG^{-1}$-local objects under the equivalence $\Sp^{h \cF} \cap \Sp^{\Phi \cG} \simeq \Sp^{\tau \cF} \cap \Sp^{\Phi \cG}$.

Next, let $(i_{\cG})^{\ast}: \Sp^{h \cF} \to \Sp^{h \cF} \cap \Sp^{\Phi \cG}$ be the composite $\Sp^{h \cF} \subset \Sp^G \xto{i^{\ast}} \Sp^{\Phi \cG} \xto{(j_{\cF})^{\ast}} \Sp^{h \cF} \cap \Sp^{\Phi \cG}$. Then $(i_{\cG})^{\ast}$ is left adjoint to the inclusion $(i_{\cG})_{\ast}$. Likewise, define the left adjoint $(i'_{\cG})^{\ast}$ to the inclusion $(i'_{\cG})_{\ast}: \Sp^{\tau \cF} \cap \Sp^{\Phi \cG} \to \Sp^{\tau \cF}$. Finally, note that $\Sp^{h \cF} \cap \Sp^{\Phi \cG}$ inherits a symmetric monoidal structure from the localization $(j_\cF)^{\ast} \dashv (j_{\cF})_{\ast}$, with respect to which $(i_{\cG})^{\ast}$ is symmetric monoidal. Under the equivalence of Rmk.~\ref{rem:torsionCompleteEquivalence}, this transports to a monoidal structure on $\Sp^{\tau \cF}$ and $\Sp^{\tau \cF} \cap \Sp^{\Phi \cG}$ for which the adjunction $(i'_{\cG})^{\ast} \dashv (i'_{\cG})_{\ast}$ is monoidal.

\begin{prp} \label{prp:RecollementsOfRecollements} Let $\cG$ be a subfamily of $\cF$. We have stable monoidal recollements
\[ \begin{tikzcd}[row sep=4ex, column sep=6ex, text height=1.5ex, text depth=0.5ex]
\Sp^{h \cG} \ar[shift right=1,right hook->]{r}[swap]{(j_{\cG})_{\ast}} & \Sp^{h \cF} \ar[shift right=2]{l}[swap]{(j_{\cG})^{\ast}} \ar[shift left=2]{r}{(i_{\cG})^{\ast}} & \Sp^{h \cF} \cap \Sp^{\Phi \cG} \ar[shift left=1,left hook->]{l}{(i_{\cG})_{\ast}},
\end{tikzcd}
\begin{tikzcd}[row sep=4ex, column sep=6ex, text height=1.5ex, text depth=0.5ex]
\Sp^{\tau \cG} \ar[shift right=1,right hook->]{r}[swap]{(j'_{\cG})_{\ast}} & \Sp^{\tau \cF} \ar[shift right=2]{l}[swap]{(j'_{\cG})^{\ast}} \ar[shift left=2]{r}{(i'_{\cG})^{\ast}} & \Sp^{\tau \cF} \cap \Sp^{\Phi \cG} \ar[shift left=1,left hook->]{l}{(i'_{\cG})_{\ast}},
\end{tikzcd} \],
\[ \begin{tikzcd}[row sep=4ex, column sep=6ex, text height=1.5ex, text depth=0.5ex]
\Sp^{h \cF} \cap \Sp^{\Phi \cG} \ar[shift right=1,right hook->]{r}[swap]{(j_{\cF})_{\ast}} & \Sp^{\Phi \cG} \ar[shift right=2]{l}[swap]{(j_{\cF})^{\ast}} \ar[shift left=2]{r}{(i_{\cF})^{\ast}} & \Sp^{\Phi \cF} \ar[shift left=1,left hook->]{l}{(i_{\cF})_{\ast}}.
\end{tikzcd} \]
Furthermore, the equivalence $\Sp^{h \cF} \xto{\simeq} \Sp^{\tau \cF}$ of Rmk.~\ref{rem:torsionCompleteEquivalence} is an equivalence of recollements under which $(j_{\cG})_!$ is the inclusion of $\cG$-torsion objects into $\cF$-torsion objects.
\end{prp}
\begin{proof} The defining properties of a stable monoidal recollement follow immediately from the same properties for the $\cF$ and $\cG$ recollements on $\Sp^G$. For the last assertion, the equivalence of $\cF$-complete and $\cF$-torsion objects is implemented by $j_! j^{\ast}$, and as such clearly restricts to equivalences $\Sp^{h \cG} \xto{\simeq} \Sp^{\tau \cG}$ and $\Sp^{h \cF} \cap \Sp^{\Phi \cG} \xto{\simeq} \Sp^{\tau \cF} \cap \Sp^{\Phi \cG}$ compatibly with the adjunctions in view of Lem.~\ref{lm:subfamilyProperties}(4-5). Finally, the claim about $(j_{\cG})_!$ follows from a diagram chase of the right adjoints.
\end{proof}

\begin{rem}[Compact generation] \label{rem:CompactGenerationIntersection} Given a $G$-family $\cF$, the $\cF^{-1}$-local objects $\{ G/H_+ \otimes \widetilde{E \cF}: H \notin \cF \}$ form a set of compact generators for $\Sp^{\Phi \cF}$ because $\Sp^{\Phi \cF} = \Mod_{\Sp^G}(\widetilde{E \cF})$ and $G/H_+$ is $\cF$-torsion for all $H \in \cF$. Given two $G$-families $\cF$ and $\cG$, the essential image of $(j_\cF)_!$ is the localizing subcategory of $\Sp^{\Phi \cG}$ generated by $\{ G/H_+ \otimes \widetilde{E \cG}: H \notin \cG, H \in \cF \}$.
\end{rem}

\begin{rem} \label{NewRecollementSpaceCompatibility} The conclusions of Prop.~\ref{prp:RecollementsOfRecollements} are also valid for the $\cF$ and $\cG$ recollements on the $\infty$-category of $G$-spaces. We likewise have the adjunction $\adjunct{\Sigma^{\infty}_+}{\Spc^{h \cF} \cap \Spc^{\Phi \cG} }{\Sp^{h \cF} \cap \Sp^{\Phi \cG}}{\Omega^{\infty}}$ and the same compatibility relations as in Constr.~\ref{cnstr:SpacesRecollement}.
\end{rem}

\begin{rem} \label{rem:Fracture} Let us relate Prop.~\ref{prp:RecollementsOfRecollements} to the `canonical fracture' of $G$-spectra studied in \cite[\S 2.4]{AMGR-NaiveApproach}. We say that a full subcategory $C_0 \subset C$ is \emph{convex} if given any $x,z \in C_0$ such that there exists a $2$-simplex $[x \to y \to z] \in C$, then $y \in C_0$. Let $\Conv(\fS)$ denote the poset of convex subcategories of $\fS$ and let $\Loc(\Sp^G)$ denote the poset of reflective subcategories of $\Sp^G$, with the order given by inclusion. Suppose $Q \in \Conv(\fS)$ and write $Q = \cF \setminus \cG$ for some $G$-family $\cF$ and subfamily $\cG$. Then the assignment $$\mathfrak{F}_G: \Conv(\fS) \to \Loc(\Sp^G)$$ of \cite[Prop.~2.69]{AMGR-NaiveApproach} sends $Q$ to $\Sp^{h \cF} \cap \Sp^{\Phi \cG}$. Indeed, if we let $\cK_{H}$ be the localizing subcategory of $\Sp^G$ generated by $G/H_+$ and examine \cite[Notn.~2.54]{AMGR-NaiveApproach}, we see that $\cK_{\leq Q} \simeq \Sp^{h \cF}$ and $\cK_{< Q} \simeq \Sp^{h \cG}$ under the equivalence between torsion and complete objects. Thus, $\Sp^G_Q$ defined as the presentable quotient of $\cK_{< Q} \to \cK_{\leq Q}$ is equivalent to $\Sp^{h \cF} \cap \Sp^{\Phi \cG}$ in view of Prop.~\ref{prp:RecollementsOfRecollements}. Moreover, by inspection the functor $\rho: \Sp^G_Q \xto{\nu} \cK_{\leq Q} \xto{i_R} \Sp^G$ in \cite[Notn.~2.54]{AMGR-NaiveApproach} exhibiting $\Sp^G_Q$ as a reflective subcategory  embeds $\Sp^G_Q$ as $\cF$-complete and $\cG^{-1}$-local objects.

By \cite[Prop.~2.69]{AMGR-NaiveApproach} the functor $\mathfrak{F}_G: \Conv(\fS) \to \Loc(\Sp^G)$ is a \emph{fracture} in the sense of \cite[Def.~2.32]{AMGR-NaiveApproach}. Thus, for any convex subcategory $Q = \cF \setminus \cG$ and sieve-cosieve decomposition of $Q$ into $Q_0 = \cF_0 \setminus \cG_0$ and $Q_1 = \cF_1 \setminus \cG_1$, we obtain a recollement $(\Sp^{h \cF_0} \cap \Sp^{\Phi \cG_0}, \Sp^{h \cF_1} \cap \Sp^{\Phi \cG_1})$ of $\Sp^{h \cF} \cap \Sp^{\Phi \cG}$. It is easily seen that these specialize to those considered in Prop.~\ref{prp:RecollementsOfRecollements} in the case where $Q$ is itself a sieve or a cosieve.
\end{rem}

\begin{ntn} \label{ntn:locallyClosedFibers} Given a subgroup $H$ of $G$, let $\overline{H} = \fS_{\leq H}$ and $\partial \overline{H} = \fS_{<H}$ denote the $G$-family of subgroups that are subconjugate to $H$ and properly subconjugate to $H$, respectively.\footnote{This notation is consistent with viewing sieves as closed sets and cosieves as open sets for a topology on $\fS$.} Let $\fS^c_{\geq H}$ denote the $G$-family of subgroups $K$ such that $H$ is \emph{not} subconjugate to $K$.
\end{ntn}

\begin{lem} \label{lem:VanishingOutsideCone} Suppose $X \in \Sp^G$ is $\overline{H}$-complete and $(\partial \overline{H})^{-1}$-local. Then $X$ is in addition $(\fS^c_{\geq H})^{-1}$-local, i.e., for all subgroups $K$ such that $H$ is not subconjugate to $K$, $X^{\phi K} \simeq 0$.
\end{lem}
\begin{proof} Note that $\partial \overline{H} = \overline{H} \cap \fS^c_{\geq H}$ and use Lem.~\ref{lem:KeyIntersectionPropertyFamilies}.
\end{proof}

The following two lemmas are explained in \cite[Obs.~2.11-14]{AMGR-NaiveApproach}) (and the first one also in \cite[Prop.~II.2.14]{NS18}), so we will omit their proofs.

\begin{lem} \label{lem:ClosedPartRecollementNormalSubgroup} Let $N$ be a normal subgroup of $G$. Then the geometric fixed points functor $\Phi^N: \Sp^G \to \Sp^{G/N}$ has fully faithful right adjoint with essential image $\Sp^{\Phi \fS^c_{\geq H}}$. Consequently, $\Sp^{G/N}$ is equivalent to the smashing localization $\Mod_{\Sp^G}(\widetilde{E \fS^c_{\geq N}})$.
\end{lem}

\begin{lem} \label{lem:LocallyClosedFibersAreBorel} The geometric fixed points functor $\phi^H: \Sp^G \to \Fun(B W_G H, \Sp)$ has fully faithful right adjoint with essential image $\Sp^{h \overline{H}} \cap \Sp^{\Phi (\partial \overline{H})} =\Sp^{h \overline{H}} \cap \Sp^{\Phi \fS^c_{\geq H}}$.
\end{lem}

\subsection{Reconstruction from geometric fixed points}
 
\label{section:Reconstruction}

We next aim to state the reconstruction theorem \cite[Thm.~A]{AMGR-NaiveApproach} of Ayala, Mazel-Gee, and Rozenblyum. For this, we need a few preliminary notions.

\begin{dfn} \label{dfn:GeometricLocus} The $G$-\emph{geometric locus} $$\Sp^G_{\locus{\phi}} \subset \Sp^G \times \fS[G]$$ is the full subcategory on objects $(X,H)$ such that $X \in \Sp^{h \overline{H}} \cap \Sp^{\Phi (\partial \overline{H})}$, i.e., $X$ is $\overline{H}$-complete and $(\partial \overline{H})^{-1}$-local (Notn.~\ref{ntn:locallyClosedFibers}).
\end{dfn}

\begin{dfn} Given $H$ subconjugate to $K$, the \emph{generalized Tate construction}
\[ \tau^K_H: \Fun(B W_G H, \Sp) \to \Fun(B W_G K, \Sp) \]
is the functor given by the composition
\[ \Fun(B W_G H, \Sp) \to \Sp^G \xto{\phi^K} \Fun(B W_G K, \Sp) \]
where the first functor is the right adjoint to $\phi^H$. If $H=1$, then we will write $\tau^K \coloneq \tau^K_1$.
\end{dfn}

\begin{rem} \label{rem:GenTateResCompatibility} Evidently, the generalized Tate functors $\tau^K_H$ inherit some compatibility properties from the geometric fixed points functors. For example, for $H$ a subgroup of $K$ in $G$, the commutative diagrams
\[ \begin{tikzcd}[row sep=4ex, column sep=4ex, text height=1.5ex, text depth=0.25ex]
\Sp^K \ar{r}{\Phi^H} \ar{d}[swap]{\ind} & \Sp^{W_K H} \ar{d}[swap]{\ind} \ar{r} & \Fun(B W_K H, \Sp) \ar{d}[swap]{\ind} \\
\Sp^G \ar{r}{\Phi^H} & \Sp^{W_G H} \ar{r} & \Fun(B W_G H, \Sp)
\end{tikzcd},
\begin{tikzcd}[row sep=4ex, column sep=4ex, text height=1.5ex, text depth=0.25ex]
\Sp^G \ar{r}{\Phi^K} \ar{d}{\res} & \Sp^{W_K H} \ar{d}{\res} \\
\Sp^K \ar{r}{\Phi^K} & \Sp
\end{tikzcd}
 \]
imply that the diagram of generalized Tate functors defined relative to $G$ and $K$
\[ \begin{tikzcd}[row sep=4ex, column sep=4ex, text height=1.5ex, text depth=0.25ex]
\Sp^{h W_G H} \ar{r}{\tau^K_H} \ar{d}{\res} & \Sp^{h W_G K} \ar{d}{\res} \\
\Sp^{h W_K H} \ar{r}{\tau^K_H} & \Sp
\end{tikzcd} \]
commutes. The notation is therefore unambiguous (or abusive) in the same sense as that for geometric fixed points.

Also, if $N_G K = N_G H$, then the composite (with the first functor right adjoint to $\Phi^H$)
\[ \Sp^{W_G H} \to \Sp^G \xto{\Phi^K} \Sp^{W_G K} \]
is homotopic to $\Phi^{K/H}$, and thus $\tau^K_H \simeq \tau^{K/H}$ for $K/H$ regarded as a normal subgroup of $W_G H$. 

Finally, note that if $G = C_p$ is a cyclic group of prime order, then $\tau^{C_p} \simeq t^{C_p}$ is the ordinary Tate construction, but not generally otherwise.
\end{rem}

\begin{lem} The structure map $p: \Sp^G_{\locus{\phi}} \to \fS$ is the locally cocartesian fibration such that the functors $\tau^K_H$ are the pushforward functors encoded by $p$ under the equivalence of Lem.~\ref{lem:LocallyClosedFibersAreBorel}.
\end{lem}
\begin{proof} This is \cite[Constr.~2.38]{AMGR-NaiveApproach} applied to the fracture $\mathfrak{F}_G$ of Rmk.~\ref{rem:Fracture}. To spell out a few more details, we need to show that for every edge $e: \Delta^1 \to \fS$ given by $H$ subconjugate to $K$, the pullback $p|_{e}$ of $p$ over $\Delta^1$ is a cocartesian fibration. Let $C' \subset \Sp^G \times \Delta^1$ be the full subcategory on objects $\{(X,i) \}$ where if $i=0$, then $X \in (\Sp^G_{\locus{\phi}})_H$. Then we have a factorization
 \[ \Sp^G_{\locus{\phi}} \times_{\fS, e} \Delta^1 \xto{i''} C' \xto{i'} \Sp^G \times \Delta^1. \]
Note that $C' \to \Delta^1$ is a sub-cocartesian fibration of $\Sp^G \times \Delta^1$ via $i'$ (with cocartesian edges exactly those sent to equivalences via the projection to $\Sp^G$). As for the fiber over $1$, by definition we have that $(\Sp^G_{\locus{\phi}})_K$ is a localization of $\Sp^G$. By an elementary lifting argument, this extends to a localization functor $L: C' \to C'$ whose essential image is $\Sp^G_{\locus{\phi}} \times_{\fS, e} \Delta^1$. By \cite[Lem.~2.2.1.11]{HA}, we deduce that $p|_{e}$ is a cocartesian fibration.
\end{proof}


Recall the barycentric subdivision construction (Def.~\ref{dfn:barycentricSubdivision} and Rmk.~\ref{rem:barycentricSubdivisionOrdinaryCategory}). Unwinding that definition in our situation of interest, we see that $\sd(\fS)$ is the category whose objects are strings $\kappa = [H_0 < H_1 < \cdots < H_n]$ in $\fS$ with each $H_i$ properly subconjugate to $H_{i+1}$, and where a morphism
\[ \kappa = [H_0 < H_1 < \cdots < H_n] \to \lambda = [K_0 < K_1 < \cdots < K_m] \]
is the data of an injective map $\alpha: [n] \to [m]$ of totally ordered sets and a commutative diagram in $\fS$
\[ \begin{tikzcd}[row sep=4ex, column sep=4ex, text height=1.5ex, text depth=0.25ex]
H_{0} \ar{r} \ar{d} & H_{1} \ar{r} \ar{d} & \cdots \ar{r} & H_n \ar{d} \\
K_{\alpha(0)} \ar{r} & K_{\alpha(1)} \ar{r} & \cdots \ar{r} & K_{\alpha(n)}
\end{tikzcd} \]
whose vertical morphisms are equivalences. Note that if a morphism $\kappa \to \lambda$ exists, then $\alpha$ and the commutative ladder are uniquely determined. Thus, the morphism sets in $\sd(\fS)$ are either empty or singleton and $\sd(\fS)$ is also a preordered set. Regard $\sd(\fS)$ as a locally cocartesian fibration over $\fS$ via the functor which takes a string to its maximum element (Constr.~\ref{cnstr:MaxFunctorSubdivision}).


\begin{rem} Given any locally cocartesian fibration $p: C \to \fS$ whose fibers $C_H$ are stable $\infty$-categories and whose pushforward functors are exact, the right-lax limit $\Fun^{\cocart}_{/\fS}(\sd(\fS),C)$ is a stable $\infty$-category by Lem.~\ref{lem:rlaxLimitAdmitsFiniteLimitsAndStable}. Moreover, if the fibers are presentable and the pushforward functors are also accessible, then the right-lax limit is presentable by Prop.~\ref{prp:rightLaxLimitPresentable}. 
\end{rem}

We may now state \cite[Thm.~A]{AMGR-NaiveApproach}, rewritten in our notation.

\begin{thm} There is a canonical equivalence $\Sp^G \simeq \Fun^{\cocart}_{/\fS}(\sd(\fS),\Sp^G_{\locus{\phi}})$.
\end{thm}

Examining the proof of \cite[Thm.~2.40]{AMGR-NaiveApproach}, we see that this equivalence is implemented by the \emph{right-lax} functor $\Sp^G \times \fS \dashrightarrow \Sp^G_{\locus{\phi}}$ that globalizes the left adjoints $\phi^H$. This is not expressible as a functor $\Sp^G \times \fS \to \Sp^G_{\locus{\phi}}$; rather, its construction derives from an existence and uniqueness theorem on adjunctions in $(\infty,2)$-categories (\cite[Lem.~1.34]{AMGR-NaiveApproach} and \cite[Cor.~3.1.7]{gaitsgory2017study}). However, by instead working with the defining inclusion $\Sp^G_{\locus{\phi}} \subset \Sp^G \times \fS$, we can avoid serious usage of $(\infty,2)$-category theory and still define a comparison functor, as in the following construction.

\begin{cnstr} \label{cnstr:ComparisonFunctorFromRightLaxLimitToGSpectra} Let $\cF$ be a $G$-family, $\cG$ a subfamily, and $\cH = \cF \setminus \cG$. Consider the composite functor
\[ \Theta'_{\cH}: \Fun^{\cocart}_{/\cH}(\sd(\cH), \cH \times_{\fS} \Sp^G_{\locus{\phi}}) \to \Fun(\sd(\cH), \Sp^G) \xto{\lim} \Sp^G \]
where the first functor is postcomposition by the projection to $\Sp^G$ and the second takes the limit. Note that by Lem.~\ref{lm:subfamilyProperties}, if $X \in (\Sp^G_{\locus{\phi}})_H$ for any $H \in \cF \setminus \cG$, then $X \in \Sp^{h \cF} \cap \Sp^{\Phi \cG}$. Therefore, $\Theta'_{\cH}$ factors through the inclusion $\Sp^{h \cF} \cap \Sp^{\Phi \cG} \subset \Sp^G$. Denote that functor by $\Theta_{\cH}$.

In the case of $\cF = \fS$ and $\cG = \emptyset$, we also write $\Theta$ for the comparison functor.
\end{cnstr}

\begin{lem} \label{lm:GeometricFixedPointsOfComparisonFunctor} Let $\cF$ be a $G$-family, $\cG$ a subfamily, and $\cH = \cF \setminus \cG$. For every $H \in \cH$, the composition
\[ \Fun^{\cocart}_{/\cH}(\sd(\cH), \cH \times_{\fS} \Sp^G_{\locus{\phi}}) \xto{\Theta'_{\cH}} \Sp^G \xto{\phi^H} \Fun(B W_G H, \Sp) \]
is homotopic to evaluation at $H \in \sd(\cH)$ under the equivalence $(\Sp^G_{\locus{\phi}})_H \simeq \Fun(B W_G H, \Sp)$.
\end{lem}
\begin{proof} Let $f: \sd(\cH) \to \Sp^G_{\locus{\phi}}$ be an object in $\Fun^{\cocart}_{/\cH}(\sd(\cH), \cH \times_{\fS} \Sp^G_{\locus{\phi}})$, and let $f': \sd(\cH) \to \Sp^G$ denote the subsequent functor obtained by the projection to $\Sp^G$. We need to produce a natural equivalence $\phi^H \lim f' \simeq f'(H)$. Since $\sd(\cH)$ is finite, it suffices instead to show $\lim \phi^H f' \simeq f'(H)$. Note that for any $X \in (\Sp^G_{\locus{\phi}})_K$, if $K$ is not in $\overline{H}$ then $X^{\phi H} \simeq 0$; indeed, $\Phi^L(X) \simeq 0$ for all $L \in \fS_{\geq K}^c$ by definition. Therefore, if we let $J \subset \sd(\cH)$ be the full subcategory on those strings $\sigma$ with $\max(\sigma) \leq H$, the functor $\phi^H f'$ is a right Kan extension of its restriction to $J$ (for this, also note that if $\tau = [K_0 < \cdots < K_n] \in \sd(\cH)$ with $K_n \notin \overline{H}$, then $\sd(\cH)_{\tau/} \times_{\sd(\cH)} J = \emptyset$).

Next, let $I \subset J$ be the full subcategory on those strings $\sigma$ with $\max(\sigma)$ conjugate to $H$. For a string $\tau = [K_0 < ... < K_n] \in J$ with $K_n$ properly subconjugate to $H$, the unique string inclusion $e: [K_0 < ... < K_n] \to [K_0 < ... < K_n < H]$ is sent to an equivalence by $\phi^H f'$ by definition of the locally cocartesian edges in $\Sp^G_{\locus{\phi}}$; indeed, $f'(e)$ is a unit map of the localization for the reflective subcategory $(\Sp^G_{\locus{\phi}})_H \subset \Sp^G$.  Observe also that $e$ is an initial object in $I \times_J J^{\tau/}$. We deduce that $\phi^H f'$ is a right Kan extension of its further restriction to $I$. Because $H$ is an initial object of $I$, we conclude that $\lim \phi^H f' \simeq f'(H)$, as desired.
\end{proof}

For the next proposition, recall from Thm.~\ref{thm:RecollementRlaxLimitOfLlaxFunctor} the recollement of a right-lax limit defined by a sieve-cosieve decomposition of the base. 

\begin{prp} \label{prp:ComparisonFunctorIsStrictMorphismOfRecollements} Let $\cF$ be a $G$-family, $\cG$ a subfamily, and $\cH = \cF \setminus \cG$. The functor 
\[ \Theta_{\cF}: \Fun^{\cocart}_{/\cF}(\sd(\cF), \cF \times_{\fS} \Sp^G_{\locus{\phi}}) \to \Sp^{h \cF} \]
 is a strict morphism of stable recollements
\[ (\Fun^{\cocart}_{/\cG}(\sd(\cG), \cG \times_{\fS} \Sp^G_{\locus{\phi}}), \Fun^{\cocart}_{/\cH}(\sd(\cH), \cH \times_{\fS} \Sp^G_{\locus{\phi}})) \to (\Sp^{h \cG}, \Sp^{h \cF} \cap \Sp^{\Phi \cG}). \]
Moreover, the resulting functors between the open and closed parts are equivalent to $\Theta_{\cG}$ and $\Theta_{\cH}$.
\end{prp}
\begin{proof} We need to show that $\Theta_{\cF}$ sends the essential images of $j_!$, $j_{\ast}$, and $i_{\ast}$ to $\cG$-torsion\footnote{More precisely, $\cG$-torsion with respect to the embedding of $\Sp^{h \cF}$ in $\Sp^G$ as $\cF$-torsion objects.}, $\cG$-complete, and $\cG^{-1}$-local objects, respectively. Let $f: \sd(\cF) \to \cF \times_{\fS} \Sp^G_{\locus{\phi}}$ be a functor that preserves locally cocartesian edges. By Prop.~\ref{prp:openPartOfRecollementLeftAdjoint}, if $f$ is in the essential image of $j_!$, then $f(H) = 0$ for all $H \in \cH$. By Lem.~\ref{lm:GeometricFixedPointsOfComparisonFunctor}, we then have $\phi^H \Theta_{\cF}(f) \simeq 0$ for all $H \in \cH$, so $\Theta_{\cF}(f)$ is $\cG$-torsion. Similarly, using Prop.~\ref{prp:closedPartOfRecollement} and Lem.~\ref{lm:GeometricFixedPointsOfComparisonFunctor} again, the same proof shows that if $f$ is in the essential image of $i_{\ast}$, then $\phi^H \Theta_{\cF}(f) \simeq 0$ for all $H \in \cG$ and thus $\Theta_{\cF}(f)$ is $\cG^{-1}$-local. Finally, suppose that $f$ is in the essential image of $j_{\ast}$. By Prop.~\ref{prp:ExistenceLaxRightKanExtension}, $f$ is a relative right Kan extension of its restriction to the subcategory $\sd(\cF)_0$ of strings whose minimums lie in $\cG$. Because the inclusion $(\Sp^G_{\locus{\phi}})_H \subset \Sp^G$ of each fiber preserves limits, the further composition $f': \sd(\cF) \xto{f} \Sp^G_{\locus{\phi}} \to \Sp^G$ is then a right Kan extension of its restriction to $\sd(\cF)_0$ (in the non-relative sense). Moreover, the inclusion $\sd(\cG) \subset \sd(\cF)_0$ is right cofinal. Indeed, for every string $\sigma = [K_0 < ... < K_n]$ in $\sd(\cF)_0$, if we let $\sigma'$ denote its maximal substring in $\sd(\cG)$, then $\sigma'$ is a terminal object in $(\sd(\cF)_0)^{/\sigma} \times_{\sd(\cF)_0} \sd(\cG)$, so these slice categories are weakly contractible and we may thus apply Joyal's version of Quillen's Theorem A \cite[Thm.~4.1.3.1]{HTT}. It follows that $\Theta_{\cF}(f)$ is computed as a limit of $\cG$-complete spectra and is hence itself $\cG$-complete.

The two functors on the open and closed parts induced by the morphism of stable recollements are then definitionally $(j_{\cG})^{\ast} \Theta_{\cF} j_{\ast}$ and $(i_{\cG})^{\ast} \Theta_{\cF} i_{\ast}$. These are equivalent to $\Theta_{\cG}$ and $\Theta_{\cH}$ by the same cofinality arguments as above.
\end{proof}

\begin{thm} \label{thm:GeometricFixedPointsDescriptionOfGSpectra} For every $G$-family $\cF$ and subfamily $\cG$, the functor $\Theta_{\cF \setminus \cG}$ is an equivalence of $\infty$-categories. In particular, we have an equivalence
\[ \Theta: \Fun^{\cocart}_{/\fS}(\sd(\fS), \Sp^G_{\locus{\phi}}) \xto{\simeq} \Sp^G. \]
\end{thm}
\begin{proof} Our strategy is to use Prop.~\ref{prp:ComparisonFunctorIsStrictMorphismOfRecollements} in conjunction with the fact that given a strict morphism $F: \cX \to \cX'$ of stable recollements $(\cU, \cZ) \to (\cU', \cZ')$, if $F_{U}$ and $F_{Z}$ are equivalences then $F$ is an equivalence (Rmk.~\ref{rem:TwoOutOfThreePropertyEquivalencesStrictMorphismRecoll}).\footnote{This type of inductive argument is also used in the proof of \cite[Thm.~2.40]{AMGR-NaiveApproach}.} Let us first prove that $\Theta_{\cF}$ is an equivalence for all families $\cF$. We proceed by induction on the size of $\cF$. For the base case, if $\cF = \{ 1 \}$ is the trivial family, then $\sd(\cF) \cong \cF$ and $\Theta_{\cF}$ is definitionally an equivalence. Now suppose for the inductive hypothesis that $\Theta_{\cG}$ is an equivalence for all proper subfamilies $\cG$ of $\cF$. Let $H \in \cF$ be a maximal element and let $\cG \subset \cF$ be the largest subfamily excluding $H$. Then $\cF \setminus \cG = \overline{H} \setminus \partial \overline{H}$, so $\Theta_{\cF \setminus \cG}$ is definitionally an equivalence. By Prop.~\ref{prp:ComparisonFunctorIsStrictMorphismOfRecollements}, we deduce that $\Theta_{\cF}$ is an equivalence.

Finally, to deal with the general case, we note that any strict morphism of stable recollements that is also an equivalence restricts to equivalences between the open and closed parts. Thus, having proven that $\Theta_{\cF}$ is an equivalence, we further deduce that $\Theta_{\cF \setminus \cG}$ is an equivalence for any subfamily $\cG$.
\end{proof}

\begin{rem} The generalized Tate functors $\tau^K_H$ are lax monoidal, and the various natural transformations among these functors encoded by the locally cocartesian fibration are also lax monoidal. This data should assemble to a symmetric monoidal structure on the right-lax limit $\Fun^{\cocart}_{/\fS}(\sd(\fS),\Sp^G_{\locus{\phi}})$ such that the functor $\Theta$ of Thm.~\ref{thm:GeometricFixedPointsDescriptionOfGSpectra} is an equivalence of symmetric monoidal $\infty$-categories. However, in the formalism of $\infty$-operads it seems difficult to make this intuition rigorous. Instead, we may endow $\Fun^{\cocart}_{/\fS}(\sd(\fS),\Sp^G_{\locus{\phi}})$ with the symmetric monoidal structure of $\Sp^G$ obtained by transfer of structure under $\Theta$.
\end{rem}

Thm.~\ref{thm:GeometricFixedPointsDescriptionOfGSpectra} and Prop.~\ref{prp:ComparisonFunctorIsStrictMorphismOfRecollements}, along with the explicit description of the functor $j_{\ast}$ given in Prop.~\ref{prp:ExistenceLaxRightKanExtension}, gives a formula for the geometric fixed points of an $\cF$-complete spectrum in terms of a limit of generalized Tate constructions.

\begin{cor} \label{cor:FormulaForGeomFixedPointsOfCompleteSpectrum} Let $X$ be a $G$-spectrum and let $X^{\bullet}: \sd(\fS) \to \Sp^G_{\locus{\phi}}$ denote a lift of $X$ under the equivalence $\Theta$. Suppose that $X$ is $\cF$-complete. For $H \notin \cF$, let $J_H \subset \sd(\fS)$ be the full subcategory on strings $[K_0 < \cdots < K_n < H]$ such that $K_i \in \cF$. Then $$X^{\phi H} \simeq \lim_{J_H} X^{\bullet},$$
with the limit taken in the fiber $\Fun(B W_G H, \Sp) \simeq (\Sp^G_{\locus{\phi}})_H$.
\end{cor}

\begin{exm} Suppose that $G = C_{p^2}$ and let $\cP$ be the family of proper subgroups of $G$. Then $\sd(\cP) \cong \sd(\Delta^1)$, so the data of a $\cP$-complete spectrum $X$ amounts to
\begin{itemize}
\item A Borel $C_{p^2}$-spectrum $X^1$.
\item A Borel $C_{p^2}/ C_p$-spectrum $X^{\phi C_p}$.
\item A $C_{p^2}/ C_p \cong C_p$-equivariant map $\alpha: X^{\phi C_p} \to (X^1)^{t C_p}$.
\end{itemize}
The category $J_{C_{p^2}}$ as well as the functor $J_{C_{p^2}} \to \Sp$ is then identified as
\[ \left( \begin{tikzcd}[row sep=4ex, column sep=4ex, text height=1.5ex, text depth=0.25ex]
& \goesto{\left[ C_p < C_{p^2} \right] \ar{d} \\
\left[ 1 < C_{p^2} \right] \ar{r} & \left[ 1 < C_p < C_{p^2} \right]
\end{tikzcd} \right) }{ \left(
\begin{tikzcd}[row sep=4ex, column sep=4ex, text height=1.5ex, text depth=0.25ex]
& X^{\phi C_p t C_p} \ar{d}{\alpha^{t C_p}} \\
(X^1)^{\tau C_{p^2}} \ar{r}{\can} & (X^1)^{t C_p t C_p}
\end{tikzcd} \right) }, \]
where $\mit{can}$ is the canonical map encoded by the locally cocartesian fibration. Thus,
\[ X^{\phi C_{p^2}} \simeq (X^1)^{\tau C_{p^2}} \times_{(X^1)^{t C_p t C_p}} X^{\phi C_p t C_p}. \]
We will later see that $(-)^{\tau C_{p^2}} \simeq (-)^{h C_p tC_p}$ (Lem.~\ref{lm:identifyGenTate}).
\end{exm}

Let us now turn to the examples of interest for the dihedral Tate orbit lemma.

\begin{exm} \label{exm:DihedralEven} Suppose that $G = D_4 = C_2 \times \mu_2$ is the Klein four-group and let $\Gamma = \{ 1, C_2, \Delta \}$ for $\Delta$ the diagonal subgroup. The data of a $\Gamma$-complete spectrum $X$ amounts to
\begin{itemize}
\item A Borel $D_4$-spectrum $X^1$, Borel $(D_4/C_2)$-spectrum $X^{\phi C_2}$, and Borel $(D_4/\Delta)$-spectrum $X^{\phi \Delta}$.
\item A $(D_4/C_2)$-equivariant map $\alpha: X^{\phi C_2} \to (X^1)^{t C_2}$ and $(D_4/ \Delta)$-equivariant map $\beta: X^{\phi \Delta} \to (X^1)^{t \Delta}$.
\end{itemize}
Since $J_{\mu_2} = \{ [1<\mu_2] \}$ and $J_{\mu_2} \to \Fun(B(D_4/\mu_2),\Sp)$ is the pushforward of $X^1$ by $(-)^{t \mu_2}$, we see that $X^{\phi \mu_2} \simeq (X^1)^{t \mu_2}$. On the other hand, $J_{D_4} \to \Sp$ is given by
\[ \left( \begin{tikzcd}[row sep=4ex, column sep=4ex, text height=1.5ex, text depth=0.25ex]
\goesto{\left[ \Delta < D_4 \right] \ar{r} & \left[ 1<\Delta<D_4 \right] \\
\left[ 1 < D_4 \right] \ar{ru} \ar{rd} &  \\
\left[ C_2 < D_4 \right] \ar{r} & \left[ 1 < C_2 < D_4 \right]
\end{tikzcd} \right)}{\left(
\begin{tikzcd}[row sep=4ex, column sep=8ex, text height=1.5ex, text depth=0.25ex]
(X^{\phi \Delta})^{t(D_4/\Delta)} \ar{r}{\beta^{t(D_4/\Delta)}} & ((X^1)^{t \Delta})^{t (D_4/ \Delta)} \\
(X^1)^{\tau D_4} \ar{ru}[swap]{\can} \ar{rd}{\can} &  \\
(X^{\phi C_2})^{t (D_4/C_2)} \ar{r}[swap]{\alpha^{t(D_4/C_2)}} & ((X^1)^{t C_2})^{t (D_4/C_2)}
\end{tikzcd} \right)}, \]
and $X^{\phi D_4}$ is the limit of this diagram.
\end{exm}

To handle the case of the dihedral group $D_{2p}$ of order $2p$ for $p$ an odd prime, we first record a vanishing property of the generalized Tate construction.

\begin{lem} \label{lem:GenTateVanishingMultiprime} Let $G$ be a finite group and suppose $K \leq G$ is a subgroup that is not a $p$-group. Then $\tau^K \simeq 0$.
\end{lem}
\begin{proof} By the compatibility of the generalized Tate functors with restriction (Rmk.~\ref{rem:GenTateResCompatibility}), we may suppose $K = G$ without loss of generality. Note that $\tau^G$ may be computed as the left Kan extension of $(-)^{h G}$ along the functor from $\Sp^{h G}$ to its Verdier quotient by orbits $\{ G/H_+ : H <G\}$ with $H$ proper \cite[Rmk.~2.16]{AMGR-NaiveApproach}. If we let $\ul{All}$ be the family of subgroups $H$ such that $|H|=p^n$ for some prime $p$ and integer $n$ as in \cite[Fig. 1.7]{mathew2019}, then $\ul{All}$ is a subfamily of the proper subgroups under our assumption. However, by \cite[Thm.~4.25]{mathew2019}, the thick $\otimes$-ideal in $\Sp^G$ generated by $\{G/H_+: H \in \ul{All} \}$ includes the Borel completion of the unit. Therefore, the Verdier quotient in question is the trivial category, and we deduce that $\tau^G \simeq 0$.
\end{proof}

\begin{exm} \label{exm:DihedralOdd} Let $p$ be an odd prime, $G = D_{2p} = \mu_p \rtimes C_2$ the dihedral group of order $2p$, and $\Gamma$ the family of subgroups $H$ such that $H \cap \mu_p = 1$. Note that up to conjugacy, $\Gamma$ consists of the subgroups $1$ and $C_2$, and the Weyl group of $C_2$ is trivial. Thus, up to equivalence, the data of a $\Gamma$-complete spectrum $X$ amounts to 
\begin{itemize}
\item A Borel $D_{2p}$-spectrum $X^1$ and a spectrum $X^{\phi C_2}$.
\item A map $\alpha: X^{\phi C_2} \to (X^1)^{t C_2}$.
\end{itemize}
Using that $J_{\mu_p} = [1 < \mu_p]$, we compute $X^{\phi \mu_p} \simeq (X^1)^{t \mu_p}$. As for $X^{\phi D_{2p}}$, by Lem.~\ref{lem:GenTateVanishingMultiprime} we have that $(X^1)^{\tau D_{2p}} \simeq 0$. We further claim that the generalized Tate functor $\tau^{D_{2p}}_{C_2}$ vanishes:
\begin{itemize} \item[($\ast$)] Let $\cH = \{ 1, \mu_p \} = \fS^c_{\geq C_2}$. By Rmk.~\ref{ParamRecollementFamily} applied to $\ul{\Sp}^{\Phi \cH}$, the restriction and induction functors for $C_2 \subset D_{2p}$ descend to an adjunction
\[ \adjunct{\res'}{\Sp^{\Phi \cH}}{\Sp}{\ind'}, \]
where $(\ul{\Sp}^{\Phi \cH})_{D_{2p}/C_2} \simeq \Sp$ because the restriction of $\cH$ to a $C_2$-family yields the trivial family. Now consider the inclusion of the open fiber
\[ j_{\ast}: \Sp^{h W_G C_2} \to \Sp^{\Phi \cH}. \]
Because $W_{D_{2p}} C_2 \cong 1$, we have that $\res' \simeq j^{\ast}$, and we deduce that $j_! \simeq j_{\ast}$. Because $\cH^c = \fS_{\geq C_2}$ consists only of the two subgroups $C_2$ and $D_{2p}$ up to conjugacy, we may identify $\phi^{D_{2p}}: \Sp^{\Phi \cH} \to \Sp$ with the restriction $i^{\ast}$ to the closed complement of a recollement of $\Sp^{\Phi \cH}$ with $j_{\ast}$ as the inclusion of the open part. We then have $\tau^{D_{2p}}_{C_2} \simeq i^{\ast} j_{\ast}$. In view of the fiber sequence $$j_! \xto{\simeq} j_{\ast} \to i_{\ast} i^{\ast} j_{\ast} \simeq 0,$$ we deduce that $\tau^{D_{2p}}_{C_2} \simeq 0$.
\end{itemize}
Using Cor.~\ref{cor:FormulaForGeomFixedPointsOfCompleteSpectrum}, we conclude that $X^{\phi D_{2p}} \simeq 0$.
\end{exm}

We conclude this section by indicating how the comparison functor $\Theta$ is functorial in the group $G$ with respect to restriction and geometric fixed points.

\begin{cnstr}[Restriction functoriality for geometric loci] \label{restrictionGeometricLoci} Let $H$ be a subgroup of $G$ and consider the map $i: \fS[H] \to \fS[G]$ that sends a subgroup $K$ of $H$ to the same $K$ viewed as a subgroup of $G$. Since $i$ preserves the subconjugacy relation, $i$ is a functor,\footnote{However, since there may be additional conjugacy relations in $G$, $i$ is not generally the inclusion of a subcategory.} and also let $i: \sd(\fS[H]) \to \sd(\fS[G])$ denote the induced functor on barycentric subdivisions. Next, consider the functor $\res^G_H \times \id: \Sp^G \times \fS[H] \to \Sp^H \times \fS[H]$. Since for any subgroup $K \leq H$, the restriction of the $G$-families $\fS[G]_{\leq K}$, $\fS[G]_{<K}$ to $H$ yields $H$-families $\fS[H]_{\leq K}$, $\fS[H]_{<K}$, by Rmk.~\ref{ParamRecollementFamily} we have an induced functor over $\fS[H]$
\[ \res^G_H: \Sp^G_{\locus{\phi}} \times_{\fS[G]} \fS[H] \to \Sp^H_{\locus{\phi}} \]
that preserves locally cocartesian edges. Precomposition by $i$ and postcomposition by $\res^G_H$ then defines a functor
\[ \res^G_H: \Fun^{\cocart}_{/\fS[G]}(\sd(\fS[G]), \Sp^G_{\locus{\phi}}) \to \Fun^{\cocart}_{/\fS[H]}(\sd(\fS[H]), \Sp^H_{\locus{\phi}}). \]
We have a lax commutative diagram
\[ \begin{tikzcd}[row sep=4ex, column sep=6ex, text height=1.5ex, text depth=0.5ex]
\Fun^{\cocart}_{/\fS[G]}(\sd(\fS[G]), \Sp^G_{\locus{\phi}}) \ar{d}[swap]{\res^G_H} \ar{r}{\Theta_G}[swap]{\simeq} \ar[phantom]{rd}{\NEarrow} & \Sp^G \ar{d}{\res^G_H} \\
\Fun^{\cocart}_{/\fS[H]}(\sd(\fS[H]), \Sp^H_{\locus{\phi}}) \ar{r}{\Theta_H}[swap]{\simeq} & \Sp^H
\end{tikzcd} \]
where the natural transformation $\eta: \res^G_H \circ \Theta_G \to \Theta_H \circ \res^G_H$ is defined using the contravariant functoriality of the limit for $i: \sd(\fS[H]) \to \sd(\fS[G])$.

We claim that $\eta$ is an equivalence, so that this diagram commutes. Indeed, suppose given $f: \sd(\fS[G]) \to \Sp^G_{\locus{\phi}}$ and let $g = \res^G_H f: \sd(\fS[H]) \to \Sp^H_{\locus{\phi}}$. Let $f': \sd(\fS[G]) \to \Sp^G$ and $g': \sd(\fS[H]) \to \Sp^H$ be the functors obtained by postcomposition, so $g' = \res^G_H f' i$ by definition and $\eta_f$ is the comparison map
\[ \lim_{\sd(\fS[G])} \res^G_H f' \to \lim_{\sd(\fS[H])} \res^G_H f' i. \]
It suffices to check that for all subgroups $K \leq H$,  $\phi^K(\eta_f)$ is an equivalence. But then by the commutativity of the diagram
\[ \begin{tikzcd}[row sep=4ex, column sep=4ex, text height=1.5ex, text depth=0.25ex]
\Sp^G \ar{d}[swap]{\res^G_H} \ar{r}{\phi^K} & \Fun(B W_G K, \Sp) \ar{d}{\res^{W_G K}_{W_H K}} \\
\Sp^H \ar{r}{\phi^K} & \Fun(B W_H K, \Sp),
\end{tikzcd} \]
and under the equivalences $\phi^K \Theta_G \simeq \ev_K$ and $\phi^K \Theta_H \simeq \ev_K$ of Lem.~\ref{lm:GeometricFixedPointsOfComparisonFunctor}, we see that $\phi^K(\eta_f)$ is an equivalence.
\end{cnstr}

\begin{cnstr}[Geometric fixed points functoriality for geometric loci] \label{GeometricFixedPointsGeometricLoci} Let $N$ be a normal subgroup of $G$. Then we may embed $\fS[G/N]$ as a cosieve in $\fS[G]$ via the functor $i: \fS[G/N] \to \fS[G]$ that sends $M/N$ to $M$. We also let $i: \sd(\fS[G/N]) \to \sd(\fS[G])$ denote the induced functor on barycentric subdivisions, which is a cosieve inclusion. By Lem.~\ref{lem:ClosedPartRecollementNormalSubgroup}, $\Phi^N: \Sp^G \to \Sp^{G/N}$ has fully faithful right adjoint with essential image given by the $(\fS[G] \setminus \fS[G/N])^{-1}$-local objects. Therefore, $\Phi^N$ implements an equivalence over $\fS[G/N]$
\[ \Sp^G_{\locus{\phi}} \times_{\fS[G]} \fS[G/N] \xto{\simeq} \Sp^{G/N}_{\locus{\phi}}. \]
Define $\Phi^N: \Fun^{\cocart}_{/\fS[G]}(\sd(\fS[G]), \Sp^G_{\locus{\phi}}) \to \Fun^{\cocart}_{/\fS[G/N]}(\sd(\fS[G/N]), \Sp^{G/N}_{\locus{\phi}})$ to be the functor obtained by $i^{\ast}$ under that equivalence. Then because $\Theta$ is a morphism of recollements, we have a commutative diagram
\[ \begin{tikzcd}[row sep=4ex, column sep=6ex, text height=1.5ex, text depth=0.5ex]
\Fun^{\cocart}_{/\fS[G]}(\sd(\fS[G]), \Sp^G_{\locus{\phi}}) \ar{d}[swap]{\Phi^N} \ar{r}{\Theta_G}[swap]{\simeq} & \Sp^G \ar{d}{\Phi^N} \\
\Fun^{\cocart}_{/\fS[G/N]}(\sd(\fS[G/N]), \Sp^{G/N}_{\locus{\phi}}) \ar{r}{\Theta_{G/N}}[swap]{\simeq} & \Sp^{G/N}.
\end{tikzcd} \]
\end{cnstr}

\section{Theories of \texorpdfstring{$G$}{G}-spectra relative to a normal subgroup \texorpdfstring{$N$}{N}}

In classical approaches to equivariant stable homotopy theory \cite{MR866482} \cite{AlaskaNotes}, one attaches to every $G$-universe $\cU$ a corresponding theory of $G$-spectra indexed with respect to $\cU$; upon inverting the weak equivalences, this yields a stable $\infty$-category $\Sp^G_{\cU}$. For the complete $G$-universe $\cU$, one obtains \emph{genuine} $G$-spectra $\Sp^G_{\cU} \simeq \Sp^G$, whereas for the trivial $G$-universe $\sU^G$, one obtains \emph{naive} $G$-spectra $\Sp^G_{\cU^G} \simeq \Fun(\sO^{\op}_G, \Sp)$.\footnote{We identify the $\infty$-category as the ordinary stabilization of $G$-spaces $\Spc^G = \Fun(\sO_G^{\op}, \Spc)$.} Interpolating between genuine and naive $G$-spectra, for every normal subgroup $N \trianglelefteq G$, one has the fixed points $G$-universe $\cU^N$ \cite[Ch.~XVI, \S 5]{AlaskaNotes} and the associated $\infty$-category $\Sp^G_{\cU^N}$. In this section, we will revisit these notions from a different and intrinsically $\infty$-categorical perspective that makes no reference to representation theory. Using the language of parametrized $\infty$-category theory, we define $\infty$-categories $\Sp^G_{\naive{N}}$ and $\Sp^G_{\Borel{N}}$ of \emph{$N$-naive} and \emph{$N$-Borel} $G$-spectra (Def.~\ref{dfn:NaiveGSpectraRelativeToNormalSubgroup} and Def.~\ref{dfn:BorelGSpectraRelativeToNormalSubgroup}). We then show $\Sp^G_{\Borel{N}}$ canonically embeds into $\Sp^G$ as the $\Gamma_N$-complete $G$-spectra for $\Gamma_N$ the $N$-free $G$-family (Prop.~\ref{prp:BorelSpectraAsCompleteObjects}). 

\begin{rem} Although we expect the $\infty$-category $\Sp^G_{\naive{N}}$ to be equivalent to $\Sp^G_{\cU^N}$, we will not give a precise comparison in this paper.
\end{rem}

To begin with, we will need a technical lemma.

\begin{lem} \label{lem:QuotientMapCartesianFibration} \begin{enumerate}[leftmargin=*] \item Let $\adjunct{L}{C}{D}{R}$ be an adjunction such that for all $c \in C$, $d \in D$, and $f: d \to L c$ the natural map
\[ L( Rd \times_{R L c} c ) \to d \]
adjoint to the projection $Rd \times_{R L c} c \to R d$ is an equivalence. Then $L$ is a weak cartesian fibration,\footnote{A weak cartesian fibration is the version of cartesian fibration that is stable under equivalence, defined to be the obvious generalization of a Street fibration to the $\infty$-categorical context.} and hence a cartesian fibration if $L$ is assumed to be a categorical fibration.
\end{enumerate}
\begin{enumerate} \setcounter{enumi}{1}
\item Let 
\[ \begin{tikzcd}[row sep=4ex, column sep=4ex, text height=1.5ex, text depth=0.25ex]
L'' = L' \circ L: C \ar[shift left=2]{r}{L} & C' \ar[shift left=1]{l}{R} \ar[shift left=2]{r}{L'} & D : R \circ R' = R'' \ar[shift left=1]{l}{R'}
\end{tikzcd} \]
\end{enumerate}
be a diagram of adjunctions such that $L \dashv R$, $L' \dashv R'$, and $L'' \dashv R''$ all satisfy the assumption in (1). Then $L$ sends $L''$-cartesian edges to $L'$-cartesian edges.
\end{lem}
\begin{proof} For (1), under our assumption, we need only show that $R d \times_{R L c} c \to c$ is a $L$-cartesian edge. But for this, for any $c' \in C$ we have the pullback square of spaces
\[ \begin{tikzcd}[row sep=4ex, column sep=4ex, text height=1.5ex, text depth=0.25ex]
\Map_{C}(c',Rd \times_{RL c} c) \ar{r} \ar{d} & \Map_C(c',Rd) \ar{d} \ar{r}{\simeq} & \Map_D(Lc',d) \ar{d} \\
\Map_C(c',c) \ar{r} & \Map_C(c',RLc) \ar{r}{\simeq} & \Map_D(Lc',Lc), \\
\end{tikzcd} \]
and the assertion follows from the definition of $L$-cartesian edge and a simple diagram chase.

For (2), let $c \in C$, $(f:d \to L'' c) \in D$, and consider the $L''$-cartesian edge $R'' d \times_{R'' L'' c} c \to c$ (this case suffices since all $L''$-cartesian edges are of this form up to equivalence). Note that the unit map for $L'' \dashv R''$ factors as the composition $R'' L'' c \simeq R R' L' L c \to R L c \to c$ of unit maps for $L' \dashv R'$ and $L \dashv R$. Thus, we have
\[ L(R'' d \times_{R'' L'' c} c) \simeq L(R \left(R' d \times_{R' L'' c} Lc \right) \times_{R L c} c) \xto{\simeq} R' d \times_{R' L'' c} L c  \]
by our assumption on $L \dashv R$, and this equivalence respects the projection map to $L(c)$. But our assumption on $L' \dashv R'$ ensures that $R'(d) \times_{R' L''(c)} L(c) \to L(c)$ is a $L'$-cartesian edge.
\end{proof}

\subsection{\texorpdfstring{$N$}{N}-naive \texorpdfstring{$G$}{G}-spectra}

\begin{nul} \label{inflationFunctors} Let $N$ be a normal subgroup of $G$ and let $\pi: G \to G/N$ denote the quotient map. We have the adjunction
\[ \adjunct{r_N}{\FF_G}{\FF_{G/N}}{\iota_N} \]
where $r_N(U) = U/N$ and $\iota_N(V) = V$ with $V$ regarded as a $G$-set via $\pi$. Clearly, $\Hom_{G/N}(U,V) \cong \Hom_G(U,V)$, so $\iota_N$ is fully faithful. For $U \in \FF_G$, $V \in \FF_{G/N}$, and a $G/N$-map $f: V \to U/N$, we also have $(V \times_{U/N} U)/N \cong V$, so by Lem.~\ref{lem:QuotientMapCartesianFibration} $r_N$ is a cartesian fibration. Note also that the adjunction $r_N \dashv \iota_N$ restricts to $\adjunct{r_N}{\sO_G}{\sO_{G/N}}{\iota_N}$ and $V \times_{U/N} U$ is transitive if $V$ and $U$ are, hence $r_N$ remains a cartesian fibration when restricted to $\sO_G$. Given a $G$-orbit $G/H$ and a $G/N$-map $f: \frac{G/N}{K/N} \to \frac{G/N}{H N/N}$, we may identify the pullback $G/H \times_{G/H N} G/K \cong G/(H \cap K)$, and a $r_N$-cartesian edge lifting $f$ is given by $G/(H \cap K) \to G/H$. 
\end{nul}

\begin{cvn} For $N$ a normal subgroup of $G$, we will regard $\sO_G^{\op}$ as a $G/N$-category via $r_N^{\op}$.
\end{cvn}

\begin{dfn} \label{dfn:NaiveGSpectraRelativeToNormalSubgroup} Let $\Sp^G_{\naive{N}} = \Fun_{G/N}(\sO_G^{\op}, \underline{\Sp}^{G/N})$ be the $\infty$-category of \emph{naive $G$-spectra relative to $N$}, or \emph{$N$-naive $G$-spectra}.
\end{dfn}

For example, if $N = G$ we have the usual $\infty$-category $\Fun(\sO_G^{\op}, \Sp)$ of naive $G$-spectra, and if $N = 1$ we instead have the $\infty$-category $\Sp^G$ itself. 

\begin{cnstr} \label{cnstr:ForgetfulFunctorToNaiveGSpectra} We define a `forgetful' functor $\sU[N]: \Sp^G \to \Fun_{G/N}(\sO_G^{\op}, \underline{\Sp}^{G/N})$.

First, let $q_N: \omega_G \circ \iota_N \to \omega_{G/N}$ be the natural transformation defined on objects $U \in \FF_{G/N}$ by the functor $U//G \to U//(G/N)$ which sends objects $x \in U$ to the same $x \in U$ and morphisms $g: x \to g \cdot x = \pi(g) \cdot x$ to $\pi(g): x \to \pi(g) \cdot x$.

For any $\infty$-category $C$, the adjunction $r_N \dashv \iota_N$ induces an adjunction
\[ \adjunct{(r_N^{\op})^{\ast}}{\Fun(\FF_{G/N}^{\op},C)}{\Fun(\FF_G^{\op}, C)}{(\iota_N^{\op})^{\ast}} \]
where we may identify $(r_N^{\op})^{\ast}$ with the left Kan extension along $\iota_N^{\op}$. Let $\underline{\inf}[N]: \SH \omega_{G/N}^{\op} r_N^{\op} \to \SH \omega_G^{\op}$ be the natural transformation adjoint to $\SH q_N^{\op}$ and let 
\[ \underline{\inf}[N]: \sO_G^{\op} \times_{\sO_{G/N}^{\op}} \underline{\Sp}^{G/N} \to \underline{\Sp}^G \]
also denote the associated $G$-functor. Note that for a $G$-orbit $G/H$, $\underline{\inf}[N]_{G/H}$ is given by the inflation functor $\inf^{H \cap N}: \Sp^{H/H \cap N} \to \Sp^H$. By the dual of \cite[Prop.~7.3.2.6]{HA}, $\underline{\inf}[N]$ admits a relative right adjoint
\[ \widehat{\Psi}[N]: \underline{\Sp}^G \to \sO_G^{\op} \times_{\sO_{G/N}^{\op}} \underline{\Sp}^{G/N} \]
that does not preserve cocartesian edges; rather, for a map of $G$-orbits $f: G/K \to G/H$ we have a lax commutative square
\[ \begin{tikzcd}[row sep=6ex, column sep=6ex, text height=1.5ex, text depth=0.25ex]
\Sp^H \ar{r}{\gamma_{\ast}} \ar{d}{f^{\ast}} \ar[phantom]{rd}{\SWarrow}  & \Sp^{H/H \cap N} \ar{d}{f^{\ast}} \\
\Sp^K \ar{r}{\gamma_{\ast}} & \Sp^{K/ K \cap N}
\end{tikzcd}
\text{ for the square of $G$-orbits }
\begin{tikzcd}[row sep=6ex, column sep=4ex, text height=1.5ex, text depth=0.25ex]
G/K \ar{d}{f} \ar{r}{\gamma} & (G/N)/(K N/ N) \ar{d}{f} \\
G/H \ar{r}{\gamma} & (G/N)/(H N/N)
\end{tikzcd}. \]
However, for a map of $G/N$-orbits $f: \frac{G/N}{K/N} \to \frac{G/N}{H N/N}$ we have a homotopy commutative square
\[ \begin{tikzcd}[row sep=6ex, column sep=6ex, text height=1.5ex, text depth=0.25ex]
\Sp^H \ar{r}{\gamma_{\ast}} \ar{d}{f^{\ast}} & \Sp^{H/(H \cap N)} \ar{d}{f^{\ast}} \\
\Sp^{K \cap H} \ar{r}{\gamma_{\ast}} & \Sp^{(K \cap H)/(H \cap N)}
\end{tikzcd}
\text{ for the pullback square }
\begin{tikzcd}[row sep=6ex, column sep=4ex, text height=1.5ex, text depth=0.25ex]
G/(K \cap H) \ar{d}{\gamma} \ar{r}{f} & G/H \ar{d}{\gamma} \\
G/K \ar{r}{f} & G/H N
\end{tikzcd} \]
and hence the further composition $\underline{\Psi}[N] = \pr \circ \widehat{\Psi}[N]: \underline{\Sp}^G \to \underline{\Sp}^{G/N}$ \emph{does} preserve cocartesian edges over $\sO_{G/N}^{\op}$, where we regard $\underline{\Sp}^G$ as a $G/N$-$\infty$-category via $r_N^{\op}$. We also have the unit $\eta: \id \to \iota_N r_N$ which by precomposition yields the $G$-functor
\[ \underline{\res}[N]: \sO_G^{\op} \times_{(\iota_N r_N)^{\op}, \sO_G^{\op}} \underline{\Sp}^G \to \underline{\Sp}^G \]
where for a $G$-orbit $G/H$, $\underline{\res}[N]_{G/H}$ is given by the restriction functor $\res^{H N}_H: \Sp^{H N} \to \Sp^H$. The composite $\underline{\Psi}[N] \circ \underline{\res}[N]$ is a $G/N$-functor. We obtain a $G/N$-functor
\[ \widetilde{\sU}[N]: \sO_{G/N}^{\op} \times_{\sO_G^{\op}} \Sp^G \to \underline{\Fun}_{G/N}(\sO_G^{\op}, \underline{\Sp}^{G/N}) \]
via adjunction.\footnote{The ad-hoc notation $\widetilde{\sU}[N]$ for this $G/N$-functor is employed so as not to conflict with the $G$-functor $\underline{\sU}[N]$ in \ref{ExtendingNaiveSpectraToGCategory} below.} Define $\sU[N]$ to be the fiber of $\widetilde{\sU}[N]$ over $(G/N)/(G/N)$.
\end{cnstr}

\begin{nul}[\textbf{Monoidality of forgetful functor}] \label{LaxMonoidalForgetfulFunctorToNaiveSpectra} In Constr.~\ref{cnstr:ForgetfulFunctorToNaiveGSpectra}, the monoidality of inflation and restriction implies that with respect to $\underline{\Sp}^{G, \otimes}$ and $\underline{\Sp}^{G/N, \otimes}$, the $G$-functors $\underline{\inf}[N]$ and $\underline{\res}[N]$ are symmetric monoidal and the $G/N$-functor $\underline{\Psi}[N]$ is lax monoidal. Therefore, $\sU[N]$ is lax monoidal with respect to the pointwise monoidal structure on $\Sp^G_{\naive{N}}$.
\end{nul}

\begin{nul}[\textbf{Extension to $G$-$\infty$-category}] \label{ExtendingNaiveSpectraToGCategory} For any subgroup $H$ of $G$, consider the commutative diagram of restriction functors
\[ \begin{tikzcd}[row sep=4ex, column sep=8ex, text height=1.5ex, text depth=0.25ex]
\FF_H & \FF_G \ar{l}[swap]{\res^G_H} \\
\FF_{H/(H \cap N)} \ar{u}[swap]{\iota_{H \cap N}} & \FF_{G/N} \ar{l}{\res^{G/N}_{H/(H \cap N)}} \ar{u}{\iota_N}
\end{tikzcd}
\text{ that yields by adjunction }
\begin{tikzcd}[row sep=4ex, column sep=8ex, text height=1.5ex, text depth=0.25ex]
\FF_H \ar{r}{\ind^G_H} \ar{d}{r_{N \cap H}} & \FF_G \ar{d}{r_N} \\
\FF_{H/(H \cap N)} \ar{r}[swap]{\ind^{G/N}_{H/(H \cap N)}} & \FF_{G/N}
\end{tikzcd}. \]
Precomposition by $(\ind^G_H)^{\op}: \sO_H^{\op} \to \sO_G^{\op}$ yields functors
\[ \res^G_H: \Fun_{G/N}(\sO^{\op}_G, \underline{\Sp}^{G/N}) \to \Fun_{H/(H\cap N)}(\sO^{\op}_H, \underline{\Sp}^{H/(H \cap N)}) \]
that assemble to the data of a functor $\sO_G^{\op} \to \Cat_{\infty}$ and thereby define a $G$-$\infty$-category $\underline{\Sp}^G_{\naive{N}}$. Furthermore, $\sU[N]$ extends to a $G$-functor $\underline{\sU}[N]: \underline{\Sp}^G \to \underline{\Sp}^G_{\naive{N}}$, given on the fiber over $G/H$ by $\sU[N \cap H]$.
\end{nul}

\begin{nul}[\textbf{Evaluation functors}] \label{evaluationFactorizationNaiveSpectra} For any subgroup $H$ of $G$, evaluation on the orbit $G/H$ yields a functor
\[ s_H^{\ast}: \Fun_{G/N}(\sO^{\op}_G, \underline{\Sp}^{G/N}) \to \Sp^{H/(H \cap N)}. \]
By construction, this fits into a commutative diagram
\[ \begin{tikzcd}[row sep=4ex, column sep=4ex, text height=1.5ex, text depth=0.25ex]
\Sp^G \ar{r}{\sU[N]} \ar{d}[swap]{\res^G_H} & \Fun_{G/N}(\sO_G^{\op}, \underline{\Sp}^{G/N}) \ar{d}{s_H^{\ast}} \\
\Sp^H \ar{r}{\Psi^{H \cap N}} & \Sp^{H/(H \cap N)}.
\end{tikzcd} \]
Because $\underline{\Sp}^{G/N}$ is a $G/N$-presentable $G/N$-stable $\infty$-category, the same holds for $\underline{\Fun}_{G/N}(\sO^{\op}_G, \underline{\Sp}^{G/N})$ with $G/N$-limits and colimits computed as in \cite[Prop.~9.17]{Exp2}. Thus, $\Fun_{G/N}(\sO_G^{\op}, \underline{\Sp}^{G/N})$ is a presentable stable $\infty$-category such that the $s_H^{\ast}$ form a set of jointly conservative functors that preserve and detect limits and colimits. Since both the restriction and categorical fixed points functors preserve limits and colimits, it follows that $\sU[N]$ preserves limits and colimits and therefore admits a left adjoint $\sF[N]$ and right adjoint $\sF^{\vee}[N]$.

We also have a partial compatibility relation as $H$ varies. Namely, given $H$, if $K$ is a subgroup such that $N \cap H \leq K \leq H$ (so $K \cap N = H \cap N$), then
\[ \res^{H/(N \cap H)}_{K/(N \cap H)} \circ s_H^{\ast} \simeq s_K^{\ast}: \Fun_{G/N}(\sO^{\op}_G, \underline{\Sp}^{G/N}) \to \Sp^{K/N \cap H}. \]
\end{nul}

\begin{nul}[\textbf{Interaction with $G$-spaces}] By repeating the construction of $\sU[N]$ for $G$-spaces and using the compatibility of restriction and categorical fixed points with $\Omega^{\infty}$, we obtain a commutative diagram
\[ \begin{tikzcd}[row sep=4ex, column sep=4ex, text height=1.5ex, text depth=0.25ex]
\Sp^G \ar{r}{\sU[N]} \ar{d}[swap]{\Omega^{\infty}} & \Fun_{G/N}(\sO^{\op}_G, \underline{\Sp}^{G/N}) \ar{d}{\Omega^{\infty}}  \\
\Spc^G \ar{r}{\sU'[N]} & \Fun_{G/N}(\sO^{\op}_G, \underline{\Spc}^{G/N}) 
\end{tikzcd} \]
where the righthand $\Omega^{\infty}$ functor denotes postcomposition by the $G/N$-functor $\Omega^{\infty}: \underline{\Sp}^{G/N} \to \underline{\Spc}^{G/N}$. Moreover, a diagram chase reveals that under the equivalence
\[ \Fun_{G/N}(\sO_G^{\op}, \underline{\Spc}^{G/N}) \simeq \Fun(\sO^{\op}_G, \Spc) = \Spc^G \]
of \cite[Prop.~3.9]{Exp2}, $\sU'[N]$ is an equivalence.
\end{nul}

To understand the compact generation of $N$-naive $G$-spectra, we need the following lemma.

\begin{lem} \label{lem:CompactGeneration} Let $C$ and $\{C_i: i \in I \}$ be presentable stable $\infty$-categories (with $I$ a small set) such that each $C_i$ has a (small) set $\{ x_{i \alpha}: \alpha \in \Lambda_i \}$ of compact generators. Suppose we have functors $U_i: C \to C_i$ that preserve limits and colimits and are jointly conservative. Let $F_i$ be left adjoint to $U_i$. Then $C$ has a (small) set of compact generators given by $\{  F_i x_{i \alpha}: i \in I, \alpha \in \Lambda_i \}$. In particular, $C$ is compactly generated. 
\end{lem}
\begin{proof} We check directly that the indicated set generates $C$. Let $c \in C$ be any object and suppose that $\Hom_C(\Sigma^n F_i x_{i \alpha}, c) \cong 0$ for all choices of indices. Then by adjunction, $\Hom_{C_i}(\Sigma^n x_{i \alpha}, U_i c) \cong 0$, hence $U_i c \simeq 0$ for all $i \in I$. Invoking the joint conservativity of the $U_i$, we deduce that $c \simeq 0$. As for compactness, note that the assumption that each $U_i$ preserves colimits ensures that its left adjoint $F_i$ preserves compact objects.
\end{proof}

\begin{cor} \label{cor:CompactGenerationNaive} The $\infty$-category $\Fun_{G/N}(\sO^{\op}_G, \underline{\Sp}^{G/N})$ has a set of compact generators given by $\{ (s_H)_!(1) : H \leq G \}$.
\end{cor}
\begin{proof} By applying Lem.~\ref{lem:CompactGeneration} to the functors $s_H^{\ast}$ described in \ref{evaluationFactorizationNaiveSpectra}, we deduce that $\{(s_H)_! \left( \frac{H/(H \cap N)}{K/(H \cap N)}_+ \right) : H \leq G \}$ is a set of compact generators for $\Fun_{G/N}(\sO^{\op}_G, \underline{\Sp}^{G/N})$. Because $\res^{H/(N \cap H)}_{K/(N \cap H)} s_H^{\ast} \simeq s_K^{\ast}$, we may eliminate redundant expressions and reduce to the set $\{ (s_H)_!(1) : H \leq G \}$.
\end{proof}

\subsection{\texorpdfstring{$N$}{N}-Borel \texorpdfstring{$G$}{G}-spectra}

We next consider Borel $G$-spectra relative to $N$. Let $\psi$ denote the extension $N \to G \to G/N$.

\begin{dfn} \label{dfn:TwistedClassifyingSpace} Let $B^{\psi}_{G/N} N \subset \sO^{\op}_G$ be the full subcategory on those $G$-orbits that are $N$-free.
\end{dfn}

Note that $B^{\psi}_{G/N} N$ is a cosieve in $\sO^{\op}_G$: this amounts to the observation that if $U$ is $N$-free and $f: V \to U$ is a $G$-equivariant map of $G$-sets, then $V$ is $N$-free.

\begin{lem} \label{lm:TwistedClassifyingSpaceIsSpace} The cocartesian fibration $r_N^{\op}: \sO^{\op}_G \to \sO^{\op}_{G/N}$ restricts to a left fibration
\[ \rho_N: B^{\psi}_{G/N} N \to \sO^{\op}_{G/N}. \]
\end{lem}
\begin{proof} Because $B^{\psi}_{G/N} N$ is a cosieve, the inclusion $B^{\psi}_{G/N} N \subset \sO^{\op}_G$ is stable under $r_N^{\op}$-cocartesian edges, so $\rho_N$ is a cocartesian fibration such that the inclusion preserves cocartesian edges. Furthermore, if $f: U \to V$ is a $G$-equivariant map of $N$-free $G$-sets such that $\overline{f}: U/N \to V/N$ is an isomorphism, then it is easy to check that $f$ is an isomorphism. Because $\rho_N$ is conservative, we deduce that $\rho_N$ is in addition a left fibration.
\end{proof}

\begin{rem} If $N$ yields a product decomposition $G \cong G/N \times N$, then $B^{\psi}_{G/N} N$ is spanned by those orbits of the form $G/\Gamma_{\phi}$ where $\Gamma_{\phi}$ is the graph of a homomorphism $\phi: M \to G/N$ for $M \leq N$. As a $G/N$-space, $B^{\psi}_{G/N} N$ then is the classifying $G/N$-space for $G/N$-equivariant principal $N$-bundles, which is usually denoted as $B_{G/N} N$. We thus think of $B^{\psi}_{G/N} N$ as a twisted variant of $B_{G/N} N$ for non-trivial extensions $\psi$. Many other authors have also studied equivariant classifying spaces in varying levels of generality -- for example, see \cite{Guillou2017} \cite{Luck05} \cite{Lck2014}.
\end{rem}

The following definition extends \cite[Rmk.~2.23]{Qui19b} to the case of a non-trivial extension $\psi$.

\begin{dfn} \label{dfn:BorelGSpectraRelativeToNormalSubgroup} Let $\Sp^G_{\Borel{N}} = \Fun_{G/N}(B^{\psi}_{G/N} N, \underline{\Sp}^{G/N})$ be the $\infty$-category of \emph{Borel $G$-spectra relative to $N$}, or \emph{$N$-Borel $G$-spectra}. We will also refer to $G/N$-functors $$X: B^{\psi}_{G/N} N \to \ul{\Sp}^{G/N}$$
as \emph{$G/N$-spectra with $\psi$-twisted $N$-action}.
\end{dfn}

\begin{ntn} Given an abelian group $A$, we will use $B^t_{C_2} A$ as preferred alternative notation in lieu of $B^{\psi}_{C_2} A$ for the defining extension $\psi = [A \to A \rtimes C_2 \to C_2]$ of the semidirect product, where $C_2$ acts on $A$ by the inversion involution. We will also refer to $C_2$-functors $X:B^t_{C_2} A \to \ul{\Sp}^{C_2}$ as \emph{$C_2$-spectra with twisted $A$-action}, leaving $\psi$ implicit.
\end{ntn}

\begin{properties} \label{VariousPropertiesForgetfulFunctorBorelSpectra} By composing $\sU[N]$ with restriction along the $G/N$-functor $i: B^{\psi}_{G/N} N \subset \sO^{\op}_G$, we obtain a forgetful functor
\[ \sU_b[N]: \Sp^G \to \Fun_{G/N}(B^{\psi}_{G/N} N, \underline{\Sp}^{G/N}). \]

Parallel to the above discussion of the properties of $\sU[N]$, let us enumerate some of the properties of $\sU_b[N]$.

\begin{enumerate} \item Because both $\sU[N]$ and restriction along $i$ preserve limits and colimits, $\sU_b[N]$ preserves limits and colimits and thus admits left and right adjoints $\sF_b[N]$ and $\sF^{\vee}_b[N]$ that factor through $\sF[N]$ and $\sF^{\vee}[N]$.

\item For all $G/H \in  B^{\psi}_{G/N} N$ we have $H \cap N = 1$. Therefore, the smaller collection of functors $\{ s_H^{\ast}: \Sp^G_{\Borel{N}} \to \Sp^H : H \cap N = 1 \}$ is jointly conservative and preserves and detects limits and colimits. Moreover, from \ref{evaluationFactorizationNaiveSpectra} we get that $s_H^{\ast} \circ \sU_b[N] \simeq \res^G_H$. 

\item Since $\sU[N]$ is lax monoidal by \ref{LaxMonoidalForgetfulFunctorToNaiveSpectra} and restriction is symmetric monoidal, we get that $\sU_b[N]$ is a lax monoidal functor. However, because each $s_H^{\ast} \circ \sU_b[N]$ for $H \cap N = 1$ is now symmetric monoidal, $\sU_b[N]$ is in fact symmetric monoidal.

\item As in Constr.~\ref{cnstr:ForgetfulFunctorToNaiveGSpectra}, the functor $\sU_b[N]$ is the fiber over $(G/N)/(G/N)$ of a $G/N$-functor $$\widetilde{\sU}_b[N]: \sO_{G/N}^{\op} \times_{\sO_G^{\op}} \ul{\Sp}^G \to \ul{\Fun}_{G/N}(\sO_G^{\op}, \ul{\Sp}^{G/N}). $$ 

Also, as in \ref{ExtendingNaiveSpectraToGCategory}, $\Sp^G_{\Borel{N}}$ extends to a $G$-$\infty$-category $\underline{\Sp}^G_{\Borel{N}}$ and $\sU_b[N]$ extends to a $G$-functor $\underline{\sU}_b[N]: \underline{\Sp}^G \to \underline{\Sp}^G_{\Borel{N}}$, given on the fiber over $G/H$ by
\[  \sU_b[N \cap H]: \Sp^H \to \Fun_{H/(N \cap H)}(B^{\psi_H}_{H/(H \cap N)} (N \cap H), \underline{\Sp}^{H/(N \cap H)}) \]
for the restricted extension $\psi_H \coloneq [N \cap H \to H \to H/(N \cap H)]$.

\item As with $N$-naive $G$-spectra, we have a commutative diagram
\[ \begin{tikzcd}[row sep=4ex, column sep=4ex, text height=1.5ex, text depth=0.25ex]
\Sp^G \ar{r}{\sU_b[N]} \ar{d}[swap]{\Omega^{\infty}} & \Fun_{G/N}(B^{\psi}_{G/N} N, \underline{\Sp}^{G/N}) \ar{d}{\Omega^{\infty}}  \\
\Spc^G \ar{r}{\sU'_b[N]} \ar[equal]{d} & \Fun_{G/N}(B^{\psi}_{G/N} N, \underline{\Spc}^{G/N}) \ar{d}{\simeq} \\
\Fun(\sO^{\op}_G,\Sp) \ar{r}{i^{\ast}} & \Fun(B^{\psi}_{G/N} N, \Spc).
\end{tikzcd} \]
where now we may identify $\sU'_b[N]$ with restriction along $i$. Consider the transposed lax commutative diagram
\[ \begin{tikzcd}[row sep=4ex, column sep=4ex, text height=1.5ex, text depth=0.25ex]
\Sp^G \ar{r}{\sU_b[N]} \ar[phantom]{rd}{\NWarrow}  & \Fun_{G/N}(B^{\psi}_{G/N} N, \underline{\Sp}^{G/N})  \\
\Spc^G \ar{r}{\sU'_b[N]} \ar{u}{\Sigma^{\infty}_+} & \Fun_{G/N}(B^{\psi}_{G/N} N, \underline{\Spc}^{G/N}) \ar{u}{\Sigma^{\infty}_+}.
\end{tikzcd} \]
For any subgroup $H$ such that $H \cap N = 1$, we may extend this diagram to
\[ \begin{tikzcd}[row sep=4ex, column sep=4ex, text height=1.5ex, text depth=0.25ex]
\Sp^G \ar{r}{\sU_b[N]} \ar[phantom]{rd}{\NWarrow} & \Fun_{G/N}(B^{\psi}_{G/N} N, \underline{\Sp}^{G/N}) \ar{r}{\ev_H} & \Sp^H \\
\Spc^G \ar{r}{\sU'_b[N]} \ar{u}{\Sigma^{\infty}_+} & \Fun_{G/N}(B^{\psi}_{G/N} N, \underline{\Spc}^{G/N}) \ar{u}{\Sigma^{\infty}_+} \ar{r}{\ev_H} & \Spc^H \ar{u}{\Sigma^{\infty}_+}
\end{tikzcd} \]
where the horizontal composites are given by the restriction functor $\res^G_H$ since $H/(H \cap N) \cong H$ (c.f. \ref{evaluationFactorizationNaiveSpectra} and the analogous setup for $G$-spaces). The righthand square commutes by definition, and the outer square commutes by the compatibility of restriction with $\Sigma^{\infty}_+$. Since the $\ev_H$ are jointly conservative, it follows that the the lefthand square commutes.
\end{enumerate}
\end{properties}

\begin{ntn} \label{ntn:NFreeFamily} Let $\Gamma_N$ be the \emph{$N$-free} $G$-family consisting of subgroups $H$ such that $H \cap N = 1$.
\end{ntn}

\begin{prp} \label{prp:BorelSpectraAsCompleteObjects} The functors $\sF_b[N]$ and $\sF_b^{\vee}[N]$ are fully faithful with essential image the $\Gamma_N$-torsion and $\Gamma_N$-complete $G$-spectra, respectively.
\end{prp}
\begin{proof} We first check that the unit $\eta: \id \to \sU_b[N] \sF_b[N]$ is an equivalence to show that the left adjoint $\sF_b[N]$ is fully faithful. Because $\Omega^{\infty} \sU_b[N] \simeq i^{\ast} \Omega^{\infty}$ and $\Sigma^{\infty}_+ i^{\ast} \simeq \sU_b[N] \Sigma^{\infty}_+$, we have an equivalence of left adjoints $\Sigma^{\infty}_+ i_! \simeq \sF_b[N] \Sigma^{\infty}_+$  and for $X \in \Fun(B^{\psi}_{G/N} N, \Spc)$ we may identify $\eta_{\Sigma^{\infty}_+ X}$ with $\Sigma^{\infty}_+ \eta'_X$, where $\eta': \id \to i^{\ast} i_!$ is the unit of the adjunction $i_! \dashv i_{\ast}$. But $\eta'$ is an equivalence since $i_!$ is left Kan extension along the inclusion of a full subcategory. Thus, $\eta$ is an equivalence on all suspension spectra. In view of the commutative diagram for $H \in \Gamma_N$
\[ \begin{tikzcd}[row sep=4ex, column sep=4ex, text height=1.5ex, text depth=0.25ex]
\Fun_{G/N}(B^{\psi}_{G/N} N, \underline{\Sp}^{G/N}) \ar{r}{s_H^{\ast}} \ar{d}{\Omega^{\infty}} & \Sp^H \ar{d}{\Omega^{\infty}} \\
\Fun(B^{\psi}_{G/N} N, \Spc) \ar{r}{{s'_H}^{\ast}} & \Spc^H
\end{tikzcd} \]
where $s_H' = (\ind^G_H)^{\op}: \sO_H^{\op} \to B^{\psi}_{G/N} N \subset \sO_G^{\op}$, we have an equivalence of left adjoints ${s_H}_! \Sigma^{\infty}_+ \simeq \Sigma^{\infty}_+ {s'_H}_!$, so in particular ${s_H}_! (1)$ is a suspension spectrum. Elaborating upon Cor.~\ref{cor:CompactGenerationNaive}, we observe that the set $\{ {s_H}_! (1) : H \in \Gamma_N \}$ constitutes a set of compact generators for $\Sp^G_{\Borel{N}}$. Because both the domain and codomain of $\eta$ commute with colimits, we conclude that $\eta$ is an equivalence. Moreover, because $(s_H)^{\ast} \circ \sU_b[N] \simeq \res^G_H$ as noted in \ref{evaluationFactorizationNaiveSpectra}, by adjunction $\sF_b[N] \circ (s_H)_! \simeq \ind^G_H$ and thus the essential image of $\sF_b[N]$ is the localizing subcategory generated by the set $\{ G/H_+ : H \in \Gamma_N \}$. This equals the full subcategory of $\Gamma_N$-torsion $G$-spectra. Because we already have the stable recollement
\[ \begin{tikzcd}[row sep=4ex, column sep=4ex, text height=1.5ex, text depth=0.25ex]
\Sp^{h \Gamma_N} \ar[shift right=1,right hook->]{r}[swap]{j_{\ast}} & \Sp^G \ar[shift right=2]{l}[swap]{j^{\ast}} \ar[shift left=2]{r}{i^{\ast}} & \Sp^{\Phi \Gamma_N} \ar[shift left=1,left hook->]{l}{i_{\ast}},
\end{tikzcd} \]
it follows that $\sF^{\vee}_b[N]$ is fully faithful with essential image the $\Gamma_N$-complete $G$-spectra. In more detail:
\begin{itemize} \item The composite $i^{\ast} \sF_b[N] \simeq 0$ because $\sF_b[N](X) \in j_!(\Sp^{h \Gamma_N})$ and $i^{\ast} j_! \simeq 0$. Passing to adjoints, we get $\sU_b[N] i_{\ast} \simeq 0$. Then for all $X \in \Sp^G_{\Borel{N}}$, $\sF^{\vee}_b[N](X)$ is $\Gamma_N$-complete by the equivalence
\[ \Map(i_{\ast} Z, \sF^{\vee}_b[N](X)) \simeq \Map(\sU_b[N] i_{\ast} Z, X) \simeq 0. \]
\item Using that $\sF_b[N]$ is an equivalence onto its essential image, we see that the composite $\sU_b[N] j_!$ is an equivalence from $\Sp^{h \Gamma_N}$ to $\Sp^G_{\Borel{N}}$. Its right adjoint $j^{\ast} \sF^{\vee}_b[N]$ is thus an equivalence.
\item Combining these two assertions, we have that the composite
\[ \begin{tikzcd}[row sep=4ex, column sep=4ex, text height=1.5ex, text depth=0.25ex]
\Sp^G_{\Borel{N}} \ar{r}{\sF^{\vee}_b[N]} \ar[bend right=15]{rr}[swap]{\simeq} & \Sp^G \ar{r}{j^{\ast}} & \Sp^{h \Gamma_N} \ar{r}{j_{\ast}} & \Sp^G
\end{tikzcd} \]
is equivalent to $\sF^{\vee}_b[N]$ via the unit $\sF^{\vee}_b[N] \xto{\simeq} j_{\ast} j^{\ast} \sF^{\vee}_b[N]$ and is fully faithful onto $\Gamma_N$-complete $G$-spectra.
\end{itemize}
\end{proof}

\begin{rem}[Monoidality] \label{rem:MonoidalIdentificationOfBorelSpectra} Because $\sU_b[N]$ is symmetric monoidal by \ref{VariousPropertiesForgetfulFunctorBorelSpectra}(3), its right adjoint $\sF^{\vee}_b[N]$ is lax monoidal. Therefore, the equivalence $\sF^{\vee}_b[N]: \Fun_{G/N}(B^{\psi}_{G/N} N, \underline{\Sp}^{G/N}) \xto{\simeq} \Sp^{h \Gamma_N}$ of Prop.~\ref{prp:BorelSpectraAsCompleteObjects} is one of symmetric monoidal $\infty$-categories with respect to the pointwise monoidal structure on the lefthand side and the monoidal structure induced by the $\Gamma_N$-monoidal recollement on the righthand side.
\end{rem}

\begin{cor}[Compatibility with restriction] \label{cor:BorelCompatibilityWithRestriction} The left and right adjoints $\sF_b[N]$ and $\sF_b^{\vee}[N]$ extend to $G$-left and right adjoints
\[  \underline{\sF}_b[N], \underline{\sF}_b^{\vee}[N]: \underline{\Sp}^G_{\Borel{N}} \to \underline{\Sp}^G.  \]
\end{cor}
\begin{proof} Combine Rmk.~\ref{ParamRecollementFamily} and Prop.~\ref{prp:BorelSpectraAsCompleteObjects}.
\end{proof}

We conclude this section by applying Prop.~\ref{prp:BorelSpectraAsCompleteObjects} to decompose the $\infty$-category of $D_{2p^n}$-spectra.

\begin{exm} \label{exm:DihedralRecollement} Let $\Gamma = \Gamma_{\mu_{p^n}}$ be the $D_{2p^n}$-family that consists of those subgroups $H$ such that $H \cap \mu_{p^n} = 1$. Note that $H \notin \Gamma$ if and only if $\mu_p \leq H$. Therefore, $\Sp^{\Phi \Gamma} \simeq \Sp^{D_{2 p^{n-1}}}$ for $D_{2 p^{n-1}}$ viewed as the quotient $D_{2p^n}/\mu_p$. Together with Prop.~\ref{prp:BorelSpectraAsCompleteObjects}, we obtain a stable monoidal recollement
\[ \begin{tikzcd}[row sep=4ex, column sep=10ex, text height=1.5ex, text depth=0.5ex]
\Fun_{C_2}(B^t_{C_2} \mu_{p^n}, \underline{\Sp}^{C_2}) \ar[shift right=1,right hook->]{r}[swap]{j_{\ast} = \sF^{\vee}_b[\mu_{p^n}]} & \Sp^{D_{2p^n}} \ar[shift right=2]{l}[swap]{j^{\ast} = \sU_b[\mu_{p^n}]} \ar[shift left=2]{r}{i^{\ast} = \Phi^{\mu_p}} & \Sp^{D_{2p^{n-1}}} \ar[shift left=1,left hook->]{l}{i_{\ast}}.
\end{tikzcd} \]
Furthermore, using Rmk.~\ref{ParamRecollementFamily}, this extends to a $C_2$-stable $C_2$-recollement
\[ \begin{tikzcd}[row sep=4ex, column sep=8ex, text height=1.5ex, text depth=0.5ex]
\ul{\Fun}_{C_2}(B^t_{C_2} \mu_{p^n}, \underline{\Sp}^{C_2}) \ar[shift right=1,right hook->]{r}[swap]{\widetilde{\sF}^{\vee}_b[\mu_{p^n}]} &  \sO^{\op}_{C_2} \times_{\sO^{\op}_{D_{2p^n}}} \ul{\Sp}^{D_{2p^n}} \ar[shift right=2]{l}[swap]{ \widetilde{\sU}_b[\mu_{p^n}]} \ar[shift left=2]{r}{\Phi^{\mu_p}} & \sO^{\op}_{C_2} \times_{\sO^{\op}_{D_{2p^{n-1}}}} \ul{\Sp}^{D_{2p^{n-1}}} \ar[shift left=1,left hook->]{l}{i_{\ast}}
\end{tikzcd} \]
whose fiber over $C_2/1$ is the stable monoidal recollement
\[ \begin{tikzcd}[row sep=4ex, column sep=10ex, text height=1.5ex, text depth=0.5ex]
\Fun(B \mu_{p^n}, \Sp) \ar[shift right=1,right hook->]{r}[swap]{\sF^{\vee}_b[\mu_{p^n}]} & \Sp^{\mu_{p^n}} \ar[shift right=2]{l}[swap]{\sU_b[\mu_{p^n}]} \ar[shift left=2]{r}{\Phi^{\mu_p}} & \Sp^{\mu_{p^{n-1}}} \ar[shift left=1,left hook->]{l}{i_{\ast}},
\end{tikzcd} \]
with the $C_2$-action induced by the inversion action of $C_2$ on $\mu_{p^n}$.
\end{exm}

\section{Parametrized norm maps and ambidexterity}
\label{section:NormMaps}

In this section, we construct \emph{parametrized norm maps} that will permit us to define the parametrized Tate construction (Def.~\ref{dfn:ParamTateCnstr}). Our strategy is to mimic Lurie's construction of the norm maps \cite[\S 6.1.6]{HA} in a parametrized setting over a base $\infty$-category $S$, eventually specializing to $S = \sO_G^{\op}$. We first collect a few more necessary aspects of the theory of $S$-colimits, limits, and Kan extensions from \cite{Exp2}, building upon our discussion in \S\ref{section:ParamTerminology}.

\begin{nul} \label{GeneralTheoryParamColimits} Let $K$, $L$, and $C$ be $S$-$\infty$-categories (with $p: L \to S$ the structure map), and let $\phi: K \to L$ be an $S$-cocartesian fibration \cite[Def.~7.1]{Exp2}. We are interested in computing the left adjoint $\phi_!$ to the restriction functor $\phi^{\ast}: \Fun_S(L,C) \to \Fun_S(K,C)$ as a (pointwise) $S$-left Kan extension \cite[Def.~10.1 or Def.~9.13]{Exp2}. In \cite[\S 9]{Exp2}, the second author gave a pointwise existence criterion and formula for $\phi_! F$ of a $S$-functor $F: K \to C$ \cite[Thm.~9.15 and Prop.~10.7]{Exp2}. Namely, for all objects $x \in L$, let $\underline{x} = \sO^{\cocart}_{x \rightarrow}(L)$ be as in \cite[Notn.~2.28]{Exp2} and let $i_x: \underline{x} \to L$ be the $S$-functor given by evaluation at the target that extends $x \in L$. Note that the composite $p \circ i_x$ is a corepresentable left fibration equivalent to $\underline{s} = S^{s/}$ for $s = p(x)$, so we may think of $\underline{x}$ as a $S$-point of $L$. Let
\[ F_x: K_{\underline{x}} \coloneq \underline{x} \times_L K \to C_{\underline{s}} \coloneq S^{s/} \times_S C \]
denote the resulting $S^{s/}$-functor. Then $\phi_! F$ exists if the $S^{s/}$-colimit of $F_x$ exists for all $x \in L$, and then $\phi_! F i_x$ is computed as that $S^{s/}$-colimit. We will also say that $C$ \emph{admits the relevant $S$-colimits} with respect to $\phi: K \to L$ if for all $x \in L$ with $p(x)=s$, $C_{\underline{s}}$ admits all $S^{s/}$-colimits indexed by $K_{\underline{x}}$. In this case, the left adjoint $\phi_!$ to $\phi^{\ast}$ exists by \cite[Cor.~9.16]{Exp2} (note the logic of the proof allows us to replace the $S$-cocompleteness assumption there).
 
Now suppose instead that $\phi: K \to L$ is an $S$-cartesian fibration \cite[Def.~7.1]{Exp2}. In view of the discussion of vertical opposites in \cite[\S 5]{Exp2} and the observation that the formation of vertical opposites exchanges $S$-cocartesian and $S$-cartesian fibrations, we may dualize the above discussion to see that the $S$-right Kan extension $\phi_{\ast} F$ exists if the $S^{s/}$-limit of $F_x$ exists for all $x \in L$, and then $\phi_{\ast} F i_x$ is computed as that $S^{s/}$-limit. Likewise, we have the dual notion of $C$ admitting the relevant $S$-limits with respect to $\phi$, in which case the right adjoint $\phi_{\ast}$ exists.

Finally, suppose $K$ and $L$ are $S$-spaces. Using the cocartesian model structure on $s\Set^+_{/S}$ and the description of the fibrations between fibrant objects \cite[Prop.~B.2.7]{HA}, up to equivalence we may replace any $S$-functor $\phi: K \to L$ by a categorical fibration. But a categorical fibration between left fibrations over $S$ is necessarily both a $S$-cocartesian and $S$-cartesian fibration, hence both of the above formulas apply to compute $\phi_!$ and $\phi_{\ast}$.
\end{nul}

\begin{nul} We can also consider the $S$-functor $S$-$\infty$-category $\underline{\Fun}_S(K,C) \to S$ whose cocartesian sections are $\Fun_S(K,C)$. Let $\phi: K \to L$ be a $S$-cocartesian fibration and suppose that for every $s \in S$, the $S^{s/}$-$\infty$-category $C_{\underline{s}}$ admits the relevant colimits with respect to $\phi_{\underline{s}}: K_{\underline{s}} \to L_{\underline{s}}$, so that $\phi_{\underline{s}}^{\ast}: \Fun_{S^{s/}}(L_{\underline{s}}, C_{\underline{s}}) \to \Fun_{S^{s/}}(K_{\underline{s}}, C_{\underline{s}})$ admits a left adjoint $(\phi_{\underline{s}})_!$ computed as above. Then using the built-in compatibility of $S$-left Kan extension with restriction, by \cite[Prop.~7.3.2.11]{HA} these fiberwise left adjoints assemble to yield a $S$-adjunction \cite[Def.~8.1]{Exp2}
\[ \adjunct{\underline{\phi}_!}{\underline{\Fun}_S(K,C)}{\underline{\Fun}_S(L,C)}{\underline{\phi}^{\ast}} \]
(also see \cite[Cor.~9.16 and Thm.~10.4]{Exp2}). In particular, upon forgetting the structure maps\footnote{In other words, a relative adjunction yields an adjunction between the Grothendieck constructions.} we have an ordinary adjunction $\underline{\phi}_! \dashv \underline{\phi}_{\ast}$. Similarly, for $\phi$ an $S$-cartesian fibration we can consider the $S$-adjunction
\[ \adjunct{\underline{\phi}^{\ast}}{\underline{\Fun}_S(L,C)}{\underline{\Fun}_S(K,C)}{\underline{\phi}_{\ast}}. \]
For $\phi: X \to Y$ a map of $S$-spaces, we will consider $\underline{\phi}_! \dashv \underline{\phi}^{\ast} \dashv \underline{\phi}_{\ast}$.
\end{nul}

The key result that enables the construction of norm maps is the following lemma on adjointability. Note for the formulation of the statement that $S$-(co)cartesian fibrations are stable under pullback, and the property that $C$ admits the relevant $S$-(co)limits with respect to $\phi$ is stable under pullbacks in the $\phi$ variable.

\begin{lem} \label{lem:adjointability} Let $C$ be an $S$-$\infty$-category and let
\[ \begin{tikzcd}[row sep=4ex, column sep=4ex, text height=1.5ex, text depth=0.25ex]
K' \ar{r}{f'} \ar{d}[swap]{\phi'} & K \ar{d}{\phi} \\
L' \ar{r}{f} & L
\end{tikzcd} \]
be a pullback square of $S$-$\infty$-categories. Consider the resulting commutative square of $S$-functor categories and restriction functors
\[ \begin{tikzcd}[row sep=4ex, column sep=6ex, text height=1.5ex, text depth=0.25ex]
\Fun_S(K',C)  & \Fun_S(K,C) \ar{l}[swap]{{f'}^{\ast}} \\
\Fun_S(L',C) \ar{u}{{\phi'}^{\ast}} & \Fun_S(L,C) \ar{u}[swap]{\phi^{\ast}} \ar{l}[swap]{f^{\ast}}.
\end{tikzcd} \]
\begin{enumerate} \item If $\phi$ is a $S$-cocartesian fibration and $C$ admits the relevant $S$-colimits, then the square is left adjointable, i.e., the natural map ${\phi'}_! {f'}^{\ast} \to f^{\ast} \phi_!$ is an equivalence.
\item If $\phi$ is a $S$-cartesian fibration and $C$ admits the relevant $S$-limits, then this square is right adjointable, i.e., the natural map $f^{\ast} \phi_{\ast} \to {\phi'}_{\ast} {f'}^{\ast}$ is an equivalence.
\item If $K, L, K', L'$ are $S$-spaces, then the square is both left and right adjointable provided that $C$ admits the relevant $S$-colimits and $S$-limits.
\end{enumerate}
Likewise, we have the same results for the commutative square of $S$-functor $S$-$\infty$-categories
\[ \begin{tikzcd}[row sep=4ex, column sep=6ex, text height=1.5ex, text depth=0.5ex]
\underline{\Fun}_S(K',C)  & \underline{\Fun}_S(K,C) \ar{l}[swap]{\underline{f'}^{\ast}} \\
\underline{\Fun}_S(L',C) \ar{u}{\underline{\phi'}^{\ast}} & \underline{\Fun}_S(L,C) \ar{u}[swap]{\underline{\phi}^{\ast}} \ar{l}[swap]{\underline{f}^{\ast}}.
\end{tikzcd} \]
\end{lem}
\begin{proof} Let $F: K \to C$ be a $S$-functor. For (1), we need to check that ${\phi'}_! {f'}^{\ast} F \to f^{\ast} \phi_! F$ is an equivalence of $S$-functors. It suffices to evaluate on $S$-points $\underline{x}$ in $L'$, and we then have the map
\[ \colimP{S^{s/}} (F \circ f')_x \to \colimP{S^{s/}} F_{f(x)}. \]
But since $\underline{x} \times_{L'} K' \simeq \underline{f(x)} \times_L K$ as $S^{s/}$-$\infty$-categories, these $S^{s/}$-colimits are equivalent under the comparison map. The proof of (2) is similar. For (3) we replace $\phi$ by a categorical fibration and then use (1) and (2). For the corresponding assertion about $\underline{\Fun}_S(-,C)$, it suffices to check that the natural transformations of interest are equivalences fiberwise, upon which we reduce to the prior assertion for $\Fun_{S^{s/}}(-,C_{\underline{s}})$.
\end{proof}

\subsection{Ambidexterity of parametrized local systems}

In this subsection, we extend Hopkins and Lurie's study of ambidexterity for local systems \cite[\S 4.3]{hopkins2013ambidexterity} to the parametrized setting. The following definition generalizes \cite[Def.~4.3]{hopkins2013ambidexterity}.
 
\begin{dfn} Let $C$ be an $S$-$\infty$-category. The $\infty$-category of \emph{$S$-local systems} on $C$ $$\LocSys^S(C) \to \Spc^S$$ is the cartesian fibration classified by the composite 
 \[  \Fun_S(-,C): (\Spc^S)^{\op} \subset \Cat_{\infty}^{S,\op} \to \Cat_{\infty}. \]
The \emph{$S$-$\infty$-category of $S$-local systems} on $C$ $$\ul{\LocSys}^S(C) \to \Spc^S$$ is the cartesian fibration classified by the composite
\[  \underline{\Fun}_S(-,C): (\Spc^S)^{\op} \subset \Cat_{\infty}^{S,\op} \to \Cat_{\infty}^S \xto{U} \Cat_{\infty}. \]
where $U$ forgets the structure map of a cocartesian fibration.
\end{dfn}

\begin{cor} Suppose that for all $s \in S$, $C_{\underline{s}}$ admits all $S^{s/}$-colimits indexed by $S^{s/}$-spaces. Then $\LocSys^S(C)$ and $\ul{\LocSys}^S(C)$ are Beck-Chevalley fibrations \cite[Def.~4.1.3]{hopkins2013ambidexterity}.
\end{cor}
\begin{proof} Note that by the same argument as \cite[Rmk.~5.13]{Exp2}, the hypothesis ensures that $C$ admits all $S$-colimits indexed by $S$-spaces. The corollary is then immediate from Lem.~\ref{lem:adjointability}.
\end{proof}

\begin{rem} $\ul{\LocSys}^S(C)$ also admits a structure map to $S$ that is a cocartesian fibration, so is indeed an $S$-$\infty$-category.
\end{rem}

We now have the general theory of ambidexterity \cite[\S 4.1-2]{hopkins2013ambidexterity} for a Beck-Chevalley fibration, along with the attendant notions of ambidextrous and weakly ambidextrous morphisms \cite[Constr.~4.1.8 and Def.~4.1.11]{hopkins2013ambidexterity} in $\Spc^S$. For the reader's convenience, let us recall the relevance of these notions for constructing norm maps, referring to \cite[\S 4.1]{hopkins2013ambidexterity} for greater detail and precise definitions.

\begin{rem} Suppose $\sE \to \Spc^S$ is a Beck-Chevalley fibration, $f: X \to Y$ is a map of $S$-spaces, the left and right adjoints $f_!$ and $f_{\ast}$ to $f^{\ast}$ exist, and we wish to construct a norm map $\Nm_f: f_! \to f_{\ast}$. Consider the commutative diagram
\[ \begin{tikzcd}[row sep=4ex, column sep=6ex, text height=1.5ex, text depth=0.5ex]
X \ar{rd}{\delta} \ar[bend left=20]{rrd}{=} \ar[bend right=20]{rdd}[swap]{=}  & \\
& X \times_Y X \ar{r}{\pr_2} \ar{d}{\pr_1} & X \ar{d}{f} \\
& X \ar{r}{f} & Y
\end{tikzcd} \]
and suppose we have already constructed a norm map $\Nm_{\delta}: \delta_! \to \delta_{\ast}$ and shown it to be an equivalence. If $\Nm_{\delta}^{-1}: \delta_{\ast} \xto{\simeq} \delta_!$ is a choice of inverse, then we have a natural transformation
\[ \pr_1^{\ast} \xto{\eta_{\delta}} \delta_{\ast} \delta^{\ast} \pr_1^{\ast} \simeq \delta_{\ast} \xto{\Nm_{\delta}^{-1}} \delta_! \simeq \delta_! \delta^{\ast} \pr_2^{\ast} \xto{\epsilon_{\delta}} \pr_2^{\ast}. \]
By adjunction and using the Beck-Chevalley property, we obtain a map
\[ f^{\ast} f_! \simeq (\pr_1)_! \pr_2^{\ast} \to \id. \]
Finally, we may adjoint this map in turn to define $\Nm_f: f_! \to f_{\ast}$.

Thus, for an inductive construction of norm maps, we may single out a class of `ambidextrous' morphisms for which a norm map has been constructed and shown to be an equivalence, and then define `weakly ambidextrous' morphisms to be those morphisms $f: X \to Y$ whose diagonal $\delta: X \to X \times_Y X$ is ambidextrous.
\end{rem}

Continuing our study, we henceforth suppose that $C_{\underline{s}}$ also admits all $S^{s/}$-limits indexed by $S^{s/}$-spaces, so that the right adjoints $f_{\ast}$, $\underline{f}_{\ast}$ exist for all maps $f$ of $S$-spaces. Then by \cite[Rmk.~4.1.12]{hopkins2013ambidexterity}, for $\LocSys^S(C)$ a map $f: X \to Y$ in $\Spc^S$ is ambidextrous if and only if the norm map $\Nm_{f'}: f'_! \to f'_{\ast}$ is an equivalence for all pullbacks $f': X' \to Y'$ of $f$, and similarly for $\ul{\LocSys}^S(C)$.

To simplify the following discussion, we will phrase all of our statements for $\LocSys^S(C)$. However, such statements have obvious implications for $\ul{\LocSys}^S(C)$ via checking fiberwise.

\begin{lem} \label{lem:AmbidexCheckedFiberwise} \begin{enumerate}[leftmargin=*] \item Let $f: X \to Y$ be a weakly ambidextrous morphism. Then $f$ is ambidextrous if and only if for all $y \in Y$, the norm map $\Nm_{f_y}$ for the pullback $f_y: X_{\underline{y}} \to \underline{y}$ is an equivalence.
\end{enumerate}
\begin{enumerate}
    \setcounter{enumi}{1}
    \item $f: X \to Y$ is weakly ambidextrous if and only if if for all $y \in Y$, the pullback $f_y: X_{\underline{y}} \to \underline{y}$ is weakly ambidextrous.
\end{enumerate}
\end{lem}
\begin{proof} For (1), first note that the maps $f_y$ are weakly ambidextrous by \cite[Prop.~4.1.10(3)]{hopkins2013ambidexterity}, so the statement is well-posed. The `only if' direction holds by definition. For the `if' direction, suppose given a pullback square of $S$-spaces
\[ \begin{tikzcd}[row sep=4ex, column sep=4ex, text height=1.5ex, text depth=0.25ex]
X' \ar{d}[swap]{f'} \ar{r}{\phi'} & X \ar{d}{f} \\
Y' \ar{r}{\phi} & Y.
\end{tikzcd} \]
For any point $y' \in Y'$, if we let $y=\phi(y')$ then we have an equivalence $X'_{\underline{y'}} \simeq X_{\underline{y}} \to \underline{y'} \simeq \underline{y}$. Therefore, without loss of generality it suffices to prove that $\Nm_f: f_! \to f_{\ast}$ is an equivalence. Let $y \in Y$ and denote the inclusion of the $S$-point as $i_y: \underline{y} \to Y$ and the $S$-fiber as $j_y: X_{\underline{y}} \to X$. By \cite[Rmk.~4.2.3]{hopkins2013ambidexterity} and Lem.~\ref{lem:adjointability}, we have an equivalence
\[ i_y^{\ast}  \Nm_f \simeq \Nm_{f_y} j_y^{\ast} : \Fun_S(X,C) \to \Fun_S(\underline{y},C) \simeq C_s \]
(where $y$ covers $s$), which by assumption is an equivalence. Because the evaluation functors $i_y^{\ast}$ are jointly conservative, it follows that $\Nm_f$ is an equivalence.

For (2), we only need to prove the `if' direction. We will show that the diagonal $\delta: X \to X \times_Y X$ is ambidextrous. Let $\delta_y$ denote the diagonal for $f_y$. For all $y$ we have a pullback square
\[ \begin{tikzcd}[row sep=4ex, column sep=6ex, text height=1.5ex, text depth=0.25ex]
 X_{\underline{y}} \ar{r}{j_y} \ar{d}{\delta_y} & X \ar{d}{\delta} \\
 X_{\underline{y}} \times_{\underline{y}} X_{\underline{y}} \ar{r}{(j_y,j_y)} & X \times_Y X.
\end{tikzcd} \]
Given any object $(x,x') \in X \times_Y X$ with $f(x)=f(x')=y$,\footnote{We write an equality here because we are implicitly modeling $f$ as a categorical fibration of left fibrations over $S$.} the inclusion of the $S$-fiber $i_{(x,x')}: \underline{(x,x')} \to X \times_Y X$ factors through $X_{\underline{y}} \times_{\underline{y}} X_{\underline{y}}$. Therefore, if $\delta_y$ is ambidextrous, the norm map for $X_{\underline{(x,x')}} \to \underline{(x,x')}$ is an equivalence. By statement (1) of the lemma, we conclude that $\delta$ is ambidextrous.
\end{proof}

Recall from \cite[Prop.~4.1.10(6)]{hopkins2013ambidexterity} that given a weakly $n$-ambidextrous morphism $f: X \to Y$ and $-2 \leq m \leq n$, $f$ is weakly $m$-ambidextrous if and only if $f$ is $m$-truncated \cite[Def.~5.5.6.1]{HTT}. To identify the $n$-truncated maps in $\Spc^S$, we have the following result.

\begin{lem} \label{lem:TruncationPresheaves} Let $X: S \to \Spc$ be a $S$-space. Then $X$ is $n$-truncated as an object of $\Spc^S$ if and only if for each $s \in S$, $X(s)$ is an $n$-truncated space. Similarly, for a map $f: X \to Y$ of $S$-spaces, $f$ is $n$-truncated if and only if $f(s)$ is an $n$-truncated map of spaces for all $s \in S$.
\end{lem}
\begin{proof} We repeat the argument of \cite[5.5.8.26]{HTT} for the reader's convenience. It suffices to prove the result for maps. Let $j: S^{\op} \to \Spc^S$ denote the Yoneda embedding. Then if $f$ is $n$-truncated, for any $s \in S$,
\[ f(s) \simeq \Map(j(s),f): \Map(j(s),X) \simeq X(s) \to \Map(j(s),Y) \simeq Y(s) \]
 is $n$-truncated. Conversely, suppose each $f(s)$ is $n$-truncated. The collection of $S$-spaces $Z$ for which $\Map(Z,f)$ is $n$-truncated is stable under colimits, because limits of $n$-truncated spaces and maps are again $n$-truncated. Since the representable functors $j(s)$ generate $\Spc^S$ under colimits, it follows that $f$ itself is $n$-truncated.
\end{proof}

Let us now consider the $n=-1$ case.

\begin{dfn} Let $C$ be a $S$-$\infty$-category. $C$ is \emph{$S$-pointed} if for every $s \in S$, $C_s$ is pointed, and for every $\alpha: s \to t$, the pushforward functor $\alpha_{\sharp}: C_s \to C_{t}$ preserves the zero object. If $S = \sO_G^{\op}$, we also say that $C$ is \emph{$G$-pointed}.
\end{dfn}

\begin{lem} \label{lem:pointedAmbidex} $C$ is $S$-pointed if and only if for every $s \in S$, the weakly $(-1)$-ambidextrous morphism $0_s: \emptyset \to \underline{s}$ is ambidextrous.
\end{lem}
\begin{proof} For any $s \in S$, it is easy to see that the norm map $\Nm_{0_s}$ is the canonical map between the initial and final object in $C_s$. Moreover, for any $[\alpha: s \rightarrow t] \in \underline{s}$, the map $\underline{\alpha} \to \underline{s}$ is homotopic to $\alpha^{\ast}: \underline{t} \to \underline{s}$ and the pullback of $0_s$ along $\alpha^{\ast}$ is $0_t$. Thus $0_s$ is ambidextrous if and only if $C_s$ admits a zero object and for all $\alpha: s \to t$, the pushforward functor $\alpha_{\sharp}: C_s \to C_t$ preserves the zero object. The conclusion then follows.
\end{proof}

\begin{wrn} \label{wrn:ambidexCounterexample} In contrast to the non-parametrized case \cite[Prop.~6.1.6.7]{HA}, if $C$ is $S$-pointed then we may have weakly $(-1)$-ambidextrous morphisms that fail to be ambidextrous. For example, let $S = \sO_{C_2}^{\op}$ and let $p: J \to \sO_{C_2}^{\op}$ be the $C_2$-functor given by the inclusion of the full subcategory on the free transitive $C_2$-sets (so $J \simeq BC_2$). Then $J$ is a $C_2$-space via $p$, and for any $C_2$-$\infty$-category $C$, we have that $\Fun_{C_2}(J, C) \simeq (C_{C_2/1})^{h C_2}$ for the $C_2$-action on the fiber $C_{C_2/1}$ encoded by the cocartesian fibration. On the one hand, the $C_2$-diagonal functor $\delta: J \to J \times_{\sO_{C_2}^{\op}} J$ is an equivalence, so $J$ is a $(-1)$-truncated $C_2$-space. On the other hand, for $C = \underline{\Sp}^{C_2}$, the restriction $p^{\ast}$ may be identified with
\[ j^{\ast}: \Sp^{C_2} \simeq \Fun_{C_2}(\sO_{C_2}^{\op}, \underline{\Sp}^{C_2}) \to \Fun_{C_2}(J,\underline{\Sp}^{C_2}) \simeq \Fun(BC_2, \Sp) \]
which we saw has left and right adjoints $j_!$ and $j_{\ast}$ such that the norm map $j_! \to j_{\ast}$ is \emph{not} an equivalence.
\end{wrn}

We next consider the $n=0$ case. Let $T = S^{\op}$. Recall from \cite[Def.~4.1]{Exp4} that an $\infty$-category $T$ is said to be \emph{atomic orbital} if its finite coproduct completion $\FF_T$ admits pullbacks and $T$ has no non-trivial retracts (i.e., every retract is an equivalence). For example, $\sO_G$ is atomic orbital. We now assume that $T$ is atomic orbital, and we regard $V \in T$ as `$T$-orbits' and $U \in \FF_T$ as `finite $T$-sets'.

\begin{ntn} For any $U \in \FF_T$, if $U = \coprod_{i \in I} U_i$ is its unique decomposition into orbits $U_i \in T$, then we let $\underline{U} \coloneq \coprod_{i \in I} S^{U_i/} \to S$ denote the corresponding $S$-space.
\end{ntn}

\begin{lem} \label{lem:ComparisonCoproductProduct} Suppose $C$ is $S$-pointed. Then for any finite $T$-set $U$, $T$-orbit $V$ and morphism
\[ f: \underline{U} \to \underline{V} \]
(necessarily specified by a morphism $f: U \to V$ in $\FF_T$), the diagonal $\delta: \underline{U} \to \underline{U} \times_{\underline{V}} \underline{U}$ is ambidextrous. Consequently, if $g: X \to Y$ is a morphism between finite coproducts of representables, then $g$ is weakly $0$-ambidextrous.
\end{lem}
\begin{proof} By our assumption on $T$, $\underline{U} \times_{\underline{V}} \underline{U}$ decomposes as a finite disjoint union of representables $\coprod_{i \in I} \underline{V_i}$. Moreover, because $T$ admits no non-trivial retracts, for some $J \subset I$ we have that $\underline{U} \simeq \coprod_{j \in J} \underline{V_j}$ with matching orbits, and $\delta$ is a summand inclusion $\underline{U} \to (\coprod_{j \in J} \underline{V_j}) \sqcup (\coprod_{i \in I-J } \underline{V_i})$. $\delta$ is then ambidextrous by Lem.~\ref{lem:AmbidexCheckedFiberwise} and Lem.~\ref{lem:pointedAmbidex}. The final consequence also follows by Lem.~\ref{lem:AmbidexCheckedFiberwise}.
\end{proof}

By Lem.~\ref{lem:ComparisonCoproductProduct}, the following definition is well-posed.

\begin{dfn} \label{dfn:semiadditive} Let $C$ be $S$-pointed. We say that $C$ is \emph{$S$-semiadditive} if for each morphism $f: U \to V$ in $\FF_T$, the norm map $\Nm_f$ for $f: \underline{U} \to \underline{V}$ is an equivalence. If $S = \sO_G^{\op}$, we will instead say that $C$ is \emph{$G$-semiadditive}.
\end{dfn}

Equivalently, in Def.~\ref{dfn:semiadditive} we could demand only that the norm maps for $f: U \to V$ with $V$ an orbit are equivalences.

\begin{rem} Unwinding the definition of the norm maps produced via our setup and in \cite[Constr.~5.2]{Exp4}, one sees that Def.~\ref{dfn:semiadditive} is the same as the notion of $T$-semiadditive given in \cite[Def.~5.3]{Exp4}. In particular, for $T = \sO_G^{\op}$, $\underline{\Sp}^G$ is an example of a $G$-semiadditive $G$-$\infty$-category. This amounts to the familiar fact that for each orbit $G/H$, $\Sp^H$ is semiadditive, and for each map of orbits $f: G/H \to G/K$, the left and right adjoints to the restriction functor $f^{\ast}: \Sp^K \to \Sp^H$ given by induction and coinduction are canonically equivalent.
\end{rem}



In the remainder of this subsection, we further specialize to the case $S = \sO_G^{\op}$ for $G$ a finite group. We have already encountered a potential problem in Warn.~\ref{wrn:ambidexCounterexample} with developing a useful theory of $G$-ambidexterity. The issue is essentially due to the presence of fiberwise discrete $G$-spaces that do not arise from $G$-sets. To remedy this, we will restrict our attention to the Borel subclass of $G$-spaces.

\begin{dfn} \label{dfn:BorelGSpace} Suppose that $X \to \sO_G^{\op}$ is a $G$-space. Then $X$ is \emph{Borel} if the functor $\sO^{\op}_G \to \Spc$ classifying $X$ is a right Kan extension along the inclusion of the full subcategory $BG \subset \sO_G^{\op}$.
\end{dfn}

\begin{rem} Def.~\ref{dfn:BorelGSpace} is equivalent to the following condition on a $G$-space $X$: if we let $$X_H \coloneq \sO^{\op}_H \times_{(\ind^G_H)^{\op},\sO^{\op}_G} X$$ denote the restriction of $X$ to an $H$-space, then for every subgroup $H \leq G$, the natural map
\[ X_{G/H} \simeq \Map^{\cocart}_{/\sO^{\op}_H}(\sO^{\op}_H,X_H) \to \Map_{/BH}(BH,BH \times_{\sO^{\op}_H} X_H) \simeq (X_{G/1})^{h H} \]
is an equivalence.
\end{rem}

\begin{rem} Limits and coproducts of Borel $G$-spaces are Borel. Moreover, for every $G$-set $U$, the $G$-space $\underline{U}$ is Borel. Indeed, this amounts to the observation that $\Hom_G(G/H,U) \cong U^H$ for all subgroups $H \leq G$. In particular, since representables are Borel, the Borel property is stable under passage to $G$-fibers.
\end{rem}


\begin{nul} Let $f: X \to Y$ be a map of Borel $G$-spaces and let $f_0: X_0 \to Y_0$ denote the underlying map of spaces. Then by Lem.~\ref{lem:TruncationPresheaves}, $f$ is $n$-truncated if and only if $f_0$ is $n$-truncated. Furthermore, because every $G$-orbit is a finite set, the following two conditions are equivalent:
\begin{enumerate}
    \item For every $y \in Y_0$, the homotopy fiber $(X_0)_y$ is a finite $n$-type \cite[Def.~4.4.1]{hopkins2013ambidexterity}.
    \item For every $y \in Y$, the underlying space of the homotopy $G$-fiber $X_{\underline{y}}$ is a finite $n$-type.
\end{enumerate}
In this case, we say that $f$ is \emph{$\pi$-finite $n$-truncated}. Note that if $f$ is $\pi$-finite $n$-truncated, then its diagonal is $\pi$-finite $(n-1)$-truncated; indeed, this property can be checked for the underlying spaces.

More generally, if $f:X \to Y$ is a map of $G$-spaces such that for every $y \in Y$, $X_{\ul{y}}$ is Borel, then we say that $f$ is ($\pi$-finite) $n$-truncated if the above conditions hold for all $y$ and $X_{\ul{y}}$.
\end{nul}

\begin{lem} \label{lem:BorelAmbidex} Let $f: X \to Y$ be a map of $G$-spaces such that for all $y \in Y$, $X_{\ul{y}}$ is Borel.
\begin{enumerate}
\item Suppose that $C$ is $G$-pointed.
    \begin{enumerate}
    \item If $f$ is $(-1)$-truncated, then $f$ is $(-1)$-ambidextrous.
    \item If $f$ is $0$-truncated, then $f$ is weakly $0$-ambidextrous.
    \end{enumerate}
\item Suppose in addition that $C$ is $G$-semiadditive.
    \begin{enumerate}
    \item If $f$ is $\pi$-finite $0$-truncated, then $f$ is $0$-ambidextrous.
    \item If $f$ is $\pi$-finite $1$-truncated, then $f$ is weakly $1$-ambidextrous.
    \end{enumerate}
\end{enumerate}
\end{lem}
\begin{proof} By Lem.~\ref{lem:AmbidexCheckedFiberwise} and under our hypothesis on the parametrized fibers, we may suppose in the proof that $Y = \underline{G/H}$ and $X$ is Borel. For (1), if $f$ is $(-1)$-truncated then the underlying space of $X$ is a discrete set that injects into $G/H$. But if $X$ is non-empty then $f$ must also be a surjective map of $G$-sets since the $G$-action on $G/H$ is transitive. Thus either $X = \emptyset$ or $f$ is an equivalence, so by Lem.~\ref{lem:pointedAmbidex}, $f$ is $(-1)$-ambidextrous. If $f$ is $0$-truncated, then the $(-1)$-truncated diagonal $X \to X \times_Y X$ is $(-1)$-ambidextrous as just shown, so $f$ is weakly $0$-ambidextrous.

For (2), we employ the same strategy. If $f$ is $\pi$-finite $0$-truncated, then $X$ is necessarily a finite $G$-set, so $f$ is $0$-ambidextrous by hypothesis. If $f$ is $\pi$-finite $1$-truncated, then the diagonal $X \to X \times_Y X$ is $\pi$-finite $0$-truncated and hence $0$-ambidextrous, so $f$ is weakly $1$-ambidextrous.
\end{proof}

To apply the parametrized ambidexterity theory to our situation of interest, we need the following lemma.

\begin{lem} \label{lem:BorelClassifyingSpace} The $G/N$-space $B^{\psi}_{G/N} N$ of Def.~\ref{dfn:TwistedClassifyingSpace} is Borel.
\end{lem}
\begin{proof} For any subgroup $K/N$ of $G/N$, $(B^{\psi}_{G/N} N)_{K/N} \simeq (B^{\psi'}_{K/N} N)$ for $\psi' = [N \to K \to K/N]$. Therefore, without loss of generality it suffices to prove that the map of groupoids
\[ \chi: \Map^{\cocart}_{/\sO_{G/N}^{\op}}(\sO_{G/N}^{\op}, B^{\psi}_{G/N} N) \to \Map_{/B(G/N)}(B(G/N), (B^{\psi}_{G/N} N) \times_{\sO_{G/N}^{\op}} B(G/N) ) \]
is an equivalence. The fiber $E$ of $B^{\psi}_{G/N} N$ over the terminal $G/N$-set $\ast$ is spanned by those $N$-free $G$-orbits $U$ such that $U/N \cong \ast$, and an explicit inverse to the evaluation map
\[ \Map^{\cocart}_{/\sO_{G/N}^{\op}}(\sO_{G/N}^{\op}, B^{\psi}_{G/N} N) \xto{\simeq} E \]
is given by sending $U$ to the cocartesian section $s_U = (- \times U): \sO_{G/N}^{\op} \to B^{\psi}_{G/N} N$ that sends $V$ to $V \times U$: this follows from our identification of the cocartesian edges in Lem.~\ref{lem:QuotientMapCartesianFibration}. Then
\[ \chi(s_U): B(G/N) \to (B^{\psi}_{G/N} N) \times_{\sO_{G/N}^{\op}} B(G/N) \] is the section which sends $G/N$ to the free transitive $G$-set $G/N \times U$. Let us now select a basepoint to identify $U \cong G/H$ for $H$ a subgroup such that $H \cap N = 1$ and $G = N H$. We have $W_G H \cong \Aut_G(G/H)$, where a coset $\overline{a} \in W_G H$ gives an automorphism $\theta_{\overline{a} }$ of $G/H$ that sends $1H$ to the well-defined coset $aH$, and under $\chi$ this is sent to the automorphism $\id \times \theta_{\overline{a} }$ of the section $\chi(s_{G/H})$.

By elementary group theory, each coset in $G/N$ has a unique representative $x N$ with $x \in H$, and each coset in $G/H$ has a unique representative $y H$ with $y \in N$. Moreover, the inclusion $N_G(H) \cap N \to N_G(H)$ yields an isomorphism $N_G(H) \cap N  \cong W_G(H)$; the map is an injection because $N \cap H = 1$ and a surjection because $G = NH$. The two surjections $G/1 \to G/N$ and $G/1 \to G/H$ sending $1G$ to $1N$ and $1H$ define an isomorphism $G/1 \xto{\cong} G/N \times G/H$ for which an explicit inverse sends $(xN,yH)$ to $x \cdot y$. Under this isomorphism, $\id \times \theta_{\overline{a} }$ is sent to the unique element $a \in N_G(H) \cap N$ that is a representative for $\overline{a}$.

On the other hand, $(B^{\psi}_{G/N} N) \times_{\sO_{G/N}^{\op}} B(G/N) \simeq BG$ where we select $G/1$ to be the unique object of $BG$. We compute the groupoid of maps $\Map_{/B(G/N)}(B(G/N), BG)$ to have objects given by splittings $\tau: G/N \to G$ of the surjection $\pi: G \to G/N$ and morphisms $\tau \to \tau'$ given by $n \in N$ such that for every coset $bN$, $n \tau(b N) n^{-1} = \tau'(b N)$. In particular, $\Aut(\tau) = N \cap C_G(H)$ for $H = \tau(G/N)$. However, if $n h n^{-1} = h' \in H$, then $\pi(n h n^{-1} h^{-1}) = 1$ shows that $n h n^{-1} h^{-1} \in N \cap H =1$, so in fact $n h = h n$ and thus $\Aut(\tau) = N \cap N_G(H)$. Combining this with the explicit understanding of the comparison map given above, we deduce that $\chi$ is fully faithful. Essential surjectivity is also clear by the bijection between splittings of $\pi$ and subgroups $H$ with $N \cap H = 1$ and $G = N H$. We conclude that $\chi$ is an equivalence.
\end{proof}

\subsection{The parametrized Tate construction}

In view of Lem.~\ref{lem:BorelClassifyingSpace}, we may define the parametrized Tate construction (Def.~\ref{dfn:ParamTateCnstr}) by applying the parametrized ambidexterity theory to the Beck-Chevalley fibration $$\LocSys^{G/N}(\underline{\Sp}^{G/N}) \to \Spc^{G/N}.$$ Let $\rho_N: B^{\psi}_{G/N} N \to \sO_{G/N}^{\op}$ be the structure map as in Lem.~\ref{lm:TwistedClassifyingSpaceIsSpace}, and let
 \[ {\rho_N}^{\ast}: \Sp^{G/N} \simeq \Fun_{G/N}(\sO_{G/N}^{\op},\underline{\Sp}^{G/N}) \to  \Fun_{G/N}(B^{\psi}_{G/N} N, \underline{\Sp}^{G/N}) \]
be the functor given by restriction along $\rho_N$. We first introduce some alternative notation for parametrized homotopy orbits $(\rho_N)_!$ and fixed points $(\rho_N)^{\ast}$.


\begin{ntn} Given $X \in \Fun_{G/N}(B^{\psi}_{G/N} N, \underline{\Sp}^{G/N})$, we will write
\begin{align*} X_{h[\psi]} \coloneq (\rho_N)_! (X) \text{ and } X^{h[\psi]} \coloneq (\rho_N)^{\ast} X.
\end{align*}
\end{ntn}

\begin{dfn} \label{dfn:ParamTateCnstr} The $G/N$-functor $\rho_N$ has as its underlying map of spaces $B N \to \ast$, which is $\pi$-finite $1$-truncated. By Lem.~\ref{lem:BorelAmbidex}, $\rho_N$ is weakly $1$-ambidextrous, so we can construct the norm map $\Nm_{\rho_N}: (\rho_N)_! \to (\rho_N)^{\ast}$. Let
\[ t[\psi]: \Fun_{G/N}(B^{\psi}_{G/N} N, \underline{\Sp}^{G/N}) \to \Sp^{G/N} \]
denote the cofiber of $\Nm_{\rho_N}$. On objects $X$, we write $X^{t[\psi]}$ for the image of $X$ under $t[\psi]$. 
\end{dfn}

\begin{ntn} If $\psi$ is the defining extension $A \to A \rtimes C_2 \to C_2$ of the semidirect product where $C_2$ acts on the abelian group $A$ by inversion, we will instead write $X_{h_{C_2} A} = X_{h[\psi]}$, $X^{h_{C_2} A} = X^{h[\psi]}$, and $X^{t_{C_2} A} = X^{t[\psi]}$.
\end{ntn}

\begin{rem}[The norm vanishes on induced objects] \label{NormVanishesOnInduced} For $H \in \Gamma_N$, let $U = \rho_N(G/H) \cong \frac{G/N}{H N/N}$ be the $G/N$-orbit and $s_H: \underline{U} \to B^{\psi}_{G/N} N$ be the unique $G/N$-functor that selects $G/H$, so that the functor $s_H^{\ast}$ of \ref{evaluationFactorizationNaiveSpectra} is obtained by restriction along $s_H$. Note that the map of Borel $G/N$-spaces $s_H$ is $\pi$-finite $0$-truncated because its underlying map of spaces is $U \to B N$ with $U$ a finite discrete set. By Lem.~\ref{lem:BorelAmbidex}, we see that $\Nm_{s_H}: {s_H}_! \xto{\simeq} (s_H)_{\ast}$.

Now consider the composite map $p_U = \rho_N \circ (s_H)$, which is also $\pi$-finite $0$-truncated. On the one hand, the associated norm map $\Nm_{p_U}$ is an equivalence (explicitly, between induction and coinduction from $\Sp^H$ to $\Sp^{G/N}$ for $H \cong H N/N$ viewed as a subgroup of $G/N$). On the other hand, by \cite[Rmk.~4.2.4]{hopkins2013ambidexterity}, we have that $\Nm_{p_U}$ is homotopic to the composite $((\rho_N)_{\ast} \Nm_{s_H}) \circ (\Nm_{\rho_N} (s_H)_!)$. We deduce that $\Nm_{\rho_N}$ is an equivalence on the image of $(s_H)_!$. This extends the observation that the ordinary norm map $X_{h G} \to X^{h G}$ is an equivalence on objects induced from $\Sp$ to $\Fun(B G, \Sp)$.
\end{rem}

By Prop.~\ref{prp:BorelSpectraAsCompleteObjects}, we also have a norm map $\Nm': \sF_b[N] \to \sF^{\vee}_b[N]$ arising from the $\Gamma_N$-recollement of $\Sp^G$, with functors $\sF_b[N]$ and $\sF^{\vee}_b[N]$ as in \ref{VariousPropertiesForgetfulFunctorBorelSpectra}. We now proceed to show that the two norm maps $\Psi^N \Nm'$ and $\Nm_{\rho_N}$ are equivalent.

\begin{lem} \label{lm:IdentifyingDiagonalAsComposition} The functor $\rho_N^{\ast}$ is homotopic to the composite
\[ \sU_b[N] \circ {\inf}^N: \Sp^{G/N} \to \Sp^G \to \Fun_{G/N}(B^{\psi}_{G/N} N, \underline{\Sp}^{G/N}). \]
\end{lem}
\begin{proof} For the proof, we work in the setup of Constr.~\ref{cnstr:ForgetfulFunctorToNaiveGSpectra}. Let
\[ \underline{\inf}'[N]: \underline{\Sp}^{G/N} \to \sO_{G/N}^{\op} \times_{\sO_G^{\op}} \underline{\Sp}^G \]
be the $G/N$-functor defined by the natural transformation $\SH q_N^{\op}: \SH \omega_{G/N}^{\op} \to \SH \omega_G^{\op} \iota_N^{\op}$, so for a $G/N$ orbit $V = \frac{G/N}{K/N}$, the fiber $\underline{\inf}'[N]_V: \Sp^{K/N} \to \Sp^K$ is given by the inflation functor $\inf^N$. By definition, the composite
\[ \begin{tikzcd}[row sep=4ex, column sep=6ex, text height=1.5ex, text depth=0.5ex]
\underline{\Sp}^{G/N} \ar{r}{\underline{\inf}'[N]} & \sO_{G/N}^{\op} \times_{\sO^{\op}_G} \Sp^G \ar{r}{\widetilde{\sU}[N]} & \underline{\Fun}_{G/N}(\sO_G^{\op}, \underline{\Sp}^{G/N})
\end{tikzcd} \]
is adjoint to the composite (abusing notation for the first functor)
\[ \begin{tikzcd}[row sep=4ex, column sep=6ex, text height=1.5ex, text depth=0.5ex]
\sO^{\op}_G \times_{\sO_{G/N}^{\op}} \underline{\Sp}^{G/N} \ar{r}{\underline{\inf}'[N]} & \sO_G^{\op} \times_{\sO_G^{\op}} \underline{\Sp}^G \ar{r}{\underline{\res}[N]} & \underline{\Sp}^G \ar{r}{\widehat{\Psi}[N]} & \sO^{\op}_G \times_{\sO_{G/N}^{\op}} \underline{\Sp}^{G/N} \ar{r}{\pr} & \underline{\Sp}^{G/N}.
\end{tikzcd} \]
For a $G$-orbit $G/H$, the fiber of $\widehat{\Psi}[N] \circ \underline{\res}[N] \circ \underline{\inf}'[N]$ over $G/H$ is given by the composition
\[ \begin{tikzcd}[row sep=4ex, column sep=6ex, text height=1.5ex, text depth=0.5ex]
\Sp^{H N/N} \ar{r}{\inf^N} & \Sp^{H N} \ar{r}{\res^{H N}_H} & \Sp^H \ar{r}{\Psi^{H \cap N}} & \Sp^{H/(H \cap N)}.
\end{tikzcd} \]
Using that $H N/N \cong H/(H \cap N)$, the composition $\res^{H N}_H \circ \inf^N$ is homotopic to $\inf^{H \cap N}$. Therefore, if $H \in \Gamma_N$ so that $H \cap N = 1$, the entire composite is trivial. We deduce that the composite
\[ \begin{tikzcd}[row sep=4ex, column sep=6ex, text height=1.5ex, text depth=0.5ex]
B^{\psi}_{G/N} N \times_{\sO_{G/N}^{\op}} \underline{\Sp}^{G/N} \ar{r}{\underline{\inf}'[N]} & B^{\psi}_{G/N} N \times_{\sO_G^{\op}} \underline{\Sp}^G \ar{r}{\underline{\res}[N]} & \underline{\Sp}^G \ar{r}{\widehat{\Psi}[N]} & B^{\psi}_{G/N} N \times_{\sO_{G/N}^{\op}} \underline{\Sp}^{G/N} 
\end{tikzcd} \]
is homotopic to the identity, which proves the claim.
\end{proof}

\begin{nul} By Lem.~\ref{lm:IdentifyingDiagonalAsComposition}, $(\rho_N)_{\ast} \simeq \Psi^N \sF^{\vee}_b[N]$. Let $t'[\psi] : \Fun_{G/N}(B^{\psi}_{G/N} N, \underline{\Sp}^{G/N}) \to \Sp^{G/N}$ be the cofiber of $\Psi^N \Nm'$. Since the orbits $G/H_+ \in \Sp^G$ for $H \in \Gamma_N$ are both $\Gamma_N$-torsion and $\Gamma_N$-complete, $\Nm' \circ (s_H)_!(1)$ is an equivalence for all $H \in \Gamma_N$. Therefore, $t'[\psi]$ vanishes on each $(s_H)_!(1)$. Because $\{ (s_H)_!(1): H \in \Gamma_N \}$ is a set of compact generators for $\Fun_{G/N}(B^{\psi}_{G/N} N, \underline{\Sp}^{G/N})$ and $(\rho_N)_! $ is a colimit preserving functor, the composite
\[ (\rho_N)_! \xtolong{\Nm_{\rho_N}}{1} (\rho_N)_{\ast} \simeq \Psi^N \sF^{\vee}_b[N] \to t'[\psi] \]
is null-homotopic. We thereby obtain a natural transformation $\nu: t[\psi] \to t'[\psi]$. Taking fibers, we also have a natural transformation $\mu: (\rho_N)_! \to \Psi^N \sF_b[N]$. All together, for $X \in \Sp^G_{\Borel{N}}$, we have
\[ \begin{tikzcd}[row sep=4ex, column sep=8ex, text height=1.5ex, text depth=0.5ex]
X_{h[\psi]} \ar{r}{\Nm_{\rho_N}} \ar{d}{\mu_X} & X^{h[\psi]} \ar{r} \ar{d}{\simeq} & X^{t[\psi]} \ar{d}{\nu_X} \\
\Psi^N \sF_b[N] (X) \ar{r}{\Psi^N \Nm'} & \Psi^N \sF_b^{\vee}[N](X) \ar{r} & X^{t'[\psi]}.
\end{tikzcd} \]
\end{nul}

\begin{prp} \label{prp:EquivalentTateConstructions} The natural transformations $\mu$ and $\nu$ are equivalences.
\end{prp}
\begin{proof} It suffices to show that $\mu$ is an equivalence. By Rmk.~\ref{NormVanishesOnInduced}, $\Nm_{\rho_N}$ is an equivalence on $(s_H)_!(1)$ for every $H \in \Gamma_N$, and we just saw the same property for $\Nm'$. Therefore, $\mu$ is an equivalence on each $(s_H)_!(1)$ by the two-out-of-three property of equivalences. Since both $(\rho_N)_!$ and $\Psi^N \sF_b[N]$ preserve colimits and the $(s_H)_!(1)$ form a set of compact generators, we conclude that $\mu$ is an equivalence.
\end{proof}

\begin{rem}[$\infty$-categorical Adams isomorphism] By Prop.~\ref{prp:EquivalentTateConstructions}, for $X \in \Sp^G_{\Borel{N}}$, we have an equivalence of $G/N$-spectra $X_{h[\psi]} \simeq \Psi^N \sF_b[N] (X)$. Viewing $X$ as an `$N$-free' $G$-spectrum, this amounts to the Adams isomorphism for a normal subgroup $N$ of a \emph{finite} group $G$ in our context (compare \cite[Ch.~XVI, Thm.~5.4]{AlaskaNotes}). 

We also note that Sanders \cite{SandersCompactnessLocus} has recently introduced a different formal framework for producing the Adams isomorphism, in the more general situation of a closed normal subgroup of a compact Lie group. It would be interesting to understand the relationship between his results and ours.
\end{rem}

\begin{rem} In view of Prop.~\ref{prp:EquivalentTateConstructions}, we could have defined the parametrized Tate construction as $t'[\psi]$ to begin with. However, we still need the Adams isomorphism to identify the fiber term $\Psi^N \sF_b[N]$ of $(\rho_N)_{\ast} \to t'[\psi]$ as the parametrized orbits functor $(\rho_N)_!$.
\end{rem}

\begin{rem}[Point-set models] \label{rem:PointSetModels} Let $X \in \Sp^G_{\Borel{N}}$ and consider the fiber sequence of $G/N$-spectra
\[ X_{h[\psi]} \to  X^{h[\psi]} \to X^{t[\psi]}. \]
By Prop.~\ref{prp:EquivalentTateConstructions} and the monoidal recollement theory for $\Gamma_N$, this fiber sequence is obtained as $\Psi^N$ of the fiber sequence of $G$-spectra
\[ \sF_b^{\vee}[N](X) \otimes {E \Gamma_N}_+ \to \sF_b^{\vee}[N](X) \to \sF_b^{\vee}[N](X) \otimes \widetilde{E \Gamma_N}. \]
If we let $X = \sU_b[N](Y)$ for $Y \in \Sp^G$, then we may also write this as
\[ Y \otimes {E \Gamma_N}_+ \simeq F( {E \Gamma_N}_+, Y) \otimes {E \Gamma_N}_+ \to F( {E \Gamma_N}_+, Y) \to F( {E \Gamma_N}_+, Y) \otimes \widetilde{E \Gamma_N}. \]

In \cite[\S 2.3]{Qui19b}, the first author defined the parametrized fixed points, orbits, and Tate constructions using these `point-set' models in the special case of a trivial extension $[\Sigma \to G \times \Sigma \to G]$ (writing $(-)^N$ for the categorical fixed points functor in place of $\Psi^N$).
\end{rem}

\begin{rem}[Compatibility with restriction] \label{rem:ParamTateCompatibleRestriction} The norm map $\Nm'$ extends to a natural transformation of $G$-functors
\[ \left( \underline{\Nm'}: \underline{\sF}_b[N] \Rightarrow \underline{\sF}^{\vee}_b[N] \right): \underline{\Sp}^G_{\Borel{N}} \simeq \underline{\Sp}^{h \Gamma_N} \to \underline{\Sp}^G. \]
Postcomposing with the functor $\widehat{\Psi}[N]: \underline{\Sp}^G \to \sO_G^{\op} \times_{\sO_{G/N}^{\op}} \underline{\Sp}^{G/N}$ defined in Constr.~\ref{cnstr:ForgetfulFunctorToNaiveGSpectra} and taking the cofiber, we may extend $t[\psi]$ to a functor over $\sO^{\op}_G$
\[ \widehat{t}[\psi] : \underline{\Sp}^G_{\Borel{N}} \to \sO_G^{\op} \times_{\sO_{G/N}^{\op}} \underline{\Sp}^{G/N} \]
that over an orbit $G/H$ is given by
\[ t[\psi_H]: \Fun_{H/(N \cap H)}(B^{\psi_H}_{H/(N \cap N)} (N \cap H), \underline{\Sp}^{H/(N \cap H)}) \to \Sp^{H/(N \cap H)}. \]
However, because $\underline{\Psi}[N]$ is not typically a $G$-functor, $\widehat{t}[\psi]$ may also fail to be a $G$-functor. If instead we precompose by the inclusion 
\[ \underline{\Fun}_{G/N}(B^{\psi}_{G/N} N, \underline{\Sp}^{G/N}) \simeq \sO_{G/N}^{\op} \times_{\sO_G^{\op}} \underline{\Sp}^G_{\Borel{N}} \to \underline{\Sp}^G_{\Borel{N}} \] 
and postcompose by the projection to $\underline{\Sp}^{G/N}$, then we obtain a $G/N$-functor
\[ \underline{t}[\psi]: \underline{\Fun}_{G/N}(B^{\psi}_{G/N} N, \underline{\Sp}^{G/N}) \to \underline{\Sp}^{G/N}. \]
By checking fiberwise, it is easy to verify that $\underline{t}[\psi]$ is equivalent to the cofiber of the norm map $\Nm_{\rho_N}: (\underline{\rho_N})_! \to (\underline{\rho_N})_{\ast}$ produced by the ambidexterity theory for the other Beck-Chevalley fibration $\underline{\LocSys}^{G/N}(\underline{\Sp}^{G/N}) \to \Spc^{G/N}$ -- we leave further details to the reader. In any case, we obtain a compatibility between the parametrized Tate construction and restriction. For example, given a $G/N$-functor $X: B^{\psi}_{G/N} N \to \underline{\Sp}^{G/N}$, we see that the underlying spectrum of $X^{t[\psi]}$ is $X^{tN} \coloneq (\res^{G/N} X)^{t N}$ for the underlying functor $\res^{G/N} X: B N \to \Sp$.
\end{rem}


A useful consequence of Prop.~\ref{prp:EquivalentTateConstructions} is that it enables us to endow the functor $t[\psi]$ and the natural transformation $(-)^{h[\psi]} \to (-)^{t[\psi]}$ with lax monoidal structures, with respect to the pointwise symmetric monoidal structure on $\Fun_{G/N}(B^{\psi}_{G/N} N, \underline{\Sp}^{G/N})$ and the smash product on $\Sp^{G/N}$.

\begin{cor} \label{cor:TateLaxMonoidalStructure} The functor $t[\psi]$ and the natural transformation $(-)^{h[\psi]} \to (-)^{t[\psi]}$ are lax monoidal.
\end{cor}
\begin{proof} The cofiber of $\Nm'$ is the lax monoidal map $j_{\ast} \to i_{\ast} i^{\ast} j_{\ast}$ of the $\Gamma_N$-recollement of $\Sp^G$, where we use Rmk.~\ref{rem:MonoidalIdentificationOfBorelSpectra} to relate the pointwise symmetric monoidal structure on the domain $\Fun_{G/N}(B^{\psi}_{G/N} N, \underline{\Sp}^{G/N})$ to the monoidal recollement. Since the categorical fixed points functor $\Psi^N$ is also lax monoidal, we deduce that $\Psi^N \sF^{\vee}_b[N](-) \to (-)^{t'[\psi]}$ is lax monoidal. The conclusion now follows from Prop.~\ref{prp:EquivalentTateConstructions}.
\end{proof}

On the other hand, one practical benefit of defining the parametrized Tate construction via the ambidexterity theory is that we may exploit the general naturality properties of norms as detailed in \cite[\S4.2]{hopkins2013ambidexterity}. To state our next result, which involves two normal subgroups $M \trianglelefteq N$ of $G$, we first require a preparatory lemma.

\begin{lem} \label{lm:CategoricalFixedPointsProperties} Let $M \trianglelefteq N \trianglelefteq G$ be two normal subgroups of $G$ and let $\psi = [N \to G \to G/N]$, $\psi' = [N/M \to G/M \to G/N]$ denote the extensions.
\begin{enumerate}
\item $\Psi^M: \Sp^G \to \Sp^{G/M}$ sends $\Gamma_N$-torsion spectra to $\Gamma_{N/M}$-torsion spectra.
\item Let $r_M: \FF_G \to \FF_{G/M}$, $r_M(U) = U/M$ be as in \ref{inflationFunctors}, and regard $\FF_G$, $\FF_{G/M}$ as cartesian fibrations over $\FF_{G/N}$ via $r_N$, $r_{N/M}$ respectively. Then $r_M$ preserves cartesian edges, so the restricted functor $r_M^{\op}: \sO_G^{\op} \to \sO_{G/M}^{\op}$ is a $G/N$-functor. Moreover, $r_M^{\op}$ further restricts to a $G/N$-functor $$\rho_M: B^{\psi}_{G/N} N \to B^{\psi'}_{G/N} N/M.$$
\item We have a commutative diagram
\[ \begin{tikzcd}[row sep=4ex, column sep=8ex, text height=1.5ex, text depth=0.5ex]
\Sp^G_{\Borel{N}} = \Fun_{G/N}(B^{\psi}_{G/N} N, \underline{\Sp}^{G/N}) & \Sp^G \ar{l}{\sU_b[N]} \\
\Sp^{G/M}_{\Borel{N/M}} = \Fun_{G/N}(B^{\psi'}_{G/N} (N/M), \underline{\Sp}^{G/N}) \ar{u}{(\rho_M)^{\ast}} & \Sp^{G/M} \ar{u}[swap]{\inf^G_{G/M}} \ar{l}{\sU_b[N/M]}
\end{tikzcd} \]
that yields a commutative diagram of right adjoints
\[ \begin{tikzcd}[row sep=4ex, column sep=8ex, text height=1.5ex, text depth=0.5ex]
\Sp^G_{\Borel{N}} = \Fun_{G/N}(B^{\psi}_{G/N} N, \underline{\Sp}^{G/N}) \ar{r}{\sF_b^{\vee}[N]} \ar{d}[swap]{(\rho_M)_{\ast}} & \Sp^G \ar{d}{\Psi^M} \\
\Sp^{G/M}_{\Borel{N/M}} = \Fun_{G/N}(B^{\psi'}_{G/N} (N/M), \underline{\Sp}^{G/N}) \ar{r}{\sF_b^{\vee}[N/M]} & \Sp^{G/M}
\end{tikzcd} \]
where the lefthand vertical functor is computed by the $G/N$-right Kan extension along $r^{\op}_M$.
\end{enumerate}
\end{lem}
\begin{proof} For (1), note that if $G/H$ is a $N$-free $G$-orbit, then $G/H$ is also $M$-free. Thus, we may compute $\Psi^M(G/H_+)$ as by taking the quotient by the $M$-action to obtain $\frac{G/M}{H M/M}_+$, which is $N/M$-free and thus $\Gamma_{N/M}$-torsion. Because the subcategory of $\Gamma_N$-torsion spectra is the localizing subcategory generated by such $G/H_+$ and $\Psi^M$ preserves colimits, the statement follows. (2) is a direct consequence of Lem.~\ref{lem:QuotientMapCartesianFibration}(2). (3) is a relative version of Lem.~\ref{lm:IdentifyingDiagonalAsComposition}, and also follows by an elementary diagram chase after unpacking the various definitions.
\end{proof}

Now suppose $G$ is a semidirect product of $N$ and $G/N$, so we have chosen a splitting $G/N \to G$ of the quotient map such that $G \cong N \rtimes G/N$, and with respect to the $G/N$-action on $N$, the inclusion $M \subset N$ is $G/N$-equivariant. Then $M \rtimes G/N$ is a subgroup of $G$, and we let $$\psi''=[M \to M \rtimes G/N \to G/N].$$ Also regard $B^{\psi'}_{G/N} N/M$ as a based $G/N$-space via the splitting. Then we have a homotopy pullback square of $G/N$-spaces
\[ \begin{tikzcd}[row sep=4ex, column sep=4ex, text height=1.5ex, text depth=0.25ex]
B^{\psi''}_{G/N} M \ar{r} \ar{d}{\rho_M} & B^{\psi}_{G/N} N \ar{d}{\rho_M} \\
\sO_{G/N}^{\op} \ar{r} & B^{\psi'}_{G/N} N/M 
\end{tikzcd} \]
that arises from the fiber sequence $B M \to BN \to B N/M$ of spaces with $G/N$-action.

\begin{prp} \label{prp:ResidualAction} Suppose $X \in \Sp^{G}_{\Borel{N}}$ and also write $X$ for its restriction to $\Sp^G_{\Borel{M}}$. Then $X^{t[\psi'']}$ canonically acquires a `residual action' by lifting to an object in $\Sp^{G}_{\Borel{N/M}}$, and we have a fiber sequence of $G/N$-spectra
\[ (X_{h[\psi'']})^{t[\psi']} \to X^{t[\psi]} \to (X^{t[\psi'']})^{h[\psi']} \]
that restricts to a fiber sequence of Borel $G/N$-spectra
\[ (X_{h M})^{t(N/M)} \to X^{tN} \to (X^{t M})^{h(N/M)}. \]
\end{prp}
\begin{proof} We apply \cite[Rmk.~4.2.3]{hopkins2013ambidexterity} to the pullback square above to deduce the first assertion. For the second assertion, we apply \cite[Rmk.~4.2.4]{hopkins2013ambidexterity} to the factorization of $\rho_N$ as $\rho_{N/M} \circ \rho_M$ to obtain a commutative diagram of $G/N$-spectra
\[ \begin{tikzcd}[row sep=6ex, column sep=12ex, text height=1.5ex, text depth=1ex]
X_{h[\psi]} \ar{r}{\Nm_{\rho_{N/M}} \circ (\rho_M)_!} \ar{d} & (X_{h[\psi'']})^{h[\psi']} \ar{r}{(\rho_{N/M})_{\ast} \circ \Nm_{\rho_M}} \ar{d} & X^{h[\psi]} \ar{d} \\
0 \ar{r} & (X_{h[\psi'']})^{t[\psi']} \ar{r} \ar{d} & X^{t[\psi]} \ar{d} \\
& 0 \ar{r} & (X^{t[\psi'']})^{h[\psi']}
\end{tikzcd} \]
in which every rectangle is a homotopy pushout (using the two-out-of-three property of homotopy pushouts to show this for the upper righthand square and then the lower righthand square). The final assertion follows from Rmk.~\ref{rem:ParamTateCompatibleRestriction}.
\end{proof}

\section{Two theories of real \texorpdfstring{$p$}{p}-cyclotomic spectra}

In this section, we define $\infty$-categories $\RCycSp_p$ and $\RCycSp^{\mr{gen}}_p$ of \emph{Borel} and \emph{genuine} real $p$-cyclotomic spectra (Def.~\ref{dfn:RealCycSp} and Def.~\ref{dfn:GenRealCycSp}). For each $\infty$-category, we define the corresponding theories $\TCR(-,p)$ and $\TCR^{\mr{gen}}(-,p)$ of $p$-typical real topological cyclic homology (Def.~\ref{dfn:newTCR} and Def.~\ref{dfn:ClassicalTCR}) as functors to $\Sp^{C_2}$ right adjoint to functors that endow $C_2$-spectra with trivial real $p$-cyclotomic structure (Constr.~\ref{cnstr:trivialFunctor} and Constr.~\ref{cnstr:GenuineTrivialFunctor}), show $\TCR(-,p)$ and $\TCR^{\mr{gen}}(-,p)$ are $C_2$-corepresentable (Def.~\ref{dfn:Gcorepresentable}, Prop.~\ref{prp:C2representabilityOfTC}, and Prop.~\ref{prp:classicalTCRcorepresentable}), and thereby deduce fiber sequence formulas (Prop.~\ref{prp:TCRfiberSequence} and Prop.~\ref{prp:fiberSequenceGenuineRealCyc}) that in the case of $\TCR^{\mr{gen}}(-,p)$, recovers the more standard definition in terms of maps $R$ and $F$. To begin our study, we need to fix a few conventions regarding the dihedral groups and dihedral spectra.

\begin{setup} \label{setup:Dihedral} Let $O(2)$ denote the group of $2 \times 2$ orthogonal matrices, and regard the circle group $S^1$ as the subgroup $SO(2) \subset O(2)$. We fix, once and for all, a splitting of the determinant $\det: O(2) \to C_2 \cong \{ \pm 1 \}$ by choosing $\sigma \in O(2)$ to be the $\det =-1$ matrix given by
\[
   \sigma =
  \left[ {\begin{array}{cc}
   0 & 1 \\
   1 & 0 \\
  \end{array} } \right].
\]
This exhibits $O(2)$ as the semidirect product $S^1 \rtimes C_2$ for $C_2 = \angs{\sigma} \subset O(2)$, where $C_2$ acts on $S^1$ by complex conjugation, i.e., inversion. For $0 \leq n \leq \infty$, let $\mu_{p^n} \subset S^1$ be the subgroup of $p^n$th roots of unity, and let $D_{2p^{n}} \subset O(2)$ be the subgroup $\mu_{p^n} \rtimes C_2$. For $n<\infty$, we let $\{ x_n \}$ denote a compatible system of generators for $\mu_{p^n}$ (so $x_n = x_{n+1}^p$ for all $n \geq 0$), and we also set $x = x_n$ if there is no ambiguity about the ambient group. When considering restriction functors $\Sp^{D_{2p^n}} \to \Sp^{D_{2p^m}}$ for $m \leq n$, we always choose restriction to be with respect to the inclusion $D_{2p^m} \subset D_{2p^n}$ induced by $\mu_{p^m} \subset \mu_{p^n}$. Then we define
\[ \Sp^{D_{2p^{\infty}}} = \lim_n \Sp^{D_{2p^n}} \]
to be the inverse limit taken along these restriction functors. Since the restriction functors are symmetric monoidal and colimit preserving, we may take the inverse limit in $\CAlg(\Pr^{L,\st})$, and $\Sp^{D_{2p^{\infty}}}$ is then a stable presentable symmetric monoidal $\infty$-category. 

 We also define the $C_2$-space $B^t_{C_2} \mu_{p^{\infty}}$ to be the full subcategory of $\sO_{D_{2p^{\infty}}}^{\op}$ on the $\mu_{p^{\infty}}$-free orbits.\footnote{For an infinite group $G$ like $D_{2p^{\infty}}$, we let $\sO_{G}$ be the category of non-empty transitive $G$-sets, which need not be finite.} Note that an orbit $D_{2p^n}/H$ is $\mu_{p^n}$-free if and only if $H=1$ or $H = \angs{\sigma z}$ for some $z \in \mu_{p^n}$, where $H$ then has order $2$ since $(\sigma z)^2 = z^{-1} z = 1$. Indeed, supposing $H \neq 1$, since $H \cap \mu_{p^n} = 1$, if $\sigma z$ and $\sigma z'$ are two elements in $H$, then we must have $(\sigma z)(\sigma z') = z^{-1} z' = 1$, so $z = z'$. It follows that with respect to the induction functors $\sO_{D_{2p^m}} \to \sO_{D_{2p^n}}$ induced by the above inclusions, we have
\[ \colim_n B^t_{C_2} \mu_{p^n} \xto{\simeq} B^t_{C_2} \mu_{p^{\infty}}    \]
as a filtered colimit of $C_2$-spaces. Therefore, we obtain an equivalence
\[ \Fun_{C_2}(B^t_{C_2} \mu_{p^{\infty}}, \ul{\Sp}^{C_2}) \xto{\simeq} \lim_n \Fun_{C_2}(B^t_{C_2} \mu_{p^n}, \ul{\Sp}^{C_2}).  \]
Note also that the restriction functors are colimit preserving and symmetric monoidal with respect to the pointwise monoidal structure, so the equivalence may be taken in $\CAlg(\Pr^{L,\st})$.

Next, for all $0< m \leq n < \infty$, by Rmk.~\ref{rem:MonoidalIdentificationOfBorelSpectra} and Cor.~\ref{cor:BorelCompatibilityWithRestriction} we have a strict morphism of stable monoidal recollements
\[ \begin{tikzcd}[row sep=4ex, column sep=8ex, text height=1.5ex, text depth=0.5ex]
 \Fun_{C_2}(B^t_{C_2} \mu_{p^n}, \ul{\Sp}^{C_2}) \ar{d}{\res} & \Sp^{D_{2p^n}} \ar{d}{\res} \ar{l}[swap]{\sU_b[\mu_{p^n}]} \ar{r}{\Phi^{\mu_p}} &  \Sp^{D_{2p^{n-1}}} \ar{d}{\res} \\
\Fun_{C_2}(B^t_{C_2} \mu_{p^m}, \ul{\Sp}^{C_2}) & \Sp^{D_{2p^{m}}} \ar{l}[swap]{\sU_b[\mu_{p^{m}}]} \ar{r}{\Phi^{\mu_p}} &  \Sp^{D_{2p^{m-1}}}.
\end{tikzcd} \]
By Cor.~\ref{cor:LimitsOfRecollements}, passage to inverse limits defines a stable monoidal recollement
\[ \begin{tikzcd}[row sep=4ex, column sep=8ex, text height=1.5ex, text depth=0.5ex]
\Fun_{C_2}(B^t_{C_2} \mu_{p^{\infty}}, \ul{\Sp}^{C_2}) \ar[shift right=1,right hook->]{r}[swap]{j_{\ast} = \sF^{\vee}_b } & \Sp^{D_{2p^{\infty}}} \ar[shift right=2]{l}[swap]{j^{\ast} = \sU_b} \ar[shift left=2]{r}{i^{\ast} = \Phi^{\mu_p}} & \Sp^{D_{2p^{\infty}}} \ar[shift left=1,left hook->]{l}{i_{\ast}}
\end{tikzcd} \]
where we implicitly use the isomorphism $D_{2p^{\infty}}/\mu_p \cong D_{2p^{\infty}}$ induced by the $p$th power map for $\mu_{p^{\infty}}/\mu_p \cong \mu_{p^{\infty}}$ to regard $\Phi^{\mu_p}$ as an endofunctor. By also using the compatibility of restriction with categorical fixed points, we obtain the lax monoidal endofunctor $\Psi^{\mu_p}$ of $\Sp^{D_{2p^{\infty}}}$, and we retain the relation $\Psi^{\mu_p} \circ i_{\ast} \simeq \id$. Now consider the fiber sequence of functors
\[ j^{\ast} \Psi^{\mu_p} j_! \to j^{\ast} \Psi^{\mu_p} j_{\ast} \to j^{\ast} \Psi^{\mu_p} i_{\ast} i^{\ast} j_{\ast} \simeq j^{\ast} \Phi^{\mu_p} j_{\ast}. \]
By the same argument as in Prop.~\ref{prp:EquivalentTateConstructions}, this fiber sequence is equivalent to
\[ (-)_{h_{C_2} \mu_p} \xto{\Nm} (-)^{h_{C_2} \mu_p} \to (-)^{t_{C_2} \mu_p} \]
where the parametrized norm map is that associated to the weakly $1$-ambidextrous morphism $$B^t_{C_2} \mu_{p^{\infty}} \to B^t_{C_2} (\mu_{p^{\infty}} / \mu_p) \simeq B^t_{C_2} \mu_{p^{\infty}}$$
with fiber $B^t_{C_2} \mu_p$ (see \ref{cnv:basepoint} for our basepoint convention). We use this identification to endow the natural transformation $(-)^{h_{C_2} \mu_p} \to (-)^{t_{C_2} \mu_p}$ with the structure of a lax monoidal functor.
\end{setup}

\begin{cvn} \label{cnv:basepoint} We will regard $B^t_{C_2} \mu_{p^n}$ as a based $C_2$-space via the functor $\sO_{C_2}^{\op} = B^t_{C_2} (1) \to B^t_{C_2} \mu_{p^n}$ induced by $\angs{\sigma} \subset D_{2p^n}$. We then say that for an object $X \in \Fun_{C_2}(B^t_{C_2} \mu_{p^n}, \ul{\Sp}^{C_2})$, evaluation on the $C_2$-basepoint yields the \emph{underlying $C_2$-spectrum} of $X$, and further restriction via $\res^{C_2}: \Sp^{C_2} \to \Sp$ yields the \emph{underlying spectrum} of $X$. Note that if $X = \sU_b[\mu_{p^n}](Y)$, then its underlying $C_2$-spectrum is $\res^{D_{2p^n}}_{\angs{\sigma}} (Y)$.
\end{cvn}

\begin{rem} \label{rem:DihedralBasepoints} Note that for $y,z \in D_{2p^n}$, $y^{-1} (\sigma z) y = \sigma y^2 z$. Therefore, for $p$ odd and all $0 \leq n \leq \infty$, any two subgroups $\angs{\sigma z}$ and $\angs{\sigma z'}$ of $D_{2p^n}$ are conjugate. In contrast, for $p=2$ and $1<n<\infty$, there are two conjugacy classes of order $2$ subgroups $H$ with $H \cap \mu_{p^n} = 1$, with representatives $\angs{\sigma}$ and $\angs{\sigma x}$. However, for $p=2$ and $n=\infty$, we again have a single conjugacy class since we can take square roots for $z \in \mu_{2^{\infty}}$.

It follows that $B^t_{C_2} \mu_{p^n}$ is a connected $C_2$-space for $p$ odd and for $p=2$, $n=0,1,\infty$, but its fiber over $C_2/C_2$ splits into two components when $p=2$, $1<n<\infty$.
\end{rem}

We also will employ a more concise notation for the (lax) equalizer of two endofunctors.

\begin{dfn}[{\cite[Def.~II.1.4]{NS18}}] Suppose $F$ and $G$ are endofunctors of an $\infty$-category $C$. Define the \emph{lax equalizer} of $F$ and $G$ to be the pullback
\[ \begin{tikzcd}[row sep=4ex, column sep=6ex, text height=1.5ex, text depth=0.5ex]
\LEq_{F:G}(C) \ar{r} \ar{d} & \sO(C) \ar{d}{(\ev_0,\ev_1)} \\
C \ar{r}{(F,G)} & C \times C.
\end{tikzcd} \]
Define the \emph{equalizer} $\Eq_{F:G}(C) \subset \LEq_{F:G}(C)$ to be the full subcategory on objects $[x,F(x) \xto{\phi} G(x)]$ where $\phi$ is an equivalence.
\end{dfn}

\subsection{Borel real \texorpdfstring{$p$}{p}-cyclotomic spectra}

\begin{dfn} \label{dfn:RealCycSp} A \emph{real $p$-cyclotomic spectrum} is a $C_2$-spectrum $X$ with a twisted $\mu_{p^{\infty}}$-action, together with a twisted $\mu_{p^{\infty}}$-equivariant map $\varphi: X \to X^{t_{C_2} \mu_p}$. The $\infty$-category of \emph{real $p$-cyclotomic spectra} is then
$$\RCycSp_p = \LEq_{\id:t_{C_2} \mu_p}(\Fun_{C_2}(B^t_{C_2} \mu_{p^{\infty}}, \underline{\Sp}^{C_2})). $$
\end{dfn}

Such objects might be more accurately called \emph{Borel} real $p$-cyclotomic spectra, but we follow \cite{NS18} in our choice of terminology.

\begin{rem} We will sometimes abuse notation and refer to $X$ itself as the real $p$-cyclotomic spectrum, leaving the map $\varphi$ implicit.
\end{rem}

 We have the same conclusion as \cite[Cor.~II.1.7]{NS18} for $\RCycSp_p$, with the same proof.

\begin{prp} \label{prp:RCycSpPresentable} $\RCycSp_p$ is a presentable stable $\infty$-category, and the forgetful functor $$\RCycSp_p \to \Sp^{C_2}$$ is conservative, exact, and creates colimits and finite limits.
\end{prp}
\begin{proof} The endofunctor $t_{C_2} \mu_p$ is exact and accessible as the cofiber of functors that admit adjoints, or as the composite $\sU_b \circ \Phi^{\mu_p} \circ \sF^{\vee}_b$ as noted in \ref{setup:Dihedral}. By \cite[Prop.~II.1.5]{NS18}, $\RCycSp_p$ is stable and presentable and the forgetful functor $$\RCycSp_p \to \Fun_{C_2}(B^t_{C_2} \mu_{p^{\infty}}, \ul{\Sp}^{C_2})$$ is colimit-preserving and exact. It is also obviously conservative, and since the $C_2$-space $B^t_{C_2} \mu_{p^{\infty}}$ has connected fibers over $C_2/C_2$ and $C_2/1$, the further forgetful functor to $\Sp^{C_2}$ is also conservative, exact, and colimit-preserving. Finally, any conservative functor between presentable $\infty$-categories that preserves $K$-indexed (co)limits necessarily also creates $K$-indexed (co)limits.
\end{proof}

\begin{cnstr}[Symmetric monoidal structure on $\RCycSp_p$] Recall from \cite[IV.2.1]{NS18} that if $C$ is a symmetric monoidal $\infty$-category, $F$ is a symmetric monoidal functor, and $G$ is a lax monoidal functor, then $\LEq_{F:G}(C)$ acquires a `canonical' symmetric monoidal structure by forming the pullback of $\infty$-operads\footnote{This is analogous to how we defined the canonical symmetric monoidal structure on a recollement. Note again that the cotensor with $\Delta^1$ is taken relative to $\Fin_{\ast}$, and also that the righthand vertical map is induced by cotensoring with $\partial \Delta^1 \subset \Delta^1$.}
\[ \begin{tikzcd}[row sep=4ex, column sep=6ex, text height=1.5ex, text depth=0.5ex]
\LEq_{F:G}(C)^{\otimes} \ar{r} \ar{d} & (C^{\otimes})^{\Delta^1} \ar{d}{(\ev_0,\ev_1)} \\
C^{\otimes} \ar{r}{(F^{\otimes},G^{\otimes})} & C^{\otimes} \times_{\Fin_{\ast}} C^{\otimes}.
\end{tikzcd} \]
Let us then endow $\RCycSp_p$ with the symmetric monoidal structure given by taking $t_{C_2} \mu_p$ to have the lax monoidal structure as indicated in Setup~\ref{setup:Dihedral}.
\end{cnstr}

\begin{cnstr}[Trivial real $p$-cyclotomic structure] \label{cnstr:trivialFunctor} We construct an exact and colimit-preserving symmetric monoidal functor
\[ \mr{triv}_{\RR,p}: \Sp^{C_2} \to \RCycSp_p \]
that endows a $C_2$-spectrum with the structure of a real $p$-cyclotomic spectrum in a `trivial' way. Consider the maps of $C_2$-spaces
\[ \sO_{C_2}^{\op} \xto{\iota} B^t_{C_2} \mu_{p^{\infty}} \xto{\pi} B^t_{C_2} (\mu_{p^{\infty}}/\mu_p) \simeq B^t_{C_2} \mu_{p^{\infty}} \xto{p} \sO_{C_2}^{\op} \]
and the associated restriction functors
\[ \Sp^{C_2} \xto{p^{\ast}} \Fun_{C_2}(B^t_{C_2} \mu_{p^{\infty}}, \ul{\Sp}^{C_2}) \xto{\pi^{\ast}} \Fun_{C_2}(B^t_{C_2} \mu_{p^{\infty}}, \ul{\Sp}^{C_2}) \xto{\iota^{\ast}} \Sp^{C_2}. \]
Because $\pi^{\ast} p^{\ast} \simeq p^{\ast}$, by adjunction we obtain a natural transformation $p^{\ast} \to \pi_{\ast} p^{\ast} = (-)^{h_{C_2} \mu_p} \circ p^{\ast}$. Then let
\[ \lambda_{\RR,p}: p^{\ast} \to (-)^{h_{C_2} \mu_p} \circ p^{\ast} \to (-)^{t_{C_2} \mu_p} \circ p^{\ast} \]
be the composite natural transformation. Note that since $p^{\ast}$ and $\pi^{\ast}$ are symmetric monoidal, the adjoint natural transformation $p^{\ast} \to \pi_{\ast} p^{\ast}$ is canonically lax monoidal. With the lax monoidal structure on $(-)^{h_{C_2} \mu_p} \to (-)^{t_{C_2} \mu_p}$ as in \ref{setup:Dihedral}, $\lambda_{\RR,p}$ acquires the structure of a lax monoidal transformation. We then define $\mr{triv}_{\RR,p}$ to be the functor determined by the data of $p^{\ast}$ and $\lambda_{\RR,p}$.

Finally, note that since $\iota^{\ast} \circ p^{\ast} \simeq \id$, the composite of $\mr{triv}_{\RR,p}$ and the forgetful functor to $\Sp^{C_2}$ is also homotopic to the identity. By Prop.~\ref{prp:RCycSpPresentable}, we deduce that $\mr{triv}_{\RR,p}$ is exact and preserves colimits.
\end{cnstr}

\begin{dfn} \label{dfn:newTCR} The \emph{$p$-typical real topological cyclic homology} functor $$\TCR(-,p): \RCycSp_p \to \Sp^{C_2}$$ is the right adjoint to the trivial functor $\mr{triv}_{\RR,p}$ of Constr.~\ref{cnstr:trivialFunctor}.
\end{dfn}

We would like to say that $\TCR(-,p)$ is corepresentable by the unit in $\RCycSp_p$. However, because $\TCR(-,p)$ is valued in $C_2$-spectra, any such representability result must be understood in the $C_2$-sense. We now digress to give a general account of $G$-corepresentability.

\begin{cnstr}[$G$-mapping spectrum] \label{cnstr:MappingSpectrum} Suppose $C \to \sO^{\op}_G$ is $G$-$\infty$-category. In \cite[Def.~10.2]{Exp1}, Barwick et al. defined the $G$-mapping space $G$-functor
\[ \ul{\Map}_C: C^{\vop} \times_{\sO_G^{\op}} C \to \ul{\Spc}^G. \]
Informally, this sends an object $(x,y)$ over $G/H$ to the $H$-space determined by $\Map_{C_{G/K}}(\res^H_K x, \res^H_K y)$ varying over subgroups $K \leq H$. In \cite[Cor.~11.9]{Exp2} (taking $F$ there to be the identity on $C$), the second author showed that for any $x \in C_{G/G}$, the $G$-functors
\[ \ul{\Map}_C(x,-): C \to \ul{\Spc}^G, \quad \ul{\Map}_C(-,x): C^{\vop} \to \ul{\Spc}^G \]
preserve $G$-limits, with $G$-limits in $C^{\vop}$ computed as $G$-colimits in $C$. Now suppose $C$ is $G$-stable, let $\Fun^{\lex}_G(-,-)$ denote the full subcategory on those $G$-functors that preserve finite $G$-limits, and let $\ul{\Fun}^{\lex}_G(-,-)$ denote the full $G$-subcategory (i.e., sub-cocartesian fibration) of $\ul{\Fun}_G(-,-)$ that over the fiber $G/H$ is given by $\Fun_H^{\lex}(-,-)$. In \cite[Thm.~7.4]{Exp4}, Nardin proved\footnote{We obtain our formulation involving $G$-left-exact functors from his using that $C$ is $G$-stable.} that the $G$-functor $\Omega^{\infty}: \ul{\Sp}^G \to \ul{\Spc}^G$ induces equivalences
\begin{align*} \Omega^{\infty}_{\ast}: & \Fun^{\lex}_G(C,\ul{\Sp}^G) \xto{\simeq} \Fun^{\lex}_G(C,\ul{\Spc}^G), \\
\Omega^{\infty}_{\ast}: & \ul{\Fun}^{\lex}_G(C,\ul{\Sp}^G) \xto{\simeq} \ul{\Fun}^{\lex}_G(C,\ul{\Spc}^G).
\end{align*}
In particular, for fixed $x \in C_{G/G}$ (that selects a cocartesian section $x: \sO_G^{\op} \to C^{\vop}$), we may lift the $G$-mapping space $G$-functor $\ul{\Map}_C(x,-): C \to \ul{\Spc}^G$ to a $G$-mapping spectrum $G$-functor
\[ \ul{\map}_C(x,-): C \to \ul{\Sp}^G. \]
Moreover, as in the non-parametrized setting, the $G$-mapping space $G$-functor is adjoint to a $G$-functor $$C^{\vop} \to \ul{\Fun}^{\lex}_G(C, \ul{\Spc}^G) \subset \ul{\Fun}_G(C,\ul{\Spc}^G),$$ which we may lift to $\ul{\Fun}^{\lex}_G(C,\ul{\Sp}^G)$ via $\Omega^{\infty}_{\ast}$ and then adjoint over to obtain
\[ \ul{\map}_C(-,-): C^{\vop} \times_{\sO_G^{\op}} C \to \ul{\Sp}^G. \]
Note that $\ul{\map}_C(-,-)$ continues to transform $G$-colimits into $G$-limits in the first variable and to preserve $G$-limits in the second variable.

Finally, by restriction to the fiber over $G/H$, we obtain the $H$-mapping spectrum functor for $C_{G/H}$:
\[ \ul{\map}_{C}(-,-): C_{G/H}^{\op} \times C_{G/H} \to \Sp^H, \]
which lifts the $H$-mapping space functor for $C_{G/H}$
\[ \ul{\Map}_{C}(-,-): C_{G/H}^{\op} \times C_{G/H} \to \Spc^H \]
through $\Omega^{\infty}: \Sp^H \to \Spc^H$. Thus, for $x,y \in C_{G/H}$, we may compute $\Omega^{\infty}$ of the categorical fixed points as
\[ \Omega^{\infty}(\ul{\map}_{C}(x,y)^K) \simeq \ul{\Map}_{C}(x,y)(H/K). \]
In particular, taking the fiber over $G/G$, for all $G/H$ the diagram
\[ \begin{tikzcd}[row sep=4ex, column sep=10ex, text height=1.5ex, text depth=0.75ex]
C_{G/G}^{\op} \times C_{G/G} \ar{r}{\ul{\map}_{C}(-,-)} \ar{d}[swap]{((\res^G_H)^{\op}, \res^G_H)} & \Sp^G \ar{r}{\Omega^{\infty}} \ar{d}{\Psi^{H}} & \Spc^G \ar{d}{\ev_{G/H}} \\
C_{G/H}^{\op} \times C_{G/H} \ar{r}{\map_{C_{G/H}}(-,-)} & \Sp \ar{r}{\Omega^{\infty}} & \Spc
\end{tikzcd} \]
is homotopy commutative, where the top horizontal composite is $\ul{\Map}_{C}(-,-)$ and the bottom horizontal composite is $\Map_{C_{G/H}}(-,-)$.

\end{cnstr}

\begin{lem} \label{lem:adjunctionMappingSpacesEquivalence} Suppose $C$ and $D$ are $G$-$\infty$-categories and $\adjunct{L}{C}{D}{R}$ is a $G$-adjunction. Then we have natural equivalences in $\ul{\Spc}^G$
\[ \ul{\Map}_{C}(L x, y) \simeq \ul{\Map}_{D}(x, R y).  \]
If $C$ and $D$ are also $G$-stable, then we have natural equivalences in $\ul{\Sp}^G$
\[ \ul{\map}_{C}(L x, y) \simeq \ul{\map}_{D}(x, R y). \]
\end{lem}
\begin{proof} For a $G$-adjunction, we have a unit transformation $\eta: \id \to R L$ such that $\eta$ cover the identity in $\sO_G^{\op}$ \cite[Prop.~7.3.2.1(2)]{HA}. We then obtain the comparison map
\[ \ul{\Map}_D(L x,y) \xto{R_{\ast}} \ul{\Map}_C(R L x,R y) \xto{\eta^{\ast}} \ul{\Map}_C(R L x, y)  \]
in $\ul{\Spc}^G$. Because we may restrict to subgroups $H$ of $G$, without loss of generality it suffices to consider the case where $x \in C_{G/G}$ and $y \in D_{G/G}$, so the comparison map is a map of $G$-spaces. For an orbit $G/H$, let $x' = \res^G_H x \in C_{G/H}$ and $y' = \res^G_H y \in D_{G/H}$. Then on $G/H$ this map evaluates to
\[ \Map_{D_{G/H}}(L_{H} x', y') \xto{(R_H)_{\ast}} \Map_{C_{G/H}}(R_H L_H x',R_H y') \xto{\eta_H^{\ast}} \Map_{C_{G/H}}(x', R_H y'), \]
which implements the equivalence of mapping spaces for the adjunction $$\adjunct{L_H}{C_{G/H}}{D_{G/H}}{R_H}$$ between the fibers over $G/H$. The conclusion then follows. Finally, the subsequent claim about $\ul{\map}$ follows by reduction to $F$ in the same manner, where instead of using the jointly conservative family of evaluation functors at orbits $G/H$ to detect equivalences in $\Spc^G$, we use the categorical fixed points functors ranging over all subgroups $H \leq G$ to detect equivalences in $\Sp^G$.
\end{proof}

\begin{prp} \label{prp:GenericRepresentabilityByUnit} Suppose $C$ is a $G$-stable $G$-$\infty$-category and we have a $G$-adjunction $$\adjunct{L}{\ul{\Sp}^G}{C}{R}.$$  Then for all $c \in C_{G/H} \subset C$, if we let $S^0$ denote the unit of $\Sp^H$, then we have a natural equivalence
\[ R(c) \simeq \ul{\map}_{C}(L(S^0),c). \]
Consequently, we have equivalences of functors
\[ R_{G/H}(-) \simeq \ul{\map}_{C}(L(S^0),-): C_{G/H} \to \Sp^H. \]
Moreover, if we let $S^0: \sO^{\op}_G \to (\ul{\Sp}^G)^{\vop}$ also denote the cocartesian section that selects each $S^0$, then we have an equivalence of $G$-functors $$R(-) \simeq \ul{\map}_{C}(L(S^0),-): C \to \underline{\Sp}^{G}.$$
\end{prp}
\begin{proof} Without loss of generality we may suppose $H=G$, so we let $S^0$ be the unit for $\Sp^G$. Note that since the identity on $\ul{\Sp}^G$ lifts the $G$-mapping space $G$-functor $\ul{\Map}_G(S^0,-): \ul{\Sp}^G \to \ul{\Spc}^G$ through $\Omega^{\infty}$, we have that $\ul{\map}_{\ul{\Sp}^G}(S^0,-) \simeq \id$. Then for $c \in C_{G/G}$, using Lem.~\ref{lem:adjunctionMappingSpacesEquivalence} we have the equivalences
\[ R(c) \simeq \ul{\map}_{\ul{\Sp}^G}(S^0,R(c)) \simeq \ul{\map}_{C}(L(S^0),c). \]
The naturality of these equivalences in $c$ then imply the remaining statements.
\end{proof}

\begin{dfn} \label{dfn:Gcorepresentable} In the situation of Prop.~\ref{prp:GenericRepresentabilityByUnit}, we say that $R$ and $R_{G/G}$ are \emph{$G$-corepresentable} by $L(S^0)$, for the unit $S^0 \in \Sp^G$.
\end{dfn}

\begin{exm} \label{exm:LimitRepresentability} Let $p: K \to \sO_{G}^{\op}$ be a (small) $G$-$\infty$-category and consider the adjunction
\[ \adjunct{p^{\ast}}{\ul{\Sp}^G}{\ul{\Fun}_G(K,\ul{\Sp}^G)}{p_{\ast}} \]
where $p_{\ast}$ takes the $G$-limit. Then $p_{\ast}$ is $G$-corepresentable by the constant $G$-diagram at the unit. With respect to the pointwise symmetric monoidal structure on a $G$-functor $\infty$-category, this is the unit in $\Fun_G(K,\ul{\Sp}^G)$.
\end{exm}

We want to apply Prop.~\ref{prp:GenericRepresentabilityByUnit} to prove that real topological cyclic homology is $C_2$-corepresentable by $\mr{triv}_{\RR,p}(S^0)$, the $C_2$-sphere spectrum endowed with the trivial real $p$-cyclotomic structure, which by construction is also the unit in $\RCycSp_p$. For this, we need to refine $\RCycSp_p$ to a $C_2$-stable $C_2$-$\infty$-category and to refine $\TCR(-,p)$ to a $C_2$-right adjoint. We first extend the definition of lax equalizer to the parametrized setting.

\begin{dfn} Suppose $C \to S$ is a $S$-$\infty$-category and $F$,$G$ are $S$-endofunctors of $C$. Let $\sO_S(C) = S \times_{\sO(S)} \sO(C)$ be the $S$-$\infty$-category of arrows in $C$ \cite[Notn.~4.29]{Exp2} and define the \emph{$S$-lax equalizer} to be the pullback of $S$-$\infty$-categories
\[ \begin{tikzcd}[row sep=4ex, column sep=4ex, text height=1.5ex, text depth=0.25ex]
\ul{\LEq}_{F:G}(C) \ar{r} \ar{d} & \sO_S(C) \ar{d}{(\ev_0,\ev_1)} \\
C \ar{r}{(F,G)} & C \times_S C.
\end{tikzcd} \]
Note that for all $s \in S$, we have an isomorphism of simplicial sets $\ul{\LEq}_{F:G}(C)_s \cong \LEq_{F_s:G_s}(C_s)$, and for all $\alpha: s \to t$, the pushforward $\alpha_{\sharp}: \LEq_{F_s:G_s}(C_s) \to \LEq_{F_t:G_t}(C_t)$ sends $[x,F_s(x) \xto{\phi} G_s(x)]$ to $[\alpha_{\sharp} x, \alpha_{\sharp}(\phi)]$. Define the \emph{$S$-equalizer} $\ul{\Eq}_{F:G}(C) \subset \ul{\LEq}_{F:G}(C)$ to be the full $S$-subcategory that fiberwise is given by $\Eq_{F_s:G_s}(C_s)$.
\end{dfn}

We have the following parametrized analogue of \cite[Lem.~II.1.5(iii)]{NS18}.

\begin{lem} \label{lem:ParamLaxEqualizerStable} If $C$ is a $G$-$\infty$-category that admits finite $G$-limits and $F$,$G$ are $G$-left exact endofunctors, then $\ul{\LEq}_{F:G}(C)$ admits finite $G$-limits and the forgetful functor $\ul{\LEq}_{F:G}(C) \to C$ preserves finite $G$-limits. If $C$ is moreover $G$-stable, then $\ul{\LEq}_{F:G}(C)$ is $G$-stable.
\end{lem}
\begin{proof} We already know that the fibers of $\ul{\LEq}_{F:G}(C)$ admit finite limits or are stable and the pushforward functors are left-exact or exact, given our respective hypotheses. For the first statement, it thus suffices to show that for all maps of $G$-orbits $f:U \to V$, the restriction functor $f^{\ast}: \LEq_{F_V:G_V}(C_V) \to \LEq_{F_U:G_U}(C_U)$ admits a right adjoint $f_{\ast}$ computed by postcomposing by the right adjoint in $C$, and moreover that these adjunctions satisfy the Beck-Chevalley condition. Given the adjunction $\adjunct{f^{\ast}}{C_V}{C_U}{f_{\ast}}$, let $\cM \to \Delta^1$ be the bicartesian fibration that encodes this adjunction. Then since $F$ and $G$ commute with both $f^{\ast}$ and $f_{\ast}$, we obtain induced endofunctors $F_f$ and $G_f$ of $\cM$ over $\Delta^1$ that preserve both cocartesian and cartesian edges and restrict to $F_U$, $F_V$ and $G_U$, $G_V$ on the fibers. Therefore, the $\Delta^1$-lax equalizer $\ul{\LEq}_{F_f:G_f}(\cM) \to \Delta^1$ is again a bicartesian fibration that encodes the adjunction between $\LEq_{F_V:G_V}(C_V)$ and $\LEq_{F_U:G_U}(C_U)$, so $f^{\ast}$ as a functor on lax equalizers admits a right adjoint computed by postcomposition by $f_{\ast}: C_U \to C_V$. For the Beck-Chevalley condition, suppose a pullback square of finite $G$-sets
\[ \begin{tikzcd}[row sep=4ex, column sep=4ex, text height=1.5ex, text depth=0.25ex]
U \times_V W \ar{r}{f} \ar{d}{g} & W \ar{d}{g} \\
U \ar{r}{f} & V
\end{tikzcd} \]
where without loss of generality we may suppose $U,V,W$ are $G$-orbits. We need to show the natural transformations
\[ (\eta: f^{\ast} g_{\ast} \Rightarrow g_{\ast} f^{\ast}): \LEq_{F_W:G_W}(C_W) \to \LEq_{F_U:G_U}(C_U) \]
are equivalences. However, since the forgetful functor $\LEq_{F_U:G_U}(C_U) \to C_U$ detects equivalences, this follows from the Beck-Chevalley conditions assumed on $C$ itself. For the second statement, ambidexterity for the adjunctions $f^{\ast} \dashv f_{\ast}$ on $C$ promotes to the same for $\ul{\LEq}_{F:G}(C)$, using the same methods.
\end{proof}

The restriction functors $\Sp^{D_{2p^n}} \to \Sp^{D_{2p^m}}$ extend to $C_2$-functors $$\sO_{C_2}^{\op} \times_{\sO_{D_{2p^n}}^{\op}} \Sp^{D_{2p^n}} \to \sO_{C_2}^{\op} \times_{\sO_{D_{2p^m}}^{\op}} \Sp^{D_{2p^m}}$$
given on the fiber $C_2/1$ by the restriction functors $\Sp^{\mu_{2p^n}} \to \Sp^{\mu_{2p^m}}$. Taking the inverse limit, we obtain a $C_2$-stable $C_2$-$\infty$-category $\Sp^{D_{2p^{\infty}}}_{C_2}$. By taking inverse limits of the dihedral $C_2$-stable $C_2$-recollements of Exm.~\ref{exm:DihedralRecollement} along the restriction $C_2$-functors, we obtain a $C_2$-stable $C_2$-recollement
\[ \begin{tikzcd}[row sep=4ex, column sep=8ex, text height=1.5ex, text depth=0.5ex]
\ul{\Fun}_{C_2}(B^t_{C_2} \mu_{p^{\infty}}, \ul{\Sp}^{C_2}) \ar[shift right=1,right hook->]{r}[swap]{j_{\ast}} & \Sp^{D_{2p^{\infty}}}_{C_2} \ar[shift right=2]{l}[swap]{j^{\ast}} \ar[shift left=2]{r}{\Phi^{\mu_p}} & \Sp^{D_{2p^{\infty}}}_{C_2} \ar[shift left=1,left hook->]{l}{i_{\ast}}
\end{tikzcd} \]
that over $C_2/C_2$ restricts to the recollement in \ref{setup:Dihedral} and whose fiber over $C_2/1$ is the recollement
\[ \begin{tikzcd}[row sep=4ex, column sep=8ex, text height=1.5ex, text depth=0.5ex]
\Fun(B \mu_{p^{\infty}}, \Sp) \ar[shift right=1,right hook->]{r}[swap]{j_{\ast}} & \Sp^{\mu_{p^{\infty}}} \ar[shift right=2]{l}[swap]{j^{\ast}} \ar[shift left=2]{r}{\Phi^{\mu_p}} & \Sp^{\mu_{p^{\infty}}} \ar[shift left=1,left hook->]{l}{i_{\ast}}.
\end{tikzcd} \]

Observe that the $C_2$-exact $C_2$-endofunctor $j^{\ast} \Phi^{\mu_p} j_{\ast}$ restricts over the fiber $C_2/C_2$ to $t_{C_2} \mu_p$ and over the fiber $C_2/1$ to $t^{\mu_p}$. We let $\ul{t_{C_2} \mu_p} = j^{\ast} \Phi^{\mu_p} j_{\ast}$.

\begin{dfn} \label{dfn:C2CategoryOfCyclotomicSpectra} The \emph{$C_2$-$\infty$-category of real $p$-cyclotomic spectra} is
\[ \ul{\RCycSp}_p = \ul{\LEq}_{\id:t_{C_2} \mu_p}(\ul{\Fun}_{C_2}(B^t_{C_2} \mu_{p^{\infty}}, \ul{\Sp}^{C_2})). \]
\end{dfn}

Observe that the fiber of $\ul{\RCycSp}_p$ over $C_2/C_2$ is $\RCycSp_p$, and the fiber over $C_2/1$ is the $\infty$-category of $p$-cyclotomic spectra $\CycSp_p$ as defined in \cite[Def.~II.1.6(ii)]{NS18}. Also, by repeating the construction in \ref{cnstr:trivialFunctor} with $C_2$-$\infty$-categories, we obtain a $C_2$-functor
\[ \ul{\mr{triv}}_{\RR,p}: \ul{\Sp}^{C_2} \to \ul{\RCycSp}_p \]
that restricts over $C_2/C_2$ to $\mr{triv}_{\RR,p}$ and over $C_2/1$ to the trivial functor $\mr{triv}_p: \Sp \to \CycSp_p$ whose right adjoint is by definition $p$-typical topological cyclic homology $\TC(-,p)$ (see \cite[Prop.~IV.4.14]{NS18} for the integral version of this adjunction). By the dual of \cite[Prop.~7.3.2.6]{HA}, the fiberwise right adjoints refine to the structure of a relative right adjoint $$\ul{\TCR}(-,p): \ul{\RCycSp}_p \to  \ul{\Sp}^{C_2} $$ to $\ul{\mr{triv}}_{\RR,p}$. Moreover, because the composite $$\ul{\Sp}^{C_2} \to \ul{\RCycSp}_p \to \ul{\Fun}_{C_2}(B^t_{C_2} \mu_{p^{\infty}}, \ul{\Sp}^{C_2})$$ of $\ul{\mr{triv}}_{\RR,p}$ and the forgetful functor is $C_2$-left exact, by Lem.~\ref{lem:ParamLaxEqualizerStable}, $\ul{\mr{triv}}_{\RR,p}$ itself is $C_2$-left exact. Therefore, $\ul{\TCR}(-,p)$ preserves cocartesian edges, i.e., is a $C_2$-functor, and is thus $C_2$-right adjoint to $\ul{\mr{triv}}_{\RR,p}$.

\begin{prp} \label{prp:C2representabilityOfTC} $\ul{\TCR}(-,p)$ and $\TCR(-,p)$ are $C_2$-corepresentable by the unit.
\end{prp}
\begin{proof} Since $\ul{\RCycSp}_p$ is $C_2$-stable by Lem.~\ref{lem:ParamLaxEqualizerStable}, this follows immediately from Prop.~\ref{prp:GenericRepresentabilityByUnit} and the above discussion.
\end{proof}

We now apply Prop.~\ref{prp:C2representabilityOfTC} to derive an equalizer formula for $\TCR(-,p)$. We first explain how to compute $G$-mapping spaces and spectra in a limit of $G$-$\infty$-categories and then in the $G$-lax equalizer, analogous to \cite[II.1.5(ii)]{NS18}.

\begin{lem} \label{lem:LimitOfGMappingSpaces} Let $C_{\bullet}: K \to \Cat_{\infty}^G$ be a diagram of $G$-$\infty$-categories, and let $C = \lim C_{\bullet}$ be the limit. Let $x,y \in C_{G/G}$. Then the natural comparison map
\[ \ul{\Map}_{C}(x,y) \to \lim_{i \in K} \left(\ul{\Map}_{C_i}(x_i,y_i) \right)\]
is an equivalence of $G$-spaces. Furthermore, if $C_{\bullet}$ is a diagram of $G$-stable $G$-$\infty$-categories and $G$-exact $G$-functors, then the natural comparison map
\[ \ul{\map}_{C}(x,y) \to \lim_{i \in K} \left(\ul{\map}_{C_i}(x_i,y_i) \right) \]
is an equivalence of $G$-spectra.
\end{lem}
\begin{proof} By either evaluation at $G/H$ or taking $H$-categorical fixed points and using the commutative diagram at the end of Constr.~\ref{cnstr:MappingSpectrum}, we may reduce to the known non-parametrized statements.
\end{proof}

\begin{lem} \label{lem:EqualizerMappingSpaces} Suppose $C$ is a $G$-$\infty$-category and $F,F'$ are $G$-endofunctors of $C$. Let $X = [x,\phi: F(x) \to F'(x)]$ and $Y=[y,\psi: F(y) \to F'(y)]$ be two objects in $\LEq_{F_{G/G}:F'_{G/G}}(C_{G/G})$. Then we have a natural equivalence of $G$-spaces
\[ \ul{\Map}_{\ul{\LEq}_{F:G}(C)}(X,Y) \simeq \mr{eq} \left( 
\begin{tikzcd}[row sep=4ex, column sep=6ex, text height=1.5ex, text depth=0.25ex]
\ul{\Map}_{C}(x,y) \ar[shift left=1]{r}{\psi_{\ast} \circ F} \ar[shift right=1]{r}[swap]{\phi^{\ast} \circ F'} & \ul{\Map}_{C}(F(x),F'(y)).
\end{tikzcd} \right) \]
If $C$ is $G$-stable and $F,F'$ are $G$-exact, then we have a natural fiber sequence of $G$-spectra
\[ \ul{\map}_{\ul{\LEq}_{F:G}(C)}(X,Y) \to \ul{\map}_{C}(x,y) \xtolong{\psi_{\ast}  F - \phi^{\ast} F'}{2} \ul{\map}_{C}(F(x),F'(y)).  \]
\end{lem}
\begin{proof} In view of Lem.~\ref{lem:LimitOfGMappingSpaces}, the same arguments as in the proof of \cite[II.1.5.(ii)]{NS18} apply to produce the formulas.
\end{proof}

\begin{dfn} For $X \in \Fun_{C_2}(B^t_{C_2} \mu_{p^{\infty}}, \ul{\Sp}^{C_2})$, we define the \emph{canonical} map
\[ \can_p: X^{h_{C_2} \mu_{p^{\infty}}} \simeq (X^{h_{C_2} \mu_p})^{h_{C_2} \mu_{p^{\infty}}} \to (X^{t_{C_2} \mu_p})^{h_{C_2} \mu_{p^{\infty}}} \]
where for the first equivalence we use that $B^t_{C_2} (\mu_{p^{\infty}}/\mu_p) \simeq B^t_{C_2} \mu_{p^{\infty}}$ as before.
\end{dfn}

\begin{prp} \label{prp:TCRfiberSequence} Let $[X,\varphi:X \to X^{t_{C_2} \mu_p}]$ be a real $p$-cyclotomic spectrum. Then we have a natural fiber sequence of $C_2$-spectra
\[ \TCR(X,p) \to X^{h_{C_2} \mu_{p^\infty}} \xtolong{\varphi^{h_{C_2} \mu_{p^{\infty}}} -\can_p}{2} (X^{t_{C_2} \mu_p})^{h_{C_2} \mu_{p^\infty}}. \]
\end{prp}
\begin{proof} We mimic the proof of \cite[Prop.~II.1.9]{NS18}. Let $C = \ul{\Fun}_{C_2}(B^t_{C_2} \mu_{p^{\infty}}, \ul{\Sp}^{C_2})$, let $S^0 \in C$ be the unit (i.e., $S^0$ with trivial action), and note that by Exm.~\ref{exm:LimitRepresentability}, $$\ul{\map}_{C}(S^0,X) \simeq X^{h_{C_2} \mu_{p^{\infty}}}.$$
The claim then follows from Prop.~\ref{prp:C2representabilityOfTC} and Lem.~\ref{lem:EqualizerMappingSpaces}. In more detail, if we let $\lambda = \lambda_{\RR,p}: S^0 \to (S^0)^{t_{C_2} \mu_p}$ denote the trivial real $p$-cyclotomic structure map as in Constr.~\ref{cnstr:trivialFunctor}, then given a map $f: S^0 \to X$ in $C$, we have a commutative diagram (again in $C$)
\[ \begin{tikzcd}[row sep=4ex, column sep=8ex, text height=1.5ex, text depth=0.5ex]
S^0 \ar{r} \ar{rd}[swap]{\lambda} & (S^0)^{h_{C_2} \mu_p} \ar{r}{f^{h_{C_2} \mu_p}} \ar{d} & X^{h_{C_2} \mu_p } \ar{d} \\
& (S^0)^{t_{C_2} \mu_p} \ar{r}{f^{t_{C_2} \mu_p}} & X^{t_{C_2} \mu_p}.
\end{tikzcd} \]
Therefore, the composite map
\[ \ul{\map}_C(S^0, X) \xto{t_{C_2} \mu_p} \ul{\map}_C((S^0)^{t_{C_2} \mu_p}, X^{t_{C_2} \mu_p}) \xto{\lambda^{\ast}} \ul{\map}_C((S^0), X^{t_{C_2} \mu_p}) \]
is homotopic to $\can_p$. It is then clear that the desired fiber sequence is given by Lem.~\ref{lem:EqualizerMappingSpaces}.
\end{proof}

\subsection{Genuine real \texorpdfstring{$p$}{p}-cyclotomic spectra}
\label{section:genRCycSp}

\begin{dfn} \label{dfn:GenRealCycSp} A \emph{genuine real $p$-cyclotomic spectrum} is a $D_{2p^{\infty}}$-spectrum $X$, together with an equivalence $\Phi^{\mu_p} X \xto{\simeq} X$ in $\Sp^{D_{2p^{\infty}}}$. The $\infty$-category of \emph{genuine real $p$-cyclotomic spectra} is then
\[ \RCycSp^{\mathrm{gen}}_p = \Eq_{\Phi^{\mu_p}:\id}(\Sp^{D_{2p^{\infty}}}). \]
\end{dfn}

\begin{rem} \label{rem:EqualizerSwap} Let $F_i: C \to D$, $i=0,1$, be two functors and let us temporarily revert to the notation of \cite[Def.~II.1.4]{NS18} for (lax) equalizers. Let $J = (a \twoarrows b)$, so equalizers in an $\infty$-category $\cE$ are limits over diagrams $J \to \cE$ (\cite[\S 4.4.3]{HTT}). Let $q: J \to \Cat_{\infty}$ be the diagram that sends one arrow to $F_0$ and the other to $F_1$. Because
\[  \mr{eq}(\begin{tikzcd}[row sep=4ex, column sep=4ex, text height=1.5ex, text depth=0.25ex]
C \ar[shift left=1]{r}{F_0} \ar[shift right=1]{r}[swap]{F_1} & D
\end{tikzcd}) \coloneq \lim_J q \simeq \Eq(\begin{tikzcd}[row sep=4ex, column sep=4ex, text height=1.5ex, text depth=0.25ex]
C \ar[shift left=1]{r}{F_0} \ar[shift right=1]{r}[swap]{F_1} & D
\end{tikzcd}) \]
and the former expression is symmetric in $F_i$, we have an equivalence 
\[ \Eq(\begin{tikzcd}[row sep=4ex, column sep=4ex, text height=1.5ex, text depth=0.25ex]
C \ar[shift left=1]{r}{F_0} \ar[shift right=1]{r}[swap]{F_1} & D
\end{tikzcd}) \simeq \Eq(\begin{tikzcd}[row sep=4ex, column sep=4ex, text height=1.5ex, text depth=0.25ex]
C \ar[shift left=1]{r}{F_1} \ar[shift right=1]{r}[swap]{F_0} & D
\end{tikzcd}). \]
These equivalences are already implicit in \cite{NS18}, but in more detail, let $\cX \to J$ be the cocartesian fibration classified by $q$, so $\lim_J q \simeq \Sect(\cX) \coloneq \Fun^{\cocart}_{/J}(J,\cX)$. Let $\alpha_i: \Delta^1 \to J$, $i = 0,1$ be the two arrows in $J$, with $q \alpha_i$ selecting $F_i$, and let $\cX_{\alpha_i}$ be the pullback. Then $\ev_a: \Sect(\cX_{\alpha_i}) \to C$ is a trivial fibration, and choosing a section, the composite
\[ C \xto{\simeq} \Sect(\cX_{\alpha_i}) \xto{\ev_b} D \]
is homotopic to $F_i$. Therefore, if we let $\iota: D \to \sO(D)$ denote the identity section, we have a homotopy commutative diagram
\[ \begin{tikzcd}[row sep=4ex, column sep=4ex, text height=1.5ex, text depth=0.25ex]
\Sect(\cX) \ar{r}{\ev_b} \ar{d}{\ev_a} & D \ar{r}{\iota} & \sO(D) \ar{d}{(\ev_0,\ev_1)} \\
C \ar{rr}{(F,G)} & & D \times D
\end{tikzcd} \]
and an induced functor $\Sect(\cX) \to \LEq(\begin{tikzcd}[row sep=4ex, column sep=4ex, text height=1.5ex, text depth=0.25ex]
C \ar[shift left=1]{r}{F_0} \ar[shift right=1]{r}[swap]{F_1} & D
\end{tikzcd})$, which is fully faithful by comparing the formulas for mapping spaces in the limit over $J$ and in the lax equalizer. Because the essential image of $\iota$ consists of the equivalences in $\sO(D)$, it follows that the essential image of $\Sect(\cX)$ is $\Eq(\begin{tikzcd}[row sep=4ex, column sep=4ex, text height=1.5ex, text depth=0.25ex]
C \ar[shift left=1]{r}{F_0} \ar[shift right=1]{r}[swap]{F_1} & D
\end{tikzcd})$. Repeating the analysis with $F_0$ and $F_1$ exchanged, we obtain a zig-zag of equivalences
\[ \begin{tikzcd}[row sep=4ex, column sep=4ex, text height=1.5ex, text depth=0.25ex]
\Eq(C \ar[shift left=1]{r}{F_1} \ar[shift right=1]{r}[swap]{F_0} & D) & \Sect(\cX) \ar{r}{\simeq} \ar{l}[swap]{\simeq} & \Eq(C \ar[shift left=1]{r}{F_0} \ar[shift right=1]{r}[swap]{F_1} & D).
\end{tikzcd} \]

It follows that in defining genuine real $p$-cyclotomic spectra, the choice of direction of the equivalence $\Phi^{\mu_p} X \simeq X$ is immaterial. Thus, in lieu of Def.~\ref{dfn:GenRealCycSp} we could have let
\[ \RCycSp^{\mr{gen}}_p = \Eq_{\id:\Phi^{\mu_p}}(\Sp^{D_{2p^{\infty}}}). \]
This definition is more convenient when comparing to Def.~\ref{dfn:RealCycSp}, whereas Def.~\ref{dfn:GenRealCycSp} is more suitable for defining the structure maps $R$ (Constr.~\ref{dfn:ParamRStructureMaps}) that define the term $\TRR_p$ in the fiber sequence formula for $\TCR^{\mr{gen}}(-,p)$ of Prop.~\ref{prp:fiberSequenceGenuineRealCyc}.
\end{rem}

\begin{prp} $\RCycSp^{\mr{gen}}_p$ is a stable presentable symmetric monoidal $\infty$-category such that the forgetful functor to $\Sp^{D_{2p^{\infty}}}$ is conservative, creates colimits and finite limits, and is symmetric monoidal.
\end{prp}
\begin{proof} Because $\Phi^{\mu_p}$ is colimit-preserving and symmetric monoidal, we may lift the equalizer diagram to $\CAlg(\Pr^{L,\st})$. Since limits there are computed as for the underlying $\infty$-categories, the claim then follows, with conservativity proven as for the lax equalizer.
\end{proof}

\begin{cnstr} \label{cnstr:InflationToInfinity} Since for all $0 \leq m<n<\infty$, the diagram
\[ \begin{tikzcd}[row sep=4ex, column sep=6ex, text height=1.5ex, text depth=0.25ex]
\Sp^{C_2} \ar{r}{\inf^{\mu_{p^n}}} \ar{rd}[swap]{\inf^{\mu_{p^m}}} & \Sp^{D_{2p^n}} \ar{d}{\res}  \\
&  \Sp^{D_{2p^m}}
\end{tikzcd} \]
commutes in $\CAlg(\Pr^{L,\st})$, we may define an exact, colimit-preserving, and symmetric monoidal functor
\[ \inf{}^{\mu_{p^{\infty}}}: \Sp^{C_2} \to \Sp^{D_{2p^{\infty}}} \]
as the inverse limit of the functors $\inf^{\mu_{p^n}}$. Let $$\Psi^{\mu_{p^{\infty}}}: \Sp^{D_{2p^{\infty}}} \to \Sp^{C_2}$$ denote its right adjoint, and also write $$\Psi^{\mu_{p^{n}}}: \Sp^{D_{2p^{\infty}}} \to \Sp^{C_2}$$ for the composite of the restriction to $\Sp^{D_{2p^n}}$ and $\Psi^{\mu_{p^{n}}}$. Recall that for a diagram $\overline{C}_{\bullet}: K^{\lhd} \to \Pr^L$, if we write $C$ for its value on the cone point $v$ and $D = \lim_K (C_{\bullet})$, then for the induced adjunction
\[ \adjunct{L}{C}{D}{R} \]
we may compute $R$ in terms of the description of $D$ as an $\infty$-category of cocartesian sections as follows:
\begin{itemize} \item[($\ast$)] Let $\overline{\cX} \to K^{\lhd}$ be the presentable fibration classified by $\overline{C}_{\bullet}$, with restriction $\cX \to K$. Let $p: \cX \subset \overline{\cX} \to \overline{\cX}_v \simeq C$ be the \emph{cartesian} pushforward to the fiber over the initial object $v \in K^{\lhd}$. Then the functor
\[ p_{\ast}: D \simeq \Sect(\cX) \to C \]
obtained via postcomposition by $p$ is homotopic to $R$.
\end{itemize}

See \cite[\S 2]{BehrensShah} for a reference. To specialize to our situation, we note that $\Psi^{\mu_{p^n}}: \Sp^{D_{2p^n}} \to \Sp^{C_2}$ applied to the unit map $X \to \ind \res X$ for the restriction-induction adjunction $$\adjunct{\res}{\Sp^{D_{2p^n}}} {\Sp^{D_{2p^{n-1}}}} {\ind}$$ defines the map $F: \Psi^{\mu_{p^n}}(X) \to \Psi^{\mu_{p^{n-1}}}(X)$ of $C_2$-spectra that lifts the map $F: X^{\mu_{p^n}} \to X^{\mu_{p^{n-1}}}$ of spectra given by inclusion of fixed points.  Using the formula above, we conclude that 
\[ \Psi^{\mu_{p^{\infty}}}(X) \xto{\simeq} \lim_{n,F} \Psi^{\mu_{p^n}}(X). \]
\end{cnstr}

\begin{rem} For $n< \infty$, the functors $\Psi^{\mu_{p^n}}$ all commute with colimits, but the inverse limit $\Psi^{\mu_{p^{\infty}}}$ does not commute with colimits in general.
\end{rem}

\begin{cnstr} \label{cnstr:GenuineTrivialFunctor} Because the diagram
\[ \begin{tikzcd}[row sep=4ex, column sep=6ex, text height=1.5ex, text depth=0.25ex]
\Sp^{C_2} \ar{r}{\inf^{\mu_{p^n}}} \ar{rd}[swap]{\inf^{\mu_{p^{n-1}}}} & \Sp^{D_{2p^n}} \ar{d}{\Phi^{\mu_p}}  \\
&  \Sp^{D_{2p^{n-1}}}
\end{tikzcd} \]
commutes for all $0<n<\infty$, we have an equivalence $\Phi^{\mu_p} \inf^{\mu_{p^{\infty}}} \simeq \inf^{\mu_{p^{\infty}}}$ in $\CAlg(\Pr^{L,\st})$. Therefore, $\inf^{\mu_{p^{\infty}}}$ lifts to the equalizer of $\id$ and $\Phi^{\mu_p}$ to define an exact, colimit-preserving, and symmetric monoidal functor
\[ \mr{triv}^{\mr{gen}}_{\RR,p}: \Sp^{C_2} \to \RCycSp^{\mr{gen}}_p. \]

\end{cnstr}

\begin{dfn} \label{dfn:ClassicalTCR} The \emph{classical $p$-typical real topological cyclic homology} functor
$$\TCR^{\mr{gen}}(-,p): \RCycSp^{\mr{gen}}_p \to \Sp^{C_2} $$
is the right adjoint to $\mr{triv}^{\mr{gen}}_{\RR,p}$.
\end{dfn}

As with $\TCR(-,p)$, we can prove that $\TCR^{\mr{gen}}(-,p)$ is $C_2$-corepresentable by the unit and thereby deduce a fiber sequence formula for the functor. Recall that in the course of formulating Def.~\ref{dfn:C2CategoryOfCyclotomicSpectra}, we extended $\Sp^{D_{2p^{\infty}}}$ and $\Phi^{\mu_p}$ to a $C_2$-$\infty$-category $\Sp^{D_{2p^{\infty}}}_{C_2}$ with $\Phi^{\mu_p}$ as a $C_2$-endofunctor.

\begin{dfn} The \emph{$C_2$-$\infty$-category of genuine real $p$-cyclotomic spectra} is 
\[ \ul{\RCycSp}^{\mr{gen}}_p = \ul{\Eq}_{\Phi^{\mu_p}:\id}(\Sp^{D_{2p^{\infty}}}_{C_2}). \]
\end{dfn}

Note that the fiber of $\ul{\RCycSp}^{\mr{gen}}_p$ over $C_2/C_2$ is $\RCycSp^{\mr{gen}}_p$, and the fiber over $C_2/1$ is $\CycSp^{\mr{gen}}_p$ as defined in \cite[Def.~II.3.1]{NS18}. By repeating the constructions \ref{cnstr:InflationToInfinity} and \ref{cnstr:GenuineTrivialFunctor} in the $C_2$-sense, we construct $C_2$-functors
\begin{align*} \ul{\inf}^{\mu_{p^{\infty}}} &: \ul{\Sp}^{C_2} \to \Sp^{D_{2p^{\infty}}}_{C_2}, \\
\ul{\mr{triv}}^{\mr{gen}}_{\RR,p} &: \ul{\Sp}^{C_2} \to \ul{\RCycSp}^{\mr{gen}}_p, \end{align*}
which are $C_2$-left exact in view of the commutativity of the diagrams
\[ \begin{tikzcd}[row sep=4ex, column sep=6ex, text height=1.5ex, text depth=0.5ex]
\Sp^{C_2} \ar{r}{\inf^{\mu_{p^n}}} & \Sp^{D_{2p^n}} \\
\Sp \ar{r} \ar{u}{\ind^{C_2}} \ar{r}{\inf^{\mu_{p^n}}} & \Sp^{\mu_{p^n}} \ar{u}[swap]{\ind^{D_{2p^n}}_{\mu_{p^n}} }
\end{tikzcd} \]
for all $n \geq 0$. Therefore, we obtain $C_2$-right adjoints
\begin{align*} \ul{\Psi}^{\mu_{p^{\infty}}} &: \Sp^{D_{2p^{\infty}}}_{C_2} \to \ul{\Sp}^{C_2}, \\
\ul{\TCR}^{\mr{gen}}(-,p) &: \ul{\RCycSp}^{\mr{gen}}_p \to \ul{\Sp}^{C_2}. \end{align*}

\begin{prp} \label{prp:classicalTCRcorepresentable} $\ul{\TCR}^{\mr{gen}}(-,p)$ and $\TCR^{\mr{gen}}(-,p)$ are $C_2$-corepresentable by the unit.
\end{prp}
\begin{proof} Since $\Sp^{D_{2p^{\infty}}}_{C_2}$ is $C_2$-stable and $\Phi^{\mu_p}$ is $C_2$-exact, by Lem.~\ref{lem:ParamLaxEqualizerStable} we see that $\ul{\LEq}_{\Phi^{\mu_p}:\id}(\Sp^{D_{2p^{\infty}}}_{C_2})$ is $C_2$-stable. Note that  the $G$-equalizer as a full $G$-subcategory of the $G$-lax equalizer of $F$ and $F'$ is closed under finite $G$-limits if $F$ and $F'$ are $G$-left exact. Thus, we deduce that $\ul{\RCycSp}^{\mr{gen}}_p$ is also $C_2$-stable. The claim then follows by applying Prop.~\ref{prp:GenericRepresentabilityByUnit} to the $C_2$-adjunction $\ul{\mr{triv}}^{\mr{gen}}_{\RR,p} \dashv \ul{\TCR}^{\mr{gen}}(,-p)$ and using that $\mr{triv}^{\mr{gen}}_{\RR,p}$ is symmetric monoidal.
\end{proof}

\begin{cnstr}[Structure maps $R$] \label{dfn:ParamRStructureMaps} First note that we have natural transformations $$\Psi^{\mu_{p^n}} \to \Psi^{\mu_{p^{n-1}}} \Phi^{\mu_p}$$ defined via applying $\Psi^{\mu_{p^n}}$ to the unit map $\id \to i_{\ast} i^{\ast} \simeq i_{\ast} \Phi^{\mu_p}$ of the recollement on $\Sp^{D_{2p^n}}$ with closed part $\Sp^{D_{2p^{n-1}}}$, using again that $\Psi^{\mu_{p^n}} i_{\ast} \simeq \Psi^{\mu_{p^{n-1}}}$.

Let $[X, \alpha: \Phi^{\mu_p} X \xto{\simeq} X]$ be a genuine real $p$-cyclotomic spectrum. For all $0 < n < \infty$, we define natural maps of $C_2$-spectra
\[ R: \Psi^{\mu_{p^n}} X \to \Psi^{\mu_{p^{n-1}}} \Phi^{\mu_p} (X) \xto{\simeq} \Psi^{\mu_{p^{n-1}}} (X) \]
to be the composite of the above map and $\alpha$. Note that $R$ lifts the maps of spectra $$R: X^{\mu_{p^n}} \to (\Phi^{\mu_p} (X))^{\mu_{p^{n-1}}} \xto{\simeq} X^{\mu_{p^{n-1}}}$$ for the underlying $p$-cyclotomic spectrum (c.f. the discussion prior to \cite[Def.~II.4.4]{NS18}). Note also that the diagram
\[ \begin{tikzcd}[row sep=4ex, column sep=4ex, text height=1.5ex, text depth=0.25ex]
\Psi^{\mu_{p^n}} X \ar{r} \ar{d}{F} & \Psi^{\mu_{p^{n-1}}} \Phi^{\mu_p} (X) \ar{r}{\simeq} \ar{d}{F} & \Psi^{\mu_{p^{n-1}}} (X) \ar{d}{F} \\
\Psi^{\mu_{p^{n-1}}} X \ar{r} & \Psi^{\mu_{p^{n-2}}} \Phi^{\mu_p} (X) \ar{r}{\simeq} & \Psi^{\mu_{p^{n-2}}} (X)
\end{tikzcd} \]
commutes for all $1<n<\infty$, with the maps $F$ defined as in Constr.~\ref{cnstr:InflationToInfinity}. By taking the inverse limit along the maps $F$, the maps $R$ then induce a map
\[ R: \lim_{n,F}\Psi^{\mu_{p^n}}(X) \to \lim_{n,F}\Psi^{\mu_{p^n}}(X). \]
On the other hand, taking the inverse limit along the maps $R$, the maps $F$ induce a map
\[ F:  \lim_{n,R}\Psi^{\mu_{p^n}}(X) \to \lim_{n,R}\Psi^{\mu_{p^n}}(X).\]
\end{cnstr}

\begin{dfn} For a genuine real $p$-cyclotomic spectrum $[X, \Phi^{\mu_p} X \xto{\simeq} X]$, let 
\begin{align*} \TRR(X,p) &= \lim_{n,R} \Psi^{\mu_{p^n}}(X), \\
\TFR(X,p) &= \lim_{n,F} \Psi^{\mu_{p^n}}(X).
\end{align*}
\end{dfn}

\begin{prp} \label{prp:fiberSequenceGenuineRealCyc} Let $[X, \Phi^{\mu_p} X \xto{\simeq} X]$ be a genuine real $p$-cyclotomic spectrum. We have natural fiber sequences of $C_2$-spectra
\begin{align*} \TCR^{\mr{gen}}(X,p) &\to \TFR(X,p) \xtolong{\id - R}{1.5} \TFR(X,p), \\
\TCR^{\mr{gen}}(X,p) &\to \TRR(X,p) \xtolong{\id - F}{1.5} \TRR(X,p),
\end{align*}
and a natural equivalence of $C_2$-spectra
\[ \TCR^{\mr{gen}}(X,p) \simeq \lim_{n, F, R} \Psi^{\mu_{p^n}}(X) \coloneq \lim_{J_{\infty}} \Psi^{\mu_{p^n}}(X).  \]
\end{prp}

Before giving the proof, we define the category $J_{\infty}$ and prove a few necessary results about it.

\begin{dfn} \label{dfn:EqualizerInfinity} Let $J_{\infty}$ be the category freely generated by
\[ \begin{tikzcd}[row sep=4ex, column sep=4ex, text height=1.5ex, text depth=0.25ex]
\cdots \ar[shift left=1]{r}{\alpha_{n+1}} \ar[shift right=1]{r}[swap]{\beta_{n+1}} & n+1 \ar[shift left=1]{r}{\alpha_n} \ar[shift right=1]{r}[swap]{\beta_n} & n \ar[shift left=1]{r}{\alpha_{n-1}} \ar[shift right=1]{r}[swap]{\beta_{n-1}} & \cdots \ar[shift left=1]{r}{\alpha_1} \ar[shift right=1]{r}[swap]{\beta_1} & 1 \ar[shift left=1]{r}{\alpha_0} \ar[shift right=1]{r}[swap]{\beta_0} & 0 \:,
\end{tikzcd} \]
modulo the relation $\beta_n \circ \alpha_{n+1} = \alpha_n \circ \beta_{n+1}$ for all $n \geq 0$. More concretely, the objects of $J_{\infty}$ are non-negative integers, there are no morphisms $n \to n+k$ for $k>0$, there is only the identity $n \to n$, and morphisms $n+k \to n$, $k > 0$ are in bijection with non-empty sieves in $[k]$, where we attach to $S \subset [k]$ the composition $$\beta_{n} \cdots \beta_{n+l-1} \alpha_{n+l} \cdots \alpha_{n+k-1}$$ for $l = \max(S)$ (so if $l=0$, we have $\alpha_n \cdots \alpha_{n+k-1}$, and if $l=k$, we have $\beta_n \cdots \beta_{n+k-1}$).
\end{dfn}

\begin{rem} \label{rem:spineOfEqualizerInfinity} Let $\pi: J_{\infty} \to \ZZ_{\geq 0}^{\op}$ be the functor that sends $n$ to $n$, and $\alpha_n, \beta_n$ to $n+1 \to n$. For $n \geq m$, let $[n:m] \subset \ZZ_{\geq 0}^{\op}$ denote the full subcategory on integers $n \geq k \geq m$, and let $$J_{[n:m]} = J_{\infty} \times_{\ZZ_{\geq 0}^{\op}} [n:m].$$
We claim that the square
\[ \begin{tikzcd}[row sep=4ex, column sep=4ex, text height=1.5ex, text depth=0.25ex]
J_{[2:1]} \ar{r} \ar{d} & J_{[2:0]} \ar{d} \\
J_{[3:1]} \ar{r} & J_{[3:0]}
\end{tikzcd} \]
is a homotopy pushout square of $\infty$-categories. Indeed, for clarity write $a<b<c$ for the vertices of $\Delta^2$, and let $q: J_{[3:0]} \to \Delta^2$ be the functor that sends $3$ to $a$, $2,1$ to $b$, and $0$ to $c$, and maps in the obvious way. The claim amounts to showing that $q$ is a flat inner fibration (\cite[Def.~B.3.1]{HA}), for which we may use the criterion of \cite[Prop.~B.3.2]{HA}. Suppressing subscripts of morphisms in $J_{[3:0]}$ for clarity, we need to check that for morphisms $$\gamma \in \{ \beta^3, \beta^2 \alpha, \beta \alpha^2, \alpha^3 \} \in \Hom(3,0),$$
the resulting category $(J_{[2:1]})_{3/ /1}$ of factorizations of $\gamma$ through $J_{[2:1]}$ is weakly contractible. For $\gamma = \delta \circ \epsilon$ with the domain of $\delta$ equal to $i=1,2$, write $[\delta|\epsilon]_i$ for the object in $(J_{[2:1]})_{3/ /1}$. If $\gamma = \beta^3$, then $(J_{[2:1]})_{3/ /1}$ is given by $[\beta^2|\beta]_2 \to [\beta|\beta^2]_1$, so is weakly contractible, and likewise for $\gamma = \alpha^3$. If $\gamma = \beta \alpha^2$, then $(J_{[2:1]})_{3/ /1}$ is the category
\[ \begin{tikzcd}[row sep=4ex, column sep=4ex, text height=1.5ex, text depth=0.25ex]
\left[\beta \alpha | \alpha \right]_2 \ar{d} \ar{rd} & \left[ \alpha^2 | \beta\right]_2 \ar{d} \\
\left[ \beta| \alpha^2 \right]_1 & \left[ \alpha| \beta \alpha \right]_1,
\end{tikzcd} \]
using that $\beta \alpha = \alpha \beta$ and always writing maps as $\beta^i \alpha^j$. Thus, $(J_{[2:1]})_{3/ /1}$ is weakly contractible, and likewise for $\gamma = \beta \alpha^2$, proving the claim. Continuing this line of reasoning, we see that the cofibration
\[ \left( J_{[2:0]} \bigcup_{J_{[2:1]}} J_{[3:1]} \bigcup_{J_{[3:2]}} J_{[4:2]} \bigcup_{J_{[4:3]}} \cdots \right) \to J_{\infty} \]
is a categorical equivalence. Therefore, for an $\infty$-category $C$ and two diagrams $\ZZ_{\geq 0}^{\op} \to C$ written as
\begin{align*} \cdots \xto{F} X_2 \xto{F} X_1 \xto{F} X_0, \\
\cdots \xto{R} X_2 \xto{R} X_1 \xto{R} X_0,
\end{align*}
to extend this data to a diagram $J_{\infty} \to C$ that sends $\alpha$ to $F$ and $\beta$ to $R$, we only need to supply the data of commutative squares in $C$
\[ \begin{tikzcd}[row sep=4ex, column sep=4ex, text height=1.5ex, text depth=0.25ex]
X_{n+2} \ar{r}{F} \ar{d}{R} & X_{n+1} \ar{d}{R} \\
X_{n+1} \ar{r}{F} & X_n 
\end{tikzcd} \]
for all $n \geq 0$.
\end{rem}

\begin{lem} \label{lm:abstractEqualizerExchangeFormula} \begin{enumerate}[leftmargin=*] \item Let $p, p': J_{\infty} \to B \NN$ be the functors determined by sending $\alpha_n$ to $0$ and $\beta_n$ to $1$, resp. $\beta_n$ to $0$ and $\alpha_n$ to $1$. Then $p$ and $p'$ are cartesian fibrations classified by the functor $(B \NN)^{\op} \simeq B \NN \to \Cat \subset \Cat_{\infty}$ that sends the unique object to $\ZZ_{\geq 0}^{\op}$ and $1$ to $s^{\op}$, the (opposite of the) successor endofunctor (c.f. Notn.~\ref{ntn:nonnegativeintegers}).
\item Suppose $C$ is an $\infty$-category and $X_{\bullet}: J_{\infty} \to C$ is a diagram. Denote all maps $(X_{\bullet})(\alpha_n)$ by $F$ and all maps $(X_{\bullet})(\beta_n)$ by $R$. Then assuming the limits exist in $C$, we have equivalences
\[ \begin{tikzcd}[row sep=4ex, column sep=4ex, text height=1.5ex, text depth=0.25ex]
\mr{eq}(\lim_{n,R} X_n \ar[shift left=1]{r}{\id} \ar[shift right=1]{r}[swap]{F} & \lim_{n,R} X_n) \simeq \lim_{J_{\infty}} (X_{\bullet}) \simeq \mr{eq}(\lim_{n,F} X_n \ar[shift left=1]{r}{\id} \ar[shift right=1]{r}[swap]{R} & \lim_{n,F} X_n),
\end{tikzcd} \]
where we also write $F$ and $R$ for the induced maps on the limits.
\end{enumerate} 
\end{lem}
\begin{proof} (1): To show that $p$ is a cartesian fibration, it suffices to show that $\beta_n$ is a $p$-cartesian edge for all $n \geq 0$. For this, suppose given a map $f:m \to n$ in $J_{\infty}$ such that $p(f)$ factors as $\ast \xto{k} \ast \xto{1} \ast$, i.e., $p(f) \geq 1$. Then we must have $m \geq n+1$ and $f \neq \alpha_n \cdots \alpha_{m-1}$, so $f$ factors uniquely through $\beta_n$ and the claim is proven. The case of $p'$ is identical. Finally, the description of the resulting action of $\NN$ on the fiber $\ZZ_{\geq 0}^{\op}$ is clear in view of the commutative diagram
\[ \begin{tikzcd}[row sep=4ex, column sep=4ex, text height=1.5ex, text depth=0.25ex]
0 & 1 \ar{l}[swap]{\alpha_0} & 2 \ar{l}[swap]{\alpha_1} & 3 \ar{l}[swap]{\alpha_2} & \cdots \ar{l}[swap]{\alpha_3} \\
1 \ar{u}{\beta_0} & 2 \ar{l}{\alpha_1} \ar{u}{\beta_2} & 3 \ar{l}{\alpha_2} \ar{u}{\beta_2} & 4 \ar{l}{\alpha_3} \ar{u}{\beta_3} & \cdots \ar{l}{\alpha_4}
\end{tikzcd} \]
and similarly with the roles of $\alpha_{\bullet}$ and $\beta_{\bullet}$ exchanged.

(2): Factoring $J_{\infty} \to \ast$ through the cartesian fibration $p$ and using the transitivity of right Kan extensions, we get that 
$$\lim_{J_{\infty}} (X_{\bullet}) \simeq \lim_{B \NN} \lim_{n, F} X_n, $$
where $\NN$ acts on $\lim_{n, F} X_n$ via $R$. But the limit over $B\NN$ is computed also as the equalizer of $\id$ and $R$, so we deduce the equivalence 
\[ \begin{tikzcd}[row sep=4ex, column sep=4ex, text height=1.5ex, text depth=0.25ex]
\lim_{J_{\infty}} (X_{\bullet}) \simeq \mr{eq}(\lim_{n,F} X_n \ar[shift left=1]{r}{\id} \ar[shift right=1]{r}[swap]{R} & \lim_{n,F} X_n).
\end{tikzcd} \]
Doing the same with $p'$ shows the other equivalence.
\end{proof}

\begin{proof}[Proof of Prop.~\ref{prp:fiberSequenceGenuineRealCyc}] Let $\alpha: \Phi^{\mu_p} X \xto{\simeq} X$ denote the structure map. Let $C = \Sp^{D_{2p^{\infty}}}_{C_2}$. For $X \in \Sp^{D_{2p^{\infty}}}$, we have that $\Psi^{\mu_{p^{\infty}}}(X) \simeq \ul{\map}_{C}(S^0,X)$ as the $C_2$-right adjoint to $\ul{\inf}^{\mu_{p^{\infty}}}$, and $\Psi^{\mu_{p^{\infty}}}(X) \simeq \lim_{n,F} \Psi^{\mu_{p^n}}(X)$ as we saw in Constr.~\ref{cnstr:InflationToInfinity}. Also, the $G$-equalizer is a full $G$-subcategory of the $G$-lax equalizer, so $G$-mapping spaces and spectra may be computed as in Lem.~\ref{lem:EqualizerMappingSpaces} if we take $F = \Phi^{\mu_p}$, $F'=\id$. We thus obtain a fiber sequence with the objects as in the first fiber sequence in the statement, and it remains to identify the maps. Because $\Phi^{\mu_p}(S^0) \simeq S^0$, one of the maps in that fiber sequence is homotopic to $\id$. On the other hand, we claim that
\[ \Phi^{\mu_p}: \ul{\map}_C(S^0,X) \to \ul{\map}_C(S^0,\Phi^{\mu_p} X) \]
is homotopic to the map $$\lim_{n,F} \Psi^{\mu_{p^n}}(X) \to \lim_{n,F} \Psi^{\mu_{p^n}} \Phi^{\mu_p}(X)$$ induced by taking the limit of the natural transformations $\Psi^{\mu_{p^n}} X \to \Psi^{\mu_{p^{n-1}}} \Phi^{\mu_p} X$ defined in Constr.~\ref{dfn:ParamRStructureMaps}. Indeed, since the functor $\Phi^{\mu_p}$ is obtained as the inverse limit of functors $\Phi^{\mu_p}: \Sp^{D_{2p^n}} \to \Sp^{C_2}$, the map $\Phi^{\mu_p}$ of $C_2$-spectra is also obtained as the inverse limits of maps
\[\Phi^{\mu_p}: \ul{\map}_{C_n}(S^0,X_n) \to \ul{\map}_{C_{n-1}}(S^0,\Phi^{\mu_p} X_n) \]
where $C_n = \sO_{C_2}^{\op} \times_{\sO_{D_{2p^n}}^{\op}} \ul{\Sp}^{D_{2p^{n}}}$ and $X_n$ is the restriction of $X$ to $\Sp^{D_{2p^{n}}}$. But with respect to the $C_2$-adjunction $\Phi^{\mu_p} \dashv i_{\ast}$ and the resulting equivalence
\[ \ul{\map}_{C_{n-1}}(S^0,\Phi^{\mu_p} X_n) \simeq \ul{\map}_{C_{n}}(S^0, i_{\ast} \Phi^{\mu_p} X_n),  \]
we may identify this map as given by $\ul{\map}_{C_{n}}(S^0,-) \simeq \Psi^{\mu_{p^n}}$ on the unit for $X_n$, which is the map of Constr.~\ref{dfn:ParamRStructureMaps}. It follows that the composite
\[ \ul{\map}_C(S^0,X) \xto{\Phi^{\mu_p}} \ul{\map}_C(S^0,\Phi^{\mu_p} X) \xto{\alpha_{\ast}} \ul{\map}_C(S^0, X) \]
is homotopic to $R$, and we deduce the first fiber sequence.

Because the maps $F$ and $R$ commute, by Rmk.~\ref{rem:spineOfEqualizerInfinity} the $F$ and $R$ maps extend to define a diagram $J_{\infty} \to \Sp^{C_2}$. Then by Lem.~\ref{lm:abstractEqualizerExchangeFormula}, we deduce the last equivalence and second fiber sequence.
\end{proof}

\begin{rem} Although they allude to the corepresentability of $\TC^{\mr{gen}}$ in the introduction \cite[p.~207]{NS18}, Nikolaus and Scholze choose to \emph{define} $\TC^{\mr{gen}}(-,p)$ via the fiber sequence \cite[Def.~II.4.4]{NS18}
\[ \TC^{\mr{gen}}(X,p) \to \mr{TR}(X,p) \xtolong{\id - F}{1} \mr{TR}(X,p). \]
The $C_2$-corepresentability of $\ul{\TCR}^{\mr{gen}}(-,p)$, or simple repetition of the proof of Prop.~\ref{prp:classicalTCRcorepresentable}, immediately implies that $\TC^{\mr{gen}}(-,p)$ is corepresentable by the unit. Alternatively, one may deduce this from results of Blumberg-Mandell \cite{BM16} and the comparison \cite[Thm.~II.3.7]{NS18} as noted in \cite[Rmk.~II.6.10]{NS18}.
\end{rem}

\section{Comparison of the theories}
\label{section:ComparisonSection}

Let $[X, \alpha: \Phi^{\mu_p} X \xto{\simeq} X]$ be a genuine real $p$-cyclotomic spectrum. From Setup \ref{setup:Dihedral}, consider the recollement
\[ \begin{tikzcd}[row sep=4ex, column sep=8ex, text height=1.5ex, text depth=0.5ex]
\Fun_{C_2}(B^t_{C_2} \mu_{p^{\infty}}, \ul{\Sp}^{C_2}) \ar[shift right=1,right hook->]{r}[swap]{\sF^{\vee}_b } & \Sp^{D_{2p^{\infty}}} \ar[shift right=2]{l}[swap]{\sU_b} \ar[shift left=2]{r}{\Phi^{\mu_p}} & \Sp^{D_{2p^{\infty}}} \ar[shift left=1,left hook->]{l}{i_{\ast}}
\end{tikzcd} \]
and the morphism induced by the unit of $\sU_b \dashv \sF^{\vee}_b$
$$ \beta: \sU_b \Phi^{\mu_p} (X) \to  \sU_b \Phi^{\mu_p} \sF^{\vee}_b \sU_b (X) \simeq (\sU_b X)^{t_{C_2} \mu_p}. $$

Choosing an inverse $\alpha^{-1}$, let $\varphi = \beta \circ (\sU_b \alpha^{-1})$. Then $[\sU_b X, \varphi]$ is a real $p$-cyclotomic spectrum. More generally, the lax monoidal natural transformation
\[ \sU^b \Phi^{\mu_p} \to \sU_b \Phi^{\mu_p} \sF^{\vee}_b \sU_b \simeq t_{C_2} \mu_p \sU_b \]
defines a symmetric monoidal functor
\[ \Eq_{\id:\Phi^{\mu_p}}(\Sp^{D_{2p^{\infty}}}) \to \RCycSp_p \]
via the universal property of the lax equalizer (as explained in \cite[Prop.~II.3.2]{NS18}), and precomposition with the (symmetric monoidal) equivalence of Rmk.~\ref{rem:EqualizerSwap} defines a symmetric monoidal functor
\[ \sU_{\RR}: \RCycSp^{\mr{gen}}_p \to \RCycSp_p \]
that lifts the functor $\sU_b: \Sp^{D_{2p^{\infty}}} \to \Fun_{C_2}(B^t_{C_2} \mu_{p^{\infty}}, \ul{\Sp}^{C_2})$ through the functors that forget the structure maps.

\begin{dfn} Let $[X,\varphi]$ be a real $p$-cyclotomic spectrum. Then $[X,\varphi]$ is \emph{underlying bounded below} if the underlying spectrum of $X$ is bounded below.

Similarly, we say that a genuine real $p$-cyclotomic spectrum $[X,\alpha]$ is \emph{underlying bounded below} if the underlying spectrum of $X$ is bounded below.\footnote{This implies that the underlying spectra of $\Phi^{\mu_{p^n}}(X)$ are bounded below for all $n \geq 0$.}
\end{dfn}

Let us restate the main theorem of this paper from the introduction.

\begin{thm}[See Thm.~\ref{thm:MainTheoremRestated}] \label{thm:MainTheoremEquivalenceBddBelow} $\sU_{\RR}$ restricts to an equivalence on the full subcategories of underlying bounded below objects.
\end{thm}

\begin{cor} \label{cor:TCRFormulasEquivalent} Let $X$ be a genuine real $p$-cyclotomic spectrum that is bounded below. Then we have a canonical equivalence
\[ \TCR^{\mr{gen}}(X,p) \simeq \TCR(\sU_{\RR} X,p). \]
\end{cor}
\begin{proof} Because $\sU_{\RR}$ preserves the unit (which is bounded below), this follows immediately from Thm.~\ref{thm:MainTheoremEquivalenceBddBelow} and the $C_2$-corepresentability of $\TCR(-,p)$ and $\TCR^{\mr{gen}}(-,p)$ by the unit (Prop.~\ref{prp:C2representabilityOfTC} and Prop.~\ref{prp:classicalTCRcorepresentable}).
\end{proof}

As with the comparison theorem for $p$-cyclotomic spectra \cite[Thm.~II.6.3]{NS18}, the key computation that establishes Thm.~\ref{thm:MainTheoremEquivalenceBddBelow} is a dihedral extension of the Tate orbit lemma, to which we turn first. Using this, we indicate how the `decategorified' version of Thm.~\ref{thm:MainTheoremEquivalenceBddBelow} in the form of Cor.~\ref{cor:TCRFormulasEquivalent} follows by the same arguments as in \cite[\S II.4]{NS18}. We then prove Thm.~\ref{thm:MainTheoremEquivalenceBddBelow}, proceeding in two stages: first, we obtain a comparison result at `finite level' (Prop.~\ref{prp:EquivalenceOnBoundedBelowAtFiniteLevel}) as a formal consequence of the equivalence between $1$-generated and extendable objects (Thm.~\ref{thm:OneGenerationAndExtension}), and we then promote this to Thm.~\ref{thm:MainTheoremEquivalenceBddBelow} by executing a few more formal maneuvers. Finally, we apply Cor.~\ref{cor:TCRFormulasEquivalent} to compute $\TCR^{\mr{gen}}(-,p)$ of $\THR(H \ul{\FF_p})$ for $p$ an odd prime (Thm.~\ref{thm:TCROddPrimeComputation}).

\subsection{The dihedral Tate orbit lemma}

In \cite[\S I.2]{NS18}, Nikolaus and Scholze prove the \emph{Tate orbit lemma}: for a Borel $C_{p^2}$-spectrum $X$ that is bounded below, the spectrum $(X_{h C_p})^{t_{C_p}}$ vanishes \cite[Lem.~I.2.1]{NS18}. In this subsection, we give a dihedral refinement of the Tate orbit lemma (Lem.~\ref{lem:dihedralTOLEven} for $p=2$ and Lem.~\ref{lem:dihedralTOLOdd} for $p$ odd). As a corollary, we then deduce that $\TCR^{\mr{gen}}(-,p)$ is computed by the fiber sequence formula for $\TCR(-,p)$ on bounded below genuine real $p$-cyclotomic spectra (Cor.~\ref{cor:decategorifiedEasy}).

\begin{dfn}[{\cite[\S 6]{BarwickGlasmanShah}}] The \emph{homotopy $t$-structure} on $\Sp^G$ is the $t$-structure \cite[Def.~1.2.1.1]{HA} determined by the pair of full subcategories $\Sp^G_{\geq 0}$, $\Sp^G_{\leq 0}$ of $G$-spectra $X$ such that $X^H$ is connective, resp. coconnective for all subgroups $H \leq G$.

A $G$-spectrum $X$ is \emph{bounded below} if $X$ is bounded below in the homotopy $t$-structure on $\Sp^G$, i.e., for all subgroups $H \leq G$, $X^H$ is bounded below.
\end{dfn}

\begin{rem}[{\cite[Exm.~6.3]{BarwickGlasmanShah}}] The heart of the homotopy $t$-structure on $\Sp^G$ is the category of abelian group-valued Mackey functors on finite $G$-sets. In addition, the homotopy $t$-structure on $\Sp^G$ is accessible \cite[Def.~1.4.4.12]{HA} and left and right complete \cite[\S 1.2.1]{HA}.
\end{rem}

\begin{dfn} \label{dfn:sliceBoundedBelow} A $G$-spectrum $X$ is \emph{slice bounded below} if for all subgroups $H \leq G$, $X^{\phi H}$ is bounded below.
\end{dfn}

\begin{rem} By \cite[Thm.~A]{HillYarnall}, a $G$-spectrum $X$ is slice bounded below in the sense of Def.~\ref{dfn:sliceBoundedBelow} if and only if it is slice $n$-connective for some $n>-\infty$ in the sense of the slice filtration \cite[\S 4]{HHR}.
\end{rem}

When $G = C_{p^n}$, there is no distinction between bounded below and slice bounded below $G$-spectra.

\begin{lem} \label{lem:BddBelowEqualsSliceBddBelow} Suppose $X \in \Sp^{C_{p^n}}$. Then $X$ is bounded below if and only if $X$ is slice bounded below.
\end{lem}
\begin{proof} We proceed by induction on $n$. The base case $n=0$ is trivial. Let $n>0$ and suppose we have proven the lemma for $C_{p^{n-1}}$. Let $X \in \Sp^{C_{p^n}}$ and consider the recollement
\[ \begin{tikzcd}[row sep=4ex, column sep=6ex, text height=1.5ex, text depth=0.25ex]
\Fun(B C_{p^n},\Sp) \ar[shift right=1,right hook->]{r}[swap]{j_{\ast}} & \Sp^{C_{p^n}} \ar[shift right=2]{l}[swap]{j^{\ast}} \ar[shift left=2]{r}{i^{\ast} = \Phi^{C_p}} & \Sp^{C_{p^{n-1}}} \simeq \Sp \ar[shift left=1,left hook->]{l}{i_{\ast}}
\end{tikzcd} \]
from which we obtain the fiber sequence $(X^1)_{h C_{p^n}} \to X^{C_{p^n}} \to (\Phi^{C_p} X)^{C_{p^{n-1}}}$ as in \cite[Prop.~II.2.13]{NS18}. By the inductive hypothesis, we may suppose that both $X^{C_{p^k}}$ and $X^{\phi C_{p^k}}$ are bounded below for all $0 \leq k < n$. Then noting that $(\Phi^{C_p} X)^{\phi C_{p^k}} \simeq X^{\phi C_{p^{k+1}}}$, we deduce from the fiber sequence that $X^{C_{p^n}}$ is bounded below if and only if $X^{\phi C_{p^n}}$ is bounded below.
\end{proof}

Note that the restriction functors $\res^G_H: \Sp^G \to \Sp^H$ are $t$-exact with respect to the homotopy $t$-structures. Consequently, we can make the following definition.

\begin{dfn} Let $\underline{\Sp}^G_{\geq n}, \underline{\Sp}^G_{\leq m} \subset \underline{\Sp}^G$ be the full $G$-subcategories defined fiberwise over $G/H$ on objects $X \in \Sp^H_{\geq n}$, resp. $X \in \Sp^H_{\leq m}$.
\end{dfn}

\begin{lem} The inclusions $\underline{\Sp}^G_{\geq n} \subset \underline{\Sp}^G$, resp. $\underline{\Sp}^G_{\leq m} \subset \underline{\Sp}^G$ admit right $G$-adjoints $\tau_{\geq n}$, resp. left $G$-adjoints $\tau_{\leq m}$.
\end{lem}
\begin{proof} These adjunctions exist fiberwise, so we deduce both statements from the $t$-exactness of the restriction and induction functors using \cite[Prop.~7.3.2.6]{HA} and \cite[Prop.~7.2.3.11]{HA}.
\end{proof}

\begin{rem} For a $G$-$\infty$-category $K$, we have an induced `pointwise' $t$-structure on $\Fun_G(K, \ul{\Sp}^G)$ determined by $\Fun_G(K, \ul{\Sp}^G_{\geq 0})$ and $\Fun_G(K, \ul{\Sp}^G_{\leq 0})$.
\end{rem}

\begin{lem} \label{lem:TateConvergence} Let $K$ be a $G$-$\infty$-category and $f: K \to \underline{\Sp}^G$ a $G$-functor. Then the canonical maps
\[ \begin{tikzcd}[row sep=4ex, column sep=4ex, text height=1.5ex, text depth=0.25ex]
\colim^G_K f \ar{r} & \lim_{n} \colim^G_K \tau_{\leq n} f
\end{tikzcd}, \quad 
\begin{tikzcd}[row sep=4ex, column sep=4ex, text height=1.5ex, text depth=0.25ex]
\colim_{n} \lim^G_K \tau_{\geq -n} f \ar{r} & \lim^G_K f
\end{tikzcd} \]
are equivalences. Consequently, if $X: B_{G/N}^{\psi} N \to \ul{\Sp}^{G/N}$ is a $G/N$-spectrum with $\psi$-twisted $N$-action, then the canonical maps
\[ X^{t[\psi]} \to \lim_n (\tau_{\leq n} X)^{t[\psi]}, \quad \colim_n (\tau_{\geq -n} X)^{t[\psi]} \to X^{t[\psi]} \]
are equivalences
\end{lem}
\begin{proof} For the first equivalence, using the cofiber sequences $\tau_{>n} \to \id \to \tau_{\leq n}$, it suffices to show that $\lim_{n} \colim^G_K \tau_{> n} f \simeq 0$. But this follows by completeness of the homotopy $t$-structure on $\Sp^G$, since the inclusion $\underline{\Sp}^G_{>n} \subset \underline{\Sp}^G$ preserves $G$-colimits as a left $G$-adjoint \cite[Cor.~8.7]{Exp2}. The second equivalence is proven by a dual argument. The final two equivalences then follow from the first two in view of the defining fiber sequence
\[ X_{h[\psi]} \to X^{h[\psi]} \to X^{t[\psi]} \]
and the commutativity of parametrized orbits with colimits and parametrized fixed points with limits.
\end{proof}

For applying the next lemma, note that we have canonical lax monoidal natural transformations $(-)^{t G} \to (-)^{\tau G}$ and $(-)^{t G} \to ((-)^{t N})^{t(G/N)}$, defined via the universal property of the Verdier quotient \cite[\S1.3]{NS18}.

\begin{lem} \label{lem:TateNilpotent} Let $G$ be a finite $p$-group, $X$ a Borel $G$-spectrum, and $(-)^{t G} \to F(-)$ a lax monoidal natural transformation.
\begin{enumerate}
\item Suppose that $X$ is bounded. Then $F(X)$ is $p$-nilpotent.
\item Suppose that $X$ is bounded below. Then $F(X)$ is $p$-complete, and the map $F(X) \to F(X^{\wedge}_{p})$ is an equivalence.
\end{enumerate}
\end{lem}
\begin{proof} First suppose that $F = (-)^{t G}$ itself. Then the same proof as in \cite[Lem.~I.2.9]{NS18} applies: for (1), we reduce to $X = HM$ by induction on the Postnikov tower and use that the order of $G$ annihilates Tate cohomology $\widehat{H}^{\ast}(G;M)$, and (2) then follows, using that $p$-complete spectra are closed under limits. For the general situation, again we may reduce to the case $X = H M$. Then because $F(-)$ is a lax monoidal functor, $F(HM)$ is a $F(H \ZZ)$-module, so it suffices to show that $F(H \ZZ)$ is $p$-nilpotent. For this, via the lax monoidal natural transformation $(-)^{t G} \to F(-)$, we obtain an $E_{\infty}$-map $(H \ZZ)^{tG} \to F(H \ZZ)$, and because $(H \ZZ)^{tG}$ is $p$-nilpotent, we deduce that $F(H \ZZ)$ is $p$-nilpotent.
\end{proof}

By Prop.~\ref{prp:BorelSpectraAsCompleteObjects} and Prop.~\ref{prp:EquivalentTateConstructions}, we may identify $$(-)^{t_{C_2} \mu_p}: \Fun_{C_2}(B^t_{C_2} \mu_p, \ul{\Sp}^{C_2}) \to \Sp^{C_2}$$ with the gluing functor of the $\Gamma_{\mu_p}$-recollement on $\Sp^{D_{2p}}$. We thus obtain the following corollary of Lem.~\ref{lem:TateNilpotent}.

\begin{cor} \label{cor:ParamTatePcomplete} Let $X$ be a $C_2$-spectrum with twisted $\mu_p$-action. Suppose that for all choices of $C_2$-basepoints $\iota: \sO_{C_2}^{\op} \to B^t_{C_2} \mu_p$, $\iota^{\ast} X$ is bounded below as a $C_2$-spectrum (c.f. Rmk.~\ref{rem:DihedralBasepoints}).\footnote{In particular, if $X$ arises as the restriction of a $C_2$-spectrum with twisted $\mu_{p^{\infty}}$-action, then this bounded below condition is equivalent to stipulating that the underlying $C_2$-spectrum is bounded below.} Then $X^{t_{C_2} \mu_p}$ is $p$-complete (i.e., $S^0/p$-local in $\Sp^{C_2}$).
\end{cor}
\begin{proof} Note that a $C_2$-spectrum $E$ is $p$-complete if its geometric fixed points are $p$-complete, by reference to the usual fracture square. The claim then follows from Lem.~\ref{lem:TateNilpotent} and Exm.~\ref{exm:DihedralEven} (for $p=2$) or Exm.~\ref{exm:DihedralOdd} (for $p$ odd).
\end{proof}

We now turn to our dihedral refinement of the Tate orbit lemma. In the proofs of Lem.~\ref{lem:dihedralTOLEven} and Lem.~\ref{lem:dihedralTOLOdd}, we let $x = x_{2}$ be a generator for $\mu_{p^2}$ (c.f. Setup \ref{setup:Dihedral}).

\begin{lem} \label{lem:dihedralTOLEven} The functor given by the composite
\[ \begin{tikzcd}[row sep=4ex, column sep=8ex, text height=2ex, text depth=0.75ex]
\Fun_{C_2}(B^t_{C_2} \mu_4, \underline{\Sp}^{C_2}) \ar{r}{(-)_{h_{C_2} \mu_2}} & \Fun_{C_2}(B^t_{C_2} \mu_2, \underline{\Sp}^{C_2}) \ar{r}{(-)^{t_{C_2} \mu_2}} & \Sp^{C_2}
\end{tikzcd} \]
evaluates to $0$ on those objects $X$ such that the underlying spectrum $X^1$ is bounded below.
\end{lem}

For the proof, we first need the following lemma on $\Phi^{C_2}$ as a $C_2$-functor.

\begin{lem} \label{lem:GeometricFixedPointsPreservesParametrizedColimits} The $C_2$-functor $\Phi^{C_2}: \underline{\Sp}^{C_2} \to \underline{\Sp}^{\Phi C_2}$ preserves $C_2$-colimits, so for every $C_2$-functor $f: I \to J$, the diagram
\[ \begin{tikzcd}[row sep=4ex, column sep=8ex, text height=1.5ex, text depth=0.5ex]
\Fun_{C_2}(I, \underline{\Sp}^{C_2}) \ar{r}{f_!} \ar{d}{\Phi^{C_2}} & \Fun_{C_2}(J,\underline{\Sp}^{C_2}) \ar{d}{\Phi^{C_2}} \\
\Fun_{C_2}(I, \underline{\Sp}^{\Phi C_2}) \ar{r}{f_!} \ar{d}{\simeq} & \Fun_{C_2}(J, \underline{\Sp}^{\Phi C_2}) \ar{d}{\simeq} \\
\Fun(I_{C_2/C_2}, \Sp) \ar{r}{(f_{C_2/C_2})_!} & \Fun(J_{C_2/C_2},\Sp)
\end{tikzcd} \]
commutes, where $f_!$ denotes $C_2$-left Kan extension along $f$ and $(f_{C_2/C_2})_!$ denotes left Kan extension along $f_{C_2/C_2}$.
\end{lem}
\begin{proof} By Rmk.~\ref{ParamRecollementFamily}, $\Phi^{C_2}$ is a $C_2$-left adjoint and hence preserves $C_2$-colimits \cite[Cor.~8.7]{Exp2}, so the upper square commutes. By definition, the $C_2$-left Kan extension $f_!$ is left adjoint to restriction along $f$. Since $({\underline{\Sp}^{\Phi C_2}})_{C_2/1} \simeq \ast$, we have the vertical equivalences of the lower square under which restriction along $f$ is identified with restriction along $f_{C_2/C_2}$. This implies the commutativity of the lower square.
\end{proof}

\begin{proof}[Proof of Lem.~\ref{lem:dihedralTOLEven}] In view of the compatiblity of the functors with restriction as described by the commutative diagram
\[ \begin{tikzcd}[row sep=4ex, column sep=8ex, text height=2ex, text depth=0.75ex]
\Fun_{C_2}(B^t_{C_2} \mu_4, \underline{\Sp}^{C_2}) \ar{r}{(-)_{h_{C_2} \mu_2}} \ar{d} & \Fun_{C_2}(B^t_{C_2} \mu_2, \underline{\Sp}^{C_2}) \ar{r}{(-)^{t_{C_2} \mu_2}} \ar{d} & \Sp^{C_2} \ar{d} \\
\Fun(B D_8, \Sp) \ar{r}{(-)_{h \mu_2}} & \Fun(B D_4, \Sp) \ar{r}{(-)^{t \mu_2}} & \Fun(B C_2, \Sp) 
\end{tikzcd} \]

we have that $((X_{h_{C_2} \mu_2})^{t_{C_2} \mu_2})^1 \simeq ((X^1)_{h \mu_2})^{t \mu_2}$, which vanishes by the Tate orbit lemma for bounded below $\mu_4$-Borel spectra \cite[Lem.~I.2.1]{NS18}. Thus, it suffices to show that $((X_{h_{C_2} \mu_2})^{t_{C_2} \mu_2})^{\phi C_2} \simeq 0$. Let $\rho = \rho_{\mu_2}: B^t_{C_2} \mu_4 \to B^t_{C_2} \mu_2$ be as in Lem.~\ref{lm:CategoricalFixedPointsProperties}(2), so the $C_2$-orbits functor $(-)_{h_{C_2} \mu_2}$ is $C_2$-left Kan extension along $\rho$. By Lem.~\ref{lem:GeometricFixedPointsPreservesParametrizedColimits}, we have a commutative diagram
\[ \begin{tikzcd}[row sep=4ex, column sep=8ex, text height=2ex, text depth=0.75ex]
\Fun_{C_2}(B^t_{C_2} \mu_4, \underline{\Sp}^{C_2}) \ar{r}{(-)_{h_{C_2} \mu_2}} \ar{d}{\Phi^{C_2}} & \Fun_{C_2}(B^t_{C_2} \mu_2, \underline{\Sp}^{C_2}) \ar{d}{\Phi^{C_2}} \\
\Fun( (B^t_{C_2} \mu_4)_{C_2/C_2}, \Sp) \ar{r}{\rho'_!} & \Fun((B^t_{C_2} \mu_2)_{C_2/C_2}, \Sp)
\end{tikzcd} \]
where the bottom horizontal functor is left Kan extension along the restriction $\rho'$ of $\rho$ to the fiber $(B^t_{C_2} \mu_4)_{C_2/C_2}$. Picking $\angs{\sigma}$ and $\angs{\sigma x}$ as representatives of their respective conjugacy classes of subgroups in $D_8$, we have 
\begin{align*}
(B^t_{C_2} \mu_4)_{C_2/C_2} \simeq B W_{D_8} \angs{\sigma} \bigsqcup B W_{D_8} \angs{\sigma x} \simeq B C_2 \sqcup B C_2, \\
(B^t_{C_2} \mu_2)_{C_2/C_2} \simeq B W_{D_4} \angs{\sigma} \bigsqcup B W_{D_4} \angs{\sigma x} \simeq B C_2 \sqcup B C_2.
\end{align*}

Note that $\rho$ sends the generator $x^2 \in W_{D_8} \angs{\sigma} \cong C_2$ to $1 \in W_{D_4} \angs{\sigma} \cong C_2$ and likewise for $W_{D_8} \angs{\sigma x}$. Therefore, $\rho'$ may be identified with the map $BC_2 \bigsqcup BC_2 \to \ast \bigsqcup \ast \to BC_2 \bigsqcup BC_2$, and we see that
\[ (X_{h_{C_2} \mu_2})^{\phi \angs{\sigma x^i}} \simeq \ind^{C_2}((X^{\phi \angs{\sigma x^i}})_{h C_2}), \quad i = 0, 1. \]

As for the functor $(-)^{t_{C_2} \mu_2}$, given $Y \in \Fun_{C_2}(B^t_{C_2} \mu_2, \underline{\Sp}^{C_2})$, by Exm.~\ref{exm:DihedralEven} we have that
\[ (Y^{t_{C_2} \mu_2})^{\phi C_2} \simeq (Y^1)^{\tau D_4} \times_{((Y^1)^{t \angs{\sigma} t C_2} \times (Y^1)^{t \angs{\sigma x} t C_2})} ((Y^{\phi \angs{\sigma}})^{t C_2} \times (Y^{\phi \angs{\sigma x}})^{t C_2}). \]

Using that $(-)^{t C_2}$ vanishes on $C_2$-induced objects, we deduce that
\[ \tag{$\ast$} ((X_{h_{C_2} \mu_2})^{t_{C_2} \mu_2})^{\phi C_2} = \fib((X^1_{h \mu_2})^{\tau D_4} \to (X^1_{h \mu_2})^{t \angs{\sigma} tC_2} \times (X^1_{h \mu_2})^{t \angs{\sigma x} tC_2} ). \]

Thus, the terms $X^{\phi \angs{\sigma}}$ and $X^{\phi \angs{\sigma x}}$ are irrelevant for the computation, in the sense that the counit map $j_! j^{\ast} X = X \otimes E {D_8}_+ \to X$ for the adjunction
\[ \adjunct{j_!}{\Fun(B D_8, \Sp)}{\Fun_{C_2}(B^t_{C_2} \mu_4, \underline{\Sp}^{C_2})}{j^\ast} \]
is sent to an equivalence under $(((-)_{h_{C_2} \mu_2})^{t_{C_2} \mu_2})^{\phi C_2}$. We may therefore extend our hypothesis that $X^1 = j^{\ast} X$ is bounded below to further suppose that $X$ is bounded below with respect to the homotopy $t$-structure on $\Sp^{C_2}$. Then by Lem.~\ref{lem:TateConvergence}, the cofiber sequence
\[ (-)_{h_{C_2} \mu_4} \to ((-)_{h_{C_2} \mu_2})^{h_{C_2} \mu_2} \to  ((-)_{h_{C_2} \mu_2})^{t_{C_2} \mu_2}, \]
and induction up the Postnikov tower of $X$, we reduce to the case of $X = j_! HM$ for $M$ a $\ZZ[D_8]$-module. Moreover, in view of the fiber sequence ($\ast$) and Lem.~\ref{lem:TateNilpotent}, $((X_{h_{C_2} \mu_2})^{t_{C_2} \mu_2})^{\phi C_2}$ is $2$-complete. Thus, to show vanishing we may further suppose that $M$ is a $\FF_2[D_8]$-module.

Let us now consider the $\FF_2[D_8]$-free resolution of $M$ from \cite[\S IV.2, p.~129]{AdemMilgram} (and with all signs suppressed since $2=0$), given by taking the total complex of the bicomplex

\[ \begin{tikzcd}[row sep=4ex, column sep=4ex, text height=1.5ex, text depth=0.5ex]
\vdots \ar{d}{\sigma+1} & \vdots \ar{d}{\sigma x +1} & \vdots \ar{d}{\sigma+1} & \vdots \ar{d}{\sigma x +1} &  \\
M[D_8] \ar{d}{\sigma+1} & M[D_8] \ar{l}{x+1} \ar{d}{\sigma x + 1} & M[D_8] \ar{l}{\Sigma_x} \ar{d}{\sigma+1} & M[D_8] \ar{d}{\sigma x +1} \ar{l}{x+1} & \cdots \ar{l}{\Sigma_x} \\
M[D_8] \ar{d}{\sigma+1} & M[D_8] \ar{l}{x+1} \ar{d}{\sigma x + 1} & M[D_8] \ar{l}{\Sigma_x} \ar{d}{\sigma+1} & M[D_8] \ar{d}{\sigma x +1} \ar{l}{x+1} & \cdots \ar{l}{\Sigma_x} \\
M[D_8] \ar{d}{\sigma+1} & M[D_8] \ar{l}{x+1} \ar{d}{\sigma x + 1} & M[D_8] \ar{l}{\Sigma_x} \ar{d}{\sigma+1} & M[D_8] \ar{d}{\sigma x +1} \ar{l}{x+1} & \cdots \ar{l}{\Sigma_x} \\
M[D_8] & M[D_8] \ar{l}{x+1} & M[D_8] \ar{l}{\Sigma_x} & M[D_8] \ar{l}{x+1} & \cdots \ar{l}{\Sigma_x}
\end{tikzcd} \]
where $\Sigma_x = 1+x+x^2+x^3$. Application of the functor $(-)/\mu_2$ to this bicomplex yields the bicomplex of $\FF_2[D_4]$-modules
\[ \begin{tikzcd}[row sep=4ex, column sep=4ex, text height=1.5ex, text depth=0.5ex]
\vdots \ar{d}{\sigma+1} & \vdots \ar{d}{\sigma x +1} & \vdots \ar{d}{\sigma+1} & \vdots \ar{d}{\sigma x +1} &  \\
M[D_4] \ar{d}{\sigma+1} & M[D_4] \ar{l}{x+1} \ar{d}{\sigma x + 1} & M[D_4] \ar{l}{0} \ar{d}{\sigma+1} & M[D_4] \ar{d}{\sigma x +1} \ar{l}{x+1} & \cdots \ar{l}{0} \\
M[D_4] \ar{d}{\sigma+1} & M[D_4] \ar{l}{x+1} \ar{d}{\sigma x + 1} & M[D_4] \ar{l}{0} \ar{d}{\sigma+1} & M[D_4] \ar{d}{\sigma x +1} \ar{l}{x+1} & \cdots \ar{l}{0} \\
M[D_4] \ar{d}{\sigma+1} & M[D_4] \ar{l}{x+1} \ar{d}{\sigma x + 1} & M[D_4] \ar{l}{0} \ar{d}{\sigma+1} & M[D_4] \ar{d}{\sigma x +1} \ar{l}{x+1} & \cdots \ar{l}{0} \\
M[D_4] & M[D_4] \ar{l}{x+1} & M[D_4] \ar{l}{0} & M[D_4] \ar{l}{x+1} & \cdots \ar{l}{0}
\end{tikzcd} \]
whose total complex is quasi-isomorphic to $M_{h \mu_2}$ in the derived category of $\FF_2[D_4]$ (crucially, we use that $2=0$ to see that $(\Sigma_x)/\mu_2 = 0$). Let $F^n(M_{\mu_2})$ be the total complex obtained by truncating the bicomplex to the first $2n$ columns, viewed in the derived category. Because of the zero maps that appear horizontally in the bicomplex, we have retractions $r_n: M_{\mu_2} \to F^n(M_{\mu_2})$ splitting the natural inclusions such that
\begin{enumerate}
    \item The induced map $M_{h \mu_2} \to \lim_{n} F^n(M_{h \mu_2})$ is an equivalence.
    \item The connectivity of the fiber of $M_{h \mu_2} \to F^n(M_{h \mu_2})$ goes to $\infty$ as $n \to \infty$. 
\end{enumerate}

Moreover, in view of the commutative diagram
\[ \begin{tikzcd}[row sep=4ex, column sep=4ex, text height=1.5ex, text depth=0.25ex]
\Fun(B D_8, \Sp) \ar{r}{j_!} \ar{d}{(-)_{h \mu_2}} & \Fun_{C_2}(B^t_{C_2} \mu_4, \underline{\Sp}^{C_2}) \ar{d}{(-)_{h_{C_2} \mu_2}} \\
\Fun(B D_4, \Sp) \ar{r}{j_!} & \Fun_{C_2}(B^t_{C_2} \mu_2, \underline{\Sp}^{C_2})
\end{tikzcd} \]
we obtain a filtration $j_! (F^n(M_{\mu_2})$ of $(j_! M)_{h_{C_2} \mu_2} \simeq j_!(M_{h \mu_2})$ such that
\begin{enumerate}
    \item The induced map $j_! (M_{h \mu_2}) \to \lim_{n} j_! (F^n(M_{h \mu_2}))$ is an equivalence. For this, to commute $j_!$ past the inverse limit we use that
    \[ (\lim_{n} F^n(M_{h \mu_2}))^{t \mu_2} \simeq \lim_n F^n(M_{h \mu_2})^{t \mu_2} \]
    in view of the increasing connectivity of the fibers.
    \item The $C_2$-connectivity of the fiber of $j_! (M_{h \mu_2}) \to j_! (F^n(M_{h \mu_2}))$ goes to $\infty$ as $n \to \infty$.\footnote{A priori, when considering $C_2$-connectivity of the underlying object in $\Sp^{C_2}$ of a $C_2$-functor $B^t_{C_2} \mu_2 \to \underline{\Sp}^{C_2}$, we must consider all $C_2$-basepoints of $B^t_{C_2} \mu_2$. However, because the objects in question are Borel-torsion, any choice of $C_2$-basepoint yields the same object.} For this, note that for Borel-torsion objects $E \in \Sp^{C_2}$, $E^{C_2} \simeq E_{h C_2}$ since $E^{\phi C_2} \simeq 0$, so the connectivity of $E^{C_2}$ is bounded below by that of $E$.
\end{enumerate}

Now by Lem.~\ref{lem:TateConvergence} applied to $(-)^{t_{C_2} \mu_2}$, in order to show that $(((j_! M)_{h_{C_2} \mu_2})^{t_{C_2} \mu_2})^{\phi C_2} \simeq 0$ it suffices to consider the vanishing of the functor $((j_!(-))^{t_{C_2} \mu_2})^{\phi C_2}$ on the filtered quotients $F^{n+1}/F^n(M_{\mu_2})$. For this, we observe that the alternating vertical columns of the bicomplex are free resolutions of $M[D_4/\angs{\sigma}]$ and $M[D_4/\angs{\sigma x}]$, respectively. Therefore, the filtered quotients $F^{n+1}/F^n(M_{\mu_2})$ are extensions of objects induced from proper subgroups of $D_4$, and are thus annihilated by $(-)^{\tau D_4}$, $(-)^{t \angs{\sigma} tC_2}$, and $(-)^{t \angs{\sigma x} tC_2}$.
\end{proof}

In contrast to Lem.~\ref{lem:dihedralTOLEven}, the proof of the dihedral Tate orbit lemma at an odd prime is far simpler.

\begin{lem} \label{lem:dihedralTOLOdd} Let $p$ be an odd prime. The functor given by the composite
\[ \begin{tikzcd}[row sep=4ex, column sep=8ex, text height=2ex, text depth=0.75ex]
\Fun_{C_2}(B^t_{C_2} \mu_{p^2}, \underline{\Sp}^{C_2}) \ar{r}{(-)_{h_{C_2} \mu_p}} & \Fun_{C_2}(B^t_{C_2} \mu_p, \underline{\Sp}^{C_2}) \ar{r}{(-)^{t_{C_2} \mu_p}} & \Sp^{C_2}
\end{tikzcd} \]
evaluates to $0$ on those objects $X$ such that the underlying spectrum $X^1$ is bounded below.
\end{lem}
\begin{proof} As in the proof of Lem.~\ref{lem:dihedralTOLEven}, one has the commutative diagram
\[ \begin{tikzcd}[row sep=4ex, column sep=8ex, text height=2ex, text depth=0.75ex]
\Fun_{C_2}(B^t_{C_2} \mu_{p^2}, \underline{\Sp}^{C_2}) \ar{r}{(-)_{h_{C_2} \mu_p}} \ar{d} & \Fun_{C_2}(B^t_{C_2} \mu_p, \underline{\Sp}^{C_2}) \ar{r}{(-)^{t_{C_2} \mu_p}} \ar{d} & \Sp^{C_2} \ar{d} \\
\Fun(B D_{2 p^2}, \Sp) \ar{r}{(-)_{h \mu_p}} & \Fun(B D_{p^2}, \Sp) \ar{r}{(-)^{t \mu_p}} & \Fun(B C_2, \Sp) ,
\end{tikzcd} \]

so $((X_{h_{C_2} \mu_p})^{t_{C_2} \mu_p})^1 \simeq ((X^1)_{h \mu_p})^{t \mu_p}$, which vanishes by the Tate orbit lemma for bounded below $\mu_{p^2}$-Borel spectra \cite[Lem.~I.2.1]{NS18}. Then by Exm.~\ref{exm:DihedralOdd},  we have that $(Y)^{t_{C_2} \mu_p})^{\phi C_2} \simeq 0$ for all $Y$ and thus $((X_{h_{C_2} \mu_p})^{t_{C_2} \mu_p})^{\phi C_2} \simeq 0$ unconditionally.
\end{proof}

\begin{rem} The restriction of the $C_2$-functor $B^t_{C_2} \mu_{p^2} \to B^t_{C_2} \mu_{p}$ to the fiber over $C_2/C_2$ is equivalent to the trivial map $\ast \to \ast$, so by Lem.~\ref{lem:GeometricFixedPointsPreservesParametrizedColimits}, for $X \in \Fun_{C_2}(B^t_{C_2} \mu_{p^2}, \underline{\Sp}^{C_2})$ we have that $(X_{h_{C_2} \mu_p})^{\phi \angs{\sigma}} \simeq X^{\phi \angs{\sigma}}$.
\end{rem}

We now prove a few corollaries of the dihedral Tate orbit lemma. These results are all obvious analogues of those in \cite[\S II.4]{NS18}.

\begin{lem} \label{lem:BoundedBelowEquivalence} Suppose $X$ is a $C_2$-spectrum with twisted $\mu_{p^n}$-action whose underlying spectrum $X^1$ is bounded below. Then the canonical map of Prop.~\ref{prp:ResidualAction}
$$ X^{t_{C_2} \mu_{p^n}} \to (X^{t_{C_2} \mu_p})^{h_{C_2} \mu_{p^{n-1}}} $$
is an equivalence of $C_2$-spectra.
\end{lem}
\begin{proof} We mimic the proof of \cite[Lem.~II.4.1]{NS18}. Note that $X_{h_{C_2} \mu_{p^{n-1}}}$ has bounded below underlying spectrum $(X^1)_{h \mu_{p^{n-1}}}$. By the dihedral Tate orbit lemma, we see that the norm map
\[ X_{h_{C_2} \mu_{p^n}} \simeq (X_{h_{C_2} \mu_{p^{n-1}}})_{h_{C_2} \mu_p} \to (X_{h_{C_2} \mu_{p^{n-1}}})^{h_{C_2} \mu_p} \]
is an equivalence. By induction, it follows that the norm map
\[ X_{h_{C_2} \mu_{p^n}} \to (X_{h_{C_2} \mu_p})^{h_{C_2} \mu_{p^{n-1}}} \]
is an equivalence. Therefore, the left and middle vertical maps in the commutative diagram
\[ \begin{tikzcd}[row sep=4ex, column sep=4ex, text height=1.5ex, text depth=0.25ex]
X_{h_{C_2} \mu_{p^n}} \ar{r} \ar{d} & X^{h_{C_2} \mu_{p^n}} \ar{r} \ar{d} & X^{t_{C_2} \mu_{p^n}} \ar{d} \\
(X_{h_{C_2} \mu_p})^{h_{C_2} \mu_{p^{n-1}}} \ar{r} & (X^{h_{C_2} \mu_p})^{h_{C_2} \mu_{p^{n-1}}} \ar{r} & (X^{t_{C_2} \mu_p})^{h_{C_2} \mu_{p^{n-1}}}
\end{tikzcd} \]
are equivalences, so the right vertical map is also an equivalence.
\end{proof}

\begin{dfn} For $X \in \Fun_{C_2}(B^t_{C_2} \mu_{p^{\infty}}, \ul{\Sp}^{C_2})$, let
$$ X^{t_{C_2} \mu_{p^{\infty}}} = \lim_n X^{t_{C_2} \mu_{p^{n}}} $$
where the inverse limit is taken along the maps
$$ X^{t_{C_2} \mu_{p^{n}}} \to (X^{t_{C_2} \mu_{p^{n-1}}})^{h_{C_2} \mu_p} \to X^{t_{C_2} \mu_{p^{n-1}}}. $$
\end{dfn}


\begin{cor} \label{cor:TateToFixedPointsEquivalence} Suppose that $X$ is a $C_2$-spectrum with twisted $\mu_{p^{\infty}}$-action whose underlying spectrum is bounded below. Then the canonical map $$X^{t_{C_2} \mu_{p^{\infty}}} \to (X^{t_{C_2} \mu_p})^{h_{C_2} \mu_{p^{\infty}}}$$ is an equivalence.
\end{cor}
\begin{proof} The map in question is the inverse limit of the maps $X^{t_{C_2} \mu_{p^n}} \to (X^{t_{C_2} \mu_p})^{h_{C_2} \mu_{p^{n-1}}}$, which by Lem.~\ref{lem:BoundedBelowEquivalence} are equivalences under our assumption on $X$.
\end{proof}

\begin{rem} Using Cor.~\ref{cor:ParamTatePcomplete}, if we further suppose that the underlying $C_2$-spectrum of $X$ is bounded below, then $(X^{t_{C_2} \mu_p})^{h_{C_2} \mu_{p^n}}$ is $p$-complete for all $0 \leq n \leq \infty$, and hence $X^{t_{C_2} \mu_{p^n}}$ is also $p$-complete for all $1 \leq n \leq \infty$.
\end{rem}

The following lemma extends \cite[II.4.5-7]{NS18}.

\begin{lem} \label{lm:SameLemmaNS} Suppose $X$ is a $D_{2p^n}$-spectrum.
\begin{enumerate} \item We have a natural pullback square of $C_2$-spectra
\[ \begin{tikzcd}[row sep=4ex, column sep=4ex, text height=1.5ex, text depth=0.25ex]
\Psi^{\mu_{p^n}} X \ar{r} \ar{d} & \Psi^{\mu_{p^{n-1}}} (\Phi^{\mu_p} X) \ar{d} \\
X^{h_{C_2} \mu_{p^n}} \ar{r} & X^{t_{C_2} \mu_{p^n}}.
\end{tikzcd} \]
\item Suppose in addition that the underlying spectrum of $X$ is bounded below. Then we have a natural pullback square of $C_2$-spectra
\[ \begin{tikzcd}[row sep=4ex, column sep=4ex, text height=1.5ex, text depth=0.25ex]
\Psi^{\mu_{p^n}} X \ar{r} \ar{d} & \Psi^{\mu_{p^{n-1}}} (\Phi^{\mu_p} X) \ar{d} \\
X^{h_{C_2} \mu_{p^n}} \ar{r} & (X^{t_{C_2} \mu_{p}})^{h_{C_2} \mu_{p^{n-1}}}.
\end{tikzcd} \] 
\item Suppose in addition that the underlying spectra of
\[ X, \Phi^{\mu_p} X, \Phi^{\mu_{p^2}} X, \cdots, \Phi^{\mu_{p^{n-1}}} X \]
are all bounded below. Then we have a natural limit diagram of $C_2$-spectra
\[ \begin{tikzcd}[row sep=4ex, column sep=4ex, text height=1.5ex, text depth=0.25ex]
\Psi^{\mu_{p^n}} X \ar{rrrr} \ar{dddd} & & & & \Phi^{\mu_{p^n}} X \ar{d} \\
 & & & (\Phi^{\mu_{p^{n-1}}} X)^{h_{C_2} \mu_p} \ar{r} \ar{d}  &  (\Phi^{\mu_{p^{n-1}}} X)^{t_{C_2} \mu_p} \\
 & & (\Phi^{\mu_{p^{2}}} X)^{h_{C_2} \mu_{p^{n-2}}} \ar{r} \ar{d} & \cdots \\
 & (\Phi^{\mu_p} X)^{h_{C_2} \mu_{p^{n-1}}} \ar{r} \ar{d} & ((\Phi^{\mu_p} X)^{t_{C_2} \mu_p})^{h_{C_2} \mu_{p^{n-2}}} \\
X^{h_{C_2} \mu_{p^n}} \ar{r} & (X^{t_{C_2} \mu_p})^{h_{C_2} \mu_{p^{n-1}}}
\end{tikzcd} \]
\end{enumerate}
\end{lem}
\begin{proof} In view of Prop.~\ref{prp:EquivalentTateConstructions}, the first pullback square arises from applying $\Psi^{\mu_{p^n}}$ to the fracture square for the $\Gamma_{\mu_{p^n}}$-recollement on $\Sp^{D_{2p^n}}$. The second pullback square then follows by Lem.~\ref{lem:BoundedBelowEquivalence}, and the last limit diagram follows by induction on $n$.
\end{proof}

We may now equate the fiber sequences for $\TCR(-,p)$ (Prop.~\ref{prp:TCRfiberSequence}) and $\TCR^{\mr{gen}}(-,p)$ (Prop.~\ref{prp:fiberSequenceGenuineRealCyc}) in the bounded below situation, giving a direct proof of Cor.~\ref{cor:TCRFormulasEquivalent}.

\begin{cor} \label{cor:decategorifiedEasy} Let $X$ be a genuine real $p$-cyclotomic spectrum whose underlying spectrum is bounded below. Then there is a canonical and natural fiber sequence
\[ \TCR^{\mr{gen}}(X,p) \to X^{h_{C_2} \mu_{p^{\infty}}} \xtolong{\varphi^{h_{C_2} \mu_{p^{\infty}}} - \can}{2} (X^{t_{C_2} \mu_p})^{h_{C_2} \mu_{p^{\infty}}} \]
and thus an equivalence $\TCR^{\mr{gen}}(X,p) \simeq \TCR(\sU_{\RR} X,p)$.
\end{cor}
\begin{proof} Using Lem.~\ref{lm:SameLemmaNS}, we may transcribe the proof of \cite[Thm.~II.4.10]{NS18} into the $C_2$-parametrized setting to prove the claim, with no change of detail.
\end{proof}

\subsection{The comparison at finite level}

\begin{dfn} Let $\Sp^{C_{p^n}}_{bb} \subset \Sp^{C_{p^n}}$ be the full subcategory on those $C_{p^n}$-spectra that are bounded below, or equivalently, slice bounded below (Lem.~\ref{lem:BddBelowEqualsSliceBddBelow}).

Let $\Sp^{D_{2p^n}}_{ubb} \subset \Sp^{D_{2p^n}}$ be the full subcategory on those $D_{2p^n}$-spectra whose underlying $\mu_{p^n}$-spectrum is bounded below.
\end{dfn}

We wish to give an iterative decomposition of $\Sp^{D_{2p^n}}_{ubb}$  that categorifies the `staircase' limit diagram in Lem.~\ref{lm:SameLemmaNS}. We also take the opportunity to give a similar iterative decomposition of $\Sp^{C_{p^n}}_{bb}$, along the lines described in \cite[Rmk.~II.4.8]{NS18}.\footnote{We write $C_{p^n}$ instead of $\mu_{p^n}$ here in adherence to \cite{NS18}.} Our main tool in achieving this will be Thm.~\ref{thm:OneGenerationAndExtension} in conjunction with Thm.~\ref{thm:GeometricFixedPointsDescriptionOfGSpectra}. For the case of $D_{2p^n}$, we will need a \emph{relative} version of the geometric locus construction of Def.~\ref{dfn:GeometricLocus}.

\begin{dfn} \label{dfn:relativeGeometricLocus} Suppose $\pi: \fS[G] \to \Delta^n$ is a surjective functor, and for $0 \leq k \leq n$, let $\fS_{\pi \leq k}, \fS_{\pi<k} \subset \fS[G]$ be the sieves containing subgroups $H$ such that $\pi(H) \leq k$, resp. $\pi(H) < k$ (by convention, $\fS_{\pi<0} = \emptyset$). Let $L[\pi]_k: \Sp^G \to \Sp^{h \fS_{\pi \leq k}} \cap \Sp^{\Phi \fS_{\pi<k}}$ denote the localization functors.

Define the \emph{$\pi$-relative geometric locus} $\Sp^G_{\locus{\phi,\pi}} \subset \Sp^G \times \Delta^n$ be the full subcategory on $(X,k)$ such that $X$ is $\fS_{\pi \leq k}$-complete and $\fS_{\pi<k}^{-1}$-local. For $0\leq i \leq j \leq n$, define the \emph{$\pi$-relative generalized Tate construction} to be the composite
\[ \tau[\pi]^j_i: \Sp^{h \fS_{\pi \leq i}} \cap \Sp^{\Phi \fS_{\pi<i}} \to \Sp^G \to \Sp^{h \fS_{\pi \leq j}} \cap \Sp^{\Phi \fS_{\pi<j}} \]
of the inclusion and localization functors.

Let $[i:j] \subset \Delta^n$ denote the full subcategory on the vertices $i$ through $j$, so $\Delta^{j-i} \cong [i:j]$. Define the comparison functor
\[ \Theta'[\pi]_{i,j}: \Fun^{\cocart}_{/[i:j]}(\sd([i:j]), [i:j] \times_{\Delta^n} \Sp^G_{\locus{\phi,\pi}} ) \to \Fun(\sd([i:j]), \Sp^G) \xto{\lim} \Sp^G. \]
As before, the essential image of $\Theta'[\pi]_{i,j}$ lies in $\Sp^{h \fS_{\pi \leq j}} \cap \Sp^{\Phi \fS_{\pi < i}}$. Let $\Theta[\pi]_{i,j}$ denote the comparison functor with this codomain, and also write $\Theta[\pi] = \Theta[\pi]_{0,n}$.
\end{dfn}

\begin{vrn} \label{vrn:ReconstructionEquivalence} We have the following variants of the results in \S\ref{section:Reconstruction}, with the same proofs.
\begin{enumerate}
\item $\Sp^G_{\locus{\phi,\pi}} \to \Delta^n$ is a locally cocartesian fibration such that the pushforward functors are given by $\tau[\pi]^j_i$.
\item For all $0 \leq i \leq j \leq n$, $\Theta[\pi]_{i,j}$ is an equivalence of $\infty$-categories.
\item Let $0 \leq i \leq j <k \leq n$, so $[i:j]$, $[j+1:k]$ is a sieve-cosieve decomposition of $[i:k]$. Then we have a strict morphism of stable recollements through equivalences
\[ \begin{tikzcd}[row sep=4ex, column sep=8ex, text height=1.5ex, text depth=0.5ex]
\Fun^{\cocart}_{/[i:j]}(\sd([i:j]), [i:j] \times_{\Delta^n} \Sp^G_{\locus{\phi,\pi}}) \ar[shift left=4, left hook->]{d} \ar[shift right=4, left hook->]{d} \ar{r}{\Theta[\pi]_{i,j}}[swap]{\simeq} & \Sp^{h \fS_{\pi \leq j}} \cap \Sp^{\Phi \fS_{\pi < i}} \ar[shift left=4, left hook->]{d} \ar[shift right=4, left hook->]{d} \\
\Fun^{\cocart}_{/[i:k]}(\sd([i:k]), [i:k] \times_{\Delta^n} \Sp^G_{\locus{\phi,\pi}}) \ar{u} \ar{d} \ar{r}{\Theta[\pi]_{i,k}}[swap]{\simeq} & \Sp^{h \fS_{\pi \leq k}} \cap \Sp^{\Phi \fS_{\pi < i}}  \ar{u} \ar{d} \\
\Fun^{\cocart}_{/[j+1:k]}(\sd([j+1:k]), [j+1:k] \times_{\Delta^n} \Sp^G_{\locus{\phi,\pi}} )  \ar{r}{\Theta[\pi]_{j+1,k}}[swap]{\simeq} \ar[shift right=4, right hook->]{u} & \Sp^{h \fS_{\pi \leq k}} \cap \Sp^{\Phi \fS_{\pi \leq j}} \ar[shift right=4, right hook->]{u}
\end{tikzcd} \]
In particular, for any $i \leq l \leq j$, the composite
\[ \Fun^{\cocart}_{/[i:j]}(\sd([i:j]), [i:j] \times_{\Delta^n} \Sp^G_{\locus{\phi,\pi}}) \to \Sp^{h \fS_{\pi \leq j}} \cap \Sp^{\Phi \fS_{\pi < i}} \to \Sp^{h \fS_{\pi \leq l}} \cap \Sp^{\Phi \fS_{\pi < l}} \]
is homotopic to evaluation at $l \in [i:j] \subset \sd[i:j]$.
\end{enumerate}
\end{vrn}

Specializing to the situation of interest, we have the following definition, which exploits a key self-similarity property of the dihedral groups.

\begin{dfn} Given a subgroup $H \subset D_{2p^n}$, let $\zeta(H) \geq 0$ be the integer such that $H \cap \mu_{p^n} = \mu_{p^{\zeta(H)}}$, and let $\zeta: \fS[D_{2 p^n}] \to \Delta^n$ denote the resulting map. Note that if $H$ is subconjugate to $K$, then $\zeta(H) \leq \zeta(K)$, so $\zeta$ defines a functor.
\end{dfn}

Note that for all $0 \leq k \leq n$, the cosieve $\fS_{\zeta \geq k}$ equals $\fS[D_{2p^n}]_{\geq \mu_{p^k}}$, so $\Sp^{\Phi \fS_{\zeta<k}} \simeq \Sp^{D_{2p^n}/\mu_{p^k}}$ and
\[ \Fun_{C_2}(B^t_{C_2} (\mu_{p^n} / \mu_{p^k}), \underline{\Sp}^{C_2}) \simeq  \Sp^{h \fS_{\zeta \leq k}} \cap \Sp^{\Phi \fS_{\zeta < k}}, \]
using the equivalence of Prop.~\ref{prp:BorelSpectraAsCompleteObjects}. We will also write
\[ L[\zeta]_k: \Sp^{D_{2 p^n}} \to \Fun_{C_2}(B^t_{C_2} \mu_{p^{n-k}}, \underline{\Sp}^{C_2})  \]
for the composite of $\Phi^{\mu_{p^k}}: \Sp^{D_{2 p^n}} \to \Sp^{D_{2 p^{n-k}}}$ and $\sU_b[\mu_{p^{n-k}}]$, which is identified with the functor $L[\zeta]_k$ of Def.~\ref{dfn:relativeGeometricLocus} under this equivalence and the isomorphism $D_{2 p^{n-k}} \cong D_{2 p^n} / \mu_{p^k}$.

\begin{dfn} The \emph{$C_2$-generalized Tate functors}
\[ \tau_{C_2} \mu_{p^k}: \Fun_{C_2}(B^t_{C_2} \mu_{p^n}, \underline{\Sp}^{C_2}) \to \Fun_{C_2}(B^t_{C_2} (\mu_{p^n} / \mu_{p^k}), \underline{\Sp}^{C_2}) \]
are defined as in Def.~\ref{dfn:relativeGeometricLocus} with respect to $\zeta$ under the above equivalence.
\end{dfn}

To apply the (dihedral) Tate orbit lemma, we need to re-express the functors $\tau C_{p^n}$ and $\tau_{C_2} \mu_{p^n}$ in terms of more familiar functors. For expositional purposes, we deal with these cases separately, although the statement for $\tau_{C_2} \mu_{p^n}$ logically implies those for $\tau C_{p^n}$. We note at the outset that we have already identified $\tau C_p \simeq t C_p$ and $\tau_{C_2} \mu_p \simeq t_{C_2} \mu_p$.

\begin{lem} \label{lm:identifyGenTate} Suppose $X$ is a Borel $C_{p^n}$-spectrum. For $1 < k \leq n$, we have $C_{p^{n-k}}$-equivariant equivalences
\[ X^{\tau C_{p^k}} \simeq X^{h C_p \tau C_{p^{k-1}}} \simeq X^{h C_{p^2} \tau C_{p^{k-2}}} \simeq \cdots \simeq X^{h C_{p^{k-1}} t C_p} \]
with respect to which the canonical map $X^{\tau C_{p^k}} \to X^{t C_p \tau C_{p^{k-1}}}$ fits into the fiber sequence
\[ (X_{h C_p})^{\tau C_{p^{k-1}}} \to X^{h C_p \tau C_{p^{k-1}}} \to X^{t C_p \tau C_{p^{k-1}}}.  \]
\end{lem}
\begin{proof} Consider the recollement
\[ \begin{tikzcd}[row sep=4ex, column sep=6ex, text height=1.5ex, text depth=0.25ex]
\Sp^{h C_{p^n}} \ar[shift right=1,right hook->]{r}[swap]{j_{\ast}} & \Sp^{C_{p^n}} \ar[shift right=2]{l}[swap]{j^{\ast}} \ar[shift left=2]{r}{i^{\ast} \simeq \Phi^{C_p}} & \Sp^{C_{p^{n-1}}} \ar[shift left=1,left hook->]{l}{i_{\ast}}
\end{tikzcd} \]
and the associated fiber sequence $j_! \to j_{\ast} \to i_{\ast} i^{\ast} j_{\ast}$. By Lem.~\ref{lm:CategoricalFixedPointsProperties} applied to $C_p \trianglelefteq C_{p^n} \trianglelefteq C_{p^n}$, we see that $\Psi^{C_p}(j_! X)$ is $C_{p^{n-1}}$-Borel torsion and $\Psi^{C_p}(j_{\ast} X) \simeq j'_{\ast}( X^{h C_p})$ for $j'_{\ast}: \Sp^{h C_{p^{n-1}}} \to \Sp^{C_{p^{n-1}}}$. Therefore, the fiber sequence of $C_{p^{n-1}}$-spectra
\[ \Psi^{C_p}(j_! X) \to \Psi^{C_p}(j_{\ast} X) \to \Psi^{C_p}(i_{\ast} i^{\ast} j_{\ast} X) \simeq \Phi^{C_p} (j_{\ast} X) \]
yields the fiber sequence of underlying Borel $C_{p^{n-1}}$-spectra
\[ X_{h C_p} \to X^{h C_p} \to X^{t C_p} \]
and, applying $\phi^{C_{p^i}}:\Sp^{C_{p^{n-1}}} \to \Fun(B C_{p^{n-1-i}}, \Sp)$ for $0 < i \leq n-1$, the fiber sequence of Borel $(C_{p^{n-1-i}})$-spectra
\[ 0 \to X^{h C_p \tau C_{p^i}} \to X^{\tau C_{p^{i+1}}}. \]
We thereby deduce the equivalence $X^{h C_p \tau C_{p^i}} \simeq X^{\tau C_{p^{i+1}}}$, and the remaining equivalences follow by replacing $X$ by $X^{h C_{p^k}}$.

Next, we map the fiber sequence in $\Sp^{C_{p^{n-1}}}$ to its Borel completion
\[ \begin{tikzcd}[row sep=4ex, column sep=4ex, text height=1.5ex, text depth=0.25ex]
\Psi^{C_p}(j_! X) \ar{r} \ar{d} & \Psi^{C_p}(j_{\ast} X) \ar{r} \ar{d}{\simeq} &  \Phi^{C_p} (j_{\ast} X) \ar{d} \\
(j'_{\ast} {j'}^{\ast}) (\Psi^{C_p}(j_! X))  \ar{r} & (j'_{\ast} {j'}^{\ast}) (\Psi^{C_p}(j_{\ast} X)) \ar{r} & (j'_{\ast} {j'}^{\ast}) (\Phi^{C_p} (j_{\ast} X)).
\end{tikzcd} \]
We note that by definition, the canonical map $X^{\tau C_{p^k}} \to X^{t C_p \tau C_{p^{k-1}}}$ is obtained as $\phi^{C_{p^{k-1}}}$ of the unit map $\Phi^{C_p} (j_{\ast} X) \to (j'_{\ast} {j'}^{\ast}) (\Phi^{C_p} (j_{\ast} X))$. Because the middle map is an equivalence, the fiber of the righthand map is canonically equivalent to the cofiber of the lefthand map. But because $(j_! X)^{C_p}$ is Borel-torsion with $X_{h C_p}$ as its underlying Borel $C_{p^{n-1}}$-spectrum, $\phi^{C_{p^{k-1}}}$ of that cofiber is definitionally $(X_{ hC_p})^{\tau C_{p^{k-1}}}$.
\end{proof}

\begin{cor} \label{cor:CanonicalFunctorsEquivalences} Suppose that $X$ is a bounded below Borel $C_{p^n}$-spectrum. Then the canonical map
\[ X^{\tau C_{p^k}} \to X^{t C_p \tau C_{p^{k-1}}} \]
is an equivalence for all $1 < k \leq n$.
\end{cor}
\begin{proof} By Lem.~\ref{lm:identifyGenTate}, we may equivalently show that $(X_{h C_p})^{\tau C_{p^{k-1}}} \simeq 0$. We proceed by induction on $k$ (and prove the claim for all $n \geq k$ with $k$ fixed since we may ignore residual action for vanishing). For the base case $k=2$, the fiber of the canonical map is $(X_{h C_p})^{t C_p}$, which vanishes by the Tate orbit lemma. Now suppose $k>2$ and $(Y_{h C_p})^{\tau C_{p^i}} \simeq 0$ for all bounded below $Y \in \Sp^{h C_{p^m}}$, $m \geq i+1$ and $1 \leq i<k-1$. By Lem.~\ref{lm:identifyGenTate}, we have a fiber sequence
\[ ((X_{h C_{p}})_{h C_p})^{\tau C_{p^{k-2}}} \to (X_{h C_p})^{\tau C_{p^{k-1}}} \to (X_{h C_p})^{t C_p \tau C_{p^{k-2}}}. \]
Since $X_{h C_{p}}$ remains bounded below, the inductive hypothesis ensures that the left term vanishes, and the right term vanishes by the Tate orbit lemma again. We conclude that the middle term vanishes.
\end{proof}

\begin{lem} \label{lem:dihedralIdentifyGenTate} Suppose $X$ is a $C_2$-spectrum with twisted $\mu_{p^n}$-action. For $1 < k \leq n$, we have twisted $\mu_{p^{n-k}}$-equivariant equivalences
\[ X^{\tau_{C_2} \mu_{p^k}} \simeq X^{h_{C_2} \mu_p \tau_{C_2} \mu_{p^{k-1}}} \simeq X^{h_{C_2} \mu_{p^2} \tau_{C_2} \mu_{p^{k-2}}} \simeq \cdots \simeq X^{h_{C_2} \mu_{p^{k-1}} t_{C_2} \mu_p}, \]
with respect to which the canonical map $$X^{\tau_{C_2} \mu_{p^k}} \to X^{t_{C_2} \mu_p \tau_{C_2} \mu_{p^{k-1}}}$$ fits into the fiber sequence
\[ (X_{h_{C_2} \mu_p})^{\tau_{C_2} \mu_{p^{k-1}}} \to X^{h_{C_2} \mu_p \tau_{C_2} \mu_{p^{k-1}}} \to X^{t_{C_2} \mu_p \tau_{C_2} \mu_{p^{k-1}}}.  \]
\end{lem}
\begin{proof} The strategy of the proof is the same as that of Lem.~\ref{lm:identifyGenTate}, where we instead consider the recollement 
\[ \begin{tikzcd}[row sep=4ex, column sep=6ex, text height=1.5ex, text depth=0.25ex]
\Fun_{C_2}(B^t_{C_2} \mu_{p^n}, \underline{\Sp}^{C_2}) \ar[shift right=1,right hook->]{r}[swap]{j_{\ast}} & \Sp^{D_{2 p^n}} \ar[shift right=2]{l}[swap]{j^{\ast}} \ar[shift left=2]{r}{i^{\ast} \simeq \Phi^{{\mu}_p}} & \Sp^{D_{2p^{n-1}}} \ar[shift left=1,left hook->]{l}{i_{\ast}}.
\end{tikzcd} \]
Let $\Gamma = \Gamma_{\mu_{p^{n-1}}}$ be the $\mu_{p^{n-1}}$-free $D_{2p^{n-1}}$-family, and let 
\[ j'_{\ast} = \sF_b^{\vee}[\mu_{p^{n-1}}]: \Fun_{C_2}(B^t_{C_2} \mu_{p^{n-1}}, \underline{\Sp}^{C_2}) \to \Sp^{D_{2 p^{n-1}}}. \]
Applying Lem.~\ref{lm:CategoricalFixedPointsProperties} to the functor $\Psi^{\mu_p}: \Sp^{D_{2 p^n}} \to \Sp^{D_{2 p^{n-1}}}$, we see that in the fiber sequence
\[ \Psi^{\mu_p}(j_! X) \to \Psi^{\mu_p}(j_{\ast} X) \to \Phi^{\mu_p}(j_{\ast} X), \]
$\Psi^{\mu_p}(j_! X)$ is $\Gamma$-torsion and $\Psi^{\mu_p}(j_{\ast} X) \simeq j'_{\ast}(X^{h_{C_2} \mu_p})$. As before, for $0 \leq i \leq n-1$ let
\[ L[\zeta]_i: \Sp^{D_{2 p^{n-1}}} \to \Fun_{C_2}(B^t_{C_2} \mu_{p^{n-1-i}}, \underline{\Sp}^{C_2}) \]
be the localization functor. Then $L[\zeta]_0$ of the fiber sequence yields
\[ X_{h_{C_2} \mu_p} \to X^{h_{C_2} \mu_p} \to X^{t_{C_2} \mu_p} \]
whereas for $0< i\leq n-1$, $L[\zeta]_i$ of the fiber sequence yields
\[ 0 \to (X^{h_{C_2} \mu_p})^{\tau_{C_2} \mu_{p^i}} \xto{\simeq} X^{\tau_{C_2} \mu_{p^{i+1}}}  \]
from which we deduce the string of equivalences in the statement. Finally, if we map the fiber sequence of $D_{2p^{n-1}}$-spectra into its $\Gamma$-completion, then as in the proof of Lem.~\ref{lm:identifyGenTate} we obtain the fiber sequence as in the statement.
\end{proof}

\begin{cor} \label{cor:dihedralTOLyieldsEquivs} Suppose $X$ is a $C_2$-spectrum with twisted $\mu_{p^n}$-action whose underlying spectrum is bounded below. Then the canonical map
\[ X^{\tau_{C_2} \mu_{p^k}} \to (X^{t_{C_2} \mu_p})^{\tau_{C_2} \mu_{p^{k-1}}} \]
is an equivalence for all $1 < k \leq n$.
\end{cor}
\begin{proof} Using the dihedral Tate orbit lemma and Lem.~\ref{lem:dihedralIdentifyGenTate}, the proof is the same as that of Cor.~\ref{cor:CanonicalFunctorsEquivalences}.
\end{proof}

We are now prepared to deploy Thm.~\ref{thm:OneGenerationAndExtension} together with Cor.~\ref{cor:CanonicalFunctorsEquivalences} and Cor.~\ref{cor:dihedralTOLyieldsEquivs} to decompose $\Sp^{C_{p^n}}_{bb}$ and $\Sp^{D_{2p^n}}_{ubb}$. To give a concise statement, we need to introduce some more notation.

\begin{ntn} Let 
\begin{align*} \Fun^{\cocart}_{/\Delta^n}(\sd(\Delta^n), \Sp^{C_{p^n}}_{\locus{\phi}})_{bb} \subset \Fun^{\cocart}_{/\Delta^n}(\sd(\Delta^n), \Sp^{C_{p^n}}_{\locus{\phi}}) \\
\Fun^{\cocart}_{/\Delta^n}(\sd(\Delta^n), \Sp^{D_{2 p^n}}_{\locus{\phi,\zeta}})_{ubb} \subset \Fun^{\cocart}_{/\Delta^n}(\sd(\Delta^n), \Sp^{D_{2 p^n}}_{\locus{\phi,\zeta}})
\end{align*}
be the full subcategories on functors that evaluate on all singleton strings to bounded below, resp. underlying bounded below spectra.
\end{ntn}

\begin{lem} We have equivalences
\begin{align*} \Theta: & \Fun^{\cocart}_{/\Delta^n}(\sd(\Delta^n), \Sp^{C_{p^n}}_{\locus{\phi}})_{bb} \xto{\simeq} \Sp^{C_{p^n}}_{bb}, \\
\Theta[\zeta]: & \Fun^{\cocart}_{/\Delta^n}(\sd(\Delta^n), \Sp^{D_{2 p^n}}_{\locus{\phi,\zeta}})_{ubb} \xto{\simeq} \Sp^{D_{2 p^n}}_{ubb}
\end{align*}
obtained by restriction from the equivalences of Thm.~\ref{thm:GeometricFixedPointsDescriptionOfGSpectra} and Var.~\ref{vrn:ReconstructionEquivalence}.
\end{lem}
\begin{proof} This follows from the definitions after recalling that $\Phi^{H} \circ \Theta$ and $L[\zeta]_k \circ \Theta[\zeta]$ are homotopic to evaluation at $H$ and $k$, respectively.
\end{proof}

\begin{ntn} Let $(t C_p)_{\bullet}: \Delta^n \to \Cat_{\infty}$ be the functor defined by
\[ \Sp^{h C_{p^n}} \xto{t^{C_p}} \Sp^{h C_{p^{n-1}}} \xto{t^{C_p}} \cdots \xto{t^{C_p}} \Sp^{h C_p} \xto{t^{C_p}} \Sp \]
and let $\Sp^{h C_{p^n}}_{\Tate} \to (\Delta^n)^{\op}$ be the cartesian fibration classified by $(t C_p)_{\bullet}$.

Similarly, $(t_{C_2} \mu_p)_{\bullet}: \Delta^n \to \Cat_{\infty}$ be the functor defined by
\[ \Fun_{C_2}(B^t_{C_2} \mu_{p^n}, \underline{\Sp}^{C_2}) \xto{t_{C_2} {\mu_p}} \Fun_{C_2}(B^t_{C_2} \mu_{p^{n-1}}, \underline{\Sp}^{C_2}) \xto{t_{C_2} {\mu_p}} \cdots \xto{t_{C_2} {\mu_p}} \Sp^{C_2} \]
and let $\Sp^{h_{C_2} \mu_{p^n}}_{\Tate} \to (\Delta^n)^{\op}$ be the cartesian fibration classified by $(t_{C_2} \mu_p)_{\bullet}$.
\end{ntn}

\begin{dfn} Given a section $X: (\Delta^n)^{\op} \to \Sp^{h C_{p^n}}_{\Tate}$, resp. $X:(\Delta^n)^{\op} \to \Sp^{h_{C_2} \mu_{p^n}}_{\Tate}$, we say that $X$ is \emph{bounded below}, resp. \emph{underlying bounded below} if for all $0 \leq k \leq n$, the underlying spectrum of $X(k)$ is bounded below. Let $\Fun_{/(\Delta^n)^{\op}}((\Delta^n)^{\op},\Sp^{h C_{p^n}}_{\Tate})_{bb}$, resp. $\Fun_{/(\Delta^n)^{\op}}((\Delta^n)^{\op},\Sp^{h_{C_2} \mu_{p^n}}_{\Tate})_{ubb}$ denote the corresponding full subcategories.
\end{dfn}

\begin{prp} \label{prp:EquivalenceOnBoundedBelowAtFiniteLevel} We have inclusions of full subcategories
\begin{align*} \Fun^{\cocart}_{/\Delta^n}(\sd(\Delta^n), \Sp^{C_{p^n}}_{\locus{\phi}})_{bb} &\subset \Fun^{\cocart}_{/\Delta^n}(\sd(\Delta^n), \Sp^{C_{p^n}}_{\locus{\phi}})_{\gen{1}},  \\
\Fun_{/(\Delta^n)^{\op}}((\Delta^n)^{\op},\Sp^{h C_{p^n}}_{\Tate})_{bb} &\subset \Fun_{/(\Delta^n)^{\op}}((\Delta^n)^{\op},\Sp^{h C_{p^n}}_{\Tate})_{\ext}, \\
\Fun^{\cocart}_{/\Delta^n}(\sd(\Delta^n), \Sp^{D_{2 p^n}}_{\locus{\phi,\zeta}})_{ubb} &\subset \Fun^{\cocart}_{/\Delta^n}(\sd(\Delta^n), \Sp^{D_{2 p^n}}_{\locus{\phi,\zeta}})_{\gen{1}}, \\
\Fun_{/(\Delta^n)^{\op}}((\Delta^n)^{\op},\Sp^{h_{C_2} \mu_{p^n}}_{\Tate})_{ubb} &\subset \Fun_{/(\Delta^n)^{\op}}((\Delta^n)^{\op},\Sp^{h_{C_2} \mu_{p^n}}_{\Tate})_{\ext},
\end{align*}
such that the equivalences of Thm.~\ref{thm:OneGenerationAndExtension} between $1$-generated and extendable objects restrict to
\begin{align*} \Fun^{\cocart}_{/\Delta^n}(\sd(\Delta^n), \Sp^{C_{p^n}}_{\locus{\phi}})_{bb} & \xto{\simeq} \Fun_{/(\Delta^n)^{\op}}((\Delta^n)^{\op},\Sp^{h C_{p^n}}_{\Tate})_{bb}, \\
\Fun^{\cocart}_{/\Delta^n}(\sd(\Delta^n), \Sp^{D_{2 p^n}}_{\locus{\phi,\zeta}})_{ubb} & \xto{\simeq} \Fun_{/(\Delta^n)^{\op}}((\Delta^n)^{\op},\Sp^{h_{C_2} \mu_{p^n}}_{\Tate})_{ubb}.
\end{align*}
\end{prp}
\begin{proof} The inclusions follow from Cor.~\ref{cor:CanonicalFunctorsEquivalences}, Cor.~\ref{cor:dihedralTOLyieldsEquivs}, and Lem.~\ref{lm:equivalentOneGenerationConditions}. Matching the (underlying) bounded below conditions then implies that the equivalence of Thm.~\ref{thm:OneGenerationAndExtension} restricts as claimed.
\end{proof}

\begin{rem} Using Prop.~\ref{prp:DualDescriptionOfSections}, we may further unwind the equivalences of Prop.~\ref{prp:EquivalenceOnBoundedBelowAtFiniteLevel}. For example, analogous to \cite[Rmk.~II.4.8]{NS18}, we see that the data of an object $X \in \Sp^{D_{2p^n}}_{ubb}$ is equivalent to a sequence of objects
\[ L[\zeta]_0(X), L[\zeta]_1(X), \cdots, L[\zeta]_n(X) \]
where $L[\zeta]_k(X)$ is a $C_2$-spectrum with twisted $\mu_{p^n}/\mu_{p^k}$-action whose underlying spectrum is bounded below, together with twisted $\mu_{p^n}/\mu_{p^k}$-equivariant maps
\[ L[\zeta]_k(X) \to L[\zeta]_{k-1}(X)^{t_{C_2} \mu_p} \]
for $1 \leq k \leq n$.
\end{rem}

\subsection{Exchanging a lax equalizer for an equalizer}

In this subsection, we record a technical lemma regarding the lax equalizer of the identity and an endofunctor $F:C \to C$ that we will need for the proof of Thm.~\ref{thm:MainTheoremEquivalenceBddBelow}.

\begin{ntn} \label{ntn:nonnegativeintegers} Let $\ZZ_{\geq 0}$ denote the totally ordered set of non-negative integers regarded as a category, and let $\NN$ denote the monoid of non-negative integers under addition. Let $s: \ZZ_{\geq 0} \to \ZZ_{\geq 0}$ denote the successor functor that sends $n$ to $n+1$.
\end{ntn}

\begin{dfn} \label{dfn:spines} The \emph{spine} $\spine(\Delta^n) \subset \Delta^n$ is the subsimplicial set $\bigcup_{k=0}^{n-1} \{k<k+1\}$. Likewise, the \emph{spine} $\spine(\ZZ_{\geq 0}) \subset \ZZ_{\geq 0}$ is the subsimplicial set $\bigcup_{k=0}^{\infty} \{k<k+1\}$.
\end{dfn}

\begin{rem} \label{rem:innerAnodyneSpineInclusion} Recall that the spine inclusions of Def.~\ref{dfn:spines} are inner anodyne; indeed, a simple inductive argument with inner horn inclusions shows the maps $\spine(\Delta^n) \subset \Delta^n$ are inner anodyne, and it follows that $\spine(\ZZ_{\geq 0}) \subset \ZZ_{\geq 0}$ is inner anodyne by the stability of inner anodyne maps under filtered colimits.
\end{rem}

\begin{cnstr} \label{cnstr:SetupForEndofunctorFibration} Let $C$ be an $\infty$-category and $F: C \to C$ an endofunctor. Let $$\widehat{C} \to B \NN \cong B \NN^{\op}$$ be the cartesian fibration classified by the functor $B \NN \to \Cat_{\infty}$ that deloops the map of monoids $\NN \to \Fun(C,C)$ uniquely specified by $1 \mapsto F$.\footnote{This is the operadic left Kan extension of the functor $\ast \to \Fun(C,C)$ selecting $F$ for the monoidal structure on $\Fun(C,C)$ defined by composition of endofunctors.} Define a structure map $$p: \ZZ_{\geq 0}^{\op} \times \ZZ_{\geq 0} \to B \NN$$ by $p[(n+k,m) \rightarrow (n,m+l)] = k$ and note that $p$ is a cartesian fibration. We will regard any subcategory of $\ZZ_{\geq 0}^{\op} \times \ZZ_{\geq 0}$ as over $B \NN$ via $p$, so $\ZZ_{\geq 0}^{\op} \to B \NN$ is a cartesian fibration whereas $\ZZ_{\geq 0} \to B \NN$ is the constant functor at $\ast$. Precomposition by the successor functor $s$ defines two `shift' functors
\begin{align*} \sh = s^{\ast} &: \Fun(\ZZ_{\geq 0},C) \to \Fun(\ZZ_{\geq 0},C), \\ 
\sh = (s^{\op})^{\ast} &: \Fun_{/B \NN}(\ZZ_{\geq 0}^{\op}, \widehat{C}) \to \Fun_{/B \NN}(\ZZ_{\geq 0}^{\op}, \widehat{C}).
\end{align*}

Let $F_{\ast}$ be the endofunctor of $\Fun(\ZZ_{\geq 0},C)$ defined by postcomposition by $F$. Observe that under the straightening correspondence, $\Fun(\ZZ_{\geq 0}, C) \simeq \Fun_{/ B \NN}(\ZZ_{\geq 0}, \widehat{C})$ (since $C \simeq \ast \times_{B \NN} \widehat{C}$) and $F_{\ast}$ is encoded by the exponentiated cartesian fibration $(\widehat{C})^{\ZZ_{\geq 0}} \to B \NN$. Elaborating upon this, it is easily seen that the lax equalizer $\LEq_{\sh:F_{\ast}}(\Fun(\ZZ_{\geq 0}, C))$ is equivalent to the pullback of the diagram
\[ \begin{tikzcd}[row sep=4ex, column sep=4ex, text height=1.5ex, text depth=0.25ex]
& \Fun_{/B \NN}(\{1 < 0\} \times \ZZ_{\geq 0} , \widehat{C}) \ar{d}{(\ev_1, \ev_0)} \\
\Fun_{/ B \NN}(\ZZ_{\geq 0}, \widehat{C}) \ar{r}{(s^{\ast},\id)} & \Fun_{/B \NN}(\ZZ_{\geq 0}, \widehat{C}) \times \Fun_{/B \NN}(\ZZ_{\geq 0}, \widehat{C}),
\end{tikzcd} \]
since objects of that pullback are equivalent to diagrams
\[ \begin{tikzcd}[row sep=4ex, column sep=4ex, text height=1.5ex, text depth=0.25ex]
X_1 \ar{r}{\beta_0} \ar{d}{\alpha_1} & F(X_0) \ar{d}{F(\alpha_0)} \ar{r} & X_0 \ar{d}{\alpha_0} \\
X_2 \ar{r}{\beta_1} \ar{d}{\alpha_2} & F(X_1) \ar{r} \ar{d}{F(\alpha_1)} & X_1 \ar{d}{\alpha_1} \\
\vdots & \vdots & \vdots
\end{tikzcd} \]
where the labeled arrows are in $C$ and the right horizontal edges are cartesian in $\widehat{C}$ over $1 \in \NN$. Rather than give a complete account of the details, for the subsequent lemma let us abuse notation and instead \emph{define} the expression $\LEq_{\sh:F_{\ast}}(\Fun(\ZZ_{\geq 0}, C))$ to refer to this pullback.
\end{cnstr}

\begin{lem} \label{lem:LaxEqualizerGenericEquivalence} There is an equivalence of $\infty$-categories
\[ \chi: \LEq_{\id:\sh}(\Fun_{/B \NN}(\ZZ_{\geq 0}^{\op}, \widehat{C})) \simeq \LEq_{\sh:F_{\ast}}(\Fun(\ZZ_{\geq 0}, C)) \]
that restricts to an equivalence of $\infty$-categories
\[  \chi_0: \Eq_{\id:\sh}(\Fun_{/B \NN}(\ZZ_{\geq 0}^{\op}, \widehat{C})) \simeq \LEq_{\id:F}(C).  \]
\end{lem}

\begin{proof} Intuitively, the first equivalence $\chi$ exchanges diagrams
\[ \begin{tikzcd}[row sep=4ex, column sep=4ex, text height=1.5ex, text depth=0.25ex]
\cdots \ar{r}{\phi_2} & X_2 \ar{r}{\phi_1} \ar{d}{\alpha_2} & X_1 \ar{r}{\phi_0} \ar{d}{\alpha_1} & X_0 \ar{d}{\alpha_0} \\
\cdots \ar{r}{\phi_3} & X_3 \ar{r}{\phi_2} & X_2 \ar{r}{\phi_1} & X_1
\end{tikzcd} \]
with diagrams
\[ \begin{tikzcd}[row sep=4ex, column sep=4ex, text height=1.5ex, text depth=0.25ex]
X_1 \ar{r}{\phi_0} \ar{d}{\alpha_1} & X_0 \ar{d}{\alpha_0} \\
X_2 \ar{r}{\phi_1} \ar{d}{\alpha_2} & X_1 \ar{d}{\alpha_1} \\
\vdots & \vdots
\end{tikzcd} \]
with one such diagram uniquely determining the other. To make this idea precise, we need to introduce some auxiliary constructions. Given $n \geq 0$, define
\[ \sh^n: \LEq_{\id:\sh}(\Fun_{/B \NN}(\ZZ_{\geq 0}^{\op}, \widehat{C}))  \to \Fun_{/B \NN}(\ZZ_{\geq 0}^{\op} \times \{n<n+1\}, \widehat{C}) \] to be the composite of the projection to $\Fun_{/B \NN}(\ZZ_{\geq 0}^{\op} \times \{0<1\}, \widehat{C})$ and precomposition by the $n$-fold successor functor $\ZZ_{\geq 0}^{\op} \times \{n < n+1 \} \to \ZZ_{\geq 0}^{\op} \times \{0 <1 \} $, $(i,n+j) \mapsto (i+n,j)$. Then form the pullback
\[ \begin{tikzcd}[row sep=4ex, column sep=6ex, text height=1.5ex, text depth=0.25ex]
\LEq^{\infty}_{\id:\sh}(\Fun_{/B \NN}(\ZZ_{\geq 0}^{\op}, \widehat{C})) \ar{r} \ar{d}{\pi'} & \Fun_{/B \NN}(\ZZ_{\geq 0}^{\op} \times \ZZ_{\geq 0}, \widehat{C}) \ar{d} \\
\LEq_{\id:\sh}(\Fun_{/B \NN}(\ZZ_{\geq 0}^{\op}, \widehat{C})) \ar{r}{(\sh^k)} & \prod_{k=0}^{\infty} \Fun_{/B \NN}(\ZZ_{\geq 0}^{\op} \times \{k < k+1 \}, \widehat{C}).
\end{tikzcd} \]
where the lower right object is the iterated fiber product. The righthand vertical map is obtained via precomposition by the inclusion $\ZZ_{\geq 0}^{\op} \times \spine(\ZZ_{\geq 0}) \to \ZZ_{\geq 0}^{\op} \times \ZZ_{\geq 0}$, which is inner anodyne by Rmk.~\ref{rem:innerAnodyneSpineInclusion} and \cite[Prop.~2.3.2.4]{HTT}. Therefore, the vertical maps are trivial fibrations. Similarly, define 
\[ \sh^n: \LEq_{\sh:F_{\ast}}(\Fun(\ZZ_{\geq 0}, C)) \to \Fun_{/B \NN}( \{ n+1<n \} \times \ZZ_{\geq 0}, \widehat{C}) \]
as the composite of the projection to $ \Fun_{/B \NN}( \{ 1<0 \} \times \ZZ_{\geq 0}, \widehat{C})$ and precomposition by the $n$-fold successor functor $\{ n+1<n \} \times \ZZ_{\geq 0} \to \{ 1<0 \} \times \ZZ_{\geq 0}$, $(n+j,i) \mapsto (j,i+n)$. Form the pullback square
\[ \begin{tikzcd}[row sep=4ex, column sep=6ex, text height=1.5ex, text depth=0.25ex]
\LEq^{\infty}_{\sh:F_{\ast}}(\Fun(\ZZ_{\geq 0}, C)) \ar{r} \ar{d}{\pi''} & \Fun_{/B \NN}(\ZZ_{\geq 0}^{\op} \times \ZZ_{\geq 0}, \widehat{C}) \ar{d} \\
\LEq_{\sh:F_{\ast}}(\Fun(\ZZ_{\geq 0}, C)) \ar{r}{(\sh^k)} & \prod_{k=0}^{\infty} \Fun_{/B \NN}(\{k+1<k\} \times \ZZ_{\geq 0}, \widehat{C}).
\end{tikzcd} \]
The righthand vertical map is precomposition by the inner anodyne map $\spine(\ZZ_{\geq 0}^{\op}) \times \ZZ_{\geq 0} \to \ZZ_{\geq 0}^{\op} \times \ZZ_{\geq 0}$, so both vertical maps are trivial fibrations. Next, the product map $\spine(\ZZ_{\geq 0}^{\op}) \times \spine(\ZZ_{\geq 0}) \subset \ZZ_{\geq 0}^{\op} \times \ZZ_{\geq 0}$ is also inner anodyne, and via precomposition we get a trivial fibration
\[ \rho: \Fun_{/B \NN}(\ZZ_{\geq 0}^{\op} \times \ZZ_{\geq 0}, \widehat{C}) \to \Fun_{/B \NN}(\spine(\ZZ_{\geq 0}^{\op}) \times \spine(\ZZ_{\geq 0}), \widehat{C}). \]
Let $\sB$ be the full subcategory of $\Fun_{/B \NN}(\spine(\ZZ_{\geq 0}^{\op}) \times \spine(\ZZ_{\geq 0}), \widehat{C})$ on objects $X_{\bullet,\bullet}$ such that for all $m \geq 0$, $n >0$ we have that
\[ X|_{ \{n+1<n\} \times \{m<m+1\} } = X|_{ \{n<n-1\} \times \{m+1,m+2\} }. \]
By definition, objects of $\LEq^{\infty}_{\id:\sh}(\Fun_{/B \NN}(\ZZ_{\geq 0}^{\op}, \widehat{C}))$ are diagrams $X_{\bullet, \bullet}: \ZZ_{\geq 0}^{\op} \times \ZZ_{\geq 0} \to \widehat{C}$ such that for every $m \geq 0$,
\[ X_{\bullet,\bullet}|_{\ZZ_{\geq 0}^{\op} \times \{m\} } = X_{\bullet,\bullet}|_{\ZZ_{\geq 0}^{\op} \times \{0\}} \circ s^m \text{ and } X_{\bullet,\bullet}|_{\ZZ_{\geq 0}^{\op} \times \{m<m+1\} } = X_{\bullet,\bullet}|_{\ZZ_{\geq 0}^{\op} \times \{0<1\}} \circ s^m,\]
and similarly for $\LEq^{\infty}_{\sh:F_{\ast}}(\Fun(\ZZ_{\geq 0}, C))$. The conditions on edges are implied by those for squares in $\cB$, and the functor $\rho$ thereby restricts to trivial fibrations
\[ \rho', \rho'': \LEq^{\infty}_{\id:\sh}(\Fun_{/B \NN}(\ZZ_{\geq 0}^{\op}, \widehat{C})), \LEq^{\infty}_{\sh:F_{\ast}}(\Fun(\ZZ_{\geq 0}, C)) \to \sB.  \]
Choosing sections of the trivial fibrations $\pi', \rho''$ or $\pi'', \rho'$ then furnishes the equivalence $\chi$. 

For the second equivalence $\chi_0$, let $\Fun^{\simeq}_{/B \NN}(\ZZ_{\geq 0}, \widehat{C}) \subset \Fun_{/B \NN}(\ZZ_{\geq 0}, \widehat{C})$ be the full subcategory on those objects $X_{\bullet}: \ZZ_{\geq 0} \to C$ that send every edge to an equivalence, and form the pullback
\[ \begin{tikzcd}[row sep=4ex, column sep=4ex, text height=1.5ex, text depth=0.25ex]
\LEq_{\sh:F_{\ast}}(\Fun^{\simeq}(\ZZ_{\geq 0}, C)) \ar{r} \ar{d} & \Fun_{/B \NN}(\{1 < 0\} \times \ZZ_{\geq 0} , \widehat{C}) \ar{d}{(\ev_1, \ev_0)} \\
\Fun^{\simeq}_{/ B \NN}(\ZZ_{\geq 0}, \widehat{C}) \ar{r}{(s^{\ast},\id)} & \Fun_{/B \NN}(\ZZ_{\geq 0}, \widehat{C}) \times \Fun_{/B \NN}(\ZZ_{\geq 0}, \widehat{C}),
\end{tikzcd} \]
which defines $\LEq_{\sh:F_{\ast}}(\Fun^{\simeq}(\ZZ_{\geq 0}, C))$ as a full subcategory of $\LEq_{\sh:F_{\ast}}(\Fun(\ZZ_{\geq 0}, C))$. It follows from the definitions that $\chi$ restricts to an equivalence
\[ \chi_0': \Eq_{\id:\sh}(\Fun_{/B \NN}(\ZZ_{\geq 0}^{\op}, \widehat{C})) \simeq \LEq_{\sh:F_{\ast}}(\Fun^{\simeq}(\ZZ_{\geq 0}, C)).  \]

Let $P =( \{ 1 < 0\} \times \ZZ_{\geq -1}) \setminus \{ (0,-1)\} $ and regard it as over $B \NN$ via the projection to $\{1<0\}$. Note by Lem.~\ref{lem:posetPushoutViaFlatness} that the cofibration $\{ (1,-1) \rightarrow (1,0) \} \cup_{(1,0)} (\{1<0\} \times \ZZ_{\geq 0}) \to P$ is a categorical equivalence. Therefore, if we form the pullback
\[ \begin{tikzcd}[row sep=4ex, column sep=6ex, text height=1.5ex, text depth=0.25ex]
\LEq^+_{\sh:F_{\ast}}(\Fun(\ZZ_{\geq 0}, C)) \ar{r} \ar{d} & \Fun_{/B \NN}(P, \widehat{C}) \ar{d}{(\ev_1,\ev_0)} \\
\Fun_{/B \NN}(\ZZ_{\geq 0}, \widehat{C}) \ar{r}{(s^{\ast},\id)} & \Fun_{/B \NN}(\ZZ_{\geq -1}, \widehat{C}) \times \Fun_{/B \NN}(\ZZ_{\geq 0}, \widehat{C}),
\end{tikzcd} \]
precomposition by $\{ 1 < 0\} \times \ZZ_{\geq 0} \subset P$ induces a trivial fibration
\[ \xi: \LEq^+_{\sh:F_{\ast}}(\Fun(\ZZ_{\geq 0}, C)) \to \LEq_{\sh:F_{\ast}}(\Fun(\ZZ_{\geq 0}, C)). \]
Defining $\LEq^+_{\sh:F_{\ast}}(\Fun^{\simeq}(\ZZ_{\geq 0}, C))$ in a similar fashion, we also obtain a trivial fibration
\[ \xi_0: \LEq^+_{\sh:F_{\ast}}(\Fun^{\simeq}(\ZZ_{\geq 0}, C)) \to \LEq_{\sh:F_{\ast}}(\Fun^{\simeq}(\ZZ_{\geq 0}, C)), \]
which is obtained by restricting $\xi$.

We now observe that a functor $f: P \to \widehat{C}$ over $B \NN$ is a relative left Kan extension of its restriction to $P_0 = \{(1,-1) \to (0,0)\}$ if and only if it sends the edges $\{ (1,m) \to (1,m+k) \}$, $m \geq -1$ and $\{0,m) \to (0,m+k) \}$, $m \geq 0$ to equivalences, since each slice category $P_0 \times_P P_{/(i,m)}$ has $(i,-i)$ as a terminal object. Therefore, if we form the pullback
\[ \begin{tikzcd}[row sep=4ex, column sep=6ex, text height=1.5ex, text depth=0.25ex]
\LEq'_{\id:F}(C) \ar{r} \ar{d} & \Fun_{/P_0}(P_0, P_0 \times_{B \NN} \widehat{C}) \ar{d}{(\ev_1, \ev_0)} \\
C \ar{r}{(\id,\id)} & C \times C
\end{tikzcd} \]
the restriction functor induced by $P_0 \subset P$
\[ \LEq^+_{\sh:F_{\ast}}(\Fun^{\simeq}(\ZZ_{\geq 0}, C)) \to \LEq'_{\id:F}(C)\]
is a trivial fibration. Let us now write $\Delta^1 = P_0$ and $\sM = \Delta^1 \times_{B \NN} \widehat{C}$ for clarity. Since the source functor $\sO(\Delta^1) \to \Delta^1$ is the free cartesian fibration (\cite[Exm.~2.6 or Def.~7.5]{Exp2}) on the identity, we obtain a trivial fibration
\[ \Fun^{\cart}_{/\Delta^1}(\sO(\Delta^1), \sM) \to \Fun_{/\Delta^1}(\Delta^1, \sM). \]
Moreover, writing $\sO(\Delta^1) = [00<01<11]$, the square
\[ \begin{tikzcd}[row sep=4ex, column sep=8ex, text height=1.5ex, text depth=0.5ex]
\Fun^{\cart}_{/\Delta^1}(\sO(\Delta^1), \sM) \ar{r}{\ev|_{[00 <01]}} \ar{d}[swap]{(\ev_{00}, \ev_{11})} & \sO(C) \ar{d}{(\ev_0, \ev_1)} \\
C \times C \ar{r}{(\id,F)} & C \times C
\end{tikzcd} \]
is homotopy commutative. We thereby obtain an equivalence $\LEq'_{\id:F}(C) \simeq \LEq_{\id:F}(C)$. Chaining together the various equivalences above then produces the desired equivalence $\chi_0$.
\end{proof}

\begin{rem} The equivalence $\chi_0$ of Lem.~\ref{lem:LaxEqualizerGenericEquivalence} sends an object 
\[ \begin{tikzcd}[row sep=4ex, column sep=4ex, text height=1.5ex, text depth=0.25ex]
\cdots \ar{r}{\phi_2} & X_2 \ar{r}{\phi_1} \ar{d}{\alpha_2}[swap]{\simeq} & X_1 \ar{r}{\phi_0} \ar{d}{\alpha_1}[swap]{\simeq} & X_0 \ar{d}{\alpha_0}[swap]{\simeq} \\
\cdots \ar{r}{\phi_3} & X_3 \ar{r}{\phi_2} & X_2 \ar{r}{\phi_1} & X_1
\end{tikzcd} \]
 to the composite $X_0 \xto{\alpha_0} X_1 \xto{\beta_0} F(X_0)$, where we factor the edge $\phi_0$ through $\beta_0$ and a cartesian edge $F(X_0) \to X_0$.
\end{rem}


\subsection{Proof of the main theorem}

We have almost assembled all of the ingredients needed to prove Thm.~\ref{thm:MainTheoremEquivalenceBddBelow}. In fact, we will also reprove \cite[Thm.~II.6.3]{NS18} by way of illustrating the formal nature of our proof. In order to make effective use of Prop.~\ref{prp:EquivalenceOnBoundedBelowAtFiniteLevel} in the dihedral situation, we first establish the compatibility of the relative geometric locus construction with restriction and geometric fixed points (compare \ref{restrictionGeometricLoci} and \ref{GeometricFixedPointsGeometricLoci}). 

\begin{ntn} \label{ntn:ConciseRestrictionNotation} For $0 \leq k \leq n$, we have the inclusions $C_{p^k} \subset C_{p^n}$ and $D_{2 p^k} \subset D_{2 p^n}$. Let
\begin{align*} \res^n_k: \Sp^{C_{p^n}} \to \Sp^{C_{p^k}}, \quad & \res^n_k: \Fun(B C_{p^n} ,\Sp) \to \Fun(B C_{p^k} ,\Sp) \\
\res^n_k: \Sp^{D_{2 p^n}} \to \Sp^{D_{2 p^k}}, \quad & \res^n_k: \Fun_{C_2}(B^t_{C_2} \mu_{p^{n}}, \underline{\Sp}^{C_2}) \to \Fun_{C_2}(B^t_{C_2} \mu_{p^{k}}, \underline{\Sp}^{C_2})
\end{align*}
be alternative notation for the restriction functors.
\end{ntn}

\begin{vrn} \label{vrn:DihedralRestrictionGeometricLoci} For $0 \leq k \leq n$ and the inclusion $D_{2 p^k} \subset D_{2 p^n}$ given by $\mu_{p^k} \subset \mu_{p^n}$, we have a commutative diagram
\[ \begin{tikzcd}[row sep=4ex, column sep=4ex, text height=1.5ex, text depth=0.25ex]
\fS[D_{2p^k}] \ar{r}{i} \ar{d}{\zeta} & \fS[D_{2p^n}] \ar{d}{\zeta} \\
\Delta^k \ar{r}{i} & \Delta^n
\end{tikzcd} \]
where the bottom functor is the inclusion of $\Delta^k$ as a sieve. As in \ref{restrictionGeometricLoci}, the restriction functor $\res^n_k: \Sp^{D_{2 p^n}} \to \Sp^{D_{2 p^k}}$ induces a functor 
\[ \res^n_k: \Sp^{D_{2p^n}}_{\locus{\phi,\zeta}} \times_{\Delta^n} \Delta^k \to \Sp^{D_{2p^k}}_{\locus{\phi,\zeta}} \]
that on the fiber over $i \in \Delta^k$ is equivalent to the functor
\[ \res^n_k: \Fun_{C_2}(B^t_{C_2} \mu_{p^{n-i}}, \underline{\Sp}^{C_2}) \to \Fun_{C_2}(B^t_{C_2} \mu_{p^{k-i}}, \underline{\Sp}^{C_2}). \]
Precomposition by $i: \sd(\Delta^k) \to \sd(\Delta^n)$ and postcomposition by $\res^n_k$ yields the functor
\[ \res^n_k: \Fun^{\cocart}_{/\Delta^n}(\sd(\Delta^n), \Sp^{D_{2p^n}}_{\locus{\phi,\zeta}}) \to \Fun^{\cocart}_{/\Delta^k}(\sd(\Delta^k), \Sp^{D_{2p^k}}_{\locus{\phi,\zeta}}). \]
Furthermore, by the same argument as in \ref{restrictionGeometricLoci} we have a commutative diagram
\[ \begin{tikzcd}[row sep=4ex, column sep=6ex, text height=1.5ex, text depth=0.5ex]
\Fun^{\cocart}_{/\Delta^n}(\sd(\Delta^n), \Sp^{D_{2p^n}}_{\locus{\phi,\zeta}}) \ar{r}{\Theta[\zeta]} \ar{d}[swap]{\res^n_k} & \Sp^{D_{2p^n}} \ar{d}{\res^n_k} \\
\Fun^{\cocart}_{/\Delta^k}(\sd(\Delta^{k}), \Sp^{D_{2p^k}}_{\locus{\phi,\zeta}}) \ar{r}{\Theta[\zeta]} & \Sp^{D_{2p^k}}.
\end{tikzcd} \]
\end{vrn}

\begin{vrn} \label{vrn:DihedralGeometricFixedPoints} For $0 \leq k \leq n$ and the quotient homomorphism $D_{2p^n} \to D_{2p^n}/\mu_{p^k} \cong D_{2p^{n-k}}$, we have a commutative diagram of cosieve inclusions
\[ \begin{tikzcd}[row sep=4ex, column sep=4ex, text height=1.5ex, text depth=0.25ex]
\fS[D_{2p^{n-k}}] \ar{r}{i} \ar{d}{\zeta} & \fS[D_{2p^n}] \ar{d}{\zeta} \\
\Delta^{n-k} \ar{r}{i} & \Delta^n.
\end{tikzcd} \]
As in \ref{GeometricFixedPointsGeometricLoci}, $\Phi^{\mu_{p^k}}$ implements an equivalence
\[ \Sp^{D_{2p^n}}_{\locus{\phi,\zeta}} \times_{\Delta^n} \Delta^{n-k} \simeq \Sp^{D_{2p^{n-k}}}_{\locus{\phi,\zeta}} \]
with respect to which we write $i^{\ast}$ as
\[ \Phi^{\mu_{p^k}}: \Fun^{\cocart}_{/\Delta^n}(\sd(\Delta^n), \Sp^{D_{2p^n}}_{\locus{\phi,\zeta}}) \to \Fun^{\cocart}_{/\Delta^{n-k}}(\sd(\Delta^{n-k}), \Sp^{D_{2p^{n-k}}}_{\locus{\phi,\zeta}}). \]
We then obtain a commutative diagram
\[ \begin{tikzcd}[row sep=4ex, column sep=6ex, text height=1.5ex, text depth=0.5ex]
\Fun^{\cocart}_{/\Delta^n}(\sd(\Delta^n), \Sp^{D_{2p^n}}_{\locus{\phi,\zeta}}) \ar{r}{\Theta[\zeta]} \ar{d}[swap]{\Phi^{\mu_{p^k}}} & \Sp^{D_{2p^{n}}} \ar{d}{\Phi^{\mu_{p^k}}} \\
\Fun^{\cocart}_{/\Delta^{n-k}}(\sd(\Delta^{n-k}), \Sp^{D_{2p^{n-k}}}_{\locus{\phi,\zeta}}) \ar{r}{\Theta[\zeta]} & \Sp^{D_{2p^{n-k}}}.
\end{tikzcd} \]
\end{vrn}

We now consider an axiomatic setup that will handle the $C_{p^{\infty}}$ and $D_{2 p^{\infty}}$-situations simultaneously.

\begin{lem} \label{lm:joinLocallyCocartesian} Suppose $p: C \to S$ is a locally cocartesian fibration.
\begin{enumerate}
\item For any $\infty$-category $T$, $p': C \star T \to S \star T$ is a locally cocartesian fibration.
\item For any locally cocartesian fibration $D \to S$, the restriction functor implements an equivalence
\[ \Fun^{\cocart}_{/S \star T}(C \star T, D \star T) \xto{\simeq} \Fun^{\cocart}_{/S}(C, D). \]
\item Suppose $S \cong S_0 \star S_1$ and let $C_0 = S_0 \times_S C$. Then for any locally cocartesian fibration $D \to S_0$, the restriction functor implements an equivalence 
\[ \Fun^{\cocart}_{/S}(C,D \star S_1) \xto{\simeq} \Fun^{\cocart}_{/S_0}(C_0,D). \]
\item For any $\infty$-category $T$, the restriction functor
\[ j^{\ast}: \Fun^{\cocart}_{/S \star T}(\sd(S \star T), C \star T) \to \Fun^{\cocart}_{/S}(\sd(S), C) \]
is an equivalence of $\infty$-categories.
\end{enumerate}
\end{lem}
\begin{proof} For (1), first recall that the join is defined by the right Quillen functor $j_{\ast}: s\Set_{/ \partial \Delta^1} \to s\Set_{/\Delta^1}$ for the inclusion $j: \partial \Delta^1 \to \Delta^1$ (c.f. \cite[Dfn 4.1]{Exp2}). Therefore, given two categorical fibrations $X \to A$ and $Y \to B$, $X \star Y \to A \star B$ is a categorical fibration, so in particular $p'$ is a categorical fibration. It is clear that for any edge $e: \Delta^1 \to S \star T$ with image in $S$ or $T$ that the pullback over $e$ of $p'$ is a cocartesian fibration. Suppose $e$ is specified by $e(0) = s \in S$ and $e(1) = t \in T$. Then the pullback over $e$ equals $(C_s)^{\rhd} \to \Delta^1$, which is obviously cocartesian. Thus, $p'$ is locally cocartesian.

For (2), by definition of the join we actually have an isomorphism of simplicial sets
\[ \Fun_{/S \star T}(C \star T, D \star T) \cong \Fun_{/S}(C, D), \]
under which functors preserving locally cocartesian edges are identified with each other. (3) follows by the same argument.

For (4), note that the hypotheses of Prop.~\ref{prp:ExistenceLaxRightKanExtension} are satisfied because the zero category admits all limits, so any functor $F \in \Fun^{\cocart}_{/S \star T}(\sd(S \star T), C \star T)$ is necessarily a $(p \star \id_T)$-right Kan extension of its restriction to $\sd(S \star T)_0$. It follows that $j^{\ast}$ is an equivalence.
\end{proof}

\begin{cnstr} \label{cnstr:FamilyOfLocallyCocartesianFibrations} Suppose given a set $\{ p_n: C^n \to \Delta^n: n \geq 0\}$ of locally cocartesian fibrations, together with structure maps
\[ r_n: [0:n] \times_{\Delta^{n+1}} C^{n+1} \to C^n \]
over $\Delta^n \cong [0:n]$, where $r_n$ preserves locally cocartesian edges. Then, viewing $\Delta^n \subset \ZZ_{\geq 0}$ as the subcategory $[0:n]$, let
\[ r_n : C^{n+1} \star \ZZ_{>n+1} \to C^{n} \star \ZZ_{>n} \]
also denote the functor over $\ZZ_{\geq 0}$ obtained as in Lem.~\ref{lm:joinLocallyCocartesian}, and let
\[ C^{\infty} \coloneq \lim_n \left( C^{n} \star \ZZ_{>n} \right) \]
be the locally cocartesian fibration over $\ZZ_{\geq 0}$ with the inverse limit taken over the maps $r_n$.

Suppose further that for all $n \geq 0$, we have functors $i_n: C^{n} \to C^{n+1}$ over the cosieve inclusion $\Delta^{n} \cong [1:n+1] \sub \Delta^{n+1}$ that preserve locally cocartesian edges, such that the commutative square
\[ \begin{tikzcd}[row sep=4ex, column sep=4ex, text height=1.5ex, text depth=0.25ex]
C^{n} \ar{r}{i_n} \ar{d} & C^{n+1} \ar{d} \\
\Delta^{n} \ar{r} & \Delta^{n+1}
\end{tikzcd} \]
is a homotopy pullback, and for all $n>0$, the diagram
\[ \begin{tikzcd}[row sep=4ex, column sep=6ex, text height=1.5ex, text depth=0.5ex]
C^n \ar{r}{i_{n}} \ar{d}[swap]{r_{n-1}} & C^{n+1} \ar{d}{r_n} \\
C^{n-1} \star \{n\} \ar{r}{i_{n-1}} & C^n \star \{n\}
\end{tikzcd} \]
is homotopy commutative (where we denote the various extensions of maps $i_n$ and $r_n$ by the same symbols). By Lem.~\ref{lm:joinLocallyCocartesian}(4),
\[ \Fun^{\cocart}_{/ \ZZ_{\geq 0}}(\sd(\ZZ_{\geq 0}), C^n \star \ZZ_{>n}) \simeq \Fun^{\cocart}_{/\Delta^n}(\sd(\Delta^n), C^n) \]
under which the maps induced by postcomposing by $r_n$ are identified. Thus, we get that
\[ \Fun^{\cocart}_{/ \ZZ_{\geq 0}}(\sd(\ZZ_{\geq 0}), C^{\infty} ) \simeq \lim_n \Fun^{\cocart}_{/\Delta^n}(\sd(\Delta^n), C^n). \]

Under our assumptions, the diagram
\[ \begin{tikzcd}[row sep=4ex, column sep=4ex, text height=1.5ex, text depth=0.25ex]
\Fun^{\cocart}_{/\Delta^{n+1}}(\sd(\Delta^{n+1}),C^{n+1}) \ar{r}{(i_n)^{\ast}} \ar{d}{(r_n)_{\ast}} & \Fun^{\cocart}_{/\Delta^{n}}(\sd(\Delta^{n}),C^{n}) \ar{d}{(r_{n-1})_{\ast}} \\
\Fun^{\cocart}_{/\Delta^{n}}(\sd(\Delta^{n}),C^{n}) \ar{r}{(i_{n-1})^{\ast}} & \Fun^{\cocart}_{/\Delta^{n-1}}(\sd(\Delta^{n-1}),C^{n-1})
\end{tikzcd} \]
is homotopy commutative, so the maps $(i_n)^{\ast}$ assemble into a natural transformation
\[ i_{\bullet}^{\ast}: \Fun^{\cocart}_{/\Delta^{\bullet+1}}(\sd(\Delta^{\bullet+1}),C^{\bullet+1}) \to \Fun^{\cocart}_{/\Delta^{\bullet}}(\sd(\Delta^{\bullet}),C^{\bullet}). \]
Taking the inverse limit, we obtain an endofunctor $i_{\infty}^{\ast}$ of $\Fun^{\cocart}_{/\ZZ_{\geq 0}}(\sd(\ZZ_{\geq 0}), C^{\infty})$. On the other hand, the successor functor $s: \ZZ_{\geq 0} \to \ZZ_{\geq 0}$ induces a endofunctor $\sd(s)$ of $\sd(\ZZ_{\geq 0})$ that preserves locally cocartesian edges, and thus a `shift' endofunctor $\sh = \sd(s)^{\ast}$ of $\Fun^{\cocart}_{/\ZZ_{\geq 0}}(\sd(\ZZ_{\geq 0}), C^{\infty})$.
\end{cnstr}

\begin{lem} \label{lm:GeometricFixedPointsAsShift} We have an equivalence $\sh \simeq i^{\ast}_{\infty}$.
\end{lem}
\begin{proof} It suffices to check that for all $n \geq 0$, the diagram
\[ \begin{tikzcd}[row sep=4ex, column sep=4ex, text height=1.5ex, text depth=0.25ex]
\Fun^{\cocart}_{/\ZZ_{\geq 0}}(\sd(\ZZ_{\geq 0}), C^{\infty}) \ar{r}{\sh} \ar{d} & \Fun^{\cocart}_{/\ZZ_{\geq 0}}(\sd(\ZZ_{\geq 0}), C^{\infty}) \ar{d} \\
\Fun^{\cocart}_{/\Delta^{n+1}}(\sd(\Delta^{n+1}), C^{n+1}) \ar{r}{i_n^{\ast}} & \Fun^{\cocart}_{/\Delta^n}(\sd(\Delta^n), C^n)
\end{tikzcd} \]
is homotopy commutative. But this follows from the commutativity of the diagram
\[ \begin{tikzcd}[row sep=4ex, column sep=4ex, text height=1.5ex, text depth=0.25ex]
\Delta^n \ar{r} \ar{d} & \Delta^{n+1} \ar{d} \\
\ZZ_{\geq 0} \ar{r}{s} & \ZZ_{\geq 0}
\end{tikzcd} \]
where the upper map is the inclusion of $\Delta^n$ as the cosieve $[1:n+1] \subset \Delta^{n+1}$.
\end{proof}

Next, let $\sd_1(\ZZ_{\geq 0}) \subset \sd(\ZZ_{\geq 0})$ be the subposet on strings $[k]$ and $[k< k+1]$ as in Notn.~\ref{ntn:convexStringsLengthOne}, and let $$ \gamma_{\infty}^{\ast}: \Fun^{\cocart}_{/\ZZ_{\geq 0}}(\sd(\ZZ_{\geq 0}), C^{\infty}) \to \Fun^{\cocart}_{/\ZZ_{\geq 0}}(\sd_1(\ZZ_{\geq 0}), C^{\infty}) $$ be the functor given by restriction along the inclusion. Parallel to the setup in \ref{dualizedDescriptionOfSpine}, let $t_{\bullet}: \ZZ_{\geq 0} \to \Cat_{\infty}$ be the functor that sends $n$ to the fiber $C^{\infty}_n$ and $[n \to n+1]$ to the pushforward functor $t^{n+1}_n: C^{\infty}_n \to C^{\infty}_{n+1}$, and let $(C^{\infty})^{\vee} \to \ZZ_{\geq 0}^{\op}$ be the cartesian fibration classified by $t_{\bullet}$. Then we may replace the codomain of $\gamma_{\infty}^{\ast}$ as in Prop.~\ref{prp:DualDescriptionOfSections} to instead write
$$ \gamma_{\infty}^{\ast}: \Fun^{\cocart}_{/\ZZ_{\geq 0}}(\sd(\ZZ_{\geq 0}), C^{\infty}) \to \Fun_{/\ZZ_{\geq 0}^{\op}}(\ZZ_{\geq 0}^{\op}, (C^{\infty})^{\vee}). $$

The functor $\gamma^{\ast}_{\infty}$ clearly commutes with the shift functor $\sh$ defined as $\sd(s)^{\ast}$ on the left and $(s^{\op})^{\ast}$ on the right, so we obtain a functor between the equalizers
\[ \Eq_{\id:\sh}(\Fun^{\cocart}_{/\ZZ_{\geq 0}}(\sd(\ZZ_{\geq 0}), C^{\infty})) \to \Eq_{\id:\sh}(\Fun_{/\ZZ_{\geq 0}^{\op}}(\ZZ_{\geq 0}^{\op}, (C^{\infty})^{\vee})).\]

Note that under our assumptions, for all $0 \leq k \leq n$ we have equivalences
\[ \begin{tikzcd}[row sep=4ex, column sep=6ex, text height=1.5ex, text depth=0.5ex]
C^{n+1}_k \ar{d}{r_{n,k}} \ar{r}{\simeq} & C^{n}_{k-1} \ar{r}{\simeq} \ar{d}{r_{n-1,k-1}} & \cdots \ar{r}{\simeq} & C^{n-k+2}_{1} \ar{r}{\simeq} \ar{d}{r_{n-k+1,1}} & C^{n-k+1}_0 \ar{d}{r_{n-k,0}} \\
C^n_k \ar{r}{\simeq} &  C^{n-1}_{k-1} \ar{r}{\simeq} & \cdots \ar{r}{\simeq} & C^{n-k+1}_{1} \ar{r}{\simeq} & C^{n-k}_0,
\end{tikzcd} \]
hence we have equivalences $C^{\infty}_{n+1} \simeq C^{\infty}_{n}$ for all $n \geq 0$, under which $t^{n+2}_{n+1} \simeq t^{n+1}_n$. Therefore, if we let $C = C^{\infty}_0$ and $F = t^1_0: C \to C^{\infty}_1 \simeq C$ as an endofunctor of $C$, then with $\widehat{C} \to B \NN$ defined as in Constr.~\ref{cnstr:SetupForEndofunctorFibration}, we have a homotopy pullback square
\[ \begin{tikzcd}[row sep=4ex, column sep=4ex, text height=1.5ex, text depth=0.25ex]
(C^{\infty})^{\vee} \ar{r} \ar{d} & \widehat{C} \ar{d} \\
\ZZ_{\geq 0}^{\op} \ar{r} & B \NN
\end{tikzcd} \]
and hence $\Fun_{/\ZZ_{\geq 0}^{\op}}(\ZZ_{\geq 0}^{\op}, (C^{\infty})^{\vee}) \simeq \Fun_{/B \NN}(\ZZ_{\geq 0}^{\op}, \widehat{C})$. Lem.~\ref{lem:LaxEqualizerGenericEquivalence} then implies the equivalence
\[ \Eq_{\id:\sh}(\Fun_{/\ZZ_{\geq 0}^{\op}}(\ZZ_{\geq 0}^{\op}, (C^{\infty})^{\vee})) \simeq \LEq_{\id:F}(C). \]
We thereby obtain the `generic' comparison functor
\begin{equation} \label{eqn:AbstractComparisonFunctor} \Eq_{\id:\sh}(\Fun^{\cocart}_{/\ZZ_{\geq 0}}(\sd(\ZZ_{\geq 0}), C^{\infty})) \to \LEq_{\id:F}(C).
\end{equation}

Let us now return to our two situations of interest. In Constr.~\ref{cnstr:FamilyOfLocallyCocartesianFibrations}, we may take either
\begin{enumerate}
\item $C^n = \Sp^{C_{p^n}}_{\locus{\phi}}$, the maps $r_n$ as in Constr.~\ref{restrictionGeometricLoci}, and the maps $i_n$ as in Constr.~\ref{GeometricFixedPointsGeometricLoci}.
\item $C^n = \Sp^{D_{2p^n}}_{\locus{\phi,\zeta}}$, the maps $r_n$ as in Var.~\ref{vrn:DihedralRestrictionGeometricLoci}, and the maps $i_n$ as in Var.~\ref{vrn:DihedralGeometricFixedPoints}.
\end{enumerate}

Let $\Sp^{C_{p^{\infty}}}_{\locus{\phi}}$ and $\Sp^{D_{2p^{\infty}}}_{\locus{\phi,\zeta}}$ be the resulting inverse limits as locally cocartesian fibrations over $\ZZ_{\geq 0}$, so we have equivalences
\begin{align*} \Theta: \Fun^{\cocart}_{/\ZZ_{\geq 0}}(\sd(\ZZ_{\geq 0}), \Sp^{C_{p^{\infty}}}_{\locus{\phi}}) \xto{\simeq} \Sp^{C_{p^{\infty}}}, \\
\Theta[\zeta]: \Fun^{\cocart}_{/\ZZ_{\geq 0}}(\sd(\ZZ_{\geq 0}), \Sp^{D_{2p^{\infty}}}_{\locus{\phi,\zeta}}) \xto{\simeq} \Sp^{D_{2p^{\infty}}}.
\end{align*}

By the identification of $\Phi^{C_p}$, resp. $\Phi^{\mu_p}$ as $i_n^{\ast}$ as observed in Constr.~\ref{GeometricFixedPointsGeometricLoci} and Var.~\ref{vrn:DihedralGeometricFixedPoints}, together with Lem.~\ref{lm:GeometricFixedPointsAsShift}, we may identify the endofunctors $\Phi^{C_p}$ of $\Sp^{C_{p^{\infty}}}$ and $\Phi^{\mu_p}$ of $\Sp^{D_{2p^{\infty}}}$ with the shift endofunctors under the equivalences $\Theta$ and $\Theta[\zeta]$. Consequently, we obtain equivalences\footnote{Here, we implicitly use the equivalence of Rmk.~\ref{rem:EqualizerSwap}.}
\begin{align*} \Theta: \Eq_{\id:\sh}(\Fun^{\cocart}_{/\ZZ_{\geq 0}}(\sd(\ZZ_{\geq 0}), \Sp^{C_{p^{\infty}}}_{\locus{\phi}})) &\xto{\simeq} \CycSp_p^{\mr{gen}}, \\
\Theta[\zeta]: \Eq_{\id:\sh}(\Fun^{\cocart}_{/\ZZ_{\geq 0}}(\sd(\ZZ_{\geq 0}), \Sp^{D_{2p^{\infty}}}_{\locus{\phi,\zeta}})) &\xto{\simeq} \RCycSp_p^{\mr{gen}}.
\end{align*}

Moreover, being defined as the inverse limit over the restriction functors, the fibers of $\Sp^{C_{p^{\infty}}}_{\locus{\phi}}$ and $\Sp^{D_{2p^{\infty}}}_{\locus{\phi,\zeta}}$ are $\Fun(B C_{p^{\infty}}, \Sp)$ and $\Fun_{C_2}(B^t_{C_2} \mu_{p^{\infty}}, \underline{\Sp}^{C_2})$, and the pushforward endofunctors are $t^{C_p}$ and $t_{C_2} \mu_p$. Choosing inverses to $\Theta$ and $\Theta[\zeta]$, the functor (\ref{eqn:AbstractComparisonFunctor}) then yields comparison functors
\begin{align*} \sU: \CycSp^{\mr{gen}}_p \to \CycSp_p, \\
\sU_{\RR}: \RCycSp^{\mr{gen}}_p \to \RCycSp_p.
\end{align*}

In general, for a functor $X: \sd(\ZZ_{\geq 0}) \to C^{\infty}$ over $\ZZ_{\geq 0}$, let $X_n \in C^{\infty}_n$ be the object given by evaluating $X$ at $n$. Then we may describe $\sU$ and $\sU_{\RR}$ on objects by the formulas
\begin{align*} \sU(X, X \simeq \Phi^{C_p} X) = (X_0, X_0 \simeq X_1 \to (X_0)^{t C_p}), \\
\sU_{\RR}(X,X \simeq \Phi^{\mu_p} X) = (X_0, X_0 \simeq X_1 \to (X_0)^{t_{C_2} \mu_p}).
\end{align*}

It is then clear that $\sU$ is equivalent to the functor of \cite[Prop.~II.3.2]{NS18} considered by Nikolaus and Scholze, and $\sU_{\RR}$ is equivalent to the functor defined at the beginning of this section. The main motivation behind our somewhat roundabout reformulation of the comparison functors is to leverage Prop.~\ref{prp:EquivalenceOnBoundedBelowAtFiniteLevel} to prove analogous statements for $\sU$ and $\sU_{\RR}$. 

\begin{thm} \label{thm:MainTheoremRestated} $\sU$ and $\sU_{\RR}$ restrict to equivalences on the full subcategories of bounded below, resp. underlying bounded below objects.
\end{thm}
\begin{proof} Let $\widehat \Sp{}^{h C_{p^{\infty}}} \to B \NN$ and $\widehat \Sp {}^{h_{C_2} \mu_{p^{\infty}}} \to B \NN$ be the cartesian fibrations classified by the endofunctors $t C_p$ on $\Fun(B C_{p^{\infty}}, \Sp)$ and $t_{C_2} \mu_p$ on $\Fun_{C_2}(B_{C_2} \mu_{p^{\infty}}, \underline{\Sp}^{C_2})$. By taking the inverse limit of the equivalences of Prop.~\ref{prp:EquivalenceOnBoundedBelowAtFiniteLevel}, we obtain equivalences
\begin{align*} \Fun^{\cocart}_{/\ZZ_{\geq 0}}(\sd(\ZZ_{\geq 0}), \Sp^{C_{p^{\infty}}}_{\locus{\phi}})_{bb} & \xto{\simeq} \Fun_{/B \NN}(\ZZ_{\geq 0}^{\op},\widehat \Sp{}^{h C_{p^{\infty}}})_{bb}, \\
\Fun^{\cocart}_{/\ZZ_{\geq 0}}(\sd(\ZZ_{\geq 0}), \Sp^{D_{2p^{\infty}}}_{\locus{\phi,\zeta}})_{ubb} & \xto{\simeq} \Fun_{/B \NN}(\ZZ_{\geq 0}^{\op},\widehat \Sp {}^{h_{C_2} \mu_{p^{\infty}}})_{ubb}.
\end{align*}
The functors $\sU$ and $\sU_{\RR}$ are induced by these functors through taking equalizers of the identity and shift functors on both sides, so the theorem follows.
\end{proof}

\subsection{\texorpdfstring{$\TCR$}{TCR} of the constant mod \texorpdfstring{$p$}{p} Mackey functor at an odd prime}

Throughout this subsection, we fix an odd prime $p$ and implicitly take all functors to be $p$-typical. We apply the fiber sequence of Cor.~\ref{cor:decategorifiedEasy} to calculate $\TCR^{\mr{gen}}(H\mfp)$. Note that although we have not provided an intrinsic construction of the real topological Hochschild homology $\THR$ of an $E_{\sigma}$-algebra, the model of $\THR$ as a $O(2)$-cyclotomic spectrum given in \cite{Hog16} possesses the same $R$ and $F$ maps as defined in \S\ref{section:genRCycSp}, which suffices to make sense of $\TCR^{\mr{gen}}$ of an $E_{\sigma}$-algebra when \emph{defined} via any of the equivalent fiber sequences of Prop.~\ref{prp:fiberSequenceGenuineRealCyc}, and the formula of Cor.~\ref{cor:decategorifiedEasy} works independently of one's choice of foundations.

\begin{cvn} Let $S^{s,t} = S^{t\sigma} \wedge S^{s-t}$, and for a $C_2$-spectrum $X$, let $\pi_{s,t}^{C_2}(X)$ denote its $RO(C_2)$-graded homotopy groups. 
\end{cvn}

We begin with the identification of the $C_2$-equivariant homotopy type of $\THR(H\mfp)$ by work of Dotto-Moi-Patchkoria-Reeh, which holds for all primes $p$. 

\begin{thm}\cite[Thm.~5.18]{DMPR17}
There is a stable equivalence of $C_2$-equivariant ring spectra
$$T_{H\FF_p}(S^{2,1}) \overset{\simeq}{\to} \THR(H\mfp)$$
where $T_{H\FF_p}(S^{2,1}) = \bigvee_{n=0}^\infty \Sigma^{2n,n} H\FF_p$ is the free associative $H\FF_p$-algebra on $S^{2,1}$. In particular, there is an isomorphism of bigraded rings 
$$H\mfp_{{**}}[\tilde{x}] \cong \pi_{{**}}^{(-)}\THR(H\mfp)$$
where $|\tilde{x}| = (2,1)$. 
\end{thm}

Recall the parametrized Tate spectral sequence defined in \cite[\S 2.4]{Qui19b} (see also \cite[\S 20]{GreenleesMay} and \cite[\S 3]{mathew2019} for the more general case). This is an $RO(C_2)$-graded Mackey functor spectral sequence of the form
$$\widehat{E}^2_{*,{**}} = \widehat{H}^{*}_{\Gamma_{\mu_{p^\infty}}}(\mu_{p^\infty}; \pi_{{**}}^{(-)}(X)) \Rightarrow \pi_{{**}}^{(-)}X^{t_{C_2}\mu_{p^\infty}}$$
where the left-hand side is Amitsur-Dress-Tate cohomology taken with respect to the $\mu_{p^{\infty}}$-free family $\Gamma_{\mu_{p^\infty}}$ of subgroups of $D_{2p^\infty}$. Similarly, we have the parametrized homotopy fixed point spectral sequence
$$E^2_{*,{**}} = H^{*}_{\Gamma_{\mu_{p^\infty}}}(\mu_{p^\infty}; \pi_{{**}}^{(-)}(X)) \Rightarrow \pi_{{**}}^{(-)}X^{h_{C_2}\mu_{p^\infty}}.$$
The $d_r$-differentials in both spectral sequences change tridegrees by $(r,r-1,\lfloor r/2 \rfloor)$.

When $\mu_{p^\infty}$ acts trivially on $\pi_{**}^{(-)}(X)$, the $E_2$-pages of the parametrized Tate and homotopy fixed point spectral sequences can be computed using the following lemma. 

\begin{lem}
Let $p$ be an odd prime. There are isomorphisms
$$\widehat{H}^{**}_{\Gamma_{\mu_{p^\infty}}}(\mu_{p^\infty};\mfp) \cong (H\mfp)^{**}[[\tilde{t}^{\pm 1}]],$$
$$H^{**}_{\Gamma_{\mu_{p^\infty}}}(\mu_{p^\infty};\mfp) \cong (H\mfp)^{**}[[\tilde{t}]],$$
where $|\tilde{t}| = (-2,-1)$. 
\end{lem}

\begin{proof}
Recall that for $p$ odd, we have $\pi_{**}^{C_2}(H\mfp) \cong \FF_p[\tau^2]$ with $|\tau^2| = (0,-2)$ \cite[Prop.~1.1]{Sta16}. Since the action of $\mu_{p^\infty}$ is restricted from the action of $S^1$, we see that $\mu_{p^\infty}$ acts trivially on $\pi_{**}^{C_2}(H\mfp)$ for dimension reasons. We therefore have $$H^{**}_{\Gamma_{\mu_{p^\infty}}}(\mu_{p^\infty};\mfp) \cong H\mfp^{**}(B_{C_2}^t \mu_{p^\infty}) \cong H\mfp^{**}[[\tilde{t}]],$$ where the second isomorphism follows from \cite{Ara79} and the fact that $H\mfp$ is Real-oriented. 
\end{proof}

\begin{rem}
This lemma completely determines the parametrized Tate and homotopy fixed point spectral sequences for $H\mfp^{t_{C_2}\mu_{p^\infty}}$ and $H\mfp^{h_{C_2}\mu_{p^\infty}}$. In particular, we see that their $E^2$-pages are given by the respective Amitsur-Dress-Tate cohomology groups and that the spectral sequences collapse for bidegree reasons. 
\end{rem}

\begin{rem}
A similar calculation shows that
$$\widehat{H}^{**}_{\Gamma_{\mu_p}}(\mu_p;\mfp) \cong (H\mfp)^{**}[[\tilde{t}^{\pm 1}]][u]/(u^2),$$
$${H}^{**}_{\Gamma_{\mu_p}}(\mu_p;\mfp) \cong (H\mfp)^{**}[[\tilde{t}]][u]/(u^2),$$
where $|\tilde{t}| = (-2,-1)$ and $|u| = (-1,-1)$ -- here, one may use the fibration 
\[
B_{C_2}\mu_p \to B_{C_2}\mu_{p^\infty} \xto{\cdot p} B_{C_2}\mu_{p^\infty}
\]
to calculate the cohomology of $B_{C_2}\mu_{p^\infty}$. As above, this completely determines that parametrized Tate and homotopy fixed point spectral sequences for $H\mfp^{t_{C_2}\mu_p}$ and $H\mfp^{h_{C_2}\mu_p}$. 
\end{rem}

\begin{dfn} Let $X$ be a $C_2$-spectrum with twisted $\mu_{p^{\infty}}$-action. The \emph{$p$-typical real topological negative cyclic homology} of $X$ is
\[ \TCR^-(X,p) = X^{h_{C_2} \mu_{p^{\infty}}}. \]
The \emph{$p$-typical real topological periodic cyclic homology} of $X$ is
\[ \TPR(X,p) = X^{t_{C_2} \mu_{p^{\infty}}}. \]
\end{dfn}

\begin{prp}
Let $p$ be an odd prime. There are isomorphisms of bigraded rings 
$$\pi_{{**}}(\TCR^-(H\mfp)) \cong H\mfp^{**}[[\tilde{t}]][\tilde{x}],$$
$$\pi_{{**}}(\TPR(H\mfp)) \cong H\mfp^{**}[[\tilde{t},\tilde{t}^{-1}]][\tilde{x}],$$
where $|\tilde{t}| = (-2,-1)$ and $|\tilde{x}| = (2,1)$. 
\end{prp}

\begin{proof}
The $E_2$-page of the parametrized homotopy fixed point spectral sequence has the form
$$E^2_{***} = H^*_{\Gamma_{\mu_{p^\infty}}}(\mu_{p^\infty}; \pi_{**}^{(-)}(\THR(H\mfp))) \cong H^*_{\Gamma_{\mu_{p^\infty}}}(\mu_{p^\infty};\mfp_{**}) \otimes \pi_{**}(\THR(H\mfp)) \cong H\mfp^{**}[[\tilde{t}]] \otimes \mfp_{**}[\tilde{x}]$$
with $|\tilde{t}| = (-2,0,-1)$ and $|\tilde{x}| = (0,2,1)$, where we  again use that the $\mu_{p^\infty}$-action is obtained by restriction from the necessarily trivial $S^1$-action on $\pi_{**}^{(-)}(\THR(H\mfp))$. The projection of the $E^2$-page onto the first two degrees gives the familiar checkerboard pattern, so the spectral sequence collapses. The same argument applies to show that the parametrized Tate spectral sequence collapses at $\widehat{E}^2$. 

To determine the multiplicative structure, we apply the functor $\res^{C_2}: \Sp^{C_2} \to \Sp$. An argument as above using the homotopy fixed point and Tate spectral sequences (c.f. the discussion around \cite[Prop.~IV.4.6]{NS18} and \cite[Cor.~IV.4.8]{NS18}) produces collapsing spectral spectral sequences
$$E^2_{**} \cong E^\infty_{**} \cong  H^*(\mu_{p^\infty};\pi_*(\THH(H\FF_p))) \cong H\FF_p^*[[t]][x]$$
 converging to $\pi_*(\TC^-(H\FF_p))$ and $\pi_*(\TP(H\FF_p))$, with $|t| = -2$ and $|x| = 2$. The functor $\res^{C_2}$ induces multiplicative maps of spectral sequences which are determined by $\tilde{t} \mapsto t$ and $\tilde{x} \mapsto x$. 

By \cite[Prop.~IV.4.6]{NS18}, the class $tu$ detects $ p \in \ZZ_p \cong \pi_0(\TC^-(H\FF_p))$ and $tu$ detects $ p \in \ZZ_p \cong \pi_0(\TP(H\FF_p))$. In particular, $(tu)^i \neq 0$ for all $i \geq 0$. We must therefore have $(\tilde{t} \tilde{u})^i \neq 0$ for all $i \geq 0$, which proves the desired isomorphism. 
\end{proof}

\begin{cor}
For all $i \in \ZZ$ and $k \geq 0$, the map
$$\pi^{C_2}_{2i,i-2k} (\varphi^{h_{C_2}\mu_{p^\infty}}) : \pi^{C_2}_{2i,i-2k}(\TCR^-(H\mfp)) \to \pi^{C_2}_{2i,i-2k}(\TPR(H\mfp))$$
is injective. If $i \geq 0$, it is an isomorphism, while if $i = -j < 0$, the map is given by multiplication by $p^j$.
\end{cor}

\begin{proof} We follow the proof of \cite[Prop.~IV.4.9]{NS18}. The map $\pi_{2i,i-2k}^{C_2} \varphi^{h_{C_2}\mu_{p^\infty}}$ is multiplicative using the lax monoidality of the parametrized Tate construction.\footnote{For this computation, we only need the structure of an associative $C_2$-equivariant ring spectrum on $\THR$ \cite[Cor.~4.2]{DMPR17}.} Since $\tilde{t}\tilde{u} = p$, the maps must be injective and they are isomorphisms either in positive or in negative degrees. Assume they are isomorphisms negative degrees and consider the diagram
\[
\begin{tikzcd}
\pi_{2i,i-2k}^{C_2} \THR(H\mfp)^{h_{C_2}\mu_{p^\infty}} \arrow{rrr}{\pi_{2i,i-2k}^{C_2}\varphi^{h_{C_2}\mu_{p^\infty}}} \arrow{d} & && \pi_{2i,i-2k}^{C_2} \THR(H\mfp)^{t_{C_2}\mu_{p^\infty}} \arrow{d} \arrow{r} & \pi_{2i,i-2k}^{C_2}H\mfp^{t_{C_2}\mu_{p^\infty}} \arrow{d} \\
\pi_{2i,i-2k}^{C_2}\THR(\mfp) \arrow{rrr}{\pi_{2i,i-2k}^{C_2}\varphi} &&& \pi_{2i,i-2k}^{C_2} \THR(H\mfp)^{t_{C_2}\mu_p} \arrow{r} &\pi_{2i,i-2k}^{C_2} H\mfp^{t_{C_2}\mu_p}
\end{tikzcd}
\]
where $i<0$. The lower left-hand corner is zero by \cite[Thm.~5.18]{DMPR17}, so if the upper-left arrow is surjective, then the map
$$\ZZ_p \cdot v \tau^{2k} = \pi_{2i,i-2k}^{C_2} \THR(H\mfp)^{t_{C_2}\mu_{p^\infty}} \to \pi_{2i,i-2k}^{C_2} H\mfp^{t_{C_2}\mu_p} \cong \FF_p$$
must be zero. But $\tilde{t}$ maps to a nonzero class in $\pi_{2i,i-2k}^{C_2}H\mfp^{t_{C_2}\mu_{p^\infty}}$, and this class maps to a nonzero element in $\pi_{2i,i-2k}^{C_2}H\mfp^{t_{C_2}\mu_p}$. 
\end{proof}

Using the long exact sequence associated to the fiber sequence in Cor.~\ref{cor:decategorifiedEasy}, we obtain the following calculation. We note that Thm.~\ref{thm:TCROddPrimeComputation} has also been obtained in forthcoming work of Dotto-Moi-Patchkoria, using different methods.

\begin{thm} \label{thm:TCROddPrimeComputation} Let $p$ be an odd prime. The $C_2$-equivariant homotopy groups of $\TCR^{\mr{gen}}(H\mfp,p)$ are given by
$$\pi_{**}^{(-)}(\TCR^{\mr{gen}}(H\mfp),p) \cong \pi_{**}^{(-)}(\underline{\ZZ_p} \vee \Sigma^{-1,0} \underline{\ZZ_p}).$$
\end{thm}

\appendix

\section{Parametrized \texorpdfstring{$\infty$}{infinity}-categories}
\label{section:ParamTerminology}

Let $S$ be an $\infty$-category and let $T = S^{\op}$. In this appendix, we fix terminology and recall a few basic results concerning the theory of $S$-$\infty$-categories and $S$-(co)limits from \cite{Exp1}, \cite{Exp2}, and \cite{Exp4} -- more involved usage of these concepts will be recalled as needed in the main body of the paper.

\begin{dfn}[{\cite{Exp1}}] An \emph{$S$-$\infty$-category} $p: C \to S$ is a cocartesian fibration. We often write only $C$ for the $S$-$\infty$-category, leaving the structure map implicit. Given a morphism $\alpha: V \to W \in T$, we write $\alpha^{\ast}: C_W \to C_V$ for the cocartesian pushforward functor and also refer to this as the \emph{restriction} functor along $\alpha$.

An $S$-$\infty$-category $C$ is an \emph{$S$-space} if the structure map $p$ is a left fibration. The \emph{corepresentable} $S$-spaces are the left fibrations $S^{V/} \to S$ for objects $V \in S$. A \emph{$S$-functor} $F: C \to D$ is a functor over $S$ that preserves cocartesian edges. We write $$\Fun_S(C,D) \coloneq \Fun^{\cocart}_{/S}(C,D)$$ for the $\infty$-category of $S$-functors from $C$ to $D$. We write $\ul{\Fun}_S(C,D)$ for the $S$-$\infty$-category defined by the universal mapping property
$$ \Fun_S(E,\ul{\Fun}_S(C,D)) \simeq \Fun_S(E \times_S C, D).  $$
Let $C_{\underline{V}} = S^{V/} \times_S C$ be notation for the pullback as a $S^{V/}$-$\infty$-category. If we let $E = S^{V/}$, then we compute the fiber of $\ul{\Fun}_S(C,D)$ over $V$ as
$$ \ul{\Fun}_S(C,D)_V \simeq \Fun_{S^{V/}}(C_{\ul{V}}, D_{\ul{V}}). $$

In \cite[\S 9]{Exp1} and \cite[\S 3]{Exp2}, the second author gave an explicit construction of $\ul{\Fun}_S(C,D) \to S$ as a marked simplicial set.

The \emph{$\infty$-category of $S$-$\infty$-categories} is $\Cat^{\cocart}_{\infty/S}$, the subcategory of the overcategory $\Cat_{\infty/S}$ on the $S$-$\infty$-categories and $S$-functors. The straightening correspondence furnishes a canonical equivalence
\[ \Cat^{\cocart}_{\infty/S} \simeq \Fun(S,\Cat_{\infty}) \]
that is natural in $S$ \cite[\S 3.2]{HTT}. The construction $\ul{\Fun}_S(-,-)$ is then the internal hom in $\Cat^{\cocart}_{\infty/S}$.
\end{dfn}

\begin{rem} If $S = \sO^{\op}_G$, then we typically write $G$-$\infty$-category, $G$-functor, etc. instead of $\sO^{\op}_G$-$\infty$-category, $\sO^{\op}_G$-functor etc. This convention also applies to all other constructions discussed in this section, e.g., $G$-(co)limit instead of $S$-(co)limit. Moreover, using the equivalence $\sO^{\op}_H \simeq (\sO^{\op}_G)^{(G/H)/}$ (Rmk.~\ref{rem:sliceCategoryPassage}), we write $H$-$\infty$-category, etc. instead of $(\sO^{\op}_G)^{(G/H)/}$-$\infty$-category, etc. Note that passage to slice categories in the theory of parametrized $\infty$-categories conceptually plays the same role as restriction to subgroups in equivariant homotopy theory.
\end{rem}

\begin{rem} In \cite{Exp4} and \cite{Exp1}, the authors speak of $T$-$\infty$-categories, etc.
\end{rem}

\begin{dfn}[{\cite[\S 4]{Exp2}}] For two $S$-$\infty$-categories $K,L$, the \emph{$S$-join} $$K \star_S L \to S \times \Delta^1$$ is the $S$-$\infty$-category defined as a simplicial set by the universal mapping property
$$\Hom_{/S \times \Delta^1}(A, K \star_S L) \cong \Hom_{/S}(A_0,K) \times \Hom_{/S}(A_1,L).$$
The $S$-join respects base-change in the variable $S$. In particular, for all $V \in S$ we have an isomorphism
$$ K \star_S L \cong K_V \star L_V. $$
\end{dfn}

Let $C$ be a $S$-$\infty$-category.

\begin{dfn}[{\cite[\S 5]{Exp2}}] A $S$-functor $\sigma: S \to C$ is an \emph{$S$-initial object} if for all $V \in S$, $\sigma(V) \in C_V$ is an initial object. Note  that all of the restriction functors then necessarily preserve initial objects.

More generally, let $K$ be a $S$-$\infty$-category, and let $\overline{F}: K \star_S S \to C$ be a $S$-functor that extends $F: K \to C$. Then $\overline{F}$ is a \emph{$S$-colimit diagram} if the corresponding $S$-functor
$$ (\id, \sigma_{\overline{F}}): S \to S \times_{\sigma_F, \ul{\Fun}_S(K,C)} \ul{\Fun}_S(K \star_S S,C)  $$
is an $S$-initial object. Here, we use the equivalence $\Fun_S(S, \ul{\Fun}_S(K,C)) \simeq \Fun_S(K,C)$ to write the cocartesian section. We then say that $\overline{F}|_S$ is a \emph{$S$-colimit} of $F$. If $S$ has an initial object $V_0$ (e.g., $V_0 = G/G$ in $\sO^{\op}_G$), then we also say that $\overline{F}|_S(V_0) \in C_{V_0}$ is a $S$-colimit of $F$ and write $$\mr{colim}{}^S F = \overline{F}|_S(V_0).$$
These definitions dualize in an obvious way, so we may consider $S$-final objects, $S$-limit diagrams $\overline{F}: S \star_S K \to C$, etc.
\end{dfn}

\begin{rem} The concept of an $S$-(co)limit is classically known in category theory as an \emph{indexed (co)limit}. In equivariant homotopy theory, Dotto and Moi have also studied homotopy $G$-(co)limits using model-categorical techniques \cite{dottomoi}.
\end{rem}

\begin{exm} Suppose $K = L \times S$ is a constant $S$-$\infty$-category at $L$. Suppose that for all $V \in S$, $C_V$ admits $L$-indexed colimits, and for all $\alpha: V \to W$ in $T$, the restriction functor $\alpha^{\ast}: C_W \to C_V$ preserves $L$-indexed colimits. Then a $S$-colimit of $F: L \times S \to C$ exists and is computed fiberwise as the colimit of $F_s: L \to C_s$.
\end{exm}

\begin{exm}[Corepresentable $S$-diagrams {\cite[\S 5.9]{Exp2}} {\cite[\S 4]{Exp4}}] Suppose that $T$ admits multipullbacks, i.e., the finite coproduct completion $\FF_T$ of $T$ admits pullbacks. For example, $\sO_G$ satisfies this condition. We call $\FF_T$ the $\infty$-category of \emph{finite $T$-sets} and $T \subset \FF_T$ the \emph{orbits}. For $U \in \FF_T$, let $\ul{U} \to S$ be the corresponding $S$-$\infty$-category of points, i.e., $\ul{U} \simeq \coprod_{i \in I} S^{U_i/}$ for an orbit decomposition $U \simeq \coprod_{i \in I} U_i$, and note that the assignment $\fromto{U}{\ul{U}}$ is covariant in morphisms in $\FF_T$. Let $\alpha: U \to V$ be a morphism in $\FF_T$ such that $V$ is an orbit. Let $x_i \in C_{U_i}$ be a set of objects for all $i \in I$ and write $(x_i): \ul{U} \to C_{\ul{V}}$ for the $S^{V/}$-functor determined by the $x_i$. Then the \emph{$S$-coproduct along $\alpha$} $$\coprod_{\alpha} x_i \in C_V$$ is defined to be the $S^{V/}$-colimit of $(x_i)$. A \emph{finite $S$-coproduct} is any $S^{V/}$-colimit of this form. We have that $C$ admits all finite $S$-coproducts if and only if the following conditions obtain \cite[Prop.~5.11]{Exp2}:
\begin{enumerate}
\item For all $V \in S$, $C_V$ admits finite coproducts, and for all morphisms $\alpha: V \to W$ in $T$, the restriction functor $\alpha^{\ast}: C_W \to C_{V}$ preserves finite coproducts.
\item For all morphisms $\alpha: V \to W$ in $T$, $\alpha^{\ast}$ admits a left adjoint $\alpha_!$.
\item Given $U \in \FF_T$ with orbit decomposition $\coprod_{i \in I} U_i$, let $C_U = \prod_{i \in I} C_{U_i}$ and extend $\alpha^{\ast}$ and $\alpha_!$ to be defined for all morphisms $\alpha$ in $\FF_T$ in the obvious way (e.g., if $\alpha: U \to V$ is a map with $V$ an orbit, then $\alpha_!(x_i) = \coprod_{i \in I}(\alpha_i)_!(x_i)$ for $\alpha_i: U_i \to V$ the restriction of $\alpha$ to $U_i$). Then the \emph{Beck-Chevalley conditions} hold: for every pullback square
\[ \begin{tikzcd}[row sep=4ex, column sep=4ex, text height=1.5ex, text depth=0.25ex]
U' \ar{r}{\alpha'} \ar{d}{\beta'} & V' \ar{d}{\beta} \\
U \ar{r}{\alpha} & V
\end{tikzcd} \]
in $\FF_T$, the canonical map $(\alpha')_! (\beta')^{\ast} \to \beta^{\ast} \alpha_!$ is an equivalence.
\end{enumerate}
In this case, the $S$-coproduct $\coprod_{\alpha} x_i$ above is computed by $\alpha_!(x_i)$.

Dually, $C$ admits all finite $S$-products if and only if the analogous conditions hold with respect to finite products in the fibers and right adjoints $\alpha_{\ast}$.
\end{exm}

\begin{dfn}[{\cite[\S 8]{Exp2}}] Let $C,D$ be $S$-$\infty$-categories and suppose that $\adjunct{F}{C}{D}{G}$ is a relative adjunction over $S$ \cite[\S 7.3.2]{HA}. Then we say that $F \dashv G$ is a \emph{$S$-adjunction} if $F$ and $G$ are $S$-functors.
\end{dfn}

\begin{nul}[{\cite[\S 9-10]{Exp2}}] Let $\pi: K \to S$ denote the structure map and consider the $S$-functor $\pi^{\ast}: C \to \ul{\Fun}_S(K,C)$. Then $C_{\ul{V}}$ admits all $K_{\ul{V}}$-indexed $S^{V/}$-colimits for all $V \in S$ if and only if $\pi^{\ast}$ admits a $S$-left adjoint $\pi_!$ \cite[Cor.~9.16]{Exp2}, so that we have an $S$-adjunction
\[ \adjunct{\pi_!}{\ul{\Fun}_S(K,C)}{C}{\pi^{\ast}}. \]
More generally, if $\pi: K \to L$ is a $S$-functor, we have an existence criterion for the $S$-left adjoint $\pi_!$ to $\pi^{\ast}$ in terms of a pointwise formula for the $S$-left Kan extension of a $S$-functor $F: K \to S$ along $\pi$ \cite[Thm.~10.3]{Exp2}. Dualizing, we have the same for the $S$-right adjoint $\pi_{\ast}$.
\end{nul}

\begin{nul} \label{Scocomplete} $C$ is $S$-cocomplete \cite[Def.~5.12]{Exp2} if and only if $C$ admits finite $S$-coproducts, each fiber $C_V$ admits geometric realizations, and the restriction functors $\alpha^{\ast}: C_W \to C_V$ preserve geometric realizations \cite[Cor.~12.15]{Exp2}. The proof uses a parametrized version of the Bousfield-Kan formula. In particular, the $G$-$\infty$-categories $\ul{\Spc}^G$ and $\ul{\Sp}^G$ are $G$-cocomplete, and dualizing the argument, also $G$-complete. Thus, the parametrized orbits and fixed points functors discussed in this paper exist.
\end{nul}

\begin{nul} Suppose that $S$ has an initial object $V_0$ and $C$ is $S$-cocomplete. Then we may compute the $S$-colimit of a $S$-functor $F: K \to C$ as a colimit over the total category of the dual cartesian fibration $K^{\vee} \to T$ \cite{BGN}.\footnote{The second author thanks Marc Hoyois for pointing out this observation to him -- also see \cite[Rmk.~16.5]{BachmannHoyoisNorms}.} In more detail, note that under our assumption, $C^{\vee} \to T$ is also a \emph{cocartesian} fibration, and let $P: C^{\vee} \to C^{\vee}_{V_0} \simeq C_{V_0}$ be the cocartesian pushforward to the fiber over the terminal object $V_0 \in T$. Let $$F_0^{\vee} = P \circ F^{\vee}: K^{\vee} \to C_{V_0}$$ be the composite and let $x = \colim F_0^{\vee}$. Then if we let $f: T \to C^{\vee}$ be the $p^{\vee}$-relative colimit of $F^{\vee}$ (where $p^{\vee}: C^{\vee} \to T$ is the structure map), $f(V_0) \simeq x$ by \cite[Prop.~4.3.1.10]{HTT} (which also shows existence of $f$). Therefore, if $\pi$ is the structure map of $K$, then $$\pi^{\ast}: \Fun_S(S,C) \simeq C_{V_0} \to \Fun_S(K,C) \simeq \Fun^{\cart}_{/T}(K^{\vee}, C^{\vee})$$ admits a left adjoint computed by $F \mapsto x$. The $S$-completeness of $C$ further implies that $\pi_!$ also computes the $S$-colimit, so $x \simeq \mr{colim}{}^S F$.

Dualizing, we may compute the $S$-limit of a $S$-functor $F: K \to C$ as a limit over the total category $K$, assuming that $C$ is $S$-complete. To do this, let $P': C \to C_{V_0}$ be the \emph{cartesian} pushforward to the fiber over the initial object,  let $F_0 = P' \circ F$, and let $y = \lim F_0$. Then by the same reasoning, we have that $y \simeq \mr{lim}{}^S F$.
\end{nul}

\begin{rem} Not assuming $S$-cocompleteness of $C$ to begin with, the above technique of reduction to the Grothendieck construction can also be used to give another proof of \cite[Cor.~12.15]{Exp2}, where we use the non-parametrized Bousfield-Kan formula for the existence and preservation of fiberwise colimits, and then the existence of finite $S$-coproducts for the requisite compatibility of the general $S$-colimit with restriction.
\end{rem}

\section{Pointwise monoidal structure}

In this appendix, we construct the `pointwise' monoidal structure on the $S$-functor $\infty$-category $\Fun_S(K,C)$, given a cocartesian $S$-family $C^{\otimes} \to S \times \Fin_{\ast}$ of symmetric monoidal $\infty$-categories. Let us first recall how to construct the non-parametrized pointwise monoidal structure on a functor $\infty$-category $\Fun(K,C)$.

\begin{nul} \label{pointwiseMonoidalStructure} Let $p: C^\otimes \to \Fin_{\ast}$ be an $\infty$-operad, and let $K$ be  a simplicial set. We have the cotensor $p^K: (C^\otimes)^K \to \Fin_{\ast}$ defined by $$\Hom_{/\Fin_{\ast}}(A,(C^\otimes)^K) \cong \Hom_{/\Fin_{\ast}}(A \times K,C^\otimes).$$ Then $p^K$ is again an $\infty$-operad: this follows from the observation that for any $\mathfrak{O}$-anodyne morphism $A \to B$ of preoperads (with $\mathfrak{O}$ the defining categorical pattern for the model structure on preoperads), $A \times K \to B \times K$ is again $\mathfrak{O}$-anodyne \cite[Prop.~B.1.9]{HA}. Moreover, if $p$ is in addition a cocartesian fibration, then $p^K$ is also a cocartesian fibration. The fiber of $p^K$ over $\angs{n}$ is $\Fun(K,C^{\times n}) \simeq \prod_{i=1}^n\Fun(K,C)$, and for the unique active map $\angs{n} \to \angs{1}$, if $\phi: C^{\times n} \to C$ is a choice of pushforward functor encoded by $p$, then the postcomposition by $\phi$ functor $\phi_{\ast}: \Fun(K,C^{\times n}) \to \Fun(K,C)$ is a choice of pushforward functor encoded by $p^K$. In other words, $p^K$ is the `pointwise' symmetric monoidal structure on $\Fun(K,C)$.
\end{nul}

\begin{lem} \label{lm:evaluationCocartesianMonoidal} Let $C^\otimes$ be a symmetric monoidal $\infty$-category. Then the functor
\[ e_L: (C^\otimes)^{K \star L} \to (C^\otimes)^L \]
induced by $L \subset K \star L$ is a cocartesian fibration of $\infty$-operads.
\end{lem}
\begin{proof} Because $e_L$ is induced by the monomorphism $L \subset K \star L$, $e_L$ is a fibration of $\infty$-operads. By \cite[Rmk.~7.3]{Exp2} and using the inert-active factorization system on an $\infty$-operad, it then suffices to prove the following two properties of $e_L$: 
\begin{enumerate} \item For every object $\angs{n} \in \Fin_{\ast}$, $(e_L)_{\angs{n}}$ is a cocartesian fibration;
\item For every active edge $\alpha: \angs{n} \to \angs{1}$ and commutative square in $(C^\otimes)^{K \star L}$
\[ \begin{tikzcd}[row sep=4ex, column sep=4ex, text height=1.5ex, text depth=0.25ex]
f = (f_1,...,f_n) \ar{r} \ar{d}{\theta} & f' = \otimes_{i=1}^n f_i \ar{d}{\theta'} \\
g = (g_1,...,g_n) \ar{r} & g' = \otimes_{i=1}^n g_i
\end{tikzcd} \]

with the horizontal edges as $p^{K \star L}$-cocartesian edges covering $\alpha$, if $\theta$ is $(e_L)_{\angs{n}}$-cocartesian then $\theta'$ is $(e_L)_{\angs{1}}$-cocartesian.
\end{enumerate} 
For (1), by \cite[Lem.~4.8]{Exp2} we have that $(e_L)_{\angs{n}}: \Fun(K \star L,C^{\times n}) \to \Fun(L, C^{\times n})$ is a cocartesian fibration. Moreover, $\theta: f \to g$ is a $(e_L)_{\angs{n}}$-cocartesian edge if and only if its image in $\Fun(K,C^{\times n})$ is an equivalence. This proves (2), since the $n$-fold tensor product of equivalences is always an equivalence.
\end{proof}

\begin{nul} We now elaborate \ref{pointwiseMonoidalStructure} to construct the pointwise monoidal structure on an $S$-functor category $\Fun_S(K,C)$ when $C$ is classified by a functor $S \to \CMon(\Cat_{\infty})$ valued in symmetric monoidal $\infty$-categories and symmetric monoidal functors thereof. In terms of fibrations, such functors correspond to \emph{cocartesian $S$-families of symmetric monoidal $\infty$-categories} $C^\otimes \to S \times \Fin_{\ast}$ \cite[Def.~4.8.3.1]{HA}.\footnote{More precisely, we have an equivalence of $\infty$-categories $(\Cat_{\infty})^{\cocart}_{/S \times \Fin_{\ast}} \simeq \Fun(S, \Fun(\Fin_{\ast}, \Cat_{\infty}))$ under which a cocartesian $S$-family of symmetric monoidal $\infty$-categories corresponds to a functor valued in commutative monoid objects.} Let $\mathfrak{P}_S$ be the categorical pattern $$(\All, \All, \{\lambda_{s,n} : (\angs{n}^{\circ})^{\lhd} \to \{s \} \times \Fin_{\ast} \subset S \times \Fin_{\ast} : s \in S \})$$ on $S \times \Fin_{\ast}$, where $\lambda_{s,n}$ is the usual map appearing in the definition of the model structure on preoperads that sends the cone point $v$ to $\angs{n}$, $i \in \angs{n}^{\circ}$ to $\angs{1}$, and the unique morphism $v \to i$ to the inert morphism $\rho^i: \angs{n} \to \angs{1}$ in $\Fin_{\ast}$ that selects $i \in \angs{n}^{\circ}$. Then cocartesian $S$-families of symmetric monoidal $\infty$-categories are by definition $\mathfrak{P}$-fibered \cite[Def.~B.0.19]{HA} and hence are the fibrant objects for the model structure on $s\Set^+_{/S \times \Fin_{\ast}}$ defined by $\mathfrak{P}$ \cite[Thm.~B.0.20]{HA}. 
\end{nul}

\begin{dfn} \label{dfn:S-PointwiseMonoidal} Suppose $C^{\otimes} \to S \times \Fin_{\ast}$ is a cocartesian $S$-family of symmetric monoidal $\infty$-categories and $q: K \to S$ is an $S$-$\infty$-category. Consider the span of marked simplicial sets
\[ \begin{tikzcd}[row sep=4ex, column sep=4ex, text height=1.5ex, text depth=0.25ex]
(\Fin_{\ast})^{\sharp} & \leftnat{K} \times (\Fin_{\ast})^{\sharp} \ar{r}{q \times \id} \ar{l}[swap]{\pr} & S^{\sharp} \times (\Fin_{\ast})^{\sharp}.
\end{tikzcd} \]
Define the \emph{pointwise monoidal structure} on $\Fun_S(K,C)$ to be
\[ \Fun_S(K,C)^{\otimes} \coloneq \pr_{\ast} (q \times \id)^{\ast} (\leftnat{C}^{\otimes}) \]
regarded as a simplicial set over $\Fin_{\ast}$.
\end{dfn}

Note that the fiber of $\Fun_S(K,C)^{\otimes} \to \Fin_{\ast}$ over $\angs{1}$ is $\Fun_S(K,C)$.

\begin{lem} \label{lm:ShowingPointwiseMonoidal} With respect to the categorical patterns $\mathfrak{P} = \mathfrak{P}_{\ast}$ on $\Fin_{\ast}$ and $\mathfrak{P}_S$ on $S \times \Fin_{\ast}$, the span of marked simplicial sets in Def.~\ref{dfn:S-PointwiseMonoidal} satisfies the hypotheses of \cite[Thm.~B.4.2]{HA}, so $\Fun_S(K,C)^{\otimes}$ is a symmetric monoidal $\infty$-category.
\end{lem} 
\begin{proof} The projection map $\pr$ is both a cartesian and cocartesian fibrations where an edge $e$ is (co)cartesian if and only if its projection to $K$ is an equivalence. Using also the basic stability property of cocartesian edges in $K$ \cite[Lem.~2.4.2.7]{HTT}, it is then easy to verify conditions (1)-(8) of \cite[Thm.~B.4.2]{HA}. By \cite[Thm.~B.4.2]{HA}, $\pr_{\ast} q^{\ast}: s\Set^+_{/\mathfrak{P}_S} \to s\Set^+_{/\mathfrak{P}}$ is right Quillen, which shows that $\Fun_S(K,C)^{\otimes}$ is a fibrant object in $s\Set^+_{/\mathfrak{P}}$ and hence a symmetric monoidal $\infty$-category.
\end{proof}

\begin{rem} An $S$-functor $f: L \to K$ yields a morphism of spans
\[ \begin{tikzcd}[row sep=4ex, column sep=4ex, text height=1.5ex, text depth=0.25ex]
& \leftnat{L} \times (\Fin_{\ast})^{\sharp} \ar{rd} \ar{ld} \ar{d}{f \times \id} & \\
(\Fin_{\ast})^{\sharp} & \leftnat{K} \times (\Fin_{\ast})^{\sharp} \ar{r} \ar{l} & S^{\sharp} \times (\Fin_{\ast})^{\sharp}
\end{tikzcd} \]
and therefore induces a map $f^{\ast}: \leftnat{\Fun_S(K,C)^{\otimes}} \to \leftnat{\Fun_S(L,C)^{\otimes}}$ of marked simplicial sets over $\Fin_{\ast}$. In other words, restriction along $f$ is a symmetric monoidal functor. 
\end{rem}

\begin{vrn} Consistent with the monoidality of restriction, the hypotheses of \cite[Thm.~B.4.2]{HA} also apply to the span 
\[ \begin{tikzcd}[row sep=4ex, column sep=4ex, text height=1.5ex, text depth=0.25ex]
S^{\sharp} \times (\Fin_{\ast})^{\sharp} & \leftnat{K} \times (\Fin_{\ast})^{\sharp} \ar{r}{q \times \id} \ar{l}[swap]{q \times \id} & S^{\sharp} \times (\Fin_{\ast})^{\sharp}.
\end{tikzcd} \]
We then define $$\underline{\Fun}_S(K,C)^{\otimes} \coloneq (q \times \id)_{\ast} (q \times \id)^{\ast} (\leftnat{C}^{\otimes})$$ as a pointwise monoidal enhancement of $\underline{\Fun}_S(K,C)$.
\end{vrn}

\bibliographystyle{amsalpha}
\bibliography{Gcats}

\end{document}